\documentclass{article}
\usepackage{amsmath,amssymb,amsthm}
\usepackage[margin=1in]{geometry}
\usepackage{xcolor}
\usepackage{booktabs}
\usepackage{tikz}

\usetikzlibrary{arrows.meta,positioning}
\usepackage[hidelinks]{hyperref}
\usepackage{mathpazo}

\newtheorem{theorem}{Theorem}[section]
\newtheorem{proposition}[theorem]{Proposition}
\newtheorem{lemma}[theorem]{Lemma}
\newtheorem{corollary}[theorem]{Corollary}
\newtheorem{assumption}[theorem]{Assumption}
\newtheorem{definition}[theorem]{Definition}
\newtheorem{remark}[theorem]{Remark}
\newtheorem{example}[theorem]{Example}

\newcommand{\bR}{\mathbb{R}}

\newcommand{\bT}{\mathbb{T}}
\newcommand{\Lip}{\operatorname{Lip}}

\newcommand{\Vnl}{\mathcal V_{\mathrm{nl}}}
\newcommand{\dep}{\operatorname{dep}}

\title{Neural Network Approximation and Structural \\
Complexity of Semilinear PDEs with Lower Bounds}
\author{Hayden Schaeffer}

\begin{document}

\date{}

\maketitle
\begin{abstract}
Neural networks are commonly used to learn the solution operators of PDEs such as the Navier-Stokes equations and viscous conservation laws. On standard benchmarks, moderately sized and often shallow networks achieve small errors that decay at algebraic rates with the network size. These observations are not well described by the current approximation theory, which treats the solution operator as a generic map and derives worst-case bounds over large classes of operators. We show that the structure of the equations and their properties explain the observed behavior. For semilinear parabolic equations, the approximation costs separate into two terms, where one is controlled by an invariant of the governing law and the other by the data. The constitutive law can be parameterized as a single block whose size for accuracy $\eta$ is $O(\eta^{-p\chi}\log\eta^{-1})$ for every $p>p_f$, where the structural exponent $p_f$ only depends on the computation graph of $f$. This exponent does not depend on the dimension and is in fact zero for polynomial laws, so for the two-dimensional Navier-Stokes equations the nonlinear block costs only $O(\log\eta^{-1})$ parameters. The data enters through encoding layers with a bounded number of learned weights per mode, so neither term grows with the grid resolution or the time step. While the analysis only uses the max-norm error of the pointwise block, we show that one obtains approximation bounds in a stronger norm along the whole trajectory. For viscous Hamilton-Jacobi equations with a $C^s$ Hamiltonian, the approximation requires on the order of $\eta^{-1/s}$ parameters, and we prove a matching lower bound, up to logarithms. This shows that solution operators for this class of problems have low structural complexity, even though they are nonlinear maps between infinite-dimensional spaces.
\end{abstract}

\section{Introduction}
\label{sec:intro}

Let $\Omega\subset\bR^{d_x}$ be a bounded spatial domain and consider the semilinear evolution equation
\begin{equation}
u_t+Lu=f(u,\nabla u),
\qquad
u(0)=u_0,
\label{eqn:intro_pde}
\end{equation}
where $L$ is a linear differential operator and $f$ the nonlinear constitutive law depending on the state and its derivative. This class of partial differential equations (PDE) covers a wide range of important dynamical models in science and engineering, for example, viscous Burgers and viscous conservation laws, viscous Hamilton--Jacobi equations, reaction-diffusion systems e.g. Allen-Cahn and the diffusive FitzHugh-Nagumo system, Cahn--Hilliard, and the two-dimensional incompressible Navier-Stokes equations \cite{henry2006geometric,evans2022partial, pazy2012semigroups, temam2024navier,temam2012infinite}. In scientific machine learning (SciML), these systems and their associated solution operators have become benchmark examples to test architectures and algorithms \cite{li2020fourier,lu2021learning,bhattacharya2021model,kovachki2023neural,liu2024prose,sun2025towards,herde2024poseidon,hao2024dpot}.

The goal of these methods is to approximate the solution operator $G_T:u_0\mapsto u$, which maps the initial data either to the solution at the final time $t=T$ \cite{li2020fourier,kovachki2021universal,chandramoorthy2026spectral} or to the full trajectory on $[0,T]$ \cite{liu2024prose,liu2024prosefd, sun2025towards, furuya2025quantitative,taniguchi2026generalization,liu2025bcat}. We adopt the latter formulation throughout this work, which is a stronger requirement. Since these equations serve as test cases throughout the literature for learning solution operators with neural networks, we study the structural complexity of approximating their solution operators \cite{li2020fourier,lu2021learning,kovachki2023neural,liu2024prose}. In particular, we quantify how many trainable parameters are needed to approximate the solution operator to accuracy $\eta$, in the sense of the parametric complexity for operator learning \cite{lanthaler2026parametric}. Our aim is to analyze these systems in a way that allows for a nuanced separation between the quantities that control the cost of approximation. This error leads to terms that depend on the complexity of the nonlinear governing law, the derived regularity from the dynamics, the spatial and temporal resolution, the encoding of the input, and the contraction of the iterates.
We show that the first two terms require trainable parameters, and that their total number does not grow with the grid size or resolution. The remaining error terms can affect the size of the network; however, they do not change the number of parameters that must be learned. In this way, we show that the solution operators of semilinear PDEs have low \emph{structural complexity}, which provides a theoretical explanation for why learned neural approximations perform well on them, even though they are seemingly complicated nonlinear operators between infinite-dimensional spaces.

In practice neural operators \cite{lu2021learning,li2020fourier,kovachki2023neural,bhattacharya2021model,zhang2023belnet} utilize an encoder to map the inputs into a latent space. For Fourier Neural Operators (FNO), this is the spectral domain given by the fast Fourier transform (FFT), and for PCA-based networks, it is the span of the leading principal components of the data generated by a PCA layer \cite{li2020fourier,bhattacharya2021model}. Most neural operators utilize linear layers (affine or convolutional) and point-wise nonlinear activation functions applied with the same weights at every spatial location \cite{kovachki2023neural,li2020fourier,lu2021learning,bhattacharya2021model}. In empirical studies, neural operators are more parameter efficient than what the theory of generic nonlinear operators shows. 
Specifically, in operator learning theory the rates have a curse of parametric complexity \cite{lanthaler2026parametric, liu2024neural} which yields approximation and sampling bounds that are exponentials of power laws. However, simulations show that the errors can decay at algebraic rates, for example in the size of the training set \cite{lu2021learning, de2022cost, subramanian2023towards}, and can be efficient over larger multi-operator settings \cite{liu2024prose, liu2025bcat, mccabe2024multiple,herde2024poseidon,hao2024dpot, cao2024vicon}. 

One theoretical limitation is that the universal approximation results \cite{chen1995universal, weihs2026deep} and the quantitative complexity results for the class of Lipschitz or smooth operators  \cite{lanthaler2026parametric,liu2024neural, weihs2026multiple} take the solution operator $G_T$ to be a generic map between Banach spaces. Since these operators are only characterized by their Lipschitz or $C^r$ regularity on infinite-dimensional spaces, the metric entropy grows like $\exp(\eta^{-p})$, and the associated cost for standard architectures is the same. In particular, this leads to exponential scalings for the parameter count \cite{lanthaler2026parametric,lanthaler2024operator, liu2024neural} and generalization rates that are only logarithmic in terms of the number of samples \cite{kovachki2024data,adcock2025towards, liu2024neural, weihs2026generalization,weihs2026multiple}. To get algebraic approximation rates for Lipschitz operators, one needs non-standard architectures such as super-expressive activation functions or non-standard ReLU architectures \cite{yarotsky2021elementary, schwab2026deep} which is consistent with the metric-entropy bound for encoding information \cite{lanthaler2024operator, liu2024neural}.

Utilizing the structure of the operator or data can help to avoid the curse of parametric complexity. One approach is to assume holomorphic dependence on the input, which yields algebraic rates \cite{schwab2019deep,opschoor2022exponential,herrmann2024neural,marcati2023exponential,reinhardt2026statistical}.
For elliptic problems, the size of networks that can approximate the solution scales like the size of networks representing the coefficients or the nonlinearity \cite{marwah2023neural, marwah2021parametric}. Another direction is to relate the architecture to the PDE by showing that it can emulate a numerical scheme. Examples include the Picard iteration for scalar reaction laws \cite{furuya2025quantitative,taniguchi2026generalization, furuya2026approximation}, learned time-stepping operators \cite{chen2025error}, iterated one-step operators for convex monotone semigroups \cite{blessing2026neural}, and emulating stable spectral schemes \cite{chandramoorthy2026spectral}. For FNO that emulate stable spectral schemes, \cite{chandramoorthy2026spectral} showed that a polynomial nonlinearity only adds a logarithmic cost, while the number of retained Fourier modes determines the overall rate.

Compared with these works, we take a complementary approach that applies to laws that are nonlinear in $(u,\nabla u)$. We separate the parameter count into a term that is controlled by the governing equation and a term controlled by the derived regularity of the data. We show that both terms are necessary. We use three properties of semilinear parabolic PDEs that are absent for a generic Lipschitz operator:

\begin{enumerate}
\item \textbf{Fixed-point Structure:} On a short horizon, the solutions are the unique fixed points of the mild map
$$\Phi[u](t)=S(t)u_0+\int_0^tS(t-s)f(u(s),\nabla u(s))\,ds,$$
which is a contraction on a ball in $C([0,T];X^\alpha)$. Thus, a solution operator can be taken as the limit of an iteration whose only nonlinearity comes from $f$.

\item \textbf{Parabolic Smoothing:} The semigroup $S(t)=e^{-tL}$ is analytic and smoothing, so certain high frequencies are damped. Therefore, we can quantify the truncation and quadrature errors with explicit algebraic rates induced directly from the dynamics itself.

\item \textbf{Algebraic Parametric Complexity of the Governing Law:} The term $f$ can be written as a finite directed acyclic graph (DAG) made from elementary operations, and is independent of any discretization. The cost of approximating the governing law is determined by the structure of the DAG rather than by the number of variables of $f$ or the number of space-time points.
\end{enumerate}

Combining these properties allows a standard architecture to use one small point-wise neural network whose parameters are shared across grid points, time levels and iterations, and thus keeps the number of trainable parameters small.


\subsection{Main result}
\label{sec:main}

The analysis uses the following pipeline to chain the approximations (see Figure~\ref{fig:pipeline}). The Galerkin truncation, temporal quadrature, and spatial sampling are used to form a discrete Picard map that approximates the mild map. The governing law $f$ is represented by its DAG form and is replaced by a deep ReLU network which approximates each node in the DAG. To apply it over the grid, weight sharing through a $1\times1$ convolution (Lemma~\ref{lem:sharing}) is used. This results in a looped structure that mimics a fixed-point method.

\begin{figure}[ht]
\centering
\begin{tikzpicture}[>=Latex, font=\small,
  bx/.style={draw, rounded corners, align=center, minimum height=1.05cm, inner sep=4pt, text width=2.25cm}]
\node[bx] (a) at (0,0) {mild map\\ $\Phi$};
\node[bx] (b) at (3.1,0) {Galerkin\\ $\Phi_N$};
\node[bx] (c) at (6.2,0) {time quadrature\\ $\Phi_{N,Q}$};
\node[bx] (d) at (9.3,0) {sampling\\ $\Phi^h_{N,Q}$};
\node[bx] (e) at (12.4,0) {shared network\\ $\Phi^h_{N,Q,\varepsilon}$};
\draw[->] (a) -- (b);
\draw[->] (b) -- (c);
\draw[->] (c) -- (d);
\draw[->] (d) -- (e);
\node[font=\footnotesize] at (1.55,-1.05) {$\lambda_{N+1}^{-(\gamma_{\mathrm s}-\alpha)/2}$};
\node[font=\footnotesize] at (4.65,-1.05) {$(\Delta t)^{\zeta_{\mathrm t}}$};
\node[font=\footnotesize] at (7.75,-1.05) {$h^{r}$};
\node[font=\footnotesize] at (10.85,-1.05) {$\varepsilon\,\xi_\alpha(N,\Delta t)$};
\node[font=\footnotesize, align=center] at (6.2,1.15)
  {Iteration error $\kappa^{n}$, with contraction constant $\kappa<1$ uniform in $N,h,Q$.};
\end{tikzpicture}
\caption{\textbf{Approximation Pipeline and Associated Errors.}
All steps are discretizations except the last one, which replaces the governing law by a shared ReLU network and the corresponding trainable parameters. The linear steps utilize fixed transforms, learned spectral weights, and parameter sharing. This leads to a loop, which is illustrated in Figure~\ref{fig:loop}.}
\label{fig:pipeline}
\end{figure}
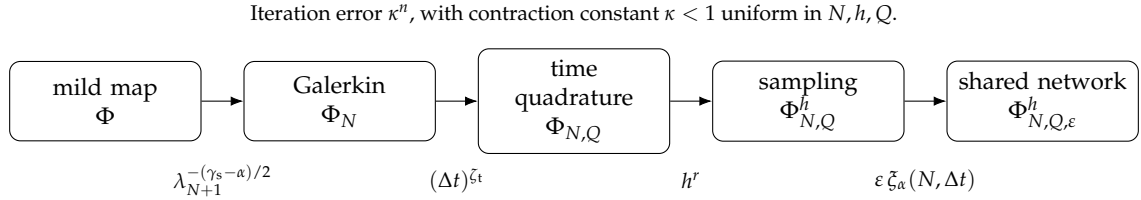

Let $X^s=D(L^{s/2})$ be the space generated by $L$, $\gamma_{\mathrm s}$ be the spatial regularity of the data, and $\zeta_{\mathrm t}\le1$ be the resulting temporal order of accuracy (see Section~\ref{sec:notation} for notation). The main complexity result is informally stated below and summarizes the rigorous results Theorems~\ref{thm:main} and \ref{thm:total_cost}.

\begin{theorem}[Main result, informal]
\label{thm:main_informal}
Let $u$ be the mild solution of \eqref{eqn:intro_pde} on $[0,T]$ with spatial regularity $\gamma_{\mathrm s}$. If $f$ can be written as a finite DAG whose nonlinear nodes are multilinear or component-wise $C^{s_j}$, with $p_{\max}=1/s_{\min}$, then for every sufficiently small $\eta>0$, the architecture from Definition~\ref{def:loop}, with
$$    \lambda_{N+1} \asymp \eta^{-2/(\gamma_{\mathrm s}-\alpha)},
\qquad
h \asymp \eta^{1/(\gamma_{\mathrm s} - \alpha)},
\qquad
  \Delta t \asymp \eta^{1 / \zeta_{\mathrm t}},
\qquad
 n \asymp  \log\eta^{-1},$$
admits a choice of trainable parameters for which $\|\widehat u-u\|_{C([0,T];X^\alpha)}\le\eta$ holds. Its trainable parameters scale as
$$\underbrace{W(\widehat f_\varepsilon)\ \le\ C\,\eta^{-p_{\max}\chi_{\mathrm s}}\log\eta^{-1}}_{\text{law}}
\qquad  +     \qquad
\underbrace{W_{\mathrm{spec}}\   \le    \ C\,\eta^{-d_x/(\gamma_{\mathrm s}-\alpha)}}_{\text{data}}, $$
and are independent of the number of grid points, time levels, and sweeps. The full network without parameter sharing scales like
$$W(\mathcal A_n)\ \le\ C\,\eta^{-\upsilon}\left(\log\eta^{-1}\right)^2,
\qquad \upsilon = \frac{d_x}{\gamma_{\mathrm s} - \alpha} + \frac{1}{\zeta_{\mathrm t}} + p_{\max} \chi_{\mathrm s},
\qquad \dep(\mathcal A_n)\ \le\ C\left(\log\eta^{-1}\right)^2,$$
where $\chi_{\mathrm s}\ge1$ equals $1$ for $\alpha<2$ and decreases to $1$ as the data becomes smoother, with an extra factor $(\log\eta^{-1})^{p_{\max}}$ when $\alpha=2$.

If every node is multilinear, then $\upsilon=\frac{d_x}{\gamma_{\mathrm s}-\alpha}+\frac{1}{\zeta_{\mathrm t}}$ and the complexity is set only by the space-time resolution.
\end{theorem}

The exponent $p_{\max}=1/s_{\min}$ is the classical ReLU rate \cite{yarotsky2017error,devore2021neural} applied to the least smooth nonlinear node in the DAG form of $f$. It is the only place where the smoothness of $f$ appears explicitly in the bounds. It is multiplied by the stability factor $\chi_{\mathrm s}$, which comes from the lack of smoothing by the Duhamel kernel when we measure the error in  $X^\alpha$, $\alpha>2$. This is sharp (see Proposition~\ref{prop:xi_sharp}), but under some assumptions it can be improved (see Section~\ref{sec:amplification}).

\subsection{Contributions}
\label{sec:contributions}
We list some of the contributions of this work. 

\begin{enumerate}
\item \textbf{Separation of the Parameter Count (Theorems~\ref{thm:main} and \ref{thm:total_cost}).} We split the number of trainable parameters needed to obtain an accuracy of $\eta$ in $C([0,T];X^\alpha)$ into a term controlled by the governing law, i.e. $C\eta^{-p\chi_{\mathrm s}}\log\eta^{-1}$, and a term controlled by the regularity of the data, i.e. $C\eta^{-d_x/(\gamma_{\mathrm s}-\alpha)}$.  The terms do not grow with the grid, the time step, or the number of iterations of the looped network due to parameter sharing.  These results also apply to popular operator learning approaches, for example, those based on FNO or DeepONet (see Remarks~\ref{rem:loop_fno} and \ref{rem:loop_deeponet}).

\item \textbf{The Structural Exponent (Section~\ref{sec:dag}).} For governing laws whose DAG depends on inputs $(u,\nabla u)$, the complexity is set by the least smooth nonlinear node (see Theorem~\ref{thm:relu_f}), building on \cite{schmidt2020nonparametric,poggio2024compositional}. Taking the infimum over all DAG forms of $f$ yields the structural exponent $p_f$ which is an invariant of the system and does not depend on $d$, $d_x$ or the number of state components. For polynomial systems, it is equal to zero.

\item \textbf{A Derived Error Decomposition (Section~\ref{sec:total}).} All of the sources of error and assumptions on the regularity are derived from the mild formulation. They are bounded separately rather than assumed through the accuracy of a particular scheme (Theorem~\ref{thm:total_error}). The error is measured in $C([0,T];X^\alpha)$ which is a stronger norm than is typically used \cite{furuya2025quantitative,taniguchi2026generalization,chandramoorthy2026spectral} but still achieves a scaling for the depth that is only $C(\log\eta^{-1})^2$ (see Theorem~\ref{thm:realization}), the same as in \cite{furuya2025quantitative,furuya2026approximation}. The decomposition is written in the interpretable form of
$$\underbrace{\kappa^{n}}_{\text{iteration}}+\underbrace{C\varepsilon\,\xi_\alpha(N,\Delta t)}_{\text{network}}+\underbrace{\varepsilon^{\mathrm{tail}}_N(u_0)+\varepsilon^{\mathrm{enc}}_{N,h}(u_0)}_{\text{encoding}}+\underbrace{C\lambda_{N+1}^{-(\gamma_{\mathrm s}-\alpha)/2}}_{\text{Galerkin}}+\underbrace{Ch^r}_{\text{spatial}}+\underbrace{C(\Delta t)^{\zeta_{\mathrm t}}}_{\text{temporal}}.$$
For $\alpha>2$ the network error is amplified, which multiplies its exponent by $\chi_{\mathrm s}$. This amplification is sharp by Proposition~\ref{prop:xi_sharp} in general, but can be removed by certain formulations for gradient-free laws when $d_x\le3$ (see Section~\ref{sec:amplification}).

\item \textbf{Lower Bounds (Section~\ref{sec:lower_bounds}).} 
Any architecture whose weights depend continuously on the governing laws needs at least $c\,\eta^{-1/(s+1)}$ parameters for gradient-free governing laws and at least $c(\eta\log\eta^{-1})^{-1/s}$ over a class of gradient dependent governing laws.  For gradient dependent laws and $s\ge2$, the lower bound rates match the upper bound rates up to logs (see Theorems~\ref{thm:operator_lower_bound} and \ref{thm:hj_lower_bound}). If the spectrum of $L$ is trainable, at least one weight mode is required, so the data term in our analysis is sharp (Proposition~\ref{prop:spectral_lower}).
\end{enumerate}

In addition, our results are extended to longer time intervals with the same network (Proposition~\ref{prop:continuation}). For gradient-free laws, the error is controlled by a one-sided rate rather than only using contraction (Lemma~\ref{lem:allen_cahn_window}). This also extends to the solution operator using just one set of weights for all data in some fixed ball (Theorem~\ref{thm:operator_rates}). We compute the exponents for seven commonly used PDEs and examine the constants in the Allen-Cahn case, where the block has a few hundred parameters on a single window and the trainable count is dominated by the spectral weights (Section~\ref{sec:examples}).

\subsection{Related work}
\label{sec:related}

There are several quantitative results for learning the solutions to PDEs using neural networks.  In \cite{lanthaler2022error}, error estimates are provided for DeepONet, and in \cite{liu2024neural}, approximation and generalization rates are provided over the class of Lipschitz operators.  Generic error bounds for operator learning that apply to nonlinear parabolic equations were shown in \cite{de2022generic}. Error bounds were derived for FNO in \cite{kovachki2021universal}, and upper and lower complexity bounds for PCA-Net can be found in \cite{lanthaler2023operator}. Approximation and generalization bounds for encoder-decoder architectures were analyzed in \cite{yang2026efficient}. In \cite{zhang2025discretization}, universal approximation was derived for a discretization-invariant extension of DeepONet, where the network has an additional component to resolve the grids, and general universal approximation results appeared in \cite{weihs2026deep}.  In \cite{lanthaler2022nonlinear}, they proved lower bounds showing that operator learning approaches that use a linear reconstruction step, e.g. DeepONet and PCA-Net, cannot efficiently approximate discontinuous solutions.
In \cite{furuya2024can}, they showed that bijective neural operators between infinite-dimensional Hilbert spaces may not admit continuous finite-dimensional discretizations.
In \cite{lanthaler2024discretization}, the authors examined the discretization error associated with evaluating an FNO on a grid and proved an algebraic rate with respect to the resolution.

As shown in \cite{lanthaler2026parametric}, the curse of parametric complexity over Lipschitz or $C^s$ classes leads to parameter counts of $\exp(\varepsilon^{-r})$ for operator learning, including DeepONet, PCA-Net, and FNO. However, using structures or other information can alleviate this restriction.  FNO can emulate pseudo-spectral methods for periodic domains using $C^3$ activation functions (not ReLU), and so can approximate the Darcy and Navier-Stokes solution operators with explicit bounds \cite{kovachki2021universal}. In \cite{lanthaler2022error}, they used the Cole-Hopf transformation to obtain explicit DeepONet bounds for viscous Burgers. In our results, the bound follows from the bilinear structure of the law. In \cite{lanthaler2023operator}, it is shown that PCA-Net can avoid the curse of parametric complexity for Darcy flow by using the structure of the parametrized coefficient and avoids it for the Navier-Stokes equations through emulation of a numerical method. Their results are measured in the averaged error. For the setting of self-adjoint linear operators, \cite{de2023convergence} showed that the generalization rates are algebraic and depend on the eigenvalue decay and data smoothness. In \cite{chandramoorthy2026spectral}, under the assumption that the solution operator has a stable and accurate Fourier spectral discretization, they proved that FNO learns an accurate approximation with polynomial sample complexity. They show that the complexity for the problem $u_t+\mathcal Au+\mathcal D_s(\mathcal G(u))=f$ on the torus, under mild conditions, leads to algebraic rates in $L^2$ and provide generalization bounds.

In \cite{furuya2025quantitative,taniguchi2026generalization}, neural operators were constructed using Picard iterations in order to approximate the solution operators of scalar reaction equations $u_t+ Lu=F(u)$. Within these networks, the nonlinear scalar $F$ is a Lipschitz function and can be approximated with order $\epsilon^{-1}$ neurons (up to logs), thus a polynomial and a Lipschitz nonlinearity scale the same. In \cite{furuya2025quantitative,taniguchi2026generalization}, the authors noted that these Picard layers share weights and that the networks' depth scales like $(\log\epsilon^{-1})^2$. In both of these papers, the linear term is approximated without a rate and discretization is not analyzed. In the recent work \cite{furuya2026approximation}, Picard layers with a truncated eigenbasis were utilized for reaction-diffusion systems, with explicit bounds in the number of retained modes and depth scaling $C(\log\epsilon^{-1})^2$.

Our work complements these prior results, since we consider governing equations that can be nonlinear in $\nabla u$ on domains with an explicit eigenbasis, and derive the explicit conditions for the numerical emulator. In addition, we show the error bound over the whole trajectory in stronger norms induced by the dynamics, with upper and lower rates. 

The architecture has loop structure due to the unrolling of fixed-point iteration, as seen in the learned iterative shrinkage \cite{gregor2010learning}, algorithm unrolling \cite{monga2021algorithm} and deep equilibrium models \cite{bai2019deep}, including weight-tied Fourier neural operators for steady-state PDEs \cite{marwah2023deep}. Banach fixed-point iterations unrolled into a loop were shown to be general operator networks in \cite{feischl2025neural}, where their depth is logarithmic in accuracy given an accurate approximation of the iteration steps. The loop structure also allows for parameter sharing across layers, utilized in neural ordinary differential equations \cite{chen2018neural,sun2020neupde}. In \cite{chen2025error}, they estimated the generalization error of networks that learn time-stepping operators of implicit schemes, and the authors of \cite{blessing2026neural} approximated convex monotone semigroups using a learned one-step operator. In both of these results, the learned time-stepper can be composed to generate predictions. Multilevel Picard methods and the associated neural network emulator can approximate single solutions of high-dimensional semilinear parabolic equations with only polynomially many parameters \cite{hutzenthaler2021multilevel,hutzenthaler2020proof,neufeld2024rectified}.  Recent work has extended the learning task to families of governing laws, either through in-context and multimodal operator learning \cite{cao2024vicon,liu2024prose}, multi-task operators and PDE foundation models \cite{sun2025towards,liu2024prosefd,liu2025bcat,mccabe2024multiple,herde2024poseidon,hao2024dpot}, and operators with heterogeneous inputs \cite{jin2022mionet,zhang2024d2no}. The analysis of multi-task operator learning, including approximation and generalization was studied in \cite{weihs2026deep,weihs2026generalization,weihs2026multiple}. Specifically, in \cite{weihs2026multiple},  the theoretical rates for learning families of operators are equivalent to those of learning one operator with the same Lipschitz regularity. In \cite{weihs2026kernel}, the rate for multi-task kernel learning is controlled by the hardest constituent task, which is similar to the least regular node in our analysis. 

\subsection{Organization and Notation}
\label{sec:notation}

Section~\ref{sec:spaces} introduces the function spaces, the mild map, and the systems' derived regularity. Section~\ref{sec:dag} provides the DAG approximation and the structural exponent used through the main results. In Section~\ref{sec:discrete}, we analyze the discretization and the neural Picard map. The main results for one solution are in Section~\ref{sec:total}, which shows the total error, the complexity theorem, and the size of the full network, and extend to  operator learning  in Section~\ref{sec:operator}. Sharpness and lower bounds are derived in Section~\ref{sec:lower_bounds}. Lastly, we show that the theory applies to a range of examples in Section~\ref{sec:examples}. A collection of proofs is deferred to Appendices~\ref{sec:deferred} and \ref{sec:appendix}.

We use the following notation throughout this work. First, $H$ is a real Hilbert space, $L$ is as above with eigenpairs $(\lambda_k,e_k)$, and define the spaces $X^s=D(L^{s/2})$, $X_T^s=C([0,T];X^s)$. The exponents satisfy $0\le\beta<\alpha<\gamma$ with $\alpha-\beta<2$ and $\gamma-\beta<2$, and we set
$$ \sigma:=\frac{\alpha-\beta}{2},
\qquad    \theta:=1-\sigma\in(0,1),
\qquad   \zeta:=\frac{\gamma-\alpha}{2},
\qquad   \bar\sigma:=\frac{\gamma-\beta}{2},
\qquad  \chi:=1+\frac{(\alpha-2)_+}{2\zeta},$$
so $\sigma+\zeta=\bar\sigma$ and $\zeta<\theta$. The term $\chi$ is the amplification of the network error by the discrete Duhamel kernel (see Theorem~\ref{thm:neural_consistency}). It is equal to  $1$ when $\alpha \le 2$ and equals $\frac{\gamma-2}{\gamma-\alpha}$ for $\alpha>2$. The state variable's dimension is $d$, i.e. $u(t,x)\in\bR^d$, and the spatial dimension is denoted by $d_x$. Thus the gradient is $\nabla u(t,x)\in\bR^{d\times d_x}$. For a ReLU network $\mathcal N$,  we denote the depth by $\dep(\mathcal N)$ and the number of nonzero parameters by $W(\mathcal N)$.   The term $\mathcal D$ denotes the Duhamel operator and $D(L)$ the domain of $L$. The main symbols are summarized in Table~\ref{tab:notation}.

\begin{table}[ht]
\centering
\small
\setlength{\tabcolsep}{4pt}
\begin{tabular}{@{}ll@{\quad}ll@{}}
\toprule
Symbol & Def. & Symbol & Def.\\
\midrule
$L,\ \lambda_k,\ e_k$ & elliptic operator, eigenpairs & $N$ & retained Galerkin modes\\
$X^s=D(L^{s/2})$ & Hilbert spaces & $h,\ M(h)$ & mesh size, sample points\\
$X^s_T=C([0,T];X^s)$ & trajectory space & $Q,\ \Delta t=T/Q$ & time levels, step\\
$\beta<\alpha<\gamma$ & forcing, well-posedness, regularity & $n$ & Picard iterations\\
$\sigma=\frac{\alpha-\beta}{2}$, $\theta=1-\sigma$ & Duhamel improvement & $\varepsilon$ & local network tolerance\\
$\zeta=\frac{\gamma-\alpha}{2}$, $\bar\sigma=\sigma+\zeta$ & regularity improvement & $\eta$ & target accuracy\\
$\gamma_{\mathrm s}\ge\gamma$, $\zeta_{\mathrm s}=\frac{\gamma_{\mathrm s}-\alpha}{2}$, $\zeta_{\mathrm t}$ & spatial index, temporal order & $W_{\mathrm{spec}}$ & learned spectral weights\\
$\chi=1+\frac{(\alpha-2)_+}{2\zeta}$ & error amplification exponent & $\upsilon$ & full-network exponent\\
$\Phi,\ \mathcal D$ & mild map, Duhamel operator & $\widehat f_\varepsilon$ & shared local network\\
$\Psi^h_{N,Q},\ \Psi^h_{N,Q,\varepsilon}$ & discrete, neural Picard maps & $\xi_\alpha(N,\Delta t)$ & network-error amplification\\
$a,b,a_\infty$ & coefficient sequences & $\kappa_0,\ \kappa$ & contraction constants\\
$\mathcal L_{N,Q},\ \mathcal J_Q$ & piecewise-linear lifting & $\mathcal L^S_{N,Q}$ & semigroup reconstruction\\
$\delta$ & DAG node tolerance & $p_j,\ p_{\max},\ p_f$ & DAG node, graph, structural exponents\\
$\mathcal A_n$ & $n$-sweep network & $\Pi_{N,h}$ & discrete encoder\\
\bottomrule
\end{tabular}
\caption{Notation: The exponents $\sigma,\theta,\zeta,\bar\sigma$ are fixed by $\beta,\alpha,\gamma$ but are used throughout. We also use $\sigma$ to denote the ReLU activation, which should be clear from context.}
\label{tab:notation}
\end{table}

We chose to use $\eta$ for the target accuracy of the trajectory and reserve $\varepsilon$ for the accuracy of the point-wise block or local networks, which is smaller when $\alpha>2$.


\section{Preliminaries} \label{sec:spaces}

We first discuss the functional analytic results that will be used throughout the work, specifically, the Hilbert spaces generated by a positive self-adjoint operator $L$ \cite{lunardi2018interpolation, temam2012infinite}.

\begin{definition}[Variable Hilbert spaces] \label{def:variable_hilbert}
Let $H$ be a Hilbert space and let $L:D(L) \subset H\to H$
be a densely defined, self-adjoint, strictly positive operator with compact inverse. The eigenvalues of $L$ are denoted by
$$0 < \lambda_1 \leq \lambda_2 \leq \cdots,$$
where $\lambda_k\to\infty$, and the corresponding orthonormal eigenfunctions are denoted by
$\{e_k\}_{k=1}^\infty$. For $s\ge 0$, define the function spaces generated by $L$ by $X^s: = D(L^{s/2})$, equipped with the norm
$$\|u\|_{X^s} := \|L^{s/2}u\|_H =
\left( \sum_{k=1}^\infty \lambda_k^{s} |\langle u,e_k\rangle_H|^2 \right)^{1/2}. $$
For fixed indices $0 \leq \beta < \alpha < \gamma$ with $\alpha \ge 1$, the following three spaces,
 $X^\gamma$, $X^\alpha$, and $X^\beta$, denote the regularity space, the well-posedness space, and the space for the nonlinear governing equation, respectively. In particular, the spaces are continuously embedded,
$X^\gamma\hookrightarrow X^\alpha\hookrightarrow X^\beta$.
For $s\geq0$ and $T>0$, define $X_T^s:=C([0,T];X^s)$
equipped with the norm
$$\|u\|_{X_T^s} := \sup_{0\leq t\leq T} \|u(t)\|_{X^s}, $$
which preserves the same embedding order
$X_T^\gamma\hookrightarrow X_T^\alpha\hookrightarrow X_T^\beta.$
\end{definition}

For simplicity, we assume that $X^\alpha$ is smooth enough so that point-wise evaluation of $u$ and $\nabla u$ is well-defined. 
\begin{assumption}[$C^1$ Embedding] \label{assumption:regularity_embedding}
For the well-posedness index $\alpha$, assume that $X^\alpha \hookrightarrow C^1(\overline{\Omega}; \mathbb{R}^d)$
continuously, i.e., there exists a constant $C_{\mathrm{emb}}>0$ such that
$$\|u\|_{C^1(\overline{\Omega})}\leq C_{\mathrm{emb}} \|u\|_{X^\alpha}.$$
\end{assumption}
With this assumption, for any $u\in X^\alpha$, the inputs to the network are well-defined for standard computations. One can weaken this assumption at the cost of additional density and approximation terms. 

 For an explicit example, consider the space $H=L^2(\Omega)$ and the standard Laplacian $L=-\Delta$ on a bounded domain $\Omega \subset \bR^{d_x}$, with boundary conditions that force $L$ to be self-adjoint and strictly positive with compact inverse (e.g. Dirichlet BC with $D(L)=H^2(\Omega)\cap H^1_0(\Omega)$). Up to norm equivalence,  $X^s$ is the Sobolev space $H^s(\Omega)$ with the corresponding boundary conditions, and $X^\gamma\hookrightarrow X^\alpha\hookrightarrow X^\beta$ is the Sobolev embedding. Since $H^\alpha(\Omega)\hookrightarrow C^1(\overline\Omega)$ for $\alpha>1+\frac{d_x}2$ \cite{evans2022partial}, Assumption~\ref{assumption:regularity_embedding} holds. The same characterization holds if we extend $L$ to be a uniformly elliptic second-order operator \cite{lions1969quelques,lunardi2018interpolation}. We will keep the abstract formulation in terms of $L$, without committing to a domain, boundary condition, or eigenbasis.

For a semigroup $S(t)$, the Duhamel operator is defined by
$$ \mathcal{D}[g](t) := \int_0^t S(t-s)g(s)\,ds.$$
The following properties for the Duhamel operator lead to estimates and regularity bounds. 

\begin{assumption}[Analytic Semigroup Regularity]
\label{assumption:semigroup_regular}
The operator $L$ generates an analytic semigroup on $H$ denoted by
$S(t)=e^{-tL}$ for $t\geq 0$. The semigroup preserves the Hilbert spaces $\{X^s\}_{s\geq 0}$, i.e., for $s\geq 0$,
$$\|S(t)v\|_{X^s}\leq \|v\|_{X^s},$$
for $t\geq0$, and for $t>0$ it smooths in the sense
$$\|S(t)v\|_{X^r} \leq C_{r-s}\,t^{-(r-s)/2}\|v\|_{X^s},
$$
where $r\geq s$. Thus, if $0\leq s<r$ and $r-s<2$, then the
Duhamel operator satisfies
$$ \|\mathcal D[g]\|_{X_T^r} \leq C\,T^{1-(r-s)/2}\|g\|_{X_T^s},
$$
where $C>0$ is independent of $T$.
\end{assumption}

 Although stated as an assumption, any $L$ from Definition~\ref{def:variable_hilbert} will satisfy Assumption~\ref{assumption:semigroup_regular}. Taking the spectral decomposition, $S(t)u=\sum_k e^{-\lambda_kt}\langle u,e_k\rangle e_k$ is contractive on every $X^s$, and $$\sup_{\lambda>0}\lambda^{(r-s)/2}e^{-\lambda t}=C_{r-s}t^{-(r-s)/2}$$ where the constant is defined as $C_{r-s}=\sup_{x>0}x^{(r-s)/2}e^{-x}$. After integration, we get the Duhamel bound with $C=C_{r-s} \, (1-\frac{r-s}{2})^{-1}$.

We also need that $\mathcal{D}[g]\in C([0,T];X^r)$ for $g\in C([0,T];X^s)$ and $0\le r-s<2$, so that the mild map acts on $X^\alpha_T$.  To see this, note that since $g([0,T])$ is compact in $X^s$ and the projections $P_K$ (see Definition~\ref{def:spectral_galerkin}) converge strongly to the identity with norm bounded by 1, we have  $\sup_t\|(I-P_K)g(t)\|_{X^s}\to0$. The function $\mathcal{D}[P_Kg]$ is continuous and lives on a finite-dimensional space, so applying the Duhamel bound yields
$$\sup_t \|\mathcal{D}[(I-P_K)g](t)\|_{X^r} \le CT^{1-(r-s)/2}\sup_t\|(I-P_K)g(t)\|_{X^s}\to0,$$
for $t \in [0,T]$ and $T$ fixed.
Therefore,  the functions converge, i.e. $\mathcal{D}[P_Kg]\to\mathcal D[g]$ in the norm of $X^r_T$, and since $C([0,T];X^r)$ is complete, we have $\mathcal D[g]\in C([0,T];X^r)$. This is needed so that $\Phi$ is a map from $X_T^\alpha$ into itself in the Banach fixed point arguments. 

\subsection{Mild map and Fixed Points}

The approximation theory uses the fixed-point of the mild map (i.e., the mild solution) as the target function. The regularity and behavior of the mild solution can be directly characterized by the semigroup formulation for semilinear evolution equations \cite{pazy2012semigroups}. In particular, this covers popular examples including: viscous Burgers-type equations, reaction-diffusion systems, and the Navier-Stokes equations (under certain settings) \cite{temam2024navier}.  We use the mild map and the associated Picard iteration rather than a general time-stepping scheme, so that we can control the bounds via consistency and contraction only, with constants independent of the discretization. This is due to the fact that every time level is updated simultaneously which will lead to a logarithmic depth. 

\begin{definition}[Mild solution]
\label{def:mild}
Consider the initial value problem
\begin{equation}
u_t + Lu = f(u,\nabla u),\qquad u(0)=u_0.
\label{eqn:PDE}
\end{equation}
Let $S(t)=e^{-tL}$ denote the semigroup generated by the linear operator $-L$.  The mild map $\Phi:X^\alpha_T\to X^\alpha_T$ is defined as
$$ \Phi[u](t):=S(t)u_0 +\int_0^tS(t-s)f(u(s),\nabla u(s))\,ds.$$
for $0\leq t\leq T$. A function $u\in X^\alpha_T$ is called a mild solution of \eqref{eqn:PDE} if and only if it is a fixed point of the mild map $\Phi$, i.e. $u=\Phi[u]$.
\end{definition}

The existence and regularity of the solutions are a consequence of fixed-point analysis. 

\begin{proposition}[Contraction of the Mild Map]
\label{prop:contraction}
Let $S(t)$ be the analytic semigroup generated by $-L$ and $\mathcal{D}[\cdot]$ be the associated Duhamel operator. Assume that the nonlinear map $u\rightarrow f(u,\nabla u)$ maps $X_T^\alpha$ into $X_T^\beta$, where
$0 \leq \beta < \alpha$ and $\alpha-\beta<2$. Assume that for every $R>0$, there exists $L_R>0$ such that
$$
\|f(u,\nabla u)-f(v,\nabla v)\|_{X_T^\beta}
\leq
L_R\|u-v\|_{X_T^\alpha},
$$
for
$\|u\|_{X_T^\alpha},\|v\|_{X_T^\alpha}\leq R$. For simplicity assume that $f(0,0)=0$. Then, for every $u_0\in X^\alpha$, there exist constants $R>0$ and $T_0>0$ such that, for every $T\in(0,T_0]$, the following hold
\begin{enumerate}
\item The ball is invariant under $\Phi$, i.e.,
$\Phi(B_R^\alpha)\subseteq B_R^\alpha$, where
$B_R^\alpha:=\left\{ u\in X_T^\alpha: \|u\|_{X_T^\alpha}\leq R \right\}. $
\item The map is contractive with rate
$\kappa_0=CT^{1-(\alpha-\beta)/2} L_R\leq\frac{1}{2}$,
i.e., $$ \|\Phi[u]-\Phi[v]\|_{X_T^\alpha} \leq \kappa_0\|u-v\|_{X_T^\alpha} $$
for all $u,v\in B_R^\alpha$.
\item The initial value problem has a unique mild solution on $[0,T]$ in $B_R^\alpha$.
\end{enumerate}
\end{proposition}

The proof is a standard Banach fixed-point argument and appears in Appendix~\ref{sec:app_spaces}. The assumption $f(0,0)=0$ is chosen for simplicity and is
not essential for the proofs. It is sufficient to assume a growth condition. For example, for every $R>0$, there exists a constant $C_R>0$ such that
$$ \|f(u,\nabla u)\|_{X_T^\beta} \leq C_R \left( 1 + \|u\|_{X_T^\alpha} \right)$$
for all $\|u\|_{X_T^\alpha}\leq R$. Then the invariance proof would follow from
$$\|\Phi[u]\|_{X_T^\alpha}
\leq \|u_0\|_{X^\alpha} + C T^{1-(\alpha-\beta)/2} C_R(1+R),$$
and $T>0$ may be chosen sufficiently small to ensure that
$\Phi(B_R^\alpha)\subseteq B_R^\alpha$. The contraction results would not change since it only depends on the local Lipschitz continuity of the nonlinear map $f$. Thus, we could keep the $f(0,0)$ in the bounds:
$$\|f(u,\nabla u)\|_{X_T^\beta}
\leq \|f(0,0)\|_{X_T^\beta} + L_R\|u\|_{X_T^\alpha},$$
and the same argument holds.

We will repeatedly use the implication that, since $f$ acts point-wise in time, the condition on $X^\beta_T$ is the same as 
\begin{equation}
\left\|f(v,\nabla v)-f(w,\nabla w)\right\|_{X^\beta}\le L_R\|v-w\|_{X^\alpha}
\qquad\text{for all }v,w\in X^\alpha \text{ with }\|v\|_{X^\alpha},\|w\|_{X^\alpha} \le R .
\label{eqn:point-wise_lipschitz}
\end{equation}
The equation \eqref{eqn:point-wise_lipschitz} is implied directly from the $X_T$ form when applied to the constant trajectories $\widetilde v\equiv v$ and $\widetilde w\equiv w$  in $B_R^\alpha$. To show the converse, the $X_T$ formulation follows from \eqref{eqn:point-wise_lipschitz} after taking a supremum over $t$. This will be used without reference in several results.

Under certain conditions on the spaces, the solutions have the standard parabolic smoothing property for $t>0$ \cite{pazy2012semigroups, lunardi2018interpolation}. The bounds degrade as $t$ approaches zero.

\begin{proposition}[Regularity]
\label{prop:instantaneous_regularity}
Let $u\in X_T^\alpha$ be the mild solution of \eqref{eqn:PDE} from Proposition~\ref{prop:contraction}, where $u_0\in X^\alpha$. Assume that for some $\gamma>\alpha$, we have $\gamma-\beta<2$. Then, for every $t\in(0,T]$, the mild solution is in
$X^\gamma$ with
$$
\|u(t)\|_{X^\gamma}
\leq
C t^{-(\gamma-\alpha)/2}\|u_0\|_{X^\alpha}
+
C t^{1-(\gamma-\beta)/2}
\|f(u,\nabla u)\|_{X_T^\beta},
$$
where $C>0$ is independent of $t$. In particular,
$u\in C((0,T];X^\gamma)$.
\end{proposition}

\noindent The proof is given in Appendix~\ref{sec:app_spaces}.

When approximating the system by the fully discrete form, we will need the following bounds and smoothness, which only require that $u_0\in X^\alpha$.

\begin{lemma}[Derived Regularity]
\label{lem:derived_regularity}
Let $u_0\in X^\alpha$ with $\|u_0\|_{X^\alpha}\le R$, let $u\in B^\alpha_R$ be the mild solution of \eqref{eqn:PDE} on $[0,T]$ obtained by Proposition~\ref{prop:contraction}, and set
$$g(t):=f\left(u(t),\nabla u(t)\right),
\qquad
F_R:=\sup_{v\in B^\alpha_R}\left\|f(v,\nabla v)\right\|_{X^\beta_T}.$$
Recall the exponents: $\sigma=\frac{\alpha-\beta}{2}$, $\theta=1-\sigma$, $\zeta=\frac{\gamma-\alpha}{2}$ and $\bar\sigma=\sigma+\zeta=\frac{\gamma-\beta}{2}<1$. Then the following holds:
\begin{enumerate}
\item $\displaystyle\sup_{0\le t\le T} \|g(t)\|_{X^\beta} \le F_R$;
\item $\displaystyle\sup_{0<t\le T} t^{\zeta} \|u(t)\|_{X^\gamma}\ \le \ R_\gamma := C\left(R+T^{\theta}F_R\right)$;
\item for all $0<s<t\le T$,
$$\left\|u(t)-u(s)\right\|_{X^\alpha}  \le  C (t-s)^{\zeta} s^{-\zeta} \left(R_\gamma+T^{\theta}F_R\right),$$
and thus using the weighted H\"older seminorm
\[
[g]_{\zeta}:=\sup_{0<s<t\le T}\frac{\left\|g(t)-g(s)\right\|_{X^\beta}}{(t-s)^{\zeta}s^{-\zeta}},
\]
we have $[g]_\zeta \le C L_R \left(R_\gamma+T^{\theta}F_R\right)<\infty$.
\end{enumerate}
The constants in each bound only depend on the differences $\alpha-\beta$ and $\gamma-\alpha$, and the constants of Assumption~\ref{assumption:semigroup_regular}.
\end{lemma}

The proof uses the regularity results in Proposition~\ref{prop:instantaneous_regularity} with an intermediate time point $s$ and is derived in Appendix~\ref{sec:app_spaces}.

Lemma~\ref{lem:derived_regularity} only uses that $u_0\in X^\alpha$ and by construction has the restriction $\gamma<\beta+2$. This is a consequence of the fact that the Duhamel operator smooths by fewer than two derivatives. For smoother data, this restriction is an artifact since the forcing inherits the regularity of the solution, and the smoothing can be iterated upon. We use a bootstrap argument to show this.

\begin{assumption}[Higher Regularity]
\label{assumption:bootstrap}
There is a $\gamma'\ge\gamma$ such that for every $R>0$ there is $K_R$ with
$$\left \|f(v,\nabla v)\right\|_{X^{s-(\alpha-\beta)}} \le K_R \left(1+\|v\|_{X^s}\right), $$
for all $s \in[\alpha,\gamma']$ and $v\in X^s$ with $\|v\|_{X^\alpha} \le R$.
\end{assumption}

As an example, on a torus with $X^s=H^s$, this should hold for every governing equation that is smooth in $(u,\nabla u)$, by applying the Moser estimates for compositions \cite{bahouri2011fourier,runst2011sobolev}. In the setting where the governing equation has sufficient but finite smoothness, this will hold for some $\gamma'$ up to a limit defined by the governing equation's smoothness.

\begin{lemma}[Regularity Bootstrap]
\label{lem:bootstrap}
If Assumption~\ref{assumption:bootstrap} holds, $u_0\in X^{\gamma'}$ with $\|u_0\|_{X^\alpha} \le R$, and $u\in B^\alpha_R$ is the mild solution on $[0,T]$, then
$$\sup_{0\le t\le T} \|u(t)\|_{X^{\gamma'}} \le R_{\gamma'}:=C\left(1+\|u_0\|_{X^{\gamma'}} \right),$$
and
$$\sup_{0\le t\le T} \left \|f(u(t),\nabla u(t)) \right\|_{X^{\gamma'-(\alpha-\beta)}} \le K_R \left( 1 + R_{\gamma'} \right),$$
where $C$ depends only on $R$, $T$, $K_R$, $\gamma'-\alpha$, $\alpha-\beta$, $\lambda_1$, and the constants of Assumption~\ref{assumption:semigroup_regular}.
\end{lemma}

The proof iterates the Duhamel estimate over steps of size $\theta$ (see Appendix~\ref{sec:app_spaces}).

The lemma lets us extend the spatial regularity $\gamma_{\mathrm s}$ up to an amount determined by the governing equation. This is used for measuring the tail of the Galerkin projection and the spatial quadrature. For notation, we set $\gamma_{\mathrm s}:=\gamma'$ to be the exponent in Assumption~\ref{assumption:bootstrap} when $u_0\in X^{\gamma'}$, and $\gamma_{\mathrm s}:=\gamma$ otherwise. Similarly, we denote $\zeta_{\mathrm s}:=\frac{\gamma_{\mathrm s}-\alpha}{2}\ge\zeta$. The order in time is set by the H\"older regularity of the force as a function in time, and Lemma~\ref{lem:derived_regularity} implies the exponent is $\zeta$, with the restriction $\gamma<\beta+2$. This is extended to $\min\{1,\zeta_{\mathrm s}\}$ by Lemma~\ref{lem:bootstrap_time}, which can be at most first order.  We write $\zeta_{\mathrm t}:=\min\{1,\zeta_{\mathrm s}\}$ when $\gamma_{\mathrm s}=\gamma'>\gamma$, and $\zeta_{\mathrm t}:=\zeta$ otherwise. Note that in both cases, the powers have the following ordering $\zeta\le\zeta_{\mathrm t}\le\zeta_{\mathrm s}$.

\begin{lemma}[Time Bootstrap]
\label{lem:bootstrap_time}
If the hypotheses of Lemma~\ref{lem:bootstrap} hold,  $\zeta':=\frac{\gamma'-\alpha}{2}$, and $\kappa:=\min\{1,\zeta'\}$, then for $0\le s<t\le T$

$$\|u(t)-u(s)\|_{X^\alpha}\ \le\ C_{\mathrm t}\,(t-s)^{\kappa},$$
and
$$\left\|f(u(t),\nabla u(t))-f(u(s),\nabla u(s))\right\|_{X^\beta}\ \le\ L_R\,C_{\mathrm t}\,(t-s)^{\kappa},$$
with $C_{\mathrm t}$ depending only on $R_{\gamma'}$, $K_R$, $T$, $\lambda_1$, $\alpha-\beta$, $\kappa$ and the constants of Assumption~\ref{assumption:semigroup_regular}. 
\end{lemma}
The bound holds for all $0\le s<t\le T$, including $s=0$, and the proof is given in Appendix~\ref{sec:app_spaces}.

The derived regularity from Lemma~\ref{lem:derived_regularity} only assumed $u_0\in X^\alpha$ so the additional smoothness comes from parabolic smoothing.  Smoother data improves the regularity (up to 1) in time through Lemma~\ref{lem:bootstrap_time}, which is the limit for our temporal quadrature arguments.  For analytic semigroups, the weighted bound is the classical regularity of the mild solution \cite{lunardi2018interpolation}. In (iii), we utilize prior works to make it a modulus of continuity for the force term in order to improve the temporal quadrature rates from Section~\ref{sec:discrete}.

We will need to estimate the error incurred when projecting the continuous solution into a finite-dimensional space. One projection is onto the subspace spanned by the first $N$ eigenfunctions of the linear operator $L$ via a spectral Galerkin approximation \cite{lord2014introduction}.

\begin{definition}[Galerkin projection associated with $L$]
\label{def:spectral_galerkin}
Let $H$, $L$, $\{e_k\}_{k=1}^\infty$, and $\{\lambda_k\}_{k=1}^\infty$ be defined from Definition~\ref{def:variable_hilbert}. For $N\in\mathbb N$, define the finite-dimensional spectral subspace
$E_N:=\text{span}\{e_1,\dots,e_N\}\subset H,$
and let $P_N:H\to E_N$ be the orthogonal projection defined as
$$ P_N u:=\sum_{k=1}^N \langle u,e_k\rangle_H\,e_k. $$
Equivalently, if we define
$\mathcal C_N:H\to\mathbb R^N$, $\mathcal C_N u:= \left(\langle u,e_1\rangle_H,\dots,\langle u,e_N\rangle_H\right)$,
and
$\mathcal S_N:\mathbb R^N\to E_N$, $\mathcal S_N a := \sum_{k=1}^N a_k e_k$,
then the projection is
$P_N=\mathcal S_N\circ \mathcal C_N$.

For $s\geq0$, the restriction of the projection $P_N$ on $X^s$ is a map of $X^s$ into $E_N\subset X^s$. For $u\in X^s$, the Galerkin approximation is defined by $u_N:=P_Nu$.  When the gradient is well-defined, we define the gradient of the Galerkin approximation by
$\nabla u_N:=\nabla(P_Nu)$.
\end{definition}

The projection $P_N$ is contractive on $X^s$ since $\|P_Nu\|_{X^s}\leq \|u\|_{X^s}$.

\begin{proposition}[Galerkin projection error]
\label{prop:galerkin_approximation}
If $\alpha>\beta\geq 0$ and $P_N$ is the Galerkin projection, then for every $u\in X^\alpha$,
$$
\|u-P_Nu\|_{X^\beta}
\leq
\lambda_{N+1}^{-(\alpha-\beta)/2}
\|u\|_{X^\alpha}.
$$
\end{proposition}

\noindent The proof is given in Appendix~\ref{sec:app_spaces}.

Let $\Lambda_N:=\operatorname{diag}(\lambda_1,\ldots,\lambda_N)$, then the Galerkin mild map is a map on the time-dependent coefficient in the Galerkin space, i.e. $a(t)\in\bR^N$, defined by
\[
\Psi_N[a](t):=e^{-t\Lambda_N}\mathcal C_Nu_0+\int_0^te^{-(t-s)\Lambda_N}\mathcal C_Nf\left(\mathcal S_Na(s),\nabla\mathcal S_Na(s)\right)\,ds .
\]
Through direct computation, we see that $S(t)\mathcal S_N=\mathcal S_Ne^{-t\Lambda_N}$ and $\mathcal C_N\mathcal S_N=I$, and thus using $\mathcal S_N$, the Galerkin map is equivalent to the $E_N$-valued map $$\Phi_N[u](t)=S(t)P_Nu_0+\int_0^tS(t-s)P_Nf(P_Nu(s),\nabla P_Nu(s))\,ds$$
with $\Phi_N[\mathcal S_Na]=\mathcal S_N\Psi_N[a]$. Instead of bounding the difference $\Phi_N-\Phi$ separately, we use the $\Psi_N$ operator. 
In Theorem~\ref{thm:fully_discrete_consistency}, we analyze the difference between the exact mild map and the fully discrete one, the spectral truncation is used in Proposition~\ref{prop:galerkin_approximation} and Lemma~\ref{lem:sampling}. 

\section{Neural Network and DAG Approximation}
\label{sec:dag}

 One part of the approximability of the solution operator comes from the structure of the constitutive law through $f$. In this section, we show that the smoothness and low-dimensionality of $f$ allow for the construction of low parameter count deep ReLU approximations for the solution operator. The DAG construction in this section is compositional, akin to \cite{schmidt2020nonparametric,poggio2024compositional}, but uses the computation graph of the governing law. In \cite{chen2025error}, compositions of linear maps and low-dimensional nonlinearities was used, where the rates were determined by Lipschitz constants rather than by the smoothness of the nodes. For component-wise nodes, $p_j=1/s_j$ is the ratio of input dimension to smoothness that governs the rates for compositions in \cite{schmidt2020nonparametric}. The difference in this section is that multilinear nodes only cost $\log\varepsilon^{-1}$, and we use the infimum over all representations of the law, i.e. the invariant $p_f$

\subsection{Deep ReLU Networks} \label{sec:networks}

The main quantity we would like to measure is the complexity of approximating the solution operator using a Deep ReLU network. The main nonlinearity approximations in our constructive arguments come from $f$ where we utilize the ReLU activation, dense affine layers, and $1\times1$ convolutional layers.

\begin{definition}[Neural Network Components]
\label{def:layers}
Let $\sigma(y):=\max\{y,0\}$ denote the ReLU activation, applied element-wise.

A fully connected layer with input dimension $n$ and output dimension $n'$ is defined as the affine map $\mathcal A_{A,b}:\bR^{n}\to\bR^{n'}$ given by
$$ \mathcal A_{A,b}(z):=Az+b,
\qquad
A \in \bR^{n'\times n},\ b\in\bR^{n'}.$$
Its trainable parameters are the components $A$ and $b$ with a total of $\|A\|_0+\|b\|_0$ nonzero weights.

For spatial discretizations, we define a variable on the grid that has $c$ channels as
$$X=(X_1,\ldots,X_M)\in\left(\bR^{c}\right)^{M}$$
where $X_i\in\bR^{c}$ is the state evaluated at grid point $x_i$. A $1\times1$ convolution layer with $c$ input and $c'$ output channels is given by the map
\begin{equation}
\mathcal K_{A,b}:\left(\bR^{c}\right)^{M}\to\left(\bR^{c'}\right)^{M},
\ \ \
\left(\mathcal K_{A,b}X\right)_i:=AX_i+b
\quad\text{for every } 1\le i \le M,
\label{eqn:conv1x1}
\end{equation}
where the parameters $A \in \bR^{c'\times c}$ and $b \in \bR^{c'}$ are shared for each $i$. The parameter count is that of a single dense layer $\mathcal A_{A,b}$ on $\bR^c$, i.e. $\|A\|_0+\|b\|_0$ total parameters independent of $M$. If the convolution is written as a single affine layer on $(\bR^c)^M$, with matrix $I_M \otimes A$ and bias $\mathbf 1_M\otimes b$, it would have $M(\|A\|_0+\|b\|_0)$ nonzero weights.
\end{definition}

Next, we define the full network class considered, which is used in Theorem~\ref{thm:relu_f} to count the total parameter complexity.

\begin{definition}[Neural Network Class]
\label{def:network}
Set the depth $L_0\in\mathbb N$, the input dimension $n_0$, the output dimension $n_{L_0}$, and the intermediate widths $n_1,\ldots,n_{L_0-1}\in\mathbb N$. We define a local ReLU neural network 
$\mathcal{N}:\bR^{n_0}\to\bR^{n_{L_0}}$, which maps $z$ to $\widehat z^{(L_0)}$ by the following three steps.
\begin{enumerate}
\item \emph{(Input layer.)} $\widehat z^{(1)} = \left(\widehat z^{(1)}_1,\ldots,\widehat z^{(1)}_{n_1}\right)$ is given by
$$\widehat z^{(1)}:=\sigma\left(A_1z+b_1\right),$$
where $A_1\in\bR^{n_1\times n_0}$ and $b_1\in\bR^{n_1}$.
\item \emph{(Hidden layers.)} For $1\le{l}\le L_0-2$, the arrays $\widehat z^{({l}+1)}=\left(\widehat z^{({l}+1)}_1,\ldots,\widehat z^{({l}+1)}_{n_{{l}+1}}\right)$ are given recursively by
$$\widehat z^{({l}+1)}:=\sigma\left(A_{{l}+1}\widehat z^{({l})}+b_{{l}+1}\right),$$
where $A_{{l}+1}\in\bR^{n_{{l}+1}\times n_{l}}$ and $b_{{l}+1}\in\bR^{n_{{l}+1}}$.

\item \emph{(Output layer.)} The output is an affine map given by
$$\widehat z^{(L_0)}:=A_{L_0}\widehat z^{(L_0-1)}+b_{L_0}$$
where  $A_{L_0}\in\bR^{n_{L_0}\times n_{L_0-1}}$ and $b_{L_0}\in\bR^{n_{L_0}}$.
\end{enumerate}
Altogether the Deep ReLU network is:
\begin{equation}
\mathcal{N}(z)=A_{L_0}\,\sigma\left(A_{L_0-1}\,\sigma\left(\cdots\sigma(A_1z+b_1)\cdots\right)+b_{L_0-1}\right)+b_{L_0} .
\label{eqn:relu_network}
\end{equation}
The weight matrices $A_{l}$ and the bias vectors $b_{l}$ together form the parameter vector $\{(A_{l},b_{l})\}_{{l}=1}^{L_0}$.  We measure the size by the following:
$$ \dep(\mathcal{N}):=L_0, \qquad
\operatorname{wid}(\mathcal{N}) := \max_{1\le{l}<L_0} n_{l}, \qquad
H(\mathcal{N}) := \sum_{{l}=1}^{L_0-1} n_{l}, \qquad
W(\mathcal{N}) := \sum_{{l}=1}^{L_0} \left(\|A_{l}\|_0+\|b_{l}\|_0\right),
$$
referred to as the depth, the width, the number of neurons, and the number of nonzero weights and biases. 

The network class is defined as follows. For a compact set $K \subset \bR^{n_0}$ and constants $L_0, p, W, H\in\mathbb N$ and $\Theta,R>0$, we define the class by:
$$\mathcal{F}^{L_0,H,\sigma}_{n_0,n_{L_0}}\left(p,W,\Theta,R;K\right)
:=\left\{\mathcal{N}\ \text{as in \eqref{eqn:relu_network}}:\
\begin{array}{l}
\dep(\mathcal{N})\le L_0,\ \ \operatorname{wid}(\mathcal{N})\le p,\ \ H(\mathcal{N})\le H,\ \ W(\mathcal{N})\le W,\\[1mm]
\max_{l}\|A_{l}\|_{\infty},\ \max_{l}\|b_{l}\|_\infty\le\Theta,
\ \ \|\mathcal{N}\|_{L^\infty(K)}\le R
\end{array}\right\}$$
which are deep ReLU networks with depth at most $L_0$, at most $H$ neurons, sparsity at most $W$, parameter magnitude at most $\Theta$, and max-norm bounded by $R$ on $K$. 

For grid-based functions, we replace each dense layer $z\mapsto A_{l} z+b_{l}$ in the three steps above by the point-wise convolutional layer $\mathcal K_{A_{l},b_{l}}$ of \eqref{eqn:conv1x1}, with no change to the activation, so that map $\mathcal{N}^h:\left(\bR^{n_0}\right)^{M}\to\left(\bR^{n_{L_0}}\right)^{M}$
has the same depth, with the same parameter vector $\{(A_{l},b_{l})\}_{l=1}^{L_0}$. 

 A layer is called fixed if its weights are pre-determined and trainable otherwise.
\end{definition}

We will utilize several properties of the ReLU networks, used throughout the analysis \cite{bolcskei2019optimal,petersen2018optimal}. We summarize them in the lemma below.

\begin{lemma}[Operations on ReLU networks]
\label{lem:nn_calculus}
Let $\mathcal{N}_1,\mathcal{N}_2$ be two ReLU networks from Definition~\ref{def:network} with compatible dimensions. Each operation produces a new parameter vector $\{(A_{l},b_{l})\}$, and by Lemma~\ref{lem:sharing} each statement holds directly for grid realizations, with $W$ multiplied by the number of points and the parameter count unchanged.
\begin{enumerate}
\item \emph{(Affine layers.)} A map $z\mapsto Az+b$ is a network of depth $1$ with $\|A\|_0+\|b\|_0$ weights. Composing two networks with an affine map does not increase the depth.
\item \emph{(Composition.)} The network $\mathcal{N}_2\circ\mathcal{N}_1$ has $\dep \le \dep(\mathcal{N}_1)+\dep(\mathcal{N}_2)$ and $W \le 2 W(\mathcal{N}_1) + 2 W(\mathcal{N}_2)$.
\item \emph{(Parallelization.)} The map $z \mapsto (\mathcal{N}_1(z), \mathcal{N}_2(z))$ is a network with $\dep = \max_i \dep(\mathcal{N}_i)$. If the two depths are the same, $\operatorname{wid} \le \sum_i\operatorname{wid}(\mathcal{N}_i)$ and $W\le\sum_iW(\mathcal{N}_i)$. If the depths differ, then we extend the shallower network $\mathcal{N}_i$ by padding with $\Delta:=|\dep(\mathcal{N}_1)-\dep(\mathcal{N}_2)|$ identity layers which adds at most $W(\mathcal{N}_i) + 2m_i \Delta$ weights and $2 m_i$ in width, where $m_i$ is the output dimension of $\mathcal N_i$.
\item \emph{(Identity and Skip Connections.)} The identity on $\bR^{k}$ can be realized exactly by $y=\sigma(y)-\sigma(-y)$, with $4k$ nonzero weights per layer. Mapping a vector of dimension $k$ through ${l}$ layers has a total cost of at most $4kl$ weights and no added depth.
\item \emph{(Clipping.)} For threshold $\mathrm{a}>0$, the clipping operation $\operatorname{Clip}_{\mathrm{a}}(y)=\min\{\max(y,-\mathrm{a}),\mathrm{a}\}=\sigma(y+\mathrm{a})-\sigma(y-\mathrm{a})-\mathrm{a}$ can be realized exactly by one hidden layer that has $7$ nonzero weights and biases per coordinate ($4$ hidden, $3$ output).
\end{enumerate}
\end{lemma}

\noindent The proof is given in Appendix~\ref{sec:app_networks}.

\begin{lemma}[Parameter Sharing]
\label{lem:sharing}
Let $g:\bR^{c}\to\bR^{c'}$ be a ReLU network as in Definition~\ref{def:network}, with parameters $\{(A_{l},b_{l})\}_{l=1}^{L_0}$, and let
$$\left(G^h(X)\right)_i:=g(X_i),\qquad i=1,\ldots,M,$$
be its point-wise application on a grid. Then $G^h$ is the grid realization $g^h$ from Definition~\ref{def:network}, obtained by replacing each dense layer $\mathcal A_{A_l,b_l}$ by the point-wise convolution $\mathcal K_{A_l,b_l}$ of \eqref{eqn:conv1x1} without changing the activation function. Thus
$$\dep(G^h)=\dep(g),\ \ 
\#\{\text{distinct parameters of }G^h\}=W(g),\ \
W(G^h)=M\,W(g),$$
where $W(G^h)$ counts the nonzero weights of $G^h$ written without parameter sharing, with block-diagonal matrices $I_M\otimes A_l$ and biases $\mathbf 1_M\otimes b_l$. 
\end{lemma}
\noindent The grid-based saving comes from considering it as a convolution layer, $W(g)$ total shared weights, instead of a dense ReLU layer. The proof is a direct computation and given in Appendix~\ref{sec:app_networks}.

\subsection{The DAG formulation}

We will form the approximation of $f$ by first writing it as a directed acyclic graph (DAG) and show that the DAG can be approximated by a deep ReLU network. This allows us to probe how the complexity depends on the actual structure of $f$.

\begin{assumption}[Local constitutive map with finite DAG representation]
\label{assumption:structured_f}
Let $R>0$ and let
$$B_R^\alpha := \left\{ u \in X_T^\alpha: \|u\|_{X_T^\alpha} \le R \right\} $$
be the ball from Proposition~\ref{prop:contraction}. Under Assumption~\ref{assumption:regularity_embedding}, the set of admissible point-wise states
$$K:=\overline{ \left\{ \left(u(t,x),\nabla u(t,x) \right)
\ : \ u \in B_R^\alpha,\ t\in[0,T],\ x\in\overline\Omega \right\}} \ \subset\ \bR^d\times\bR^{d\times d_x}$$
is bounded, since $\|u(t,\cdot)\|_{C^1(\overline\Omega)} \le C_{\mathrm{emb}}R$ for every $u\in B_R^\alpha$ and $t\in[0,T]$, and therefore $K$ is compact.

Let $U\subset\bR^d\times\bR^{d\times d_x}$ be an open neighborhood of $K$ and let $f:U\to\bR^m$ be the law from the PDE \eqref{eqn:PDE}, where $m=d$ for the law itself and $m\ge d$ for a local part $f_{\mathrm{loc}}$ as from Lemma~\ref{lem:linear_postcomposition}. Assume that $f$ admits a representation by a single finite directed acyclic graph $G=(\mathcal V,\mathcal E)$, whose input nodes are the coordinates of $z=(u,\nabla u)$ and whose (unique) output node $j_{\mathrm{out}}$ produces the vector $f(z)\in\bR^m$, and where each node takes one of the following three forms.
\begin{enumerate}
\item \emph{Affine maps:} $A:\bR^n\to\bR^p$, $A(x)=Mx+b$, with $M\in\bR^{p\times n}$ and $b\in\bR^p$.
\item \emph{Multilinear maps:} $T:(\bR^n)^{r}\to\bR^p$ for some $r\ge2$, linear in each argument separately.
\item \emph{Component-wise nonlinearities:} $G:\bR^n\to\bR^n$ with
$G(x_1,\dots,x_n)=(g(x_1),\dots,g(x_n))$, where $g\in C^{s}(I)$ for some finite integer $s\ge1$ and some open bounded interval $I\subset\bR$.
\end{enumerate}
For every $j$ that is not an input node, let $K_j$ denote the range of the intermediate variable at node $j$ over all inputs $z\in K$. Denote the set of all parent nodes of $j$ by $\operatorname{pa}(j)$, and define 
$$K_j^{\mathrm{in}}:=\prod_{i\in\operatorname{pa}(j)}K_i.$$

Assume that there exists an open set $U_j$ contained in the domain of the corresponding node map with $K_j^{\mathrm{in}}\Subset U_j$, where $\Subset$ means compact containment. Let $\mathcal V_{\mathrm{in}}$ be the set of input nodes and let $\Vnl$ be the set of non-affine (multilinear or component-wise) nodes, where we assume $\Vnl\neq\emptyset$.
\end{assumption}

For the main results, the exponents will depend on the DAG formulation, which is not unique for a given $f$. For example, the cubic map $f(u)=u^3$ can be written as two multilinear nodes, as one component-wise $C^\infty$ node, or as one component-wise $C^3$ node. Each of these DAGs are valid under Assumption~\ref{assumption:structured_f}, but they lead to different rates. To avoid this issue, we define the following admissible representation.

\begin{definition}[Admissible Representations and the Structural Exponent]
\label{def:structural_exponent}
An admissible representation of $f$ is a finite DAG $G$ satisfying the conditions of the assumption on $K$. For all admissible $G$, we define
$$p_{\max}(G):=\max_{j\in\Vnl(G)}p_j,
\qquad
p_j:=\begin{cases}0,&j\text{ multilinear},\\ 1/s_j,&j\text{ component-wise }C^{s_j},\end{cases}$$
where $p_{\max}(G):=0$ if $G$ has no component-wise node. 

The structural exponent of $f$ (on $K$) is defined by
$$p_f:=\inf \left\{ p_{\max}(G)\ : \ G\text{ an admissible representation of } f\right\}.$$
\end{definition}

The infimum defines a unique exponent that reveals the structure of $f$ directly. If $f$ is a polynomial map, then $p_f=0$ and the infimum is attained by the DAG using each monomial as a tree of multilinear nodes. If $f$ is a component-wise map $x\mapsto(g(x_1),\ldots,g(x_n))$ where $g\in C^s$, then $p_f\le1/s$ and if $g\in C^\infty$ then $p_f=0$.

Theorem~\ref{thm:main} uses a given representation, i.e. for $p_{\max}(G)$. Note that since it holds for every admissible $G$ and since the bound is monotone in the exponent, it will hold when $p_{\max}$ is replaced by any $p>p_f$. If the infimum is attained by the graph, then $p_{\max}$ can be replaced by $p=p_f$. The difference is that $p_{\max}$ refers to a particular representation used in the proofs while the main structural results will use the best one obtained through $p_f$. For the explicit examples in Section~\ref{sec:examples}, we can obtain the infimum for all but the Arrhenius case (Example~\ref{ex:arrhenius}), since $p_f=0$ is approached but not attained. Figure~\ref{fig:dags} illustrates the computation graphs for two of the examples of Section~\ref{sec:examples}.

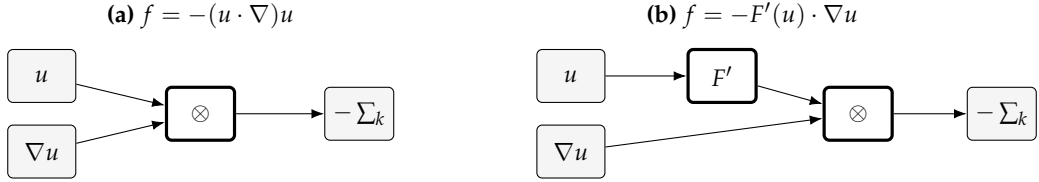
\begin{figure}[ht]
\centering
\begin{tikzpicture}[>=Latex, font=\small,
  nd/.style={draw, rounded corners=2pt, minimum height=7mm, minimum width=9mm, inner sep=3pt},
  aff/.style={nd, fill=black!4},
  nl/.style={nd, very thick}]
\node at (2.1,1.9) {\textbf{(a)} $f=-(u\cdot\nabla)u$};
\node[aff] (u1)  at (0,1.1) {$u$};
\node[aff] (du1) at (0,0.1) {$\nabla u$};
\node[nl]  (b1)  at (2.1,0.6) {$\otimes$};
\node[aff] (o1)  at (4.2,0.6) {$-\sum_k$};
\draw[->] (u1)  -- (b1);
\draw[->] (du1) -- (b1);
\draw[->] (b1)  -- (o1);
\node at (9.4,1.9) {\textbf{(b)} $f=-F'(u)\cdot\nabla u$};
\node[aff] (u2)  at (7.0,1.1) {$u$};
\node[aff] (du2) at (7.0,0.1) {$\nabla u$};
\node[nl]  (g2)  at (9.0,1.1) {$F'$};
\node[nl]  (b2)  at (10.8,0.6) {$\otimes$};
\node[aff] (o2)  at (12.7,0.6) {$-\sum_k$};
\draw[->] (u2)  -- (g2);
\draw[->] (g2)  -- (b2);
\draw[->] (du2) -- (b2);
\draw[->] (b2)  -- (o2);
\end{tikzpicture}
\caption{Computation Graphs for Two Constitutive Equations. The bold boxes are the nonlinear nodes, the shaded boxes are affine and realized exactly. In (a) the only nonlinear node is multilinear, yielding $(p,q)=(0,1)$ and a network of size $O(\log\varepsilon^{-1})$. In (b) the node $F'\in C^{s-1}$ leads to $(p,q)=(1/(s-1),1)$.}
\label{fig:dags}
\end{figure}

The next theorem constructs a deep ReLU network that approximates a DAG map node-wise. To avoid the trivial case, we assume that there is at least one nonlinear node. 

\begin{theorem}[Deep ReLU Approximation for Finite DAG]
\label{thm:relu_f}
Let $f$ have the structure from Assumption~\ref{assumption:structured_f}. If the computation graph contains no nonlinear nodes, then $f$ is affine and is exactly representable by a ReLU network of depth and size bounded independently of $\varepsilon$.

If $f$'s DAG has at least one nonlinear node, then there exist constants $\varepsilon_0\in(0,1)$ and $C>0$, depending only on the computation graph, the compact set \(K\), the node dimensions, and the node maps, such that for every
$\varepsilon\in(0,\varepsilon_0]$
there exists a ReLU network 
$\mathcal{N}_\varepsilon: \bR^d\times\bR^{d\times d_x}
\to \bR^m$ satisfying 
$$\sup_{z\in K} \| f(z)-\mathcal{N}_\varepsilon(z) \|_{\bR^m}
\le \varepsilon.$$
Additionally, there exist constants $c_j>0$, $p_j \ge 0$, and $q_j \ge 0$, indexed by the nonlinear nodes $j \in \Vnl$, such that
$$\dep(\mathcal{N}_\varepsilon) \le \sum_{j\in\Vnl} c_j \log(\varepsilon^{-1}),
\qquad
W(\mathcal{N}_\varepsilon) \le \sum_{j\in\Vnl} c_j \, \varepsilon^{-p_j} \left(\log(\varepsilon^{-1})\right)^{q_j},$$
where $(p_j,q_j)=(0,1)$ for multilinear nodes and
$(p_j,q_j)=(1/s_j,1)$ for nodes $j$ induced by a scalar function $g_j\in C^{s_j}$.
\end{theorem}

\noindent The proof approximates each nonlinear node by Theorem~\ref{thrm:quantitative_smooth_scalar} or Theorem~\ref{thrm:quantitative_multilinear}, and propagates the node-wise errors along the topological ordering of the graph. One subtlety in the proof is the use of a buffer around the input set for each node so that perturbations remain within a region where the operations are Lipschitz. The detailed analysis is given in Appendix~\ref{sec:app_networks}.

\begin{remark}[Dependence on $d$ and $d_x$]
\label{rem:dx_dependence}
The exponents $p_j$ in Theorem~\ref{thm:relu_f} are dimension-free.  Only the coefficients $c_j$ depend on the node dimensions and thus are dependent on the input dimension. A component-wise node on $\bR^n$ is just $n$ copies of one scalar network, so $c_j=O(n)$. By Theorem~\ref{thrm:quantitative_multilinear}, multilinear node $T:(\bR^n)^r\to\bR^p$ costs $O(\#\{\text{nonzero coefficients}\})\le O(p\,n^r)$.

For standard examples, the cost is $C\,d\,d_x\log\varepsilon^{-1}$ for the term $-(u\cdot\nabla)u$, $C\,d\log\varepsilon^{-1}$ for a polynomial reaction term, and $C\,d\,d_x\varepsilon^{-1/(s-1)} \log\varepsilon^{-1}$ for flux terms $-F'(u)\cdot\nabla u$ with $F\in C^s$. They are all linear in the dimension. If we applied a generic result for $C^s$ map $\bR^{d+dd_x}\to\bR^d$, the expected rate would be $\varepsilon^{-(d+dd_x)/s}$. So informally, we say that the cost or complexity is $d_x$-independent if the exponents do not directly depend on the dimension.
\end{remark}

Several important physical systems have constitutive equations that include nonlocal effects, i.e., they are in the form $\mathcal M f_{\mathrm{loc}}(u,\nabla u)$ where $f_{\mathrm{loc}}$ is local and $\mathcal M$ is a bounded nonlocal linear operator. For example, for the incompressible Navier-Stokes equation the operator is the Leray projection or the Biot-Savart operator. This additional cost is proportional to the number of spectral weights per mode, where we show in Lemma~\ref{lem:linear_postcomposition} of Section~\ref{sec:postcomposition} that it can be absorbed into the spectral layers. Note also that finite characterization of the governing PDE was also used in the context of \cite{schaeffer2017learning}.

\section{Discretization and the Neural Picard Approximation}
\label{sec:discrete}

In this section, we construct a discrete mild map and bound the error over each step of discretization. The discretization is a theoretical tool rather than a numerical method. We utilize Galerkin theory \cite{thomee2007galerkin,lord2014introduction,temam2024navier} as the approximating technique. We would like to show the existence of such networks with uniform bounds, and thus utilize an approximation that can be taken around the continuous solution $u$. We first replace the Duhamel operator by a quadrature in time and then the inner products by a spatial quadrature which defines the fully discrete neural Picard map. We show that the maps are contractive and consistent (in the numerical analysis sense). Section~\ref{sec:total} combines these results to bound the total error incurred over multiple neural Picard steps. 
\subsection{Temporal Discretization}

As a function of time, forcing along the trajectory is
\[
g(t)=f\left(u(t),\nabla u(t)\right)\in L^\infty\left(0,T;X^\beta\right),
\]
bounded by $F_R$, and by Lemma~\ref{lem:derived_regularity} Part (3), it is H\"older continuous in the weighted norm. These are the assumptions used in the next results, and only require that the data is in $X^\alpha$.

\begin{proposition}[Temporal Consistency]
\label{prop:temporal_quadrature}
Let $t_q:=q\Delta t$, $q=0,\ldots,Q$, be the nodes with time step $\Delta t=T/Q$, let $0\le\beta<\alpha$ with $\alpha-\beta<2$, and set $\sigma:=\frac{\alpha-\beta}{2}$ and $\theta:=1-\sigma>0$. Let $\zeta\in(0,\theta)$ and assume that $g:[0,T]\to X^\beta$ satisfies
$$\|g\|_{\infty,\beta}:=\sup_{0\le t\le T}\|g(t)\|_{X^\beta}<\infty,
\qquad
[g]_{\zeta}:=\sup_{0<s<t\le T}\frac{\|g(t)-g(s)\|_{X^\beta}}{(t-s)^{\zeta}s^{-\zeta}}<\infty . $$
If the discrete Duhamel operator is given by
\[
(\mathcal D_Q[g])^q:=\Delta t\sum_{r=0}^{q-1}S(t_q-t_r)g(t_r),
\]
then there is a constant $C>0$, depending only on $\sigma$ and $\zeta$ and independent of $\Delta t$, $T$ and $Q$, such that
$$\max_{0\le q\le Q}\left\|\mathcal D[g](t_q)-(\mathcal D_Q[g])^q\right\|_{X^\alpha} \le C \,\|g\|_{\infty,\beta}\,(\Delta t)^{\theta} + C\,[g]_{\zeta}\,T^{\theta-\zeta}\,(\Delta t)^{\zeta}.$$
The consistency is $\mathcal{O}((\Delta t)^{\zeta})$ since $\zeta<\theta$.

In addition, if for some $\kappa\in(0,1]$ and some $\beta'\ge\beta$ with $\kappa<\theta+\frac{\beta'-\beta}{2}$, we had
$$\sup_{0\le t\le T}\|g(t)\|_{X^{\beta'}} \le G,
\qquad \|g(t)-g(s)\|_{X^\beta}\le H\,(t-s)^{\kappa}\quad(0\le s<t\le T),$$
then, with $a:=\left(\sigma-\frac{\beta'-\beta}{2}+\kappa\right)_+<1$ and $C$ depending only on $\sigma$, $\kappa$ and $1-a$, and $\lambda_1$ (when $\beta'-2\kappa>\alpha$ due to embedding), we have
$$\max_{0\le q\le Q}\left\|\mathcal D[g](t_q)-(\mathcal D_Q[g])^q\right\|_{X^\alpha}\ \le\ C\left(T^{\theta}H+T^{1-a}G\right)(\Delta t)^{\kappa}.$$
\end{proposition}

 The proof, given in Appendix~\ref{sec:app_discrete}, splits the error into two parts. One freezes the kernel on each subinterval, which needs only boundedness of $g$ and gives the term $\|g\|_{\infty,\beta}(\Delta t)^\theta$, and the second freezes $g$ which gives the term $[g]_\zeta(\Delta t)^\zeta$. The order of accuracy is determined by the time regularity of the forcing term, see Remark~\ref{rem:gamma_free}.

  Note that the time discretization is a version of the Lawson-Euler method \cite{lawson1967generalized,hochbruck2010exponential}, which uses the exact semigroup (implicitly) and is explicit in forcing term. We use this as opposed to exponential Euler \cite{cox2002exponential} since it will allow for extension by energy estimates in Section~\ref{sec:long_time} using just a one-sided growth bound on the governing law.

\begin{remark}
\label{rem:gamma_free}
Any $\gamma$ with $\alpha<\gamma<\beta+2$ is admissible in Lemma~\ref{lem:derived_regularity}. As $\gamma$ approaches the upper limit $\beta+2$, $\zeta\to\theta$ and $\chi\to1+\frac{(\alpha-2)_+}{2\theta}$. This is the same choice that improves the spectral requirement and the limit in \eqref{eqn:quad_order_cap}, so no exponent is needed to trade off. The additional theoretical cost is the factor $(1-\bar\sigma)^{-1}=(\theta-\zeta)^{-1}$, which then appears in later constant, so we cannot set it to $\theta$.  If the force is smoother then the bootstrapping result in Lemma~\ref{lem:bootstrap_time} helps.
\end{remark}

\subsection{Spatial Discretization}

Consider a discretization $\Omega_h$ of the spatial domain $\Omega$, and define the grid spacing by
$$h:=\max_{\Omega'\in\Omega_h}\operatorname{diam}(\Omega'),$$
i.e., the maximum diameter of a subdomain of the mesh. Let
$\{x_i\}_{i=1}^M\subset\Omega$ denote the sampling points and let
$\{\omega_i\}_{i=1}^M$ denote the weights associated with the spatial quadrature rule.

\begin{definition}[Discrete Galerkin operators]
\label{def:discrete_Galerkin_operators}
Let $E_N:=\operatorname{span}\{e_1,\ldots,e_N\}$,
and recall the continuous (Galerkin) coefficient and synthesis operators
$$\mathcal C_N:H\rightarrow\bR^N,
\qquad (\mathcal C_N v)_k:=\langle v,e_k\rangle_H, $$
and
$$\mathcal S_N:\bR^N\rightarrow E_N, \qquad \mathcal S_N a:=\sum_{k=1}^N a_k e_k,$$
so that
$P_N=\mathcal S_N\mathcal C_N$ is the spectral Galerkin projection onto $E_N$. Define the sampling operator
$\mathcal I_h:C(\overline\Omega)\rightarrow\bR^M$
by $(\mathcal I_h v)_i:=v(x_i)$, for $i=1,\ldots,M$ and is defined for every continuous function. For $a\in\bR^N$, the sampled Galerkin function is defined as
$$\mathcal S_N^h a :=
\mathcal I_h\mathcal S_N a \in\bR^M,
\qquad (\mathcal S_N^h a)_i = \sum_{k=1}^N a_k e_k(x_i).$$
So $\mathcal S_N a\in E_N$ denotes the continuous Galerkin function associated with $a$, while $\mathcal S_N^h a$ denotes its values on the spatial grid.

For a grid function $v=(v_1,\ldots,v_M)^T\in\bR^M$,
define the discrete coefficient map
$\mathcal C_N^h:\bR^M\rightarrow\bR^N$
by
$$(\mathcal C_N^h v)_k :=
\sum_{i=1}^M\omega_i v_i e_k(x_i), \qquad k=1,\ldots,N,$$
which approximates the continuous coefficient map
$\mathcal C_N$ by replacing the spatial inner products with a quadrature rule. The lifted discrete Galerkin operator
$$\mathcal P_N^h := \mathcal S_N\mathcal C_N^h :
\bR^M\rightarrow E_N$$
maps grid values to a function in the continuous Galerkin space $E_N$. Its sampled version is
$$P_N^h := \mathcal I_h\mathcal P_N^h = \mathcal S_N^h\mathcal C_N^h : \bR^M\rightarrow\bR^M.$$
 The Galerkin semigroup acting on coefficient vectors is defined as the truncated operator
 $$S_N(t):=e^{-t\Lambda_N},
\qquad
\Lambda_N:=\operatorname{diag}(\lambda_1,\ldots,\lambda_N).$$
Since $S(t)e_k=e^{-\lambda_k t}e_k$,
the continuous synthesis operator satisfies
$S(t)\mathcal S_N = \mathcal S_N S_N(t)$.
For every $s\geq0$, 
$$ \|\mathcal S_N a\|_{X^s} = \|a\|_{\ell_s^2}, \qquad
\|a\|_{\ell_s^2}:= \left( \sum_{k=1}^N\lambda_k^s|a_k|^2 \right)^{1/2},
$$
and thus $\mathcal S_N$ is an isometry from
$(\bR^N,\|\cdot\|_{\ell_s^2})$ onto $E_N$ equipped with the
$X^s$ norm, uniformly in $N$.

Define the sampled gradient operator
$\mathcal I_h^{\text{grad}}:E_N\rightarrow(\bR^{d\times d_x})^M$
by
$$ (\mathcal I_h^{\text{grad}} v)_i:=\nabla v(x_i), \qquad
i=1,\ldots,M, $$
and for $a\in\bR^N$ define
$$\nabla_h\mathcal S_N^h a :=\mathcal I_h^{\text{grad}}\mathcal S_N a
=\left( \sum_{k=1}^N a_k\nabla e_k(x_i)\right)_{i=1}^M.$$
\end{definition}

Both $\mathcal S_N^h a$ and $\nabla_h\mathcal S_N^h a$ are
obtained by first constructing the continuous Galerkin function
$\mathcal S_N a$ and then sampling the function and its gradient,
respectively, at the spatial grid points. When $u$ is vector-valued, say on $\bR^d$ with $d>1$, every operator above acts component-wise and the coefficient space becomes $(\bR^{N})^{d}$. To keep the notation simple, we write $\bR^N$ and $\bR^M$, and only use $\bR^{dN}$ and $\bR^{dM}$ when necessary for examples. 

In general, $P_N^h$ is not a projection, since the quadrature rule does not need to satisfy
$\mathcal C_N^h\mathcal S_N^h = I_{\bR^N}$.
The composition is
$$(\mathcal C_N^h\mathcal S_N^h)_{jk}
=\sum_{i=1}^M \omega_i e_j(x_i)e_k(x_i),$$
and is equal to $I_{\bR^N}$
only when the quadrature rule is exact on the products $e_j e_k$, $1\leq j,k\leq N$.

To discretize the Picard map, the input to $f$ uses the reconstructed Galerkin function so the approximation is in the spatial dimension, not the discretization dimension. This is akin to a pseudo-spectral approximation and thus for $a\in\bR^N$,
$$f\left(\mathcal S_N^h a,\ \nabla_h\mathcal S_N^h a\right)$$
denotes the evaluation of the nonlinearity on the reconstructed function and its discrete gradient at the spatial sample points. The gradient $\nabla_h\mathcal S_N^ha$ is obtained by differentiating the basis elements and evaluating the result on the same grid.

\begin{definition}[Fully Discrete Neural Picard Map]
\label{def:fully_discrete_neural_picard}
Suppose we use the temporal quadrature rule from
Proposition~\ref{prop:temporal_quadrature}, the spatial
discretization and the discrete operators $\mathcal C_N^h$, $\mathcal S_N^h$,
$S_N$, and $\nabla_h\mathcal S_N^h$ from
Definition~\ref{def:discrete_Galerkin_operators}. Let
$$\left(\mathcal F^h(U,P)\right)_i:=f(U_i,P_i),
\qquad \left(\widehat{\mathcal F}_\varepsilon^h(U,P)\right)_i:= \widehat f_\varepsilon(U_i,P_i), \qquad i=1,\ldots,M,$$
denote the point-wise evaluation of the exact constitutive law and of its ReLU approximation $\widehat f_\varepsilon:=\mathcal{N}_\varepsilon$ from Theorem~\ref{thm:relu_f}. The parameters are shared since a single network is applied at every grid point.

Let $a=(a^0,\ldots,a^Q)\in(\bR^N)^{Q+1}$ and let $a_0^h:=\mathcal C_N^h\mathcal I_hu_0$ be the discrete initial coefficients. The fully discrete Picard map in coefficient space is defined as
$$\left(\Psi_{N,Q}^h(a)\right)^q
:=S_N(t_q)\,a_0^h + \Delta t\sum_{r=0}^{q-1}
S_N(t_q-t_r)\,\mathcal C_N^h \left[ \mathcal F^h \left(
\mathcal S_N^h a^r,\ \nabla_h\mathcal S_N^h a^r \right) \right]$$
for $q=0,\ldots,Q$ and where the sum is zero when $q=0$. The neural approximation is given by
$$\left(\Psi_{N,Q,\varepsilon}^h(a)\right)^q
:=S_N(t_q)\,a_0^h + \Delta t\sum_{r=0}^{q-1}
S_N(t_q-t_r)\,\mathcal C_N^h \left[ \widehat{\mathcal F}_\varepsilon^h \left( \mathcal S_N^h a^r,\ \nabla_h\mathcal S_N^h a^r \right) \right].$$
The corresponding $E_N$-valued maps are obtained by applying $\mathcal S_N$ at each time level and interpolating piecewise linearly in time, i.e., for $u\in X^\alpha_T$ we write
$$\left(\Phi_{N,Q,\varepsilon}^h[u]\right)(t_q)
:= \mathcal S_N\left(\Psi_{N,Q,\varepsilon}^h(b(u))\right)^q,
\qquad b(u)^r:=\mathcal C_N^h\mathcal I_h u(t_r),$$
and similarly for $\Phi_{N,Q}^h$. Since $E_N\subset X^\alpha$, this defines operators $\Phi_{N,Q}^h,\Phi_{N,Q,\varepsilon}^h:X_T^\alpha\rightarrow X_T^\alpha$. Lastly, we call a sequence $a_{N,Q}^h$ a fully discrete fixed point if
$a_{N,Q}^h = \Psi_{N,Q}^h(a_{N,Q}^h)$, and similarly in the neural approximation case.
\end{definition}

The difference between $\mathcal S_N$ and $\mathcal S_N^h$ is grid-evaluation and is only an implementation aspect. Definition~\ref{def:fully_discrete_neural_picard} can be stated as an architecture so that  each component of Section~\ref{sec:networks} is clear. The iteration of the Picard maps is equivalent to a parameter shared loop. 

 Similar to other neural operators, we assume that we have access to fast transforms between grid values and coefficients. These transforms are fixed, while all other weights and layers are learned. The following assumption separates the two cases.

\begin{assumption}[Fast Transforms]
\label{assumption:spectral_factorization}
There are fixed linear maps $\mathcal T_{\mathrm{syn}}$, $\mathcal T_{\mathrm{syn},1},\ldots,\mathcal T_{\mathrm{syn},d_x}$ and $\mathcal T_{\mathrm{an}}$, each a fast transform in the sense of Assumption~\ref{assumption:fast_transform}, a partition of the retained modes into blocks of at most $b_0$ modes, and maps $D^{\mathrm{ex}}_{\mathrm{an}},D^{\mathrm{ex}}_1,\ldots,D^{\mathrm{ex}}_{d_x}$ (block-diagonal), such that
$$\mathcal S^h_N=\mathcal T_{\mathrm{syn}},
\qquad \partial_{h,j}\mathcal S^h_N=\mathcal T_{\mathrm{syn},j}\,D^{\mathrm{ex}}_j,
\qquad
\mathcal C^h_N=D^{\mathrm{ex}}_{\mathrm{an}}\,\mathcal T_{\mathrm{an}}.$$
\end{assumption}

As an example, on the torus with the real Fourier basis, $\mathcal T_{\mathrm{syn}}=\mathcal T_{\mathrm{syn},j}$ is the IFFT and $\mathcal T_{\mathrm{an}}$ the FFT. The blocks are the cosine-sine pairs of one frequency, so $b_0=2$ and $D^{\mathrm{ex}}_{\mathrm{an}}$ is the normalized quadrature weights, and $D^{\mathrm{ex}}_j$ maps each pair to its derivative, a $2\times2$ block with two nonzero entries.  On these uniform grids all quadrature weights are equal to the same value $\omega$ (e.g. $\omega=|\Omega|/M$ on the torus). The exact value of $D^{\mathrm{ex}}_{\mathrm{an}}$ is this weight times the $L^2$-normalization constants of the eigenfunctions, so the quadrature weights are absorbed into the learned spectral layer $D_{\mathrm{an}}$ and cost no extra parameters.

\begin{definition}[Learned Spectral Layers]
\label{def:spectral_layers}
The learned spectral layers are maps $D_{\mathrm{an}}$, $D_1,\ldots,D_{d_x}$, $D_S$, $D_W$ and, when the law is post-composed with a fixed operator $\mathcal M$ that is diagonal in the eigenbasis (Lemma~\ref{lem:linear_postcomposition}), $D_{\mathcal M}$, each restricted to the sparsity pattern of its exact counterpart
$$D^{\mathrm{ex}}_{\mathrm{an}},\qquad D^{\mathrm{ex}}_j,\qquad D^{\mathrm{ex}}_S:=S_N(\Delta t),\qquad D^{\mathrm{ex}}_W:=\Delta t\,S_N(\Delta t),\qquad D^{\mathrm{ex}}_{\mathcal M}:=D_\mu.$$
They are used in the neural architecture (Definition~\ref{def:loop}) in place of the exact maps. The spectral weights are the elements of these maps, they are shared across grid points, time levels and sweeps, and their count satisfies
$$W_{\mathrm{spec}} \le\ C_{\mathrm{spec}}\,N,
\qquad
C_{\mathrm{spec}} \le b_0\left(d^2(3+d_x)+d\,m\,\mathbf 1_{\mathcal M}\right),$$
where $C_{\mathrm{spec}}$ is the number of nonzero elements per mode of the exact maps. 
\end{definition}

On the torus, in the component-wise Fourier basis, one can show that $C_{\mathrm{spec}}=d(3+d_x)$, the $d\,d_x$ derivative weights are removed for $f(u)$ without gradient terms. An additional linear term $\mathcal M$ adds a $d\times m$ block per mode, where $m$ is the output dimension of $f_{\mathrm{loc}}$. For the Leray projection $m=d$ and the block is the identity.

\begin{definition}[The Looped Architecture]
\label{def:loop}
Fix an admissible pair $(N,h)\in\mathcal A$ (Assumption~\ref{assumption:quadrature_stability}), $Q$, $n$, the network $\widehat f_\varepsilon$ from Definition~\ref{def:network}, with $n_0=d(1+d_x)$ (or $n_0=d$ for gradient-free and divergence-form laws) and $n_{L_0}=m$ ($m=d$ unless the law has the additional term $\mathcal M$), and learned spectral layers $D_{\mathrm{an}},D_1,\ldots,D_{d_x},D_S,D_W$, and $D_{\mathcal M}$ when needed, from Definition~\ref{def:spectral_layers}. The state propagated through the loop is the entire discrete trajectory,
$$a=(a^0,\ldots,a^Q)\in\left(\bR^{dN}\right)^{Q+1}.$$
The architecture consists of three parts, where the only fixed layers are the fast transforms from Assumption~\ref{assumption:spectral_factorization}.
\begin{enumerate}
\item \emph{(Encoder)} From the sampled data $w=\mathcal I_hu_0\in\bR^{dM}$ compute
$$a^h_0:=D_{\mathrm{an}}\,\mathcal T_{\mathrm{an}}w,
\qquad
A^0:=\left(a^h_0,\ldots,a^h_0\right),$$
which uses the fixed transform, a learned spectral layer, and replication across the $Q+1$ time levels.
\item \emph{(Looping)} One step $a\mapsto\Psi(a)$ is the composition of four layers. For each time level $r=0,\ldots,Q-1$:
\begin{enumerate}
\item[(2a)] \emph{Synthesis:} $U^r:=\mathcal T_{\mathrm{syn}}a^r\in(\bR^{d})^{M}$ and $P^r_j:=\mathcal T_{\mathrm{syn},j}D_ja^r$ for $j=1,\ldots,d_x$, so that $P^r\in(\bR^{d\times d_x})^{M}$. That is, the learned derivative multipliers act on the coefficients, and the fixed transforms then produce the grid values of the state and its gradient.
\item[(2b)] \emph{point-wise block:} $G^r:=\widehat{\mathcal F}^h_\varepsilon(U^r,P^r)$, that is, $(G^r)_i=\widehat f_\varepsilon\left(U^r_i,P^r_i\right)$ at every grid point $i=1,\ldots,M$. By Lemma~\ref{lem:sharing} this layer is the grid realization of $\widehat f_\varepsilon$ in the sense of Definition~\ref{def:network}. Each dense layer $A_{l}\,\cdot+b_{l}$ of the block acts as the $1\times1$ convolution $\mathcal K_{A_{l},b_{l}}$, with the same matrices at every point. This is the only nonlinear layer.
\item[(2c)] \emph{Analysis:} $\widehat g^{r}:=D_{\mathcal M}D_{\mathrm{an}}\,\mathcal T_{\mathrm{an}}G^r\in\bR^{dN}$. That is, a fixed transform maps the grid values to coefficients, and the learned spectral layers then normalize them. If the law is not post-composed with an operator $\mathcal M$, then $D_{\mathcal M}$ is the identity.
\end{enumerate}
Then, across time levels,
\begin{enumerate}
\item[(2d)] \emph{Collect:}
\[
v^0:=a^h_0,
\qquad
v^{q+1}:=D_S\,v^q+D_W\,\widehat g^{q}\quad(q=0,\ldots,Q-1),
\qquad
\left(\Psi(a)\right)^q:=v^q .
\]
This recurrence equation is linear, and is diagonal in the eigenbasis and causal in time. The learned parts are the decay factors and Duhamel weights as coefficients. Since these coefficients are the same at every step, it can be realized in $O(\log Q)$ depth by a parallel scan \cite{blelloch1990prefix} (Theorem~\ref{thm:realization}).
\end{enumerate}
The loop is then $A^k:=\Psi(A^{k-1})$ for $k=1,\ldots,n$, the layers (2a)-(2d) are the same at every sweep.
\item \emph{(Output:)} The output $A^n$ gives coefficients at the time nodes $t_q$. Between nodes, the trajectory is reconstructed by the semigroup lifting in Definition~\ref{def:semigroup_lift}, which applies $S_N(\tau)$ for $0<\tau<\Delta t$. The architecture has no separate layer for $S_N(\tau)$, specifically, it uses $D_S^{\tau/\Delta t}$ instead (powers applied element-wise). At the exact value $D_S=S_N(\Delta t)=e^{-\Delta t\Lambda_N}$ this equals $e^{-\tau\Lambda_N}=S_N(\tau)$. It is well defined whenever $D_S$ has positive diagonal components and it does not add parameters.
\end{enumerate}
\end{definition}

By Lemma~\ref{lem:sharing}, applied to the $MQ$ space-time points of one sweep, the block $\widehat f_\varepsilon$ is constructed by $W(\widehat f_\varepsilon)$ parameters and contributes $MQ\,W(\widehat f_\varepsilon)$ nonzero weights per sweep, thus a cost of $nMQ\,W(\widehat f_\varepsilon)$ over $n$ sweeps.

\begin{remark}[The Loop as an FNO]
\label{rem:loop_fno}
This architecture includes as a special case the FNO. On the torus, using the Fourier basis and a uniform grid, Definition~\ref{def:loop} is an FNO with parameter sharing. First, one can arrange the $Q+1$ time levels as channels. The synthesis, derivative, and analysis layers are just Fourier multipliers. Collecting (2d) is a channel-mixing multiplier (done mode-wise) that is lower triangular in the time index. The point-wise block is a sequence of Fourier layers with zero Fourier part applied block-diagonally on the channels. The result is an FNO with a channel width that is proportional to $Q$ and depth $O(n\dep(\widehat f_\varepsilon))$. Every sweep uses the same layers with the same weights. The bounds from Theorems~\ref{thm:main} and \ref{thm:total_cost} apply to this class of Fourier neural operators, with the trainable parameter count  $W(\widehat f_\varepsilon)+W_{\mathrm{spec}}$. In practice, FNOs share weights across grid points and use an autoregressive rollout in time.  Without sharing across time levels, the trainable count would grow by a factor $Q\asymp\eta^{-1/\zeta_{\mathrm t}}$. Without sharing parameters across the Picard sweeps, an extra factor $n\asymp\log\eta^{-1}$ is needed.
\end{remark}

\begin{remark}[The Loop as an Encoder-Decoder Operator]
\label{rem:loop_deeponet}
The architecture of Definition~\ref{def:loop} can also be viewed as an encoder-decoder operator network, e.g. Chen-Chen, PCA-Net, or DeepONet type. The encoder is just the linear map $w\mapsto a^h_0=D_{\mathrm{an}}\mathcal T_{\mathrm{an}}w$ on the sensor values, i.e. the evaluation $w=\mathcal I_hu_0$. The latent space processes forward using the $n$-sweeps $\mathcal A_n$. The decoder is the semigroup reconstruction from Definition~\ref{def:semigroup_lift}, which can be written as 
$$\tilde u(t,x)=\sum_{q,k}(A^n)^q_k\,\phi_{k,q}(t,x),
\qquad
\phi_{k,q}(t,x):=\mathbf 1_{[t_q,t_{q+1})}(t)\,e^{-\lambda_k(t-t_q)}\,e_k(x),$$
with the final node $t=T$ included as $\phi_{k,Q}(t,x):=\mathbf 1_{\{T\}}(t)\,e_k(x)$.

In terms of the branch-trunk construction used in DeepONet \cite{lu2021learning}, $\mathcal A_n$ is a branch network applied to the $M$ sensor values and the functions $\phi_{k,q}$ are the trunks. The decay factor $e^{-\lambda_k(t-t_q)}$ in the trunk is computed from the learned layer $D_S$, so the trunk has one trainable parameter per mode. 

In terms of PCA-Net, the PCA basis replaces the eigenbasis of $L$. As an encoder-decoder operator, the architecture belongs to the class considered in \cite{lanthaler2026parametric}, where an exponential lower bound over the Lipschitz class was shown. Our work uses the structure of the law to avoid this lower bound. 
\end{remark}

\begin{proposition}[Exact Spectral Layers]
\label{prop:spectral_exact}
Using Assumption~\ref{assumption:spectral_factorization}, if we set every learned spectral layer to its exact counterpart, the loop $\Psi$ from Definition~\ref{def:loop} agrees with the neural Picard map $\Psi^h_{N,Q,\varepsilon}$ from Definition~\ref{def:fully_discrete_neural_picard} (with $\mathcal C^h_N$ replaced by $D_\mu\mathcal C^h_N$ when needed), and the architecture computes exactly the neural Picard iterates and the semigroup reconstruction. Consequently, every associated error bound from Sections~\ref{sec:discrete}-\ref{sec:operator} holds for it, and its trainable parameters are those of $\widehat f_\varepsilon$ together with the $W_{\mathrm{spec}}\le C_{\mathrm{spec}}N$ spectral weights.
\end{proposition}

The proof is in Appendix~\ref{sec:app_discrete}.

\begin{figure}[ht]
\centering
\begin{tikzpicture}[>=Latex, font=\small,
  bx/.style={draw, rounded corners, align=center, minimum height=0.95cm, inner sep=3pt},
  fx/.style={bx, fill=black!4},
  tr/.style={bx, very thick}]
\node[fx] (enc) at (0,0) {$D_{\mathrm{an}}\mathcal T_{\mathrm{an}}$\\ \footnotesize encoder};
\node[fx] (syn) at (2.9,0) {$\mathcal T_{\mathrm{syn}},\ \mathcal T_{\mathrm{syn},j}D_j$\\ \footnotesize synthesis (2a)};
\node[tr] (blk) at (6.3,0) {$\widehat f_\varepsilon$ as $1\times1$ conv.\\ \footnotesize point-wise block (2b)};
\node[fx] (ana) at (9.7,0) {$D_{\mathcal M}D_{\mathrm{an}}\mathcal T_{\mathrm{an}}$\\ \footnotesize analysis (2c)};
\node[fx] (acc) at (12.6,0) {$D_S,\ D_W$ scan\\ \footnotesize collect (2d)};
\draw[->] (enc) -- (syn);
\draw[->] (syn) -- (blk);
\draw[->] (blk) -- (ana);
\draw[->] (ana) -- (acc);
\draw[->, thick] (acc.south) -- ++(0,-0.75) -- node[below, font=\footnotesize, pos=0.5]
   {tied weights, $n$ sweeps: $A^{k}=\Psi(A^{k-1})$} ++(-9.7,0) -- (syn.south);
\draw[->, dashed] (enc.north) -- ++(0,0.72) -- node[above, font=\footnotesize, pos=0.5]
   {input connection $D_S^{q}a^h_0$, every sweep} ++(12.6,0) -- (acc.north);
\node[font=\footnotesize, align=left] at (6.3,-2.15)
  {gray: linear, learned transforms \quad$\bullet$\quad bold: nonlinear layer};
\end{tikzpicture}
\caption{The architecture of Definition~\ref{def:loop}, drawn as a loop. One sweep consists of the four layers (2a)-(2d), the loop is applied $n$ times with the same weights, and the encoded data is connected to every sweep. The trainable parameters are the spectral weights of the shaded layers and those of the single block $\widehat f_\varepsilon$, which acts as a $1\times1$ convolution across the $M$ sample points and is reused at every one of the $Q$ time levels entering the Duhamel sum and at every sweep.}
\label{fig:loop}
\end{figure}
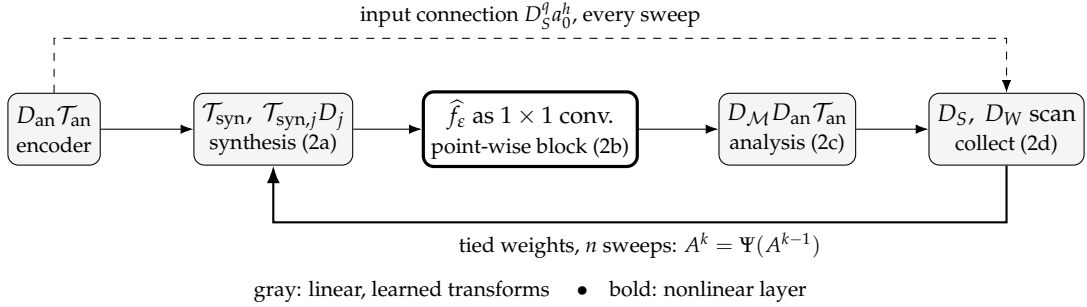

By Lemma~\ref{lem:sharing}, the block of step (2b) uses the same parameters at all $M$ grid points, $Q$ time levels and $n$ sweeps, and the spectral layers are shared across time levels and sweeps. Therefore, 
$$\#\{\text{distinct parameters}\}=W(\widehat f_\varepsilon)+W_{\mathrm{spec}},$$
is independent of $M$, $Q$ and $n$, while the weight count without sharing is at least $nMQ\,W(\widehat f_\varepsilon)$ (see Theorem~\ref{thm:realization}).

\begin{definition}[Discrete Solution Space]
\label{def:lifted_discrete_space}
For $a=(a^0,\ldots,a^Q)\in(\bR^N)^{Q+1}$, let $\mathcal J_Q$ be the piecewise-linear interpolation on the temporal grid and define the continuous lifting operation by
$$\mathcal{L}_{N,Q}a:=\mathcal{J}_Q\left(\mathcal S_Na^0,\ldots,\mathcal S_Na^Q\right)\in C([0,T];E_N),
\qquad
\|a\|_{X_{N,Q}^s}:=\|\mathcal L_{N,Q}a\|_{X_T^s},\quad s\ge0 .$$
\end{definition}

The piece-wise linear lifting fits the fixed-point formulation since its norm is the maximum of the norms at the nodes, see Proposition~\ref{prop:discrete_ball_containment}. However,  it ignores the dynamics so we reconstruct the final output using the semigroup itself, which costs one diagonal multiplier.

\begin{definition}[Semigroup Reconstruction]
\label{def:semigroup_lift}
For $a=(a^0,\ldots,a^Q)\in(\bR^N)^{Q+1}$ define $\mathcal L^S_{N,Q}a\in C([0,T];E_N)$ by
\[
\left(\mathcal L^S_{N,Q}a\right)(t):=\mathcal S_NS_N(t-t_q)\,a^q,
\qquad t\in[t_q,t_{q+1}),\ q=0,\ldots,Q-1,
\]
and $\left(\mathcal L^S_{N,Q}a\right)(T):=\mathcal S_Na^Q$. Since $S_N(\tau)$ is a contraction on $\ell^2_\alpha$ we have $\|\mathcal L^S_{N,Q}a\|_{X^\alpha_T}\le\|a\|_{X^\alpha_{N,Q}}$.
\end{definition}

The piecewise linear objects $\mathcal L_{N,Q}$ and $\mathcal J_Q$ are used only to express node-based quantities. For the difference with $u$, we use $\mathcal L^S_{N,Q}$ whose error between nodes is of order $(\Delta t)^\theta$ (see Lemma~\ref{lem:interpolation}) rather than $(\Delta t)^\zeta$. 

\begin{proposition}[Inclusion of the Discrete Ball]
\label{prop:discrete_ball_containment}
For $a\in(\bR^N)^{Q+1}$,
$$\|\mathcal L_{N,Q}a\|_{X_T^\alpha}=\max_{0\leq q\leq Q}\|a^q\|_{\ell_\alpha^2}.$$
Therefore, with $B_R^{N,Q}:=\{a\in(\bR^N)^{Q+1}: \max_q\|a^q\|_{\ell_\alpha^2} \leq R\}$, we have $\mathcal L_{N,Q}B_R^{N,Q}\subseteq B_R^\alpha$ and $\mathcal L^S_{N,Q}B_R^{N,Q}\subseteq B_R^\alpha$.
\end{proposition}

\noindent The proof is given in Appendix~\ref{sec:app_discrete}. In particular, the assumptions made on $B_R^\alpha$ apply to lifted discrete sequences, which makes parameter sharing valid.

\begin{corollary}[Discrete Input Set]
\label{cor:inputs_in_K}
Let $a\in B_R^{N,Q}$ and let $K$ be the compact set of Assumption~\ref{assumption:structured_f}. Then
$$\left((\mathcal S_N^ha^q)_i,\ (\nabla_h\mathcal S_N^ha^q)_i\right)\in K
\qquad\text{for all }q=0,\ldots,Q\text{ and }i=1,\ldots,M .$$
\end{corollary}

\noindent The proof is given in Appendix~\ref{sec:app_discrete}. Thus a network trained once on $K$ is evaluated only on $K$, at every space-time location and every Picard iterate, uniformly in $N$, $h$, $Q$ and $n$.

\subsection{Spatial Sampling and Learned Quadrature}
\label{sec:quadrature}

In the ideal setting, we would choose the quadrature and sampling to match the spectral method. However, we assume that the model is learned and thus we do not have control over the weights and associated sampling points. In this section, we place some assumptions in order to obtain reasonable approximations. 

\begin{definition}[Negative-order Spaces]
\label{def:negative_spaces}
For $s>0$, let $X^{-s}$ be the completion of $H$ under
$$\|u\|_{X^{-s}}:=\|L^{-s/2}u\|_H=\left(\sum_{k=1}^\infty\lambda_k^{-s}|\langle u,e_k\rangle_H|^2\right)^{1/2}.$$
Equivalently, $X^{-s}$ is the dual of $X^s$, and $|\langle u,v\rangle|\le\|u\|_{X^s}\|v\|_{X^{-s}}$ for $u\in X^s$, $v\in X^{-s}$.
\end{definition}

The spectral computation behind Assumption~\ref{assumption:semigroup_regular} applies directly for all real $r\ge s$, since it only uses the multiplier $\lambda^{r-s}e^{-2\lambda t}$. Thus $\|S(t)v\|_{X^r}\le C_{r-s}t^{-(r-s)/2}\|v\|_{X^s}$ holds on negative-order spaces and, in particular,
$$\|S(t)v\|_{X^{-\beta}}\le C_{\alpha-\beta} t^{-(\alpha-\beta)/2} \|v\|_{X^{-\alpha}}, \qquad v\in X^{-\alpha}.$$

\begin{assumption}[Uniform Quadrature Stability]
\label{assumption:quadrature_stability}
Let $\mathcal A\subset\mathbb N\times(0,\infty)$ be the set of admissible pairs $(N,h)$. For each $(N,h)\in\mathcal A$, let $\{x_i\}_{i=1}^{M(h)}\subset\overline{\Omega}$ be the spatial sampling locations and $\omega_i^{N,h}\ge0$ the weights, with $\sum_i\omega_i^{N,h}\le C_{\mathrm{mass}}$ independent of $(N,h)$. The quadrature rule is written as $Q_{N,h}(v,\phi):=\sum_i\omega_i^{N,h}v(x_i)\phi(x_i)$. Assume there are constants $C^0_{\mathrm{stab}},C^\beta_{\mathrm{stab}},C^\alpha_{\mathrm{stab}}>0$, independent of $(N,h)\in\mathcal A$, such that:
\begin{enumerate}
\item[(i)] For every $\phi\in E_N$, $Q_{N,h}(\phi,\phi)\le C^0_{\mathrm{stab}}\|\phi\|_H^2$. When $f$ depends on $\nabla u$, we also assume $\sum_i\omega_i^{N,h}|\nabla\phi(x_i)|^2\le C^0_{\mathrm{stab}}\|\nabla\phi\|_H^2$;
\item[(ii)] For $\beta>d_x/2$ and for $s\in\{\beta,\alpha\}$
$$\left|Q_{N,h}(v,\phi)\right|\le C^s_{\mathrm{stab}}\|v\|_{X^s}\|\phi\|_{X^{-s}}$$
for every $v\in X^s\cap C(\overline\Omega)$ and every $\phi\in E_N$.
\end{enumerate}
\end{assumption}

We write $C_{\mathrm{stab}}:=\max_s C^s_{\mathrm{stab}}$. For any grid function $v=(v_i)_{i=1}^M$, we also write $Q_{N,h}(v,\phi):=\sum_i\omega_i^{N,h}v_i\,\phi(x_i)$. Since the weights are nonnegative, the Cauchy-Schwarz inequality and (i) give, for every $\phi\in E_N$,
\begin{equation}
\left|Q_{N,h}(v,\phi)\right|\le(C^0_{\mathrm{stab}})^{1/2}\|v\|_{\omega}\|\phi\|_H,
\qquad
\|v\|_\omega:=\left(\sum_i\omega_i^{N,h}|v_i|^2\right)^{1/2}\le C_{\mathrm{mass}}^{1/2}\max_i|v_i| .
\label{eqn:sampling_pairing}
\end{equation}
Assumption~\ref{assumption:quadrature_stability}(ii) with $s=\alpha$ also implies that the discrete encoder is bounded on $X^\alpha$, uniformly in the discretization, shown in the following lemma.

\begin{lemma}[Discrete Encoder]
\label{lem:encoder_bounded}
Let $(N,h)\in\mathcal A$ and $\Pi_{N,h}:=\mathcal S_N\mathcal C_N^h\mathcal I_h:X^\alpha\to E_N$. If Assumptions~\ref{assumption:regularity_embedding} and \ref{assumption:quadrature_stability} hold, then
$$\left\|\mathcal C^h_N\mathcal I_hv\right\|_{\ell^2_\alpha}=\left\|\Pi_{N,h}v\right\|_{X^\alpha}\le C_{\mathrm{stab}}^\alpha\|v\|_{X^\alpha},$$
for every $v\in X^\alpha$.
\end{lemma}

\noindent The proof is given in Appendix~\ref{sec:app_discrete}.

\begin{assumption}[Uniform Quadrature Consistency]
\label{assumption:quadrature_consistency}
There exist $r>0$ and $C_{\mathrm{quad}}>0$, independent of $(N,h)\in\mathcal A$, such that
$$\left\|\mathcal C_N^h\mathcal I_h v-\mathcal C_Nv\right\|_{\ell_s^2}\leq C_{\mathrm{quad}}h^r\|v\|_{X^{s+r}},$$
for $s\in\{\alpha,\beta\}$ and every $v\in X^{s+r}$. We require that
\begin{equation}
r\ \le\ \gamma_{\mathrm s}-\alpha ,
\label{eqn:quad_order_cap}
\end{equation}
so that the right-hand side is finite for the regularity in Lemma~\ref{lem:derived_regularity}.
\end{assumption}

The limit imposed by \eqref{eqn:quad_order_cap} come from the regularity lemmas. Without the bootstrap the limit is set by $\gamma<\beta+2$, and with the bootstrap it is set by the data and the law (see Remark~\ref{rem:gamma_free} and Lemma~\ref{lem:bootstrap}). We will show that this makes the exponent $d_x/r$ of Theorem~\ref{thm:realization} at least $d_x/(\gamma_{\mathrm s}-\alpha)$.

\begin{remark}[A uniformly stable and consistent quadrature]
\label{rem:stable_quadrature_example}
On the torus $\Omega=\mathbb T^{d_x}$ take $L=-\nu\Delta+c_0I$ with $\nu,c_0>0$. The term $c_0u$ can be absorbed into the constitutive law. The eigenfunctions are the standard trigonometric functions with $\lambda_k=\nu|k|^2+c_0$. The retained modes form a ball $|k|\le R$, which is in the box $|k|_\infty\le\Lambda$ with $\Lambda\le R$. On the uniform grid with $M_1\ge2\Lambda+1$ points in each spatial direction, the trapezoidal rule is exact on products of these retained modes. Therefore, $\mathcal C^h_N\mathcal S^h_N=I$, (i) holds with constant $1$ (also for the gradients inequality), (ii) follows from the embedding $H^\beta\hookrightarrow C(\mathbb T^{d_x})$, and Assumption~\ref{assumption:quadrature_consistency} holds for every $r>0$. The same holds on the square $(0,\pi)^{d_x}$ with Dirichlet conditions where the eigenfunctions are the sine modes on the interior grid $x_j=\pi j/M_1$, as long as $M_1\ge\Lambda+1$. Since $\lambda_N\ge\nu R^2\ge\nu\Lambda^2$ for the largest retained $|k|=R$, the resolution condition $h\lambda_N^{1/2}\le c_{\mathrm{res}}$ implies the scale $c_{\mathrm{res}}\le\frac{2\pi}{3}\nu^{1/2}$ on the torus or $c_{\mathrm{res}}\le\frac{\pi}{2}\nu^{1/2}$ on the square.
\end{remark}
For a second-order operator $L$, the admissible discretizations will impose a resolution condition $h\lambda_N^{1/2}\le c_{\mathrm{res}}$. In two dimensions, by Weyl's law \cite{weyl1912asymptotische} we expect $\lambda_N\asymp N$ and thus $N\lesssim h^{-2}$. It agrees with the parameter choices of Theorem~\ref{thm:main} precisely when $r\le\gamma_{\mathrm s}-\alpha$.

\subsection{Contraction and Consistency of the Discrete Maps}
\label{sec:discrete_contraction}

We need to control the loop depth of the Picard iteration in the discrete setting, but contractive bounds on the neural Picard map are not available. The neural approximations are in the sup-norm while contractive or fixed-point methods require a Lipschitz bound. Instead, we show contraction of the fully discrete map and consistency between the neural approximation and the fully discrete one.

\begin{theorem}[Fully discrete fixed point]
\label{thm:fully_discrete_fixed_point}
Let $(N,h)\in\mathcal A$ and $Q\in\mathbb N$, and set
$\Delta t=T/Q$ and $t_q=q\Delta t$ for $q=0,\ldots,Q$.
Assume the conditions of Proposition~\ref{prop:contraction} on $B_R^\alpha$ hold, and set $\theta:=1-\frac{\alpha-\beta}{2}>0$.
Assume, in addition, that the spatial quadrature satisfies
Assumption~\ref{assumption:quadrature_stability}. Define $F_R:=\sup_{u\in B_R^\alpha} \|f(u,\nabla u)\|_{X_T^\beta}$ and let
$a_0^h:=\mathcal C_N^h\mathcal I_hu_0$ be in the discrete ball of radius $R$.

Then there exists a constant $C>0$, depending only on the semigroup estimate and on the difference $\alpha-\beta$, but independent of $N$, $h$, and $Q$, such that, for every $a\in B_R^{N,Q}$,
$$\|\Psi_{N,Q}^h(a)\|_{X_{N,Q}^\alpha} \leq \|a_0^h\|_{\ell_\alpha^2}+C C_{\mathrm{stab}}F_RT^\theta,$$
and for every $a,b\in B_R^{N,Q}$,
$$\|\Psi_{N,Q}^h(a)-\Psi_{N,Q}^h(b)\|_{X_{N,Q}^\alpha}
\leq \kappa\|a-b\|_{X_{N,Q}^\alpha},$$
with contraction constant $\kappa:=C C_{\mathrm{stab}}L_RT^\theta$, independent of $N$, $h$, and $Q$.

If $ \|a_0^h\|_{\ell_\alpha^2} + C C_{\mathrm{stab}}F_RT^\theta \leq R$, then $\Psi_{N,Q}^h(B_R^{N,Q}) \subseteq B_R^{N,Q}$. In addition, if $\kappa<1$, then there exists a unique $a_{N,Q}^h\in B_R^{N,Q}$ satisfying $a_{N,Q}^h=\Psi_{N,Q}^h(a_{N,Q}^h)$, i.e., a unique discrete fixed point.
\end{theorem}

The proof is given in Appendix~\ref{sec:app_discrete}. It differs from standard approximation results in numerical analysis in that the quadrature estimates build in the smoothing of the discrete semigroup.

\begin{theorem}[Neural Consistency]
\label{thm:neural_consistency}
Let $\Psi_{N,Q}^h$ and $\Psi_{N,Q,\varepsilon}^h$ denote the fully discrete and neural Picard maps of Definition~\ref{def:fully_discrete_neural_picard}. Assume that the conditions from
Proposition~\ref{prop:contraction}, Assumption~\ref{assumption:structured_f}, and Assumption~\ref{assumption:quadrature_stability} hold, and  that the neural approximation satisfies
$$\sup_{z\in K}
\left\|f(z)-\widehat f_\varepsilon(z)\right\|_{\bR^m}\leq \varepsilon$$
on the compact set $K$ from Assumption~\ref{assumption:structured_f}. Then there exists $C_R>0$, independent of $N$, $h$, $Q$, and
$\varepsilon$, such that for every $a\in B_R^{N,Q}$
$$\left\| \Psi_{N,Q,\varepsilon}^h(a) - \Psi_{N,Q}^h(a) \right\|_{X_{N,Q}^\alpha} \leq C_R \, \varepsilon\, \, \xi_\alpha(N,\Delta t),$$
where the factor is
$$\xi_\alpha(N,\Delta t):=\begin{cases}\min\left\{\lambda_N,\ \Delta t^{-1}\right\}^{\alpha/2-1}, & \alpha>2,\\[2pt] 1+\log\left(1+T\min\{\lambda_N,\ \Delta t^{-1}\}\right), & \alpha=2,\\ 1, & \alpha<2,\end{cases}$$
and $C_R$ may depend on $T$ when $\alpha\le2$.
The same bound holds for the lifted maps $\Phi^h_{N,Q,\varepsilon}[u]-\Phi^h_{N,Q}[u]$ in $X^\alpha_T$ for every $u\in B_R^\alpha$ with $b(u)\in B^{N,Q}_R$.
\end{theorem}

\begin{proof}[Proof of Theorem~\ref{thm:neural_consistency}]
For the time points $r=0,\ldots,Q-1$, the error from the neural network approximation is
$$g_\varepsilon^r := \widehat{\mathcal F}_\varepsilon^h \left( \mathcal S_N^h a^r, \nabla_h\mathcal S_N^h a^r \right)
- \mathcal F^h \left( \mathcal S_N^h a^r, \nabla_h\mathcal S_N^h a^r \right). $$
By Corollary~\ref{cor:inputs_in_K}, the inputs to the local constitutive maps live in the compact $K$, and thus 
$$\|g_\varepsilon^r\|_{{\ell}^\infty} :=\max_{1\le i\le M(h)}\left\|(g^r_\varepsilon)_i\right\|_{\bR^m} \leq \varepsilon.$$
Taking the difference and applying the continuous Galerkin operation leads to
$$\begin{aligned}
\mathcal S_N\left(\Psi_{N,Q,\varepsilon}^h(a) - \Psi_{N,Q}^h(a)\right)^q
&= \mathcal S_N \left[ \Delta t \sum_{r=0}^{q-1} S_N(t_q-t_r)
\mathcal C_N^h g_\varepsilon^r\right],
\end{aligned}$$
where the initial value terms match and cancel out. 
Since $\mathcal S_N$ is an isometry from
$\ell_0^2$ to $E_N$ equipped with the $X^0$-norm,
we first estimate $\mathcal C_N^h$ on bounded grid functions. For $w\in\bR^{M(h)}$,
$$\begin{aligned}
\|\mathcal C_N^h w\|_{\ell_0^2} &=
\sup_{\|c\|_{\ell_0^2}=1} \left| \sum_{i=1}^{M(h)} \omega_i^{N,h} w_i(\mathcal S_Nc)(x_i) \right|
\leq \|w\|_{{\ell}^\infty} \left(\sum_{i=1}^{M(h)}\omega_i^{N,h}\right)^{1/2} \sup_{\substack{\phi\in E_N\\\|\phi\|_{X^0}=1}}
\left( \sum_{i=1}^{M(h)} \omega_i^{N,h}|\phi(x_i)|^2 \right)^{1/2},
\end{aligned}
$$
where we set $\phi:=\mathcal S_Nc$, so that $\|\phi\|_{X^0}=\|c\|_{\ell^2_0}$ by orthonormality of the $e_k$.
By the uniform bound on the total mass and the sampling inequality of
Assumption~\ref{assumption:quadrature_stability}(i),
$$\sum_{i=1}^{M(h)} \omega_i^{N,h}|\phi(x_i)|^2 \leq C_{\mathrm{stab}}^0\|\phi\|_{X^0}^2.$$
Thus taking the sup over norm $1$ functions, $\|\mathcal C_N^h w\|_{\ell_0^2} \leq C_{\mathrm{samp}}\|w\|_{{\ell}^\infty},$ where $C_{\mathrm{samp}}:= (C_{\mathrm{mass}}C_{\mathrm{stab}}^0)^{1/2},$
and therefore, since $\mathcal C^h_N$ acts component-wise on the $m$ components of $g^r_\varepsilon$, each of which has ${\ell}^\infty$ norm at most $\varepsilon$,
\[
\|\mathcal C_N^h g_\varepsilon^r\|_{\ell_0^2}
=\left(\sum_{j=1}^m\|\mathcal C^h_N(g^r_\varepsilon)^{(j)}\|_{\ell^2_0}^2\right)^{1/2}
\leq
\sqrt m\,C_{\mathrm{samp}}\,\varepsilon ,
\]
and the $\sqrt m$ factor can be absorbed by $C_R$ in the next steps.

For $\tau>0$, the spectral form of $S_N(\tau)$ gives
$$\begin{aligned}
\left\|
\mathcal S_NS_N(\tau)\mathcal C_N^h g_\varepsilon^r
\right\|_{X^\alpha} &=
\left( \sum_{k=1}^N \lambda_k^\alpha e^{-2\lambda_k\tau} \left|(\mathcal C_N^h g_\varepsilon^r)_k\right|^2 \right)^{1/2}
\\ 
&\leq \min\left\{ \lambda_N^{\alpha/2}, C_\alpha\tau^{-\alpha/2}
\right\} \|\mathcal C_N^h g_\varepsilon^r\|_{\ell_0^2}\\
 &\leq C_R\varepsilon \min\left\{ \lambda_N^{\alpha/2}, \tau^{-\alpha/2}  \right\}.
\end{aligned}
$$
And so the difference is bounded by:
$$\left\|
\mathcal S_N\left(\Psi_{N,Q,\varepsilon}^h(a) - \Psi_{N,Q}^h(a)\right)^q \right\|_{X^\alpha}
\leq C_R\varepsilon\Delta t \sum_{j=1}^{q} \min\left\{ \lambda_N^{\alpha/2}, (j\Delta t)^{-\alpha/2} \right\}.$$
We estimate the discrete sum, first in the setting where $\alpha>2$, so $\alpha/2>1$. If $\lambda_N\Delta t\leq\frac{1}{2}$, set $J:=\left\lfloor\frac{1}{\lambda_N\Delta t}\right\rfloor$, so that $J\ge1$ and $J\geq \frac{1}{2\lambda_N\Delta t}$.
For the terms with $j\leq J$,
$$\Delta t \sum_{j=1}^{\min\{q,J\}} \lambda_N^{\alpha/2}
\leq \Delta t\,J\,\lambda_N^{\alpha/2}
\leq \lambda_N^{\alpha/2-1},$$
where we used $\Delta t J\le\lambda_N^{-1}$. For the remaining terms, note that since $s^{-\alpha/2}$ is monotone, we have the upper bound to the integral by
$$\Delta t
\sum_{j=J+1}^{q} (j\Delta t)^{-\alpha/2}
\leq \int_{J\Delta t}^{\infty}s^{-\alpha/2}\,ds
= \tfrac{2}{\alpha-2}(J\Delta t)^{1-\alpha/2}
\leq C\lambda_N^{\alpha/2-1},$$
using $J\Delta t\ge(2\lambda_N)^{-1}$. Thus, in the first case,
$$\Delta t \sum_{j=1}^{q}
\min\left\{ \lambda_N^{\alpha/2}, (j\Delta t)^{-\alpha/2} \right\}
\leq C\lambda_N^{\alpha/2-1}.$$
If $\lambda_N\Delta t>\frac{1}{2}$, then
$\min\{\lambda_N^{\alpha/2},(j\Delta t)^{-\alpha/2}\}\leq(j\Delta t)^{-\alpha/2}$,
and so
$$\Delta t \sum_{j=1}^{q} (j\Delta t)^{-\alpha/2}
\leq \Delta t^{1-\alpha/2}\sum_{j=1}^{\infty}j^{-\alpha/2}
\leq C\Delta t^{1-\alpha/2},$$
where we used that $\alpha/2>1$ so the sum is bounded.

Both settings give a bound of the form $C\min\{\lambda_N,\Delta t^{-1}\}^{\alpha/2-1}$. In the first setting, $\lambda_N\le(2\Delta t)^{-1}<\Delta t^{-1}$, so $\min\{\lambda_N,\Delta t^{-1}\} = \lambda_N$. In the second setting, $\Delta t^{-1}< 2 \lambda_N$, so $\Delta t^{-1}\le 2 \min\{\lambda_N,\Delta t^{-1}\}$, and since $\alpha/2-1>0$,
$$\Delta t^{1-\alpha/2}=(\Delta t^{-1})^{\alpha/2-1}\le 2^{\alpha/2-1} \min\{\lambda_N,\Delta t^{-1}\}^{\alpha/2-1}.$$ Therefore, in all cases,
$$\Delta t \sum_{j=1}^{q} \min\left\{ \lambda_N^{\alpha/2},
(j\Delta t)^{-\alpha/2} \right\}
\leq C\,\min\left\{\lambda_N,\ \Delta t^{-1}\right\}^{\alpha/2-1},$$
where $C$ depends only on $\alpha$ and is independent of $N$, $h$, $Q$, and $q$. 

For $\alpha<2$, the summand $(j\Delta t)^{-\alpha/2}$ is integrable at the origin, and since it is decreasing, 
$$\Delta t\sum_{j=1}^{q}(j\Delta t)^{-\alpha/2}\le\int_0^Ts^{-\alpha/2}\,ds=\frac{T^{1-\alpha/2}}{1-\alpha/2}.$$ 

For $\alpha=2$ the same splitting applies. If we have $\lambda_N\Delta t\le\frac{1}{2}$, the terms $j\le J$ contribute at most $\Delta t\,J\lambda_N\le1$, and the rest gives 
$$\sum_{j=J+1}^{q}j^{-1} \le \log(q/J) \le \log(2T\lambda_N),$$
since  $q\le T/\Delta t$ and $J\ge(2\lambda_N\Delta t)^{-1}$. If we have $\lambda_N\Delta t>\frac{1}{2}$, the sum is at most $$\sum_{j=1}^qj^{-1} \le 1+\log(T/\Delta t) \le 1+\log\left(2 T \min\{\lambda_N,\Delta t^{-1}\} \right).$$
 In both cases the sum is at most $3\,\xi_2(N,\Delta t)$. Taking the maximum over $q$ and using Definition~\ref{def:lifted_discrete_space} gives the result, where the statement for the lifted maps comes from applying the estimates to $a=b(u)$.
\end{proof}

We kept the minimum of $\lambda_N$ and $\Delta t^{-1}$ instead of bounding it by one term since the bound is sharp (see Proposition~\ref{prop:xi_sharp}) and under the choices of Theorem~\ref{thm:main} the two entries are comparable.
The cases of $\xi_\alpha$ come from the discrete sum in its proof, specifically, the integrability of $t^{-\alpha/2}$ for different $\alpha$'s. The network error acts as a perturbation of the forcing of size $\varepsilon$ at every time level. The discrete Duhamel sum maps it into $X^\alpha$ with the kernel $\min\{\lambda_N^{\alpha/2},\tau^{-\alpha/2}\}$, which for $\alpha>2$ is not integrable near $\tau=0$ with error scaling with $\min\{\lambda_N,\Delta t^{-1}\}^{\alpha/2-1}$. If we measure the error in $X^s$ with $s<2$, then the kernel is integrable and we avoid this issue.

\begin{theorem}[Consistency of the fully discrete mild map]
\label{thm:fully_discrete_consistency}
Let $u\in B_R^\alpha\subset X_T^\alpha$ be the mild solution of $u=\Phi[u]$ with the conditions from Propositions~\ref{prop:contraction} and \ref{prop:instantaneous_regularity}, and let $\Phi_{N,Q}^h$ be the fully discrete mild map. Assume that the spatial discretization satisfies Assumptions~\ref{assumption:quadrature_stability} and \ref{assumption:quadrature_consistency} with $0<r\leq\gamma_{\mathrm s}-\alpha$, that
$$F(t):=f(u(t),\nabla u(t))\in L^\infty\left(0,T;X^\beta\right),
\qquad \|F\|_{\infty,\beta}+[F]_{\zeta}\leq C_R,$$
and, 
$$\sup_{\substack{(N,h)\in\mathcal A\\0<t\leq T}}
t^{\varrho}\,\|f(P_Nu(t),\nabla P_Nu(t))\|_{X^{\beta+r}} \leq C_R,$$
where $\varrho=0$ if $\gamma_{\mathrm s}=\gamma'>\gamma$ or otherwise $\varrho=r/2$.
Define $E_{\mathrm{init}}^h :=\|a_0^h-\mathcal C_N u_0 \|_{\ell_\alpha^2} $. Then there exists $C_R>0$, independent of $N$, $h$ and $Q$, such that
$$\max_{0\leq q\leq Q}
\left \|\left(\Phi_{N,Q}^h[u]\right)(t_q)-u(t_q)\right\|_{X^\alpha}
\leq C_R \left[E_{\mathrm{init}}^h+\|(I-P_N) u_0\|_{X^\alpha} + \lambda_{N+1}^{-(\gamma_{\mathrm s}-\alpha)/2} + h^r + (\Delta t)^{\zeta_{\mathrm t}}\right]=:E_{\mathrm{cons}}.$$
The same bound holds for $\|\Phi_{N,Q}^h[u]-\mathcal J_Qu\|_{X_T^\alpha}$, where $\mathcal J_Qu$ is the piecewise-linear interpolant of the values $u(t_q)$.
\end{theorem}

The proof is given in Appendix~\ref{sec:app_discrete}.
The bound for $f(P_Nu,\nabla P_Nu)$ follows from Assumption~\ref{assumption:bootstrap} with $\gamma'\ge\alpha+r$. For $\varrho=0$, this is shown in the proof. For $\varrho=r/2$ it follows from the estimate $\|f(v,\nabla v)\|_{X^{\beta+r}} \le K_R(1+\|v\|_{X^{\alpha+r}})$ for $\|v\|_{X^\alpha}\le R$, i.e., Assumption~\ref{assumption:bootstrap} with $s=\alpha+r$, together with the interpolation inequality and Lemma~\ref{lem:derived_regularity} Part 2:
$$\|P_Nu(t)\|_{X^{\alpha+r}}\le\|u(t)\|_{X^\alpha}^{1-r/(2\zeta)}\|u(t)\|_{X^\gamma}^{r/(2\zeta)}\le R^{1-r/(2\zeta)}R_\gamma^{r/(2\zeta)}\,t^{-r/2}.$$
The weighted norm is needed when we only assume $X^\alpha$, since $\|P_Nu_0\|_{X^{\alpha+r}}$ cannot be controlled near zero.

\section{Total Error and Structural Complexity}
\label{sec:total}

In this section, we put together the prior consistency, contraction, and smoothing estimates to obtain an upper bound on the number of parameters needed to construct a neural network approximation to the solution of the class of semilinear equations considered. Throughout this section, we assume the conditions and results from Theorems~\ref{thm:fully_discrete_fixed_point}, \ref{thm:neural_consistency}, \ref{thm:fully_discrete_consistency}, and Lemma~\ref{lem:derived_regularity} hold. Let
$$b=(b^0,\ldots,b^Q),
\qquad
b^q:=\mathcal C_N^h\mathcal I_h u(t_q),$$
denote coefficients of the exact solution after sampling on the grid (so sampling then transformation).  We use just the contraction of the discrete Picard map $\Psi^h_{N,Q}$, whose constant is $\kappa<1$ and thus by Theorem~\ref{thm:fully_discrete_fixed_point} has a unique fixed point  $a_\infty\in B_R^{N,Q}$. We first record two elementary lemmas.

\begin{lemma}[Reconstruction]
\label{lem:interpolation}
Let $u$ be the mild solution with $\|u\|_{X^\alpha_T}\le R$ and $F_R=\sup_{v\in B_R^\alpha} \|f(v,\nabla v)\|_{X^\beta_T}$, and let $a\in(\bR^N)^{Q+1}$. Then
$$\left\|\mathcal L^S_{N,Q}a-u\right\|_{X^\alpha_T}
\le \underbrace{\max_{0\le q\le Q}\left\|\mathcal S_Na^q-u(t_q)\right\|_{X^\alpha}}_{\text{nodal error}}
 + C\, F_R\,(\Delta t)^{\theta},$$
where $C$ depends only on the difference $\alpha-\beta$. If Assumption~\ref{assumption:bootstrap} holds and $u_0\in X^{\gamma'}$, the term $CF_R(\Delta t)^\theta$ improves to $CK_R(1+R_{\gamma'})(\Delta t)^{\min\{1,\theta+\zeta_{\mathrm s}\}}$.
\end{lemma}

\noindent The proof uses the mild formulation starting at each node and is given in Appendix~\ref{sec:app_discrete}. The reconstruction error is order $(\Delta t)^\theta$ rather than $(\Delta t)^{\zeta}$, which is better since $\zeta<\theta$. It does not require additional regularity of $u$ beyond what is given through the mild formulation. Using the bootstrap, it is at most a constant times $(\Delta t)^{\zeta_{\mathrm t}}$.

When the dynamics only yields data in $X^\alpha$, we cannot control parts of the error in the $X^\gamma$ norm (which will be unbounded as $t\to 0^+$). Instead we split the error into a part that is controlled by the semi-group smoothing (an algebraic rate) and a part that measures the tail of the data (no explicit rate). We define the tail of the initial data below.

\begin{definition}[Input Data Tails and Encoding]
\label{def:data_moduli}
For $u_0\in X^\alpha$ and admissible $(N,h)$, let
$$\varepsilon^{\mathrm{tail}}_N(u_0):=\left\|(I-P_N)u_0\right\|_{X^\alpha},
\qquad \varepsilon^{\mathrm{enc}}_{N,h}(u_0):=\sup_{0\le t\le T}\left\|\left(\mathcal C^h_N\mathcal I_h-\mathcal C_N\right)S(t)u_0\right\|_{\ell^2_\alpha}.$$
\end{definition}

The supremum over the entire dynamics $S(t)u_0$ is needed since $\mathcal C^h_N\mathcal I_h$ only commutes with $S(t)$ when the quadrature is exact, which can lead to different decay rates. Both terms in the definition tend to $0$ but they do not admit rates directly. When extra regularity is given, we have $\varepsilon^{\mathrm{tail}}_N(u_0)\le\lambda_{N+1}^{-\zeta_{\mathrm s}}\|u_0\|_{X^{\gamma_{\mathrm s}}}$ if $u_0\in X^{\gamma_{\mathrm s}}$, and $\varepsilon^{\mathrm{enc}}_{N,h}(u_0)\le C_{\mathrm{quad}}h^r\|u_0\|_{X^{\alpha+r}}$ if $u_0\in X^{\alpha+r}$.

\begin{lemma}[Sampling Error of Mild Solution]
\label{lem:sampling}
Let $u\in B^\alpha_R$ be the mild solution with initial $u_0 \in X^\alpha$ and let Assumption~\ref{assumption:quadrature_consistency} hold with $0<r\le\gamma_{\mathrm s}-\alpha$. Then
$$E_{\mathrm{samp}} :=\max_{0\le q\le Q}\left\|\mathcal S_Nb^q-u(t_q)\right\|_{X^\alpha}
\le \varepsilon^{\mathrm{tail}}_N(u_0)+\varepsilon^{\mathrm{enc}}_{N,h}(u_0) + C_R \left(\lambda_{N+1}^{-\zeta_{\mathrm s}}+h^r\right),$$
where $C_R$ depends on $R$, $T$, $F_R$ and the constants of Assumptions~\ref{assumption:semigroup_regular} and \ref{assumption:quadrature_consistency}, and when using bootstrap, it also depends on $R_{\gamma'}$ and $\|u_0\|_{X^{\gamma'}}$. 

In particular, $b\in B^{N,Q}_R$ if $\|u\|_{X^\alpha_T}\le R_0<R$ and $E_{\mathrm{samp}}\le R-R_0$.
\end{lemma}

\noindent The proof splits the solution into its linear dynamics and the Duhamel operator, and is given in Appendix~\ref{sec:app_discrete}.

\subsection{Total Error of the Neural Picard Iteration}

We look at the total error induced by the Neural Picard iteration and compare the $n$-iterate to the continuous mild solution.
\begin{theorem}[Total error]
\label{thm:total_error}
Let $u$ be the mild solution with $\|u\|_{X^\alpha_T} \le R_0<R$, and assume the following hold:
\begin{enumerate}
\item[(i)] $\kappa<1$ and $\Psi^h_{N,Q}(B^{N,Q}_R)\subseteq B^{N,Q}_{R_0}$;
\item[(ii)] $E_{\mathrm{samp}}\le R-R_0$, so that $b\in B^{N,Q}_R$ by Lemma~\ref{lem:sampling};
\item[(iii)] $E_{\mathrm{net}}:=C_R\,\varepsilon\,\xi_\alpha(N,\Delta t)$, with $\xi_\alpha$ from Theorem~\ref{thm:neural_consistency}, satisfies
\begin{equation}
\frac{E_{\mathrm{net}}}{1-\kappa}\le R-R_0 .
\label{eqn:margin}
\end{equation}
\end{enumerate}
For $a^0\in B^{N,Q}_{R_0}$, define the exact and neural Picard iterates by $a^{n}:=\Psi^h_{N,Q}(a^{n-1})$,
$a^{n}_\varepsilon:=\Psi^h_{N,Q,\varepsilon}(a^{n-1}_\varepsilon)$, where $a^0_\varepsilon:=a^0$. Then $a^n,a^n_\varepsilon\in B^{N,Q}_R$ for every $n\ge0$, and
$$\left\|\mathcal L^S_{N,Q}a^n_\varepsilon-u\right\|_{X^\alpha_T}
\le \underbrace{\kappa^{n}\left\|a^0 -a_\infty\right\|_{X^\alpha_{N,Q}}}_{\text{iteration}}
+\underbrace{\frac{E_{\mathrm{net}}}{1-\kappa}}_{\text{network}}
+\underbrace{\frac{E_{\mathrm{cons}}+E_{\mathrm{samp}}}{1-\kappa}+E_{\mathrm{samp}}}_{\text{discretization}}
+\underbrace{E_{\mathrm{rec}}}_{\text{reconstruction}},$$
where $E_{\mathrm{cons}}$, $E_{\mathrm{samp}}$ and $E_{\mathrm{rec}}$ are the error terms from Theorem~\ref{thm:fully_discrete_consistency}, Lemma~\ref{lem:sampling} and Lemma~\ref{lem:interpolation}. Consequently there is $C_R>0$, independent of $N$, $h$, $Q$, $n$ and $\varepsilon$, such that
$$\left\|\mathcal L^S_{N,Q}a^n_\varepsilon-u\right\|_{X^\alpha_T}
\le C_R\left[ \kappa^{n} +\varepsilon\,\xi_\alpha(N,\Delta t) +\varepsilon^{\mathrm{tail}}_N(u_0) +\varepsilon^{\mathrm{enc}}_{N,h}(u_0) +\lambda_{N+1}^{-\zeta_{\mathrm s}} +h^{r} +(\Delta t)^{\zeta_{\mathrm t}}
\right].$$
\end{theorem}

Condition (i) holds uniformly in $N$, $h$ and $Q$ once we set  $R_0\ge\max\{\|u\|_{X^\alpha_T},\,2C^\alpha_{\mathrm{stab}}\|u_0\|_{X^\alpha}\}$ and $R>R_0$, with possible shrinking of the final time $T$. This can be seen by applying Theorem~\ref{thm:fully_discrete_fixed_point} which needs $\|a^h_0\|_{\ell^2_\alpha}+CC_{\mathrm{stab}}F_RT^\theta\le R_0$ and $CC_{\mathrm{stab}}L_RT^\theta<1$. We have $\|a^h_0\|_{\ell^2_\alpha}\le C^\alpha_{\mathrm{stab}}\|u_0\|_{X^\alpha}\le R_0/2$ by Lemma~\ref{lem:encoder_bounded}. In the final bound, we use that $E^h_{\mathrm{init}}\le\varepsilon^{\mathrm{enc}}_{N,h}(u_0)$ and $\|(I-P_N)u_0\|_{X^\alpha}=\varepsilon^{\mathrm{tail}}_N(u_0)$, while $E_{\mathrm{rec}}$ can be absorbed into $(\Delta t)^{\zeta_{\mathrm t}}$.

\begin{proof}[Proof of Theorem~\ref{thm:total_error}]
We set the norm for the coefficient sequences as $\|\cdot\|_{X^\alpha_{N,Q}}$ of Definition~\ref{def:lifted_discrete_space}, use $\|\cdot\|$ for short.

We want control over the neural discrete map, but will use the non-neural discrete map as a first approximation. First, we can show that the neural iterates remain admissible and bound its closeness to the discrete iterates. Since $a^0\in B^{N,Q}_{R_0}\subseteq B^{N,Q}_R$ and $\Psi^h_{N,Q}$ maps $B^{N,Q}_R$ into $B^{N,Q}_{R_0}$ by Theorem~\ref{thm:fully_discrete_fixed_point}, induction on $n$ shows that $a^n\in B^{N,Q}_{R_0}$ for every $n\ge1$. For the neural iterates, we claim that
\begin{equation}
a^n_\varepsilon\in B^{N,Q}_R
\qquad\text{and}\qquad
\left\|a^n_\varepsilon-a^n\right\|\ \le\ \frac{E_{\mathrm{net}}}{1-\kappa}.
\label{eqn:neural_drift}
\end{equation}
for all $n\geq 0$. For $n=0$ it is true since $a^0_\varepsilon=a^0\in B^{N,Q}_{R_0}\subseteq B^{N,Q}_R$. Next, if \eqref{eqn:neural_drift} holds for $n-1$, then $a^{n-1}_\varepsilon$ and $a^{n-1}$ both lie in $B^{N,Q}_R$, so we can apply the consistency Theorem~\ref{thm:neural_consistency} to $a^{n-1}_\varepsilon$ and the Lipschitz estimate of Theorem~\ref{thm:fully_discrete_fixed_point} to the difference, to get:
\begin{align}
\left\|a^{n}_\varepsilon-a^{n}\right\|
&\le \left\|\Psi^h_{N,Q,\varepsilon}(a^{n-1}_\varepsilon)-\Psi^h_{N,Q}(a^{n-1}_\varepsilon)\right\|
+\left\|\Psi^h_{N,Q}(a^{n-1}_\varepsilon)-\Psi^h_{N,Q}(a^{n-1})\right\| \label{eqn:total_error_eq1}\\
&\le E_{\mathrm{net}} + \kappa\left\|a^{n-1}_\varepsilon-a^{n-1} \right\|
\label{eqn:total_error_eq2}\\
&\le E_{\mathrm{net}} + \frac{\kappa\,E_{\mathrm{net}}}{1-\kappa}
 =\frac{E_{\mathrm{net}}}{1-\kappa},
\label{eqn:total_error_eq3}
\end{align}
where we use iteration definitions, Theorem~\ref{thm:neural_consistency} and Theorem~\ref{thm:fully_discrete_fixed_point} for \eqref{eqn:total_error_eq2}, and the induction hypothesis for \eqref{eqn:total_error_eq3}. Combining the resulting bound with $a^n\in B^{N,Q}_{R_0}$ and the radius condition \eqref{eqn:margin},
$$\left\|a^n_\varepsilon\right\|\ \le\ \left\|a^n\right\|+\left\|a^n_\varepsilon-a^n\right\|\ \le\ R_0+\frac{E_{\mathrm{net}}}{1-\kappa}\ \le\ R_0+(R-R_0)\ =\ R ,$$
which shows the induction and proves \eqref{eqn:neural_drift}.

Next, we look at the non-neural discrete iteration, which converges to its fixed point by Theorem~\ref{thm:fully_discrete_fixed_point}, $\Psi^h_{N,Q}$ using that it is a $\kappa$-contraction on $B^{N,Q}_R$. Specifically, there exists a unique fixed point $a_\infty$ and 
\begin{equation}
\left\|a^n-a_\infty\right\|\ \le\ \kappa^{n}\left\|a^0-a_\infty\right\|\ \le\ 2R\,\kappa^{n},
\label{eqn:total_error_eq4}
\end{equation}
using that both sequences are in $B^{N,Q}_R$.

We can compare the fixed point with the sampled exact solution $b$ (with components $b^q=\mathcal C^h_N\mathcal I_hu(t_q)$). By Lemma~\ref{lem:sampling} and the condition $E_{\mathrm{samp}}\le R-R_0$, we have $b\in B^{N,Q}_R$. Therefore, the contraction estimate from Theorem~\ref{thm:fully_discrete_fixed_point} can be applied to get:
$$\left\|a_\infty-b\right\|
\le \left\|\Psi^h_{N,Q}(a_\infty)-\Psi^h_{N,Q} (b)\right\|
+\left\|\Psi^h_{N,Q}(b)-b\right\|
\le \kappa\left\|a_\infty-b\right\|
+\left\|\Psi^h_{N,Q}(b)-b\right\|,$$
and since $\kappa<1$:
\begin{equation}
\left\|a_\infty-b\right\|
 \le \frac{1}{1-\kappa}\left\|\Psi^h_{N,Q}(b)-b\right\| .
\label{eqn:total_error_eq5}
\end{equation}
The residual measures the difference when sampling the exact solution and using the discrete map, and needs a consistency and sampling bound. Recall that $\mathcal L_{N,Q}\Psi^h_{N,Q}(b)$ and $\mathcal L_{N,Q}b$ are piecewise linear in time, with $\mathcal S_N\Psi^h_{N,Q}(b)^q$ and $\mathcal S_Nb^q$ at the nodes, respectively. Thus, by Proposition~\ref{prop:discrete_ball_containment} the norm of a piecewise linear function is attained at a node and 
\begin{align}
\left\|\Psi^h_{N,Q}(b)-b\right\|
&=\left\|\mathcal L_{N,Q}\Psi^h_{N,Q}(b)-\mathcal L_{N,Q}b\right\|_{X^\alpha_T}
=\left\|\Phi^h_{N,Q}[u]-\mathcal L_{N,Q}b\right\|_{X^\alpha_T}
\label{eqn:total_error_eq6}\\
&\le \left\|\Phi^h_{N,Q}[u]-\mathcal J_Qu\right\|_{X^\alpha_T}+\left\|\mathcal J_Qu-\mathcal L_{N,Q}b\right\|_{X^\alpha_T}
\le E_{\mathrm{cons}}+E_{\mathrm{samp}},
\label{eqn:total_error_eq7}
\end{align}
where the second equality \eqref{eqn:total_error_eq6} uses $\mathcal L_{N,Q}\Psi^h_{N,Q}(b)=\Phi^h_{N,Q}[u]$ from Definition~\ref{def:fully_discrete_neural_picard}, and the first term in \eqref{eqn:total_error_eq7} is bounded by Theorem~\ref{thm:fully_discrete_consistency} and the second by Lemma~\ref{lem:sampling} (both $\mathcal J_Qu$ and $\mathcal L_{N,Q}b$ having  nodal values of $u(t_q)$ and $\mathcal S_Nb^q$).

 The total error can be chained. By Lemma~\ref{lem:interpolation}, the  reconstruction error of $\mathcal L^S_{N,Q}$ is controlled by $E_{\mathrm{rec}}$ and by the nodal error:
\begin{align}
\max_{0\le q\le Q}\left\|\mathcal S_N(a^n_\varepsilon)^q-u(t_q)\right\|_{X^\alpha}
&\le
\left\|a^n_\varepsilon-a^n\right\|
+\left\|a^n-a_\infty\right\|
+\left\|a_\infty-b\right\|
+\max_{0\le q\le Q}\left\|\mathcal S_Nb^q-u(t_q)\right\|_{X^\alpha}
\label{eqn:total_error_eq8}\\
&\le
\frac{E_{\mathrm{net}}}{1-\kappa}
+2R\,\kappa^{n}
+\frac{E_{\mathrm{cons}}+E_{\mathrm{samp}}}{1-\kappa}
+E_{\mathrm{samp}},
\label{eqn:total_error_eq9}
\end{align}
where \eqref{eqn:total_error_eq8} is the triangle inequality and uses that $\|a\|=\max_q\|\mathcal S_Na^q\|_{X^\alpha}$ from Proposition~\ref{prop:discrete_ball_containment}. For equation \eqref{eqn:total_error_eq9}, we bound each term by \eqref{eqn:neural_drift}, \eqref{eqn:total_error_eq4}, \eqref{eqn:total_error_eq5} with \eqref{eqn:total_error_eq7}, and Lemma~\ref{lem:sampling}. Since $1-\kappa\le1$, the sum is at most
$$\frac{E_{\mathrm{net}}+E_{\mathrm{cons}}+2E_{\mathrm{samp}}}{1-\kappa}+2R\kappa^{n},$$
which gives us the main result (after adding the reconstruction term $E_{\mathrm{rec}}$ for Lemma~\ref{lem:interpolation}). 

To get the rates, we use the bounds in Theorem~\ref{thm:fully_discrete_consistency} and Lemma~\ref{lem:sampling}.
\end{proof}

\begin{remark}[Neural Approximation Only Uses Max-Norm Bounds]
\label{rem:only_exact_contracts}
The proofs were written specifically to only require that the exact discrete map be a Lipschitz contraction. The neural networks errors are included through a point-wise consistency estimate (Theorem~\ref{thm:neural_consistency}), so we do not need to control the Lipschitz constant of $\widehat f_\varepsilon$, which would be impractical. This shows that the approximation can inherit smoother control, $X^\alpha$ bounds, from a max-norm controlled network. This is related to the perturbed fixed-point iterations of \cite{feischl2025neural}, where the approximate iterations are assumed to be Lipschitz.
\end{remark}

\subsection{Parameter Complexity Bounds}

\begin{theorem}[Learnable Parameter Count, One Block]
\label{thm:main}
Assume the conditions from Theorem~\ref{thm:total_error}, Assumption~\ref{assumption:structured_f}, Assumption~\ref{assumption:bootstrap} with $\gamma'\ge\alpha+r$, and \eqref{eqn:admissible_refinement}. Set the exponent to $p_{\max}:=\max_{j\in\Vnl}p_j=1/s_{\min}$, where $p_j$ are the node exponents from Theorem~\ref{thm:relu_f} and $s_{\min}$ is the minimal smoothness exponent of a DAG node (where we set $p_{\max}=0$ if there is none). Let $C_R$ be the constant of Theorem~\ref{thm:total_error} and $\varepsilon_0$ the error threshold for Theorem~\ref{thm:relu_f}, and set
$$c_{\mathrm{err}}:=\frac{1}{7C_R},
\qquad
\eta_0:= \min\left\{1,\ \tfrac72(1-\kappa)(R-R_0),\ \frac{\varepsilon_0}{c_{\mathrm{err}}},\ \frac{T^{\zeta_{\mathrm t}}}{c_{\mathrm{err}}}\right\} $$
and $\eta\in(0,\eta_0]$ such that 
\begin{enumerate}
\item[(i)] $n:=\left\lceil\log(1/(c_{\mathrm{err}}\eta))/\log(1/\kappa)\right\rceil$;
\item[(ii)] $N:=\min\left\{N':\ \lambda_{N'+1}\ge(c_{\mathrm{err}}\eta)^{-1/\zeta_{\mathrm s}}\ \text{and}\ \  \varepsilon^{\mathrm{tail}}_{N'}(u_0)\le c_{\mathrm{err}}\eta\right\}$;
\item[(iii)] $h:=2^{-k}$ where $k$ is smallest integer such that $h^r\le c_{\mathrm{err}}\eta$, $(N,h)\in\mathcal A$ and $\varepsilon^{\mathrm{enc}}_{N,h}(u_0) \le c_{\mathrm{err}}\eta$, and $M:=M(h)$;
\item[(iv)] $Q:=\left\lceil T(c_{\mathrm{err}}\eta)^{-1/\zeta_{\mathrm t}}\right\rceil$ and $\Delta t:=T/Q$;
\item[(v)] $\varepsilon:=c_{\mathrm{err}}\eta/\max\{1,\xi_\alpha(N,\Delta t)\}$.
\end{enumerate}
Then the architecture from Definition~\ref{def:loop} with $\widehat f_\varepsilon$ from Theorem~\ref{thm:relu_f} and the spectral weights from Proposition~\ref{prop:spectral_exact}, satisfies $\|\mathcal L^S_{N,Q}a^n_\varepsilon-u\|_{X^\alpha_T}\le\eta$, and
$$\varepsilon=c_{\mathrm{err}}\eta\quad( \text{if \ } \alpha<2),
\qquad
\varepsilon\ge\frac{c_{\mathrm{err}}\eta}{1+\log\left(1+2T(c_{\mathrm{err}}\eta)^{-1/\zeta}\right)}\quad(\text{if \ } \alpha=2),
\qquad
\varepsilon\ge2^{1-\alpha/2}(c_{\mathrm{err}}\eta)^{\chi}\quad(\text{if \ }  \alpha>2).$$
In addition, if $u_0\in X^{\gamma_{\mathrm s}}$ and $\alpha>2$, then $\varepsilon\ge\max\{1,\lambda_1\}^{1-\alpha/2}m_0^{-(\alpha-2)/(2\zeta_{\mathrm s})}(c_{\mathrm{err}}\eta)^{\chi_{\mathrm s}}$, where $\chi_{\mathrm s} := 1+\frac{(\alpha-2)_+}{2\zeta_{\mathrm s}} \le \chi$ and $m_0 := \max\{1,\|u_0\|_{X^{\gamma_{\mathrm s}}}\}$. Thus the depth is bounded by $\dep(\widehat f_\varepsilon)\le C\log\eta^{-1}$ and the parameter count by
$$W(\widehat f_\varepsilon)\le
\begin{cases}
C\,\eta^{-p_{\max}}\log\eta^{-1}, & \alpha<2,\\
C\,\left(\eta^{-1}\log\eta^{-1}\right)^{p_{\max}}\log\eta^{-1}, & \alpha=2,\\
C\,\eta^{-p_{\max}\chi_{\mathrm s}}\log\eta^{-1}, & \alpha>2,
\end{cases}$$
with $C$ depending on the computation graph of $f$, $c_{\mathrm{err}}$ and $T$, and for $\alpha>2$ also on $m_0$ and $\lambda_1$.
\end{theorem}

The min in (ii) and (iii) exist by \eqref{eqn:admissible_refinement} and Lemma~\ref{lem:uniform_encoding}. The conditions on the tail of the input data is the only place without an algebraic rate so those statement are not explicit. The trainable count and the size of the $n$-sweep network are shown in Theorem~\ref{thm:total_cost}. For laws in divergence form, (v) uses $\xi_{\alpha+1}$ instead of $\xi_\alpha$ (see Proposition~\ref{prop:divergence_form}).

\begin{proof}[Proof of Theorem~\ref{thm:main}] The bound in Theorem~\ref{thm:total_error} is $C_R$ times seven terms, thus we can set $C_Rc_{\mathrm{err}}=\frac{1}{7}$, so that each term is at most $c_{\mathrm{err}}\eta$. It holds for $\kappa^n$ by assumption (i) in the theorem. By (v) it holds for $\varepsilon\xi_\alpha\le\varepsilon\max\{1,\xi_\alpha\}$. Condition (ii) gives us control for the data tail bound and $\lambda_{N+1}^{-\zeta_{\mathrm s}}$. The spatial component $h^r$ is given by (iii) and the temporal resolution $(\Delta t)^{\zeta_{\mathrm t}}$ by (iv). Lastly, the reconstruction error from Lemma~\ref{lem:interpolation} is order $(\Delta t)^{\zeta_{\mathrm t}}$ and is absorbed by the other term.

We next check that we have the assumptions for applying Theorem~\ref{thm:total_error}. Since $C_R$ is an upper-bound constant, we may make it larger before choosing the parameters.  We pick $C_R\ge2\max\{1,C_{\mathrm{smp}}\}$, where $C_{\mathrm{smp}}$ is the constant $C_R$ from Lemma~\ref{lem:sampling}. Since $\eta\le\eta_0\le\frac72(1-\kappa)(R-R_0)$, we have $E_{\mathrm{net}}\le\frac\eta7\le\frac12(1-\kappa)(R-R_0)$, which gives \eqref{eqn:margin}. Each of $\varepsilon^{\mathrm{tail}}_N$, $\varepsilon^{\mathrm{enc}}_{N,h}$, $\lambda_{N+1}^{-\zeta_{\mathrm s}}$ and $h^r$ is at most $c_{\mathrm{err}}\eta$, so Lemma~\ref{lem:sampling} and $c_{\mathrm{err}}=1/(7C_R)$ gives us the bound
$$E_{\mathrm{samp}}\le 4\max\{1,C_{\mathrm{smp}}\} c_{\mathrm{err}}\eta\le\frac{2\eta}{7}\le(1-\kappa)(R-R_0)\le R-R_0.$$
Finally $\varepsilon\le c_{\mathrm{err}}\eta_0\le\varepsilon_0$, so Theorem~\ref{thm:relu_f} applies. 

For the lower bounds on $\varepsilon$, since $c_{\mathrm{err}}\eta\le\min\{T^{\zeta_{\mathrm t}},1\}$ and $\zeta_{\mathrm t}\ge\zeta$, we have $\Delta t^{-1}=Q/T\le2(c_{\mathrm{err}}\eta)^{-1/\zeta_{\mathrm t}}\le2(c_{\mathrm{err}}\eta)^{-1/\zeta}$. For $\alpha<2$, the factor is $\xi_\alpha=1$. For $\alpha=2$, the factor is bounded by $\xi_2\le1+\log(1+T\Delta t^{-1})$, which gives the stated result. For $\alpha>2$, we have
$$\xi_\alpha\le(\Delta t^{-1})^{\alpha/2-1}\le2^{\alpha/2-1}(c_{\mathrm{err}}\eta)^{-(\alpha-2)/(2\zeta)}.$$
The right-hand side is at least $1$, so it also bounds $\max\{1,\xi_\alpha\}$. By (v),
$$\varepsilon=\frac{c_{\mathrm{err}}\eta}{ \max\{1,\xi_\alpha\}}\ge2^{1-\alpha/2}(c_{\mathrm{err}}\eta)^{1+\frac{\alpha-2}{2\zeta}} = 2^{1-\alpha/2}(c_{\mathrm{err}}\eta)^{\chi}.$$

If $u_0\in X^{\gamma_{\mathrm s}}$, set $\Lambda:=(m_0/(c_{\mathrm{err}}\eta))^{1/\zeta_{\mathrm s}}>1$ and let $N''$ be the smallest index where the eigenvalues are above $\Lambda$, i.e., $\lambda_{N''+1}\ge\Lambda$. It satisfies both conditions in (ii), the spectral one because $m_0\ge1$ and the tail bound because $\varepsilon^{\mathrm{tail}}_{N''}(u_0)\le\lambda_{N''+1}^{-\zeta_{\mathrm s}}\|u_0\|_{X^{\gamma_{\mathrm s}}}\le c_{\mathrm{err}}\eta$, so $N\le N''$. By construction, $\lambda_{N''}<\Lambda$ (we assume that $N''\ge2$ or set $\lambda_{N''}=\lambda_1$ otherwise), so $\lambda_N\le\max\{1,\lambda_1\}\Lambda$. If we bound $\xi_\alpha$ by the first entry we get $\max\{1,\xi_\alpha\}\le(\max\{1,\lambda_1\}\Lambda)^{\alpha/2-1}$, which is the stated bound with $\chi_{\mathrm s}$.

Lastly, Theorem~\ref{thm:relu_f} gives $\dep(\widehat f_\varepsilon)\le C\log\varepsilon^{-1}$ and $W(\widehat f_\varepsilon)\le C\varepsilon^{-p_{\max}}\log\varepsilon^{-1}$. Therefore, after incorporating the $\varepsilon$ lower bounds and using $\log\varepsilon^{-1}\le2\chi\log\eta^{-1}+C$ we have the result.
\end{proof}

\begin{corollary}[Structural Complexity, Block]
\label{cor:structural}
For every $p>p_f$, the neural block in Theorem~\ref{thm:main} may be chosen so that $W(\widehat f_\varepsilon)\le C_p\,\eta^{-p\chi_{\mathrm s}}\log\eta^{-1}$ and $\dep(\widehat f_\varepsilon)\le C_p\log\eta^{-1}$ for $\alpha\ne2$, with $C_p$ depending on $p$ and the quantities from Theorem~\ref{thm:main}. If the infimum for $p_f$ is attained (e.g. polynomial governing equations), we can take $p=p_f$ and, in particular, $W(\widehat f_\varepsilon)\le C\log\eta^{-1}$ when $p_f=0$ is attained.
\end{corollary}

\noindent The proof, given in Appendix~\ref{sec:app_discrete}, applies Theorem~\ref{thm:main} to a representation $G$ with $p_{\max}(G)<p$. The corollary realizes every exponent above $p_f$, which is defined independently of any architecture or discretization. The optimality of the exponent over the smooth classes is discussed in Remark~\ref{rem:block_lower_bound} and Theorems~\ref{thm:operator_lower_bound} and \ref{thm:hj_lower_bound}.

\begin{remark}[Block Lower Bounds on the Trainable Parameters]
\label{rem:block_lower_bound}
Over smooth function classes, the exponent $p_{\max}=1/s$ of a $C^s$ node cannot be improved. By the lower bound on continuous nonlinear widths from \cite{devore1989optimal}, any family of ReLU networks that approximates every function in the unit ball of $W^{s,\infty}(I)$ up to an accuracy of $\varepsilon$ and whose weights depend continuously on the target function, requires at least $c\,\varepsilon^{-1/s}$ weights. 

The bound applies to our construction in Remark~\ref{rem:taylor_layer}, which attains the rate of Theorem~\ref{thrm:quantitative_smooth_scalar} with weights that depend linearly on the Taylor series terms for $g$. This implies that the exponent $1/s$ is optimal for continuous weights. Without continuity, the VC-dimension bound of \cite{yarotsky2017error} gives a lower bound of $c\,\varepsilon^{-1/(2s)}$ up to logs, and for $s=1$ this is attained by deep networks whose weights are not continuous with respect to $g$  \cite{yarotsky2018optimal}. These are worst-case statements over the class and do not contradict the improvements for polynomial governing laws. 
\end{remark}

\subsection{Long Time Horizons}
\label{sec:long_time}

The contraction arguments shrink the time window to $[0,T]$ with $T\le T_0$. If we want a prediction up to $T_{\mathrm{final}}=JT$, we effectively can run the same (pretrained) network on $J$ consecutive time windows, where the output of one interval becomes the input to the next. A block-wise rollout of the same learned model was also analyzed in \cite{taniguchi2026generalization}. If $\Theta_{\mathrm{win}}$ is the neural network amplification factor for each window's error, summing them leads to
$\eta_{\mathrm{win}}(1+\Theta_{\mathrm{win}}+\dots+\Theta_{\mathrm{win}}^{J-1})$, where $\eta_{\mathrm{win}}$ is the error of one window started from exact data. Since $J$ and $\Theta_{\mathrm{win}}$ do not depend on $\eta$, the final time $T_{\mathrm{final}}$ just changes the constant.

We can write this as a time-stepper over the windows. For $c \in \bR^N$ let $G_N(c):=\mathcal C^h_N\left[\mathcal F^h(\mathcal S^h_Nc,\nabla_h\mathcal S^h_Nc)\right]$. For $c$ with $\|c\|_{\ell^2_\alpha}+CC_{\mathrm{stab}}F_RT^\theta\le R_0$, let $x(c)=(x^0(c),\ldots,x^Q(c))$ be the fixed point of $\Psi^h_{N,Q}$ with $a^h_0$ replaced by $c$ (Theorem~\ref{thm:fully_discrete_fixed_point}), and let $\Phi_{\mathrm{win}}(c):=x^Q(c)$ be its final value within the window (the finer time scale). Since $S_N(t_q-t_r)=(S_N(\Delta t))^{q-r}$, the fixed point solves the recursion
\begin{equation}
x^0(c)=c,\qquad x^{q+1}(c)=S_N(\Delta t)\left(x^q(c)+\Delta t\,G_N(x^q(c))\right).
\label{eqn:window_stepper}
\end{equation}
Defining the discrete inner product by $\langle a,b\rangle_s:=\sum_k\lambda_k^sa_kb_k$, the spectral theory gives
\begin{equation}
\|S_N(t)c\|_{\ell^2_s}\le e^{-\lambda_1t}\|c\|_{\ell^2_s},
\qquad
\|S_N(t)c\|_{\ell^2_{s+r}}\le c_\alpha\,t^{-r/2}\|c\|_{\ell^2_s},
\label{eqn:smoothing}
\end{equation}
for $0\le r\le\alpha$, where $c_\alpha:=\max\{1,(\alpha/2e)^{\alpha/2}\}$. Also, for any grid function $v\in\bR^M$ and $z\in\bR^N$,
\begin{equation}
\langle\mathcal C^h_Nv,z\rangle_0=Q_{N,h}(v,\mathcal S_Nz),
\label{eqn:analysis_pairing}
\end{equation}
which maps point-wise inequalities on $f$ (evaluated at the spatial grid points) into inequalities for $G_N$. For two initial vectors $a,b$ we can write the differences as $z^q:=x^q(a)-x^q(b)$ and $d^q:=G_N(x^q(a))-G_N(x^q(b))$.

\begin{proposition}[Continuation Over Windows]
\label{prop:continuation}
Set $T\le T_0$, $(N,h)\in\mathcal A$, $Q$, $n$ and $\widehat f_\varepsilon$ to have same values over all of the windows. Let $u$ be the mild solution on $[0,JT]$, and set $u_j:=u(jT)$ and $\beta_j:=\mathcal C^h_N\mathcal I_hu_j$. Assume that the following holds.
\begin{enumerate}
\item[(i)] Each of the functions $u_j$ satisfies the conditions from Theorem~\ref{thm:total_error} and Assumption~\ref{assumption:uniform_data} with the same constants, so the bound from Theorem~\ref{thm:total_error} with data $u_j$ is at most  $\eta_{\mathrm{win}}$ (uniformly).
\item[(ii)] There is $\delta>0$ such that the conditions from Theorem~\ref{thm:total_error} on the initial coefficients hold for every vector $c$ with $\|c-\beta_j\|_{\ell^2_\alpha} \le \delta$ for some $j$.
\item[(iii)] For some $s \in \{0,\alpha\}$ there are $\Theta_{\mathrm{win}},C_{\mathrm{lift}}\ge0$ such that, for coefficients $c,c'$ $\delta$ close to $\beta_j$ in $\ell^2_\alpha$
$$\|\Phi_{\mathrm{win}}(c)-\Phi_{\mathrm{win}}(c')\|_{\ell^2_s}\le\Theta_{\mathrm{win}}\|c-c'\|_{\ell^2_s},
\qquad
\|\Phi_{\mathrm{win}}(c)-\Phi_{\mathrm{win}}(c')\|_{\ell^2_\alpha}\le C_{\mathrm{lift}}\|c-c'\|_{\ell^2_s}.$$
\end{enumerate}
Let $b_0:=\beta_0$, let $b_{j+1}$ be the terminal value of the $n$-th neural Picard iterate on window $j$ started from $b_j$, and set $\tilde u_j:=\mathcal S_Nb_j$ and $\bar\eta:=\eta_{\mathrm{win}}+\max_j\varepsilon^{\mathrm{enc}}_{N,h}(u_j)$. If $\bar\eta\left(1+C_{\mathrm{lift}}\lambda_1^{-(\alpha-s)/2}\sum_{i<J}\Theta_{\mathrm{win}}^i\right)\le\delta$, then for $1\le j\le J$,
$$\|\tilde u_j-u_j\|_{X^\alpha}\le\eta_{\mathrm{win}}+C_{\mathrm{lift}}\lambda_1^{-(\alpha-s)/2}\,\bar\eta\sum_{i<j-1}\Theta_{\mathrm{win}}^i.$$
\end{proposition}

\noindent The proof is given in Appendix~\ref{sec:app_discrete}. Condition (iii) always holds with $s=\alpha$ and $\Theta_{\mathrm{win}}=C_{\mathrm{lift}}=1/(1-\kappa)$. When the governing law has an energy estimate it holds with $s=0$ and $\Theta_{\mathrm{win}}=e^{\psi T}$ (see for example Lemma~\ref{lem:allen_cahn_window} and Appendix~\ref{sec:window_stability_general} for general laws).

\begin{lemma}[Windows for the Allen-Cahn System]
\label{lem:allen_cahn_window}
Let $\Omega=\mathbb T$, $L=-\partial_{xx}+I$, $f(u)=2u-u^3$, $\alpha=\frac{7}{4}$, and assume the quadrature is the standard trapezoidal rule. Suppose the grid values of $\mathcal S_Nx^q(a)$ and $\mathcal S_Nx^q(b)$ are in the interval $[-m_K,m_K]$ for all $q$, the Lipschitz constant of $f$ on $[-m_K,m_K]$ is $\Gamma_0:=\max\{2,3m_K^2-2\}$, and assume that $\Delta t\,\Gamma_0^2\le1$. Then the following holds.
\[
\|z^q\|_{\ell^2_0}\le e^{3t_q/2}\|a-b\|_{\ell^2_0},
\qquad
\|\Phi_{\mathrm{win}}(a)-\Phi_{\mathrm{win}}(b)\|_{\ell^2_\alpha}\le\left(T^{-7/8}+8\Gamma_0e^{3T/2}T^{1/8}\right)\|a-b\|_{\ell^2_0}.
\]
\end{lemma}

\begin{proof}
Fix the (finer) time-stamp $q$. Let $v_i$ and $w_i$ be the grid values of $\mathcal S_Nx^q(a)$ and $\mathcal S_Nx^q(b)$, so that $v_i-w_i$ are the grid values of $\mathcal S_Nz^q$. Since the trapezoidal rule is exact on products of retained modes, we have $\sum_i\omega_i|v_i-w_i|^2=\|z^q\|^2_{\ell^2_0}$ and $\|\mathcal C^h_Ng\|_{\ell^2_0}\le\left(\sum_i\omega_i|g_i|^2\right)^{1/2}$ for every grid function $g$.

By \eqref{eqn:analysis_pairing}, we have 
$$\langle d^q,z^q\rangle_0 = \sum_i \omega_i\left(f(v_i)-f(w_i)\right)(v_i-w_i)
=\sum_i\omega_i\left(2|v_i-w_i|^2-(v_i^3-w_i^3)(v_i-w_i)\right)\le2\|z^q\|^2_{\ell^2_0},$$
using monotonicity $(v^3-w^3)(v-w) \ge 0$. Since $|f(v_i)-f(w_i)|\le\Gamma_0|v_i-w_i|$, we have $\|d^q\|_{\ell^2_0}\le\Gamma_0\|z^q\|_{\ell^2_0}$. Using the assumption $\Gamma_0^2\Delta t\le1$,
$$\|z^q+\Delta t\,d^q\|_{\ell^2_0}^2
\le\left(1+4\Delta t+\Gamma_0^2\Delta t^2\right)\|z^q\|^2_{\ell^2_0} 
\le e^{5\Delta t}\|z^q\|^2_{\ell^2_0}.$$
Using \eqref{eqn:window_stepper} and \eqref{eqn:smoothing} with $\lambda_1=1$, $\|z^{q+1}\|_{\ell^2_0}\le e^{-\Delta t}e^{5\Delta t/2}\|z^q\|_{\ell^2_0}=e^{3\Delta t/2}\|z^q\|_{\ell^2_0}$, and induction gives the first bound.

Iterating the stepper forward \eqref{eqn:window_stepper} gives $z^Q=S_N(T)z^0+\Delta t\sum_{p<Q}S_N(T-t_p)d^p$. By \eqref{eqn:smoothing} with $r=\alpha=\frac{7}{4}$, where $c_\alpha=1$ since $\alpha< 2 e$, and by $\|d^p\|_{\ell^2_0}\le\Gamma_0\|z^p\|_{\ell^2_0}\le\Gamma_0 e^{3T/2}\|z^0\|_{\ell^2_0}$,
$$\|z^Q\|_{\ell^2_\alpha}\le\left(T^{-7/8}+\Gamma_0e^{3T/2}\,\Delta t\sum_{p<Q}(T-t_p)^{-7/8}\right)\|z^0\|_{\ell^2_0}
\le\left(T^{-7/8}+8\Gamma_0e^{3T/2}T^{1/8}\right)\|z^0\|_{\ell^2_0},$$
using the integral bound $\Delta t\sum_{p<Q}(T-t_p)^{-7/8}\le\int_0^Ts^{-7/8}\,ds=8T^{1/8}$.

\end{proof}

The window length is set by the contraction, so $J\propto T_{\mathrm{final}}(C_{\mathrm{stab}}L_R)^{1/\theta}$, which is independent of $N$, $h$, $Q$ and $\eta$. The final accuracy $\eta$ requires $\bar\eta\le\eta/A_J$ with $A_J:=1+C_{\mathrm{lift}}\lambda_1^{-(\alpha-s)/2}\sum_{i<J}\Theta_{\mathrm{win}}^i$, so $T_{\mathrm{final}}$ changes the constant but not the exponent of $\eta$. The contraction gives $\Theta_{\mathrm{win}}=1/(1-\kappa)\le2$ for $\kappa\le\frac12$, so $A_J\lesssim J\,2^J=J\,e^{(\log2/T)\,T_{\mathrm{final}}}$. An energy estimate gives $\Theta_{\mathrm{win}}=e^{\psi T}$, so $A_J\lesssim J\,e^{\psi_+T_{\mathrm{final}}}$. Both bounds can be written as the form $A_J\lesssim J\,e^{\rho\,T_{\mathrm{final}}}$, but the contraction rate $\rho=\log 2 /T$ grows as the windows shrink, while the energy rate $\rho=\psi_+$ is set by the law and is independent of the window length. For Allen-Cahn, this is a difference of $\psi=\frac{3}{2}$ versus $\log2/T\approx6\times10^3$, for a single-window.

\subsection{Realization of the Picard Sweep}
\label{sec:realization}

While Theorem~\ref{thm:main} counts the number of trainable parameters for the neural Picard map, here we count every parameter for the entire architecture. This includes the fixed transforms, which are assumed to be fast (Assumption~\ref{assumption:fast_transform}), and the Duhamel step as a linear recurrence with diagonal coefficient. Specifically, since $S_N(t_q-t_r)=S_N(\Delta t)^{q-r}$, the iterates $v^q:=\left(\Psi^h_{N,Q,\varepsilon}(a)\right)^q$ satisfy
\begin{equation}
v^0=a^h_0,
\qquad
v^{q+1}=S_N(\Delta t)\left(v^q+\Delta t\,\mathcal C^h_N g^q\right),
\qquad q=0,\ldots,Q-1,
\label{eqn:duhamel_recurrence}
\end{equation}
where $g^q:=\widehat{\mathcal F}^h_\varepsilon\left(\mathcal S^h_Na^q,\nabla_h\mathcal S^h_Na^q\right)$. Since the multipliers $e^{-\lambda_k\Delta t}$ are fixed, each step of the recurrence is an affine map with fixed coefficients, and composition is associative. So a parallel scan \cite{blelloch1990prefix} computes all $v^0,\ldots,v^Q$ in depth $O(\log Q)$ using $O(Q\log Q)$ weights per mode, instead of the naive depth $Q$.

\begin{assumption}[Fast transforms]
\label{assumption:fast_transform}
There is a function $\mathcal T(M)$ and a constant $C_{\mathrm{ft}}$ such that each fixed transform $\mathcal T_{\mathrm{syn}}$, $\mathcal T_{\mathrm{syn},j}$ and $\mathcal T_{\mathrm{an}}$ of Assumption~\ref{assumption:spectral_factorization} factors into $C_{\mathrm{ft}}\log M$ sparse layers with at most $\mathcal T(M)$ nonzero weights in total.
\end{assumption}

 If we can use the FFT then the assumption holds with $\mathcal T(M)=CM\log M$ for the trigonometric basis and with $\mathcal T(M)=CM\log^2M$ for classical orthogonal polynomial and spherical harmonic bases \cite{driscoll1994computing,townsend2018fast}. It holds trivially with $\mathcal T(M)=NM$ and one layer, which recovers the dense transform case.

\begin{definition}[Realization of Picard Sweeps]
\label{def:realization}
Fix the parameters $(N,h)\in\mathcal A$, $Q$, $n$ and $\widehat f_\varepsilon$. For $w \in \bR^{dM}$, the sampled data $w =\mathcal I_hu_0$, set $a_0^h(w):=\mathcal C_N^hw$ and $A^0(w):= \left(a_0^h(w),\ldots,a_0^h(w)\right) \in (\bR^{dN})^{Q+1}$, and let $A^k(w) := \Psi^h_{N,Q,\varepsilon} \left(A^{k-1}(w)\right)$ for $k=1,\ldots,n$, with initial data $a^h_0(w)$. The $n$-sweep network can be realized by $\mathcal A_n(w):=A^n(w)$, the loop of Definition~\ref{def:loop} repeated $n$ times.
\end{definition}

\begin{theorem}[Size and depth of the $n$-sweep Network]
\label{thm:realization}
Assume the conditions from Theorem~\ref{thm:main} and Assumption~\ref{assumption:fast_transform}, pick $(N,h,Q,n,\varepsilon)$ and $\widehat f_\varepsilon$ accordingly, and assume that the sampled input satisfies
\begin{equation}
\left\|\mathcal C^h_N\mathcal I_hu_0\right\|_{\ell^2_\alpha}\le R_0,
\label{eqn:admissible_data}
\end{equation}
which holds when $C^\alpha_{\mathrm{stab}}\|u_0\|_{X^\alpha}\le R_0$ (Lemma~\ref{lem:encoder_bounded}). Then $\mathcal A_n$ is a ReLU network with the following properties:
\begin{enumerate}
\item The accuracy is given by $\left\|\mathcal L^S_{N,Q}\,\mathcal A_n \left(\mathcal I_hu_0\right)-u\right\|_{X^\alpha_T} \le \eta$.
\item The depth is given by  $\dep(\mathcal A_n) \le n\left(\dep(\widehat f_\varepsilon)+C\log(MQ)\right) \le C\left(\log\eta^{-1}\right)^2$.
\item The size is 
$$W(\mathcal A_n)\le C\,n\left[(2+d_x)d\,Q\,\mathcal T(M)+d\,N\, Q \log(2Q)+MQ\,W(\widehat f_\varepsilon)\right] + C\,\mathcal T(M),$$
where $M=M(h)$. If $N\le CM$ and $\mathcal T(M)\le CM\log M$, then we can write it as
\begin{equation}
W(\mathcal A_n)
\le C\,n\,\underbrace{M\,Q}_{\text{space--time points}}
\Bigl[\,\underbrace{(2+d_x)\,d\log(MQ)}_{\text{transforms and scan}}+\underbrace{W(\widehat f_\varepsilon)}_{\text{block}}\,\Bigr];
\label{eqn:realization_size}
\end{equation}
\item The number of trainable parameters is given by $W(\widehat f_\varepsilon)+W_{\mathrm{spec}}$ which are the distinct and trainable weights.
\end{enumerate}
\end{theorem}

\noindent The proof is given in Appendix~\ref{sec:app_realization}. The next theorem obtains the explicit rates by combining the results from Theorems~\ref{thm:main} and \ref{thm:realization}.

\begin{theorem}[Rates and Total Cost]
\label{thm:total_cost}
Assume the conditions from Theorem~\ref{thm:realization}, $c_Wk^{2/d_x}\le\lambda_k\le C_Wk^{2/d_x}$ for all $k$, $M(h)\asymp h^{-d_x}$, $\mathcal T(M)\le CM\log M$, and the resolution condition $(N,h)\in\mathcal A$ if $h\lambda_N^{1/2}\le c_{\mathrm{res}}$. Let $u_0\in X^{\gamma_{\mathrm s}}$. Using Theorem~\ref{thm:main} and a quadrature of order $r\le\gamma_{\mathrm s}-\alpha$, set
$$\upsilon:=\frac{d_x}{r}+\frac1{\zeta_{\mathrm t}}+p_{\max}\chi_{\mathrm s}.$$
Then the architecture from Definition~\ref{def:loop} approximates the solution up to accuracy $\eta$ in $C([0,T];X^\alpha)$ with the following cost (see Remark~\ref{rem:higher_order} for higher order operators), where $C$ is independent of $\eta$, for $\alpha\ne2$ (for $\alpha=2$, every bound involving the block includes a factor $(\log\eta^{-1})^{p_{\max}}$):
\begin{center}
\renewcommand{\arraystretch}{1.25}
\begin{tabular}{@{}ll@{}}
\toprule
Quantity & Upper Bound \\
\midrule
discretization & $N\le C\eta^{-d_x/(\gamma_{\mathrm s}-\alpha)}$,\quad $M\le C\eta^{-d_x/r}$,\quad $Q\le C\eta^{-1/\zeta_{\mathrm t}}$,\quad $n\le C\log\eta^{-1}$ \\
trainable block & $W(\widehat f_\varepsilon)\le C\eta^{-p_{\max}\chi_{\mathrm s}}\log\eta^{-1}$ \\
trainable, spectral & $W_{\mathrm{spec}}\le C\eta^{-d_x/(\gamma_{\mathrm s}-\alpha)}$ \\
All Weights (per eval.) \ \ & $\le C\eta^{-\upsilon}(\log\eta^{-1})^2$ \\
Total Depth & $\dep(\mathcal A_n)\le C(\log\eta^{-1})^2$ \\
\bottomrule
\end{tabular}
\end{center}
The trainable parameter count is independent of $M(h)$, $Q$ and $n$.
\end{theorem}

Note that when $r=\gamma_{\mathrm s}-\alpha$ we have $M\asymp N$.

\begin{remark}[Higher-order Operators]
\label{rem:higher_order}
Theorem~\ref{thm:total_cost} assumes Weyl's law and the resolution condition associated with second-order linear operators. Suppose $L$ has order $2m$. If $c_Wk^{2m/d_x}\le\lambda_k\le C_Wk^{2m/d_x}$ and the resolution condition is $h\lambda_N^{1/(2m)}\le c_{\mathrm{res}}$, then the same proof gives $N\asymp\eta^{-d_x/(m(\gamma_{\mathrm s}-\alpha))}$. With an exact quadrature $M\asymp N$, so the spatial part of the exponent $\upsilon$ becomes $\frac{d_x}{m(\gamma_{\mathrm s}-\alpha)}$. Thus $m$ is included in the rates. 
\end{remark}

\begin{proof}[Proof of Theorem~\ref{thm:total_cost}]
The proof collects all the bounds together. The accuracy and the block-wise bound from Theorem~\ref{thm:main} (where we absorb $m_0$ into the constant) and the trainable parameters come from $\widehat f_\varepsilon$ and the spectral weights $W_{\mathrm{spec}}\le C_{\mathrm{spec}}N$ (Proposition~\ref{prop:spectral_exact}).

The additional assumptions regarding the parameters make certain terms explicit. Following the proof of Theorem~\ref{thm:main},
we have $N''$ and $\Lambda:=(m_0/(c_{\mathrm{err}}\eta))^{1/\zeta_{\mathrm s}}$, where $N\le N''$. If $N''\ge2$, the lower Weyl bound in the assumptions gives $c_W(N'')^{2/d_x}\le\lambda_{N''}<\Lambda$, so $N''<(\Lambda/c_W)^{d_x/2}$ and $\Lambda^{d_x/2}=(m_0/(c_{\mathrm{err}}\eta))^{d_x/(\gamma_{\mathrm s}-\alpha)}$. By Theorem~\ref{thm:main} (iv) we have $Q\le2T(c_{\mathrm{err}}\eta)^{-1/\zeta_{\mathrm t}}$ and $n\le C\log\eta^{-1}$ by (i).

Next we look at $M$. All conditions in Theorem~\ref{thm:main}(iii) hold for every $h\le h_*:=c\,(c_{\mathrm{err}}\eta)^{1/r}$, with $c\in(0,1]$ depending only on $m_0$, $\lambda_1$, $c_{\mathrm{res}}$ and $C_{\mathrm{quad}}$. Specifically, $h^r\le c_{\mathrm{err}}\eta$ since $c\le1$. Next, $\varepsilon^{\mathrm{enc}}_{N,h}(u_0)\le C_{\mathrm{quad}}h^r\|u_0\|_{X^{\alpha+r}}\le Cm_0h^r\le c_{\mathrm{err}}\eta$ for $c\le(Cm_0)^{-1/r}$, since $u_0\in X^{\gamma_{\mathrm s}}\subset X^{\alpha+r}$. Lastly, $(N,h)\in\mathcal A$ when the resolution condition holds $h\lambda_N^{1/2}\le c_{\mathrm{res}}$, and since $\lambda_N\le\max\{1,\lambda_1\}\Lambda$ gives $\lambda_N^{1/2}\le C(c_{\mathrm{err}}\eta)^{-1/(\gamma_{\mathrm s}-\alpha)}$, this holds for $h\le h_*$ because $r\le\gamma_{\mathrm s}-\alpha$ and $c_{\mathrm{err}}\eta<1$. Since the choice $h=2^{-k}$ in (iii) stops no later than the first power of two under $h_*$, we have $h>h_*/2$, and thus $M\le Ch^{-d_x}\le C\eta^{-d_x/r}$. Conversely, $h\le(c_{\mathrm{err}}\eta)^{1/r}$ gives $M\ge c\eta^{-d_x/r}\ge cN$, since $N\le C\eta^{-d_x/(\gamma_{\mathrm s}-\alpha)}\le C\eta^{-d_x/r}$, so the size from \eqref{eqn:realization_size} holds.

Putting these bounds together in \eqref{eqn:realization_size} yields
$$W(\mathcal A_n)\le C\log\eta^{-1}\cdot\eta^{-d_x/r-1/\zeta_{\mathrm t}}\left(\log\eta^{-1}+\eta^{-p_{\max}\chi_{\mathrm s}}\log\eta^{-1}\right)\le C\eta^{-\upsilon}(\log\eta^{-1})^2,$$
and the depth comes from Theorem~\ref{thm:realization}(ii). The cost of evaluating a ReLU network is a finite number of operations per nonzero weight and per nodes, and there are at most as many units as nodes. With $r=\gamma_{\mathrm s}-\alpha$, the upper bounds on $N$ and $M$ have matching rates, and the upper Weyl bound assumption gives the matching lower bound on $N$, so $M\asymp N$.
\end{proof}

\begin{remark}[Error in $L^2$]
\label{rem:l2}
Since $\|v\|_{L^2}\le\lambda_1^{-\alpha/2}\|v\|_{X^\alpha}$, all of our bounds in $C([0,T];X^\alpha)$ give a bound in $C([0,T];L^2)$ directly, that is, with the same exponents and a constant multiplied by $\lambda_1^{-\alpha/2}$. The lower bounds of Section~\ref{sec:lower_bounds} allow any $\alpha\ge0$, so they hold for $L^2$. Repeating the proofs, using the weaker $L^2$ norm should remove most of the amplification effects, since $\xi_{\alpha'}=1$ for $\alpha'<2$, and change the mode count to $\eta^{-d_x/\gamma_{\mathrm s}}$, matching Proposition~\ref{prop:spectral_lower} at $\alpha=0$. We leave this to a future work.
\end{remark}

\subsection{Rate Comparisons with Other Approximations}
\label{sec:comparison}

We can compare these rates to two other approximations directly. For simplicity and fairness, the only difference between the approximations is in the nonlinear terms and we assume $\alpha>2$, so that $\varepsilon\asymp\eta^\chi$ is needed in our approach (Table~\ref{tab:comparison}). First, if we consider the Picard map as a generic Lipschitz operator \cite{lu2021learning,li2020fourier,kovachki2023neural, weihs2026generalization}, then resolving the trajectories from Lemma~\ref{lem:derived_regularity} requires a latent space with $n(\eta)\to\infty$, and so the exponential cost cannot be improved over the Lipschitz class \cite{lanthaler2026parametric}. To improve the rate, one would need holomorphic or structured operators, which is not given by the derived regularity of the system Lemma~\ref{lem:derived_regularity}.  In another approach, one can consider a standard approximation to $f$ on $\bR^{D}$, $D=d(1+d_x)$ \cite{yarotsky2017error}. When $s\le s_{\min}$, the common setting, our DAG formulation improves the exponent by a factor of at least $D$ (e.g., $6$ for the 2D Navier-Stokes, $2$ for a one-dimensional conservation law), and for polynomial $f$ it leads to a larger improvement, replacing inverse power laws on $\eta$ by $\log\eta^{-1}$ only. The intrinsic cost of resolving the state variable (the factor $\frac{d_x}{r}+\frac1{\zeta_{\mathrm t}}$) is used in all three approaches, but our formulation reduces the trainable components and leverages the parabolic structure directly. 

\begin{table}[ht]
\centering
\small
\begin{tabular}{@{}lcl@{}}
\toprule
Method & Dimension of nonlinear approx. & Trainable size for accuracy $\eta$ \\
\midrule
1. Black-box learning of $\Phi$
& $n(\eta)\asymp\eta^{-\frac{d_x+2}{\gamma-\alpha}}\to\infty$
& $\exp\left(c\,\eta^{-\frac{d_x+2}{\gamma-\alpha}}\log\eta^{-1}\right)$ \\[1mm]
2. Locality $+$ direct ReLU for $f$
& $D=d(1+d_x)$, fixed
& $C\,\eta^{-\frac{D}{s}\chi}\log\eta^{-1}$ \\[1mm]
3. Locality $+$ DAG (Thm.~\ref{thm:main})
& $1$, at each nonlinear node
& $C\,\eta^{-\frac{1}{s_{\min}}\chi}\log\eta^{-1}$ ($C\log\eta^{-1}$ if multilinear) \\
\bottomrule
\end{tabular}
\caption{Comparison between general and neural Picard theory. First, using the knowledge that the system is a semilinear parabolic PDE reduces the dimension of the nonlinear approximation from $n(\eta)$ to the state dimension $D=d(1+d_x)$. Then our DAG approximation reduces it from $D$ to one variable per node. The discretization and the spectral weights are the same in all three rows.}
\label{tab:comparison}
\end{table}

\subsection{Fixed Linear Operations}
\label{sec:postcomposition}

In some settings, the governing law $f$ involves a nonlocal but fixed, linear operation. We show that such systems do not change the overall costs. This was referenced in Section~\ref{sec:dag} and used in Section~\ref{sec:amplification} and in the explicit PDE examples of Section~\ref{sec:examples}. Since the linear term may change the spaces, we need to introduce an extension.

\begin{lemma}[Fixed Linear Operations]
\label{lem:linear_postcomposition}
Let $\widetilde H\supseteq H$ be a Hilbert space associated with a self-adjoint extension of $L$ that has the $e_k$ as its eigenfunctions and let $\widetilde X^s$ be the induced spaces with $X^s=\widetilde X^s\cap H$. Let $\mathcal M:\widetilde X^\beta\to X^\beta$ be a linear operator such that
\begin{equation}
\mathcal M^{*}e_k=\mu_ke_k,
\qquad
C_{\mathcal M}:=\sup_{k\ge1}|\mu_k|<\infty ,
\label{eqn:multiplier}
\end{equation}
and let $f=\mathcal M\circ f_{\mathrm{loc}}$. We assume that $f_{\mathrm{loc}}$ is local, satisfies Assumption~\ref{assumption:structured_f}, and is Lipschitz on bounded sets from $X^\alpha$ into $\widetilde X^\beta$ with constant $L_R^{\mathrm{loc}}$. If $D_\mu:=\operatorname{diag}(\mu_1,\ldots,\mu_N)$, then the following hold.
\begin{enumerate}
\item $\mathcal C_N\mathcal M=D_\mu\,\mathcal C_N$ on $\widetilde X^\beta$, $D_\mu$ commutes with $S_N(t)$, and  $\|D_\mu\|_{\ell^2_s\to\ell^2_s}\le C_{\mathcal M}$ for every $s$;
\item If $\mathcal C_N^h$ is replaced by $D_\mu\mathcal C_N^h$ in Definition~\ref{def:fully_discrete_neural_picard}, and Assumption~\ref{assumption:quadrature_stability} is replaced by $\widetilde X^s$, then Theorems~\ref{thm:fully_discrete_fixed_point}, \ref{thm:neural_consistency}, \ref{thm:fully_discrete_consistency}, \ref{thm:total_error}, and \ref{thm:main} hold but with $L_R=C_{\mathcal M}L_R^{\mathrm{loc}}$ and $F_R=C_{\mathcal M}F_R^{\mathrm{loc}}$.
\item The trainable parameters are from approximating $f_{\mathrm{loc}}$ and the spectral weights Definition~\ref{def:spectral_layers}, which includes the entries of $D_\mu$ (the $\mathbf 1_{\mathcal M}$ term). In particular, if $\mathcal M$ is the orthogonal projection of $\widetilde H$ onto $H$, then $\mu_k\equiv1$, $D_\mu=I$, and $\mathcal M$ requires no layers.
\end{enumerate}
\end{lemma}

\noindent The proof is given in Appendix~\ref{sec:app_realization}.

\section{The Solution Operator on Compact Sets of Initial Data}
\label{sec:operator}

Thus far, Theorems~\ref{thm:main} and \ref{thm:realization} are results concerning the approximation of a single trajectory of the governing system starting from a fixed $u_0$. The data (initial condition) is input to the architecture through the encoder $w\mapsto\mathcal C^h_Nw$ on the sampled function (vector) $w=\mathcal I_h u_0$. In operator learning, one aims to construct an approximation over a set of inputs and obtain a uniform accuracy for all inputs over, say, set $\mathcal U$. In this section, we show that the architecture with the same trainable parameters achieves this. 

On bounded subsets of $X^\gamma$ the regularity from prior sections is sufficient to prove that the system approximates the general solution operator. Every constant in the analysis depends only on $u_0$ via $\|u_0\|_{X^\gamma}$, so Theorem~\ref{thm:main} is already uniform and the rates are unchanged (summarized in Theorem~\ref{thm:operator_rates}). On the other hand, if we just have sets that are compact in $X^\alpha$, we get a uniform approximation but without an explicit rate, see also \cite{chen1995universal,lanthaler2022error}.

\subsection{Uniform Properties}

We denote the solution operator by $G_T:u_0\mapsto u$, see Definition~\ref{def:mild}, and denote the set of initial conditions by $\mathcal U$. 

\begin{lemma}[Uniform Well-posedness on Bounded Sets]
\label{lem:uniform_wp}
For $R_0>0$, we set $R:=1+2R_0$. Using the conditions from Proposition~\ref{prop:contraction}, there exists $T_0>0$, depending only on $R$, on $L_R$, and on the semigroup constant, such that for every $T\in(0,T_0]$ and every $u_0\in X^\alpha$ with $\|u_0\|_{X^\alpha}\le R_0$:
\begin{enumerate}
\item $\Phi$ maps $B_R^\alpha$ into itself and is a contraction with constant $\kappa_0=CT^\theta L_R\le\tfrac{1}{2}$;
\item $ G_T[u_0]$ is well defined on $B_R^\alpha$.
\end{enumerate}
Both $R$ and $\kappa_0$ are independent of $u_0$.
\end{lemma}

\noindent The proof is given in Appendix~\ref{sec:app_operator}.

\begin{lemma}[Compactness of the regularity ball]
\label{lem:compact_embedding}
For $\gamma>\alpha$ the embedding $X^\gamma\hookrightarrow X^\alpha$ is compact, in particular, the ball
$\mathcal U_{R_0}:=\{u_0\in X^\gamma:\|u_0\|_{X^\gamma}\le R_0\}$
is a compact subset of $X^\alpha$.
\end{lemma}

\noindent The proof is given in Appendix~\ref{sec:app_operator}.

To extend the results from Section~\ref{sec:total}, we need the conditions to hold uniformly over the input data set.  The parabolic estimates must be uniform on bounded sets of data.

\begin{assumption}[Uniform Estimates]
\label{assumption:uniform_data}
For every $\rho_\gamma>0$ there are constants $R_\gamma$ and $C_R$, only depending on $\rho_\gamma$, $T$, and the constants from Propositions~\ref{prop:contraction} and \ref{prop:instantaneous_regularity}, such that every $u_0$ with $\|u_0\|_{X^\gamma}\le\rho_\gamma$ satisfies the same result from Lemma~\ref{lem:derived_regularity} with
$$\sup_{0<t\le T}t^{\zeta}\|u(t)\|_{X^\gamma} \le R_\gamma,
\qquad
\|F\|_{\infty,\beta}+[F]_{\zeta}\le C_R,$$
and with the bound from Theorem~\ref{thm:fully_discrete_consistency} on $\|f(P_Nu,\nabla P_Nu)\|_{X^{\beta+r}}$ holding with the same $C_R$.
\end{assumption}

The assumption is a summary of previous bounds. The first two bounds follow from Lemma~\ref{lem:derived_regularity}, whose constants $R_\gamma=C(R+T^\theta F_R)$ and $[F]_\zeta\le CL_R(R_\gamma+T^\theta F_R)$ depend on $u_0$ only through the radius $R$ of Proposition~\ref{prop:contraction}. Thus it holds by a uniform bound on $\|u_0\|_{X^\alpha}$. The bound on $f(P_Nu,\nabla P_Nu)$ follows from Assumption~\ref{assumption:bootstrap} with $s=\alpha+r$, with constants only depending on $K_R$, $R$ and $R_\gamma$. The bootstrap shows that the same holds when we replace $X^\gamma$ by $X^{\gamma'}$, the constants then also depending on the radius $R_{\gamma'}$. The radius $\rho_\gamma$ is only needed for $E^h_{\mathrm{init}}$ and $\|(I-P_N)u_0\|_{X^\alpha}$ to be uniformly controlled over the data set.

The architecture sees the data only by its sampled values $\mathcal I_hu_0$ and reconstructs $\Pi_{N,h}u_0=\mathcal S_N\mathcal C^h_N\mathcal I_hu_0$ in $E_N$. Over a set $\mathcal U$, the errors are  $\varepsilon^{\mathrm{tail}}_N(\mathcal U)$, which measures how well the eigenbasis resolves $\mathcal U$, and $\varepsilon^{\mathrm{enc}}_{N,h}(\mathcal U)$, which measures the  error of using quadrature instead of inner products (see Definition~\ref{def:data_moduli}). For regular data, we will get a rate. Using just compactness, we can only show they tend to zero. In this subsection we assume, in addition to Assumptions~\ref{assumption:quadrature_stability} and \ref{assumption:quadrature_consistency}, the admissible refinement condition:
\begin{equation}
\text{for every }N\in\mathbb N\text{ there is }h_N>0\text{ with }(N,h)\in\mathcal A\text{ for all }h\in(0,h_N].
\label{eqn:admissible_refinement}
\end{equation}
For a set $\mathcal U\subset X^\alpha$ of data, we extend the previous tail and encoding definitions to $\varepsilon^{\mathrm{tail}}_N(\mathcal U):=\sup_{u_0\in\mathcal U}\varepsilon^{\mathrm{tail}}_N(u_0)$ and $\varepsilon^{\mathrm{enc}}_{N,h}(\mathcal U):=\sup_{u_0\in\mathcal U}\varepsilon^{\mathrm{enc}}_{N,h}(u_0)$.

\begin{lemma}[Uniform Encoding]
\label{lem:uniform_encoding}
Let $\mathcal V\subset X^\alpha$ be compact and let $N$ be fixed. Then, by \eqref{eqn:admissible_refinement},
$$\sup_{v\in\mathcal V} \left\|\left(\mathcal C^h_N\mathcal I_h-\mathcal C_N\right)v\right\|_{\ell^2_\alpha} \to 0
\qquad\text{as } h \to 0,$$
for $(N,h)\in\mathcal A$.
In particular, $\varepsilon^{\mathrm{enc}}_{N,h}(\mathcal U)\to0$ for every compact $\mathcal U\subset X^\alpha$, where we set $\mathcal V=\{S(t)u_0:u_0\in\mathcal U,\ 0 \le t \le T\}$.
\end{lemma}

The set $\mathcal V=\{S(t)u_0:u_0\in\mathcal U,\ 0 \le t \le T\}$ is compact since it is the continuous image of $\mathcal U\times[0,T]$.
The proof is given in Appendix~\ref{sec:app_operator}.

\subsection{Uniform Rates from Regularity}

\begin{theorem}[Operator Learning Error Rate]
\label{thm:operator_rates}
Let $R_0,\rho_\gamma>0$ and take the initial data to be in $$\mathcal U:=\left\{u_0\in X^\gamma:\ \|u_0\|_{X^\alpha}\le R_0,\ \|u_0\|_{X^\gamma}\le\rho_\gamma\right\}.$$ Assume the conditions from Theorem~\ref{thm:main} together with Assumption~\ref{assumption:uniform_data} hold. Let $T\le T_0$ be from Lemma~\ref{lem:uniform_wp} with the radius $\max\{1,C^\alpha_{\mathrm{stab}}\}R_0$. 

Then for every sufficiently small $\eta>0$ there is a $(N,h,Q,n,\varepsilon)$ and a block $\widehat f_\varepsilon$ with
$$\dep(\widehat f_\varepsilon)\le C\log\eta^{-1},
\qquad
W(\widehat f_\varepsilon)\le C\,\eta^{-p_{\max}\chi}\log\eta^{-1},$$
(if $\alpha=2$, include the factor $(\log\eta^{-1})^{p_{\max}}$) such that the full neural network $\mathcal A_n$ from Definition~\ref{def:realization} satisfies
$$\sup_{u_0\in\mathcal U}\left\|\mathcal L^S_{N,Q}\,\mathcal A_n\left(\mathcal I_hu_0\right)- G_T[u_0]\right\|_{X^\alpha_T}\le\eta,$$
where the size and depth bounds from Theorem~\ref{thm:realization} hold. The constants depend on $R_0$ and $\rho_\gamma$ and not on $\eta$.
\end{theorem}

\noindent The proof is given in Appendix~\ref{sec:app_operator}.

Theorem~\ref{thm:operator_rates} can be directly compared to the operator-learning literature, where one architecture, with $C\eta^{-p_{\max}\chi}\log\eta^{-1}$ block parameters and $C_{\mathrm{spec}}N$ spectral weights is used for all initial data in a bounded set of $X^\gamma$. Note that $X^\gamma$-boundedness is the only assumption used for the initial data, in particular, we do not require existence of a low-dimensional manifold or the span of finitely many terms.

\section{Sharpness and Lower Bounds}\label{sec:lower_bounds}

We will consider several results regarding the sharpness of previous results. 

\subsection{Sharpness of Amplification Factor} \label{sec:amplification}
First, we will address the neural network amplification factor $\xi_\alpha$, which exceeds $1$ only when $\alpha>2$. That setting is needed for the embedding $X^\alpha\hookrightarrow C^1$ and thus $\alpha>1+\frac{d_x}{2}$ (so $\alpha>2$ when $d_x\ge2$). We show that the estimate is sharp, that smoother data reduce it, and that it is absent for gradient-free laws when $d_x\le3$.

The problem comes from the following balance. The Duhamel operator can only improve smoothness by fewer than two derivatives. The network error is controlled in ${\ell}^\infty$, and so after applying $\mathcal C^h_N$, the integrand input to the Duhamel operator is only in $X^0$.  An output in $X^\alpha$ with $\alpha\ge2$ therefore requires more smoothing than the Duhamel integral can produce. What appears is the Galerkin truncation giving the finite factor $\xi_\alpha$. We could fix it by asking for $(f-\widehat f_\varepsilon)(u,\nabla u)$ in $X^\beta$, but that changes how learning is typically performed and would not be consistent with the practice. The following result shows that this is a sharp obstacle.

\begin{proposition}[The Amplification Estimate is Sharp]
\label{prop:xi_sharp}
For $\alpha>2$, suppose the quadrature is exact, i.e. $\mathcal C^h_N\mathcal S^h_N=I_{\bR^N}$, with $\|e_k\|_{L^\infty}\le\sqrt2$ (see Remark~\ref{rem:stable_quadrature_example}). Set $\Lambda_0:=\min\{\lambda_N,\Delta t^{-1}\}$, let $T\ge2\Lambda_0^{-1}$, and suppose some eigenvalue is in the interval $[c_{\mathrm{sp}}\Lambda_0,\Lambda_0]$ for a fixed $c_{\mathrm{sp}}\in(0,1)$ (holds by Weyl). Then there exist grid functions $g^0,\ldots,g^{Q-1}$ with $\max_r\|g^r\|_{{\ell}^\infty}\le\varepsilon$ such that
$$\left\|
\Delta t\sum_{r=0}^{Q-1}S_N(t_Q-t_r)\,\mathcal C_N^hg^r
\right\|_{\ell^2_\alpha}
 \ge c\,\varepsilon\,\Lambda_0^{\alpha/2-1},$$
where $c>0$ depends only on $\alpha$ and $c_{\mathrm{sp}}$. Therefore, no estimate that uses only the ${\ell}^\infty$-error of the neural network can remove $\xi_\alpha$ from Theorem~\ref{thm:neural_consistency}.
\end{proposition}

\begin{proof}[Proof of Proposition~\ref{prop:xi_sharp}]
Pick a $k_0\le N$ so that $\Lambda:=\lambda_{k_0}\in[c_{\mathrm{sp}}\Lambda_0,\Lambda_0]$, and then $\Lambda\Delta t\le1$ and $\Lambda T\ge2c_{\mathrm{sp}}$. We can construct the grid functions by $g^r_i:=\tfrac{\varepsilon}{\sqrt2}e_{k_0}(x_i)$ for every $r$ (using the eigenfunctions). Then $\|g^r\|_{{\ell}^\infty}\le\varepsilon$ and $\mathcal C^h_Ng^r=\tfrac{\varepsilon}{\sqrt2}\mathbf 1_{k_0}$ since the quadrature is exact on the eigenfunction. Recall that $S_N(t)$ is diagonal, which makes the LHS in the proposition statement equal to
$$\frac{\varepsilon}{\sqrt2}\, \Lambda^{\alpha/2}\,\Delta t\sum_{j=1}^{Q} e^{-\Lambda j\Delta t}
=\frac{\varepsilon}{\sqrt2} \Lambda^{\alpha/2}\Delta t
\frac{e^{-\Lambda\Delta t}\left(1-e^{-\Lambda T}\right)}{1-e^{-\Lambda\Delta t}}
 \ge \frac{\varepsilon}{\sqrt2}\,e^{-1}\left(1-e^{-2c_{\mathrm{sp}}}\right)\Lambda^{\alpha/2-1},$$
where we used $1-e^{-\Lambda\Delta t}\le\Lambda\Delta t$, $e^{-\Lambda\Delta t} \ge e^{-1}$, and $\Lambda T\ge2c_{\mathrm{sp}}$. The rest follows from $\Lambda\ge c_{\mathrm{sp}}\Lambda_0$.
\end{proof}

Improved regularity can help remove this factor. Since $\xi_\alpha$ is controlled by $\min\{\lambda_N,\Delta t^{-1}\}$, smoother data makes the exponent go from $\chi$ to $\chi_{\mathrm s}$ by lowering the spatial scale $\lambda_N\asymp\eta^{-1/\zeta_{\mathrm s}}$ (see Theorem~\ref{thm:main}), and by Remark~\ref{rem:gamma_free} it lowers $\chi$ to $1+\frac{(\alpha-2)_+}{2\theta}$.
Also, the condition $\alpha>2$ comes from the gradient input term for $f$. For a gradient-free law $f=f(u)$, e.g. reaction-diffusion system, Assumption~\ref{assumption:regularity_embedding} can be weakened to $X^\alpha\hookrightarrow C^0(\overline\Omega)$, so in dimensions $d_x\le3$ we can take $\alpha<2$ and $\xi_\alpha=1$.

\subsection{Divergence Forms}
In some settings, we can rewrite the convective terms in divergence form (see also \cite{chandramoorthy2026spectral}):
$$-(u\cdot\nabla)u=-\nabla\cdot(u\otimes u)
\quad\text{for }\nabla\cdot u=0,
\qquad
-F'(u)\partial_x u=-\partial_x F(u),$$
so that $f=\mathcal M f_{\mathrm{loc}}(u)$ where $f_{\mathrm{loc}}$ only depends on $u$ and where $\mathcal M$ is linear of order one. The derivative then becomes a diagonal spectral layer, the block no longer needs to compute $\nabla u$ directly, and so we can weaken the embedding requirements to $X^\alpha\hookrightarrow C^0$. The additional cost is that the network error now propagates through that layer.

\begin{proposition}[Divergence Form]
\label{prop:divergence_form}
Let $\widetilde X^s$ be the ambient space from Lemma~\ref{lem:linear_postcomposition}, and let $\{\tilde e_k\}$ be an orthonormal family of eigenfunctions from the extended $L$ with the same eigenvalues $\lambda_k$, i.e. $\sum_k\lambda_k^s|\langle g,\tilde e_k\rangle|^2\le\|g\|^2_{\widetilde X^s}$. Let $\mathcal M$ be linear with $\mathcal M^{*}e_k=\mu_k\tilde e_k$ and $|\mu_k|\le C_{\mathcal M}\lambda_k^{1/2}$,
and assume that $f=\mathcal M\circ f_{\mathrm{loc}}$, where $f_{\mathrm{loc}}$ depends only on $u$, satisfies Assumption~\ref{assumption:structured_f} with $K\subset\bR^d$, and is Lipschitz on bounded sets from $X^\alpha$ into $\widetilde X^{\alpha}$ with Lipschitz constant $L^{\mathrm{loc}}_R$. Set $\beta=\alpha-1$ (so $\sigma=\theta=\frac{1}{2}$), and replace $\mathcal C^h_N$ by $D_\mu\widetilde{\mathcal C}^h_N$ where $D_\mu=\operatorname{diag}(\mu_1,\ldots,\mu_N)$. Then the following hold.
\begin{enumerate}
\item $f$ is Lipschitz on bounded sets from $X^\alpha$ into $X^{\beta}$ with constant $C_{\mathcal M}L^{\mathrm{loc}}_R$, so Proposition~\ref{prop:contraction} applies.
\item The discrete smoothing estimate from Theorems~\ref{thm:fully_discrete_fixed_point}, \ref{thm:neural_consistency}, and \ref{thm:fully_discrete_consistency} holds with the updated form
$$\left\|S_N(\tau)D_\mu\widetilde{\mathcal C}^h_N v\right\|_{\ell^2_\alpha}
\le C\,C_{\mathcal M}\,C^\alpha_{\mathrm{stab}} \tau^{-1/2} \|v\|_{\widetilde X^{\alpha}},$$
where the quadrature assumptions are used only at $s=\alpha$ for functions in $\widetilde X^\alpha$.
\item Assumption~\ref{assumption:regularity_embedding} can be weakened to $X^\alpha\hookrightarrow C^0$, only requiring that $\alpha>\frac{d_x}{2}$. The entries of $D_\mu$ are included in the spectral weights.
\item Theorem~\ref{thm:neural_consistency} holds with $\xi_\alpha(N,\Delta t)$ replaced by $\xi_{\alpha+1}(N,\Delta t)$, where the amplification is at most logarithmic when $\alpha=1$.
\end{enumerate}
\end{proposition}

\noindent The proof is given in Appendix~\ref{sec:app_operator}. On the torus, $\mathcal M=-\nabla\cdot$ and so $\mathcal M^{*}e_k=\nabla e_k$, which is $|k|$ times a normalized mode, and thus the multiplier is $|\mu_k|=|k|\asymp\lambda_k^{1/2}$. Including this with the Leray projection does not change the result since $\mathbb P$ is an orthogonal projection and commutes with $L$. Since the amplification is $\xi_{\alpha+1}$, the best choice is to pick an $\alpha$ close to $1$ within what the embedding $X^\alpha\hookrightarrow C^0$ allows. For the 2D Navier-Stokes equation, this is $\alpha=1+\iota$, with $\xi\asymp\lambda_N^{\iota/2}$, the same as in velocity form. The improvement is that the block only needs to compute $u$ and the gradient transforms are no longer needed. For a viscous conservation law in one dimension, $\alpha=1$ gives only a logarithmic amplification, and since $f_{\mathrm{loc}}=F$ is a single component-wise node in the DAG, $p_{\max}$ improves from $\frac{1}{s-1}$ to $\frac{1}{s}$.

In the divergence form, the quadrature assumptions only appear when $s=\alpha$ and only through $f_{\mathrm{loc}}$. The condition $\beta>\frac{d_x}2$ in Assumption~\ref{assumption:quadrature_stability}(ii) is replaced by $\alpha>\frac{d_x}2$, and $\beta=\alpha-1$ may be as small as $0$. In Theorems~\ref{thm:fully_discrete_fixed_point} and \ref{thm:neural_consistency} this corresponds to part (2) and part (4). In Theorem~\ref{thm:fully_discrete_consistency}, the quadrature term becomes $d_l=D_\mu(\widetilde{\mathcal C}^h_N\mathcal I_h-\widetilde{\mathcal C}_N)f_{\mathrm{loc}}(P_Nu(t_l))$, and
$$\|d_l\|_{\ell^2_\beta}\le C_{\mathcal M}\left\|(\widetilde{\mathcal C}^h_N\mathcal I_h-\widetilde{\mathcal C}_N)f_{\mathrm{loc}}(P_Nu(t_l))\right\|_{\ell^2_\alpha}\le C_{\mathcal M}C_{\mathrm{quad}}h^r\left\|f_{\mathrm{loc}}(P_Nu(t_l))\right\|_{\widetilde X^{\alpha+r}},$$
so Assumption~\ref{assumption:quadrature_consistency} is needed only when $s=\alpha$, and the weighted condition of that theorem applies to $f_{\mathrm{loc}}(P_Nu)$ in $\widetilde X^{\alpha+r}$. Similarly, Assumption~\ref{assumption:bootstrap} applies to the local part, $\|f_{\mathrm{loc}}(v)\|_{\widetilde X^{s}}\le K_R(1+\|v\|_{X^s})$ for $s\in[\alpha,\gamma']$, which implies it for $f$ since $\|\mathcal Mg\|_{X^{s-1}} \le C_{\mathcal M}\|g\|_{\widetilde X^s}$. In Lemma~\ref{lem:lifting}, condition (2) holds with $\Gamma_\alpha=C_{\mathcal M}C^\alpha_{\mathrm{stab}} L^{\mathrm{loc}}_R$ since $\|D_\mu\widetilde{\mathcal C}^h_Nv\|_{\ell^2_\beta}\le C_{\mathcal M}\|\widetilde{\mathcal C}^h_Nv\|_{\ell^2_\alpha}\le C_{\mathcal M}C^\alpha_{\mathrm{stab}}\|v\|_{\widetilde X^\alpha}$. Lastly, Theorem~\ref{thm:main} can be applied where $\xi_\alpha$ is replaced by $\xi_{\alpha+1}$, so $\chi$ and $\chi_{\mathrm s}$ are computed with $\alpha+1$ instead of $\alpha$. In particular $\chi_{\mathrm{div}}=1+\frac{(\alpha-1)_+}{2\zeta}$ and there is an extra log term when $\alpha=1$.

The divergence form now gives us that the block only needs to compute $u$, so its input dimension drops from $d(1+d_x)$ to $d$ and the gradient transforms are not needed. For non-polynomial fluxes the node is $F$ rather than $F'$ which improves $p_{\max}$ from $\frac1{s-1}$ to $\frac{1}{s}$ where $F\in C^2$ is sufficient. The embedding requirement goes to $X^\alpha\hookrightarrow C^0$, which allows for a smaller $\alpha$ and thus fewer modes need to be retained.

\begin{table}[ht]
\centering
\small
\begin{tabular}{@{}lccl@{}}
\toprule
Class of constitutive law & Embedding needed & Admissible $\alpha$ & Amplification \\
\midrule
$f=f(u)$ & $X^\alpha\hookrightarrow C^0$ & $\alpha>\frac{d_x}2$ & none, if $d_x\le3$ \\
$f=\mathcal M f_{\mathrm{loc}}(u)$, $\operatorname{ord}\mathcal M=1$ & $X^\alpha\hookrightarrow C^0$ & $\alpha\ge1$, $\alpha>\frac{d_x}2$ & $\xi_{\alpha+1}$; logarithmic if $\alpha=1$ \\
$f=f(u,\nabla u)$, genuinely nonlinear in $\nabla u$ & $X^\alpha\hookrightarrow C^1$ & $\alpha>1+\frac{d_x}2$ & $\chi=1+\frac{(\alpha-2)_+}{2\zeta}$, sharp \\
\bottomrule
\end{tabular}
\caption{Embedding and Amplification: Gradient-free laws avoid the amplification when $d_x\le3$. The divergence form removes the gradient input but amplifies the network error by $\xi_{\alpha+1}$, which is only logarithmic at $\alpha=1$. The third row is sharp (Proposition~\ref{prop:xi_sharp}), and $\chi$ can be improved to $1+\frac{(\alpha-2)_+}{2\theta}$ (see Remark~\ref{rem:gamma_free}).}
\label{tab:amplification}
\end{table}

\subsection{Lower Operator Bounds}
\label{sec:operator_lower_bound}

We have constructed upper bound rates in Theorems~\ref{thm:main} and \ref{thm:operator_rates}, which count the trainable parameters that are sufficient to approximate the solution operator with one neural network. In this section, we derive lower bound rates that hold generically for the class of governing laws. In particular, we fix the class of governing equations, allow the part of the architecture that does not depend directly on the governing equation to be arbitrary, and count the parameters needed to resolve the dynamical system. 

The upper bounds thus far were formulated so that for governing equations that share the same properties, e.g. Lipschitz constants, growth, and other constants, the proof of Theorem~\ref{thm:operator_rates} provides the same discretization (see also \cite{chandramoorthy2026spectral} for uniform bounds over classes of equations). Therefore, in that class we can reuse the same transformations and spectral weights. Following the theory of continuous nonlinear widths \cite{devore1989optimal}, we will assume that if a parameter or constant depends on the governing equation its dependence is continuous. Without this restriction, a direct bound is difficult to obtain.

We provide three cases, a lower bound for gradient-free laws $f=f(u)$, a lower bound for gradient-dependent laws $f=c_0u-H(u_x)$, and a lower bound on the spectral dependencies. The proofs follow the same style where we show that the architecture has accuracy $\eta$ with $W$ parameters chosen by a continuous map $w$. The $W$ parameters depend on the governing equation. For clarity, we summarize the process as follows. First, we construct governing equations $f_{c}$ indexed by vectors $c\in\bR^{W+1}$, with $|c|_\infty=1$, so it has one coordinate for each independent feature and so that $f_{c}$ and $f_{-c}$ differ on each feature. The set $\{|c|_\infty=1\}$ is homeomorphic to the $W$-dimensional sphere by an odd map, so the Borsuk-Ulam theorem states that every continuous map from it to $\bR^W$ maps some pair of opposite points to the same value \cite{borsuk1933drei}. This was used to bound the continuous nonlinear widths from below in \cite{devore1989optimal}. Our argument is related to the sphere embedding widths of \cite{siegel2024sharp}, except that the two opposite points are separated by the PDE solution rather than by a norm. By the theorem, there is a vector $c$ with $w(f_{c})=w(f_{-c})$, that is, two opposite governing equations with the same parameters and therefore the same output. Since that output is within $\eta$ of each true solution, the two solutions are within $2\eta$ of each other. Then we can use the PDE to measure how close the solutions of $f_{c}$ and $f_{-c}$ are as a function of the parameters and accuracy. 
In particular, this gives a lower bound relationship between $\eta$ and $W$. For each case, we need to construct a solution operator that forces this to occur explicitly. 

\begin{theorem}[Lower Bound for Gradient-Free Laws]
\label{thm:operator_lower_bound}
Assume that $\Omega=\mathbb T^{d_x}$, $L=-\nu\Delta+c_0I$ with $\nu,c_0>0$, $T>0$, $\alpha\ge0$ and $s\ge1$. Set $\psi_0:=T/6$ to be a constant function. For $0<R_f\le\psi_0/2$, let $\mathcal C_{R_f}$ be the class of governing equations described by the functions
$$f(u)=c_0u+\frac1{\psi(u)},\qquad \operatorname{supp}(\psi-\psi_0)\subset[0,1],\quad \|\psi-\psi_0\|_{C^s}\le R_f ,$$
equipped with the distance $\|\psi-\psi'\|_{C^s}$, and let $G^f_T$ be the solution operator of \eqref{eqn:PDE} with $f$. Suppose that $u_0\equiv-1$, $(\mathcal A_w)_{w\in\bR^W}$ is any family of maps from a set of data containing $u_0$ into $C([0,T];X^\alpha)$, and let $w:\mathcal C_{R_f}\to\bR^W$ be continuous. If
$$\left\|\mathcal A_{w(f)}(u_0)-G^f_T(u_0) \right\|_{C([0,T];X^\alpha)} \le \eta,$$
holds for every $f\in\mathcal C_{R_f}$,
then
$$W\ \ge\ c'\,\eta^{-1/(s+1)}$$
for sufficiently small $\eta$ and where $c'$ depends only on $c_0$, $\alpha$, $d_x$, $T$, $s$, $R_f$ and a smooth bump function $\phi$.
\end{theorem}

The idea for the proof is to reduce the problem to a scalar ODE $y'=1/\psi(y)$, whose solution reaches the value $b$ at time $\int_{-1}^b\psi(y) \,dy$. If the dynamics are accurate, then they also determine the integrals of $\psi$. We perturb $\psi_0$ by $W+1$ smooth bumps with width $1/(W+1)$ with signs determined by $c$. Since we measure it in $C^s$, the height of the bumps need scale like $\epsilon_W\sim(W+1)^{-s}$ and so two opposite governing functions differ in the integral over each bump by a value scaling like $(W+1)^{-s-1}$. Relating to the accuracy $\eta$ then gives $W\gtrsim\eta^{-1/(s+1)}$.

\begin{proof}[Proof of Theorem~\ref{thm:operator_lower_bound}]
First we will reduce the PDE to scalar ODEs. Since $\|\psi-\psi_0\|_\infty\le R_f\le\psi_0/2$, every $\psi$ takes values in the interval $[\psi_0/2,3\psi_0/2]$ and so $1/\psi$ is a bounded and Lipschitz function. Therefore the scalar problem $y'=1/\psi(y)$, $y(0)=-1$ is well-posed and $y$ is increasing with speed bounded below by $2/(3\psi_0)=4/T$. For the constant function case, the linear operator $L$ multiplies it by $c_0$. Since we included $c_0$ in both $L$ and $f$, the problem reduces to $\psi$. In fact, the constant function $y(t)$ solves \eqref{eqn:PDE}, and by uniqueness of the solution operator we have $G^f_T(u_0)=y$. The trajectory passes through every value in $[-1,3]$ before time $T$. The trajectory reaches the value $b$ at time
$$t_\psi(b)=\int_{-1}^b\psi(y) \, dy.$$
The norm reduces to $\|y(t)\,1\|_{X^\alpha}=m_\alpha|y(t)|$, where $m_\alpha:=\|1\|_{X^\alpha}=c_0^{\alpha/2}(2\pi)^{d_x/2}$ is the $X^\alpha$ norm of the constant (spatial) function $1$, which is an eigenfunction of $L$ with eigenvalue $c_0$. So we have mapped the system into simpler ODE. 

We can show that trajectories being close implies that the integrals of $\psi$ are sufficiently close. Take two function $\psi,\tilde \psi$ that are within the class, with corresponding trajectories $y,\tilde y$, and assume that $\delta:=\sup_{[0,T]}|y-\tilde y| \le 1$. For a fixed value $b\in[0,1]$, we can look at the time it takes $y$ to reach $b$, which is given by $t:=t_\psi(b)$. At that same time $t$, $\tilde y$ reaches some value $\tilde b:=\tilde y(t)$, which is close $|\tilde b-b| \le \delta$. By construction we have $t_{\tilde \psi}(\tilde b)=t=t_\psi(b)$. Therefore, the two times of $b$ differ by the time $\tilde y$ needs to travel from the value $\tilde b$ to the value $b$:
$$\left|\int_{-1}^b(\psi(z)- \tilde \psi(z)) \, dz \right|
=\left|t_{\tilde \psi}(\tilde b)-t_{\tilde \psi}(b)\right|
=\left|\int_b^{\tilde b} \tilde \psi(z) \, dz \right| 
\le\frac{3\psi_0}{2}\,\delta,$$
since they start from the same initial state $-1$.
The difference between two values, by subtracting the bound above, gives us
$$\left|\int_{b_1}^{b_2}(\psi(z)-\tilde \psi(z))\, dz \right| \le 3\psi_0 \delta, $$
for $0\le b_1<b_2\le 1$. Thus accuracy in trajectories leads to accuracy in integrals of $\psi$.

Next, we perturb the trajectories by bump functions. This is a standard construction of disjoint bump functions in a $C^s$ ball \cite{devore1989optimal}. Let $\phi_i(y):=\phi((W+1)y-i+1)$, for $1\le i\le W+1$, be bump functions with disjoint supports in $[0,1]$, where $\phi\ge0$ is smooth with $\operatorname{supp}\phi\subset[0,1]$ and $\int\phi=1$. Each bump has width $1/(W+1)$, so $\int_0^1\phi_i(y)\,dy=1/(W+1)$.  Since the input to $\phi$ is multiplied by $W+1$, the $k$-th derivative is multiplied by $(W+1)^k$, and since the supports are disjoint, $\|\sum_ic_i\phi_i\|_{C^s}\le C_\phi(W+1)^s$ for $|c|_\infty\le1$, where $C_\phi:=2\|\phi\|_{C^s}$. The largest allowable height is given by
$$\epsilon_W:=\frac{R_f}{C_\phi(W+1)^s},$$
and the functions $\psi_{c}:= \psi_0+\epsilon_W\sum_ic_i\phi_i$ with $|c|_\infty=1$ define the governing equation $f_{c}\in\mathcal C_{R_f}$. The map $c \to w(f_{c})$ is continuous from $\{|c|_\infty=1\}$ to $\bR^W$, thus by the Borsuk-Ulam theorem there is a $c$ such that $w(f_{c})=w(f_{-c})$. Both governing equations are approximated by the same model $\mathcal A_{w(f_{c})}(u_0)$, so their trajectories are within $2\eta$ distance (measured in $X^\alpha$), which implies that $m_\alpha\delta\le2\eta$.

We compare the errors and the sizes. Note that if  $2\eta>m_\alpha$, then $\frac{m_\alpha R_f}{3C_\phi\psi_0\eta}<\frac{2R_f}{3C_\phi\psi_0}\le1$, since $R_f\le\psi_0/2$ and 
$$C_\phi= 2\|\phi\|_{C^s}\ge 2 \sup\phi\ge 2.$$
In that case, the bound $(W+1)^{s+1}\ge m_\alpha R_f/(3C_\phi\psi_0\eta)$ holds trivially. If that does not occur, then $\delta \le 1$ and we can pick $i$ and $|c_i|=1$. Set $[b_1,b_2]$ to be exactly the support of $\phi_i$, and thus on this interval $|\psi_{c}-\psi_{-c}|=2\epsilon_W\phi_i$. This gives 
\[
\frac{2\epsilon_W}{W+1}=\left|\int_{b_1}^{b_2}(\psi_{c}(z)-\psi_{-c}(z)) \, dz \right|\le3\psi_0\delta\le\frac{6\psi_0\,\eta}{m_\alpha}.
\]
Since $\epsilon_W=R_f/(C_\phi(W+1)^s)$, the bound gives $(W+1)^{s+1}\ge m_\alpha R_f/(3C_\phi\psi_0\eta)$, which completes the proof. 
\end{proof}

\begin{remark}[Lower and Upper Bounds on the Operator]
\label{rem:operator_lower_bound}
The class $\mathcal C_{R_f}$ has just one component-wise node defined by the function $1/\psi$. For integer $s$, we can apply Remark~\ref{rem:taylor_layer} to approximate this node with the bound from Theorem~\ref{thrm:quantitative_smooth_scalar} and with weights that depend continuously on $\psi$. Since $\psi_0=T/6$, the governing law scales like $T^{-1}$, which means that its Lipschitz constant is on the order of $T^{-1}$. One consequence is that the time window used in Proposition~\ref{prop:contraction} may be shorter than $T$ so multiple windows may be needed. The upper bound on $[0,T]$ can be continued using Proposition~\ref{prop:continuation}, where the same block is applied on $J$ windows. Since we are looking at a class of functions, $J$ depends on the class and $T$, but it does not depend on $\eta$. So we only adjust constants and not the powers of $\eta$. For integer $s$, the upper and lower bounds are thus
$$c\,\eta^{-1/(s+1)} \le W \le C\,\eta^{-\chi/s}\log\eta^{-1},$$
where $\chi=1$ when $d_x\le3$, both bounds holding uniformly for $\mathcal C_{R_f}$.
Note that there is an addition of $1$ in the lower exponent due to the construction of the proof via integrals; however, we removed this gap in Theorem~\ref{thm:hj_lower_bound} for Hamilton-Jacobi problems. The lower bound does not have a $\chi$ since the trajectories are constant in space, so no norm mismatch occurs. This bound complements the exponential lower bounds from \cite{lanthaler2026parametric}.
\end{remark}

We next consider gradient dependence, specifically, a lower bound for Hamilton-Jacobi PDE in the form $u_t+H(u_x)=\nu u_{xx}$.

\begin{theorem}[Lower Bound for Gradient Nonlinearities]
\label{thm:hj_lower_bound}
Assume that $\Omega=\mathbb T$, $L=-\nu\partial_{xx}+c_0I$ with $\nu,c_0>0$, $T>0$, $\alpha\ge0$, an integer $s\ge1$ and $0<R_f\le1$. Let $\mathcal H_{R_f}$ be the class of governing equations with $f_H(u,u_x)=c_0u-H(u_x)$ with $H\in C^\infty(\bR)$, $\operatorname{supp}H\subset[0,1]$ and $\|H\|_{C^s}\le R_f$, where the set is equipped with the distance $\|H-H'\|_{C^s}$.  Let $v_0$ be a smooth $2\pi$-periodic function with zero mean. Assume that $v_0\le1$, and $v_0=1$ on an interval $I=[x_L,x_R]$ of length $\pi<|x_R-x_L|<2\pi$. Its integral $V_0(x):=\int_0^xv_0(y)\,dy$ is periodic since $v_0$ has mean zero. The data set is $\mathcal U:=\{S V_0:0\le S\le1\}$, where the slopes are exactly $S$ on $I$. Let $(\mathcal A_w)_{w\in\bR^W}$ be maps from $\mathcal U$ into $C([0,T];X^\alpha)$, and assume that $w:\mathcal H_{R_f}\to\bR^W$ is continuous. 

If the accuracy bound
$$\sup_{u_0\in\mathcal U}\left\|\mathcal A_{w(H)}(u_0)-G^{f_H}_T(u_0) \right\|_{C([0,T];X^\alpha)} \le \eta, $$
holds for every $H\in\mathcal H_{R_f},$
then
$$W\ \ge\ c\,\left(\eta\log\eta^{-1}\right)^{-1/s},$$
for sufficiently small $\eta$, where $c>0$ only depends on $\nu$, $c_0$, $\alpha$, $T$, $s$, $R_f$ and $v_0$.
\end{theorem}

\begin{proof}[Proof of Theorem~\ref{thm:hj_lower_bound}]
We are working on $\mathbb T$ with $[0,2\pi)$ and periodic endpoints. We can translate $x$ if needed so that the set where $v_0=1$ contains an interval $I=[x_L,x_R]\subset(0,2\pi)$ whose length is $|I|=x_R-x_L>\pi$. Pick $a\in\left(0,\frac{|I|-\pi}2\right)$ and set 
$$I_a:=[x_L+a,\,x_R-a],$$
to be the interval obtained by shrinking the interval by $a$ at each end. The length is given by $|I_a|=|I|-2a>\pi$, so $I_a$ is longer than $\mathbb T\setminus I_a$, and
$$\varsigma:=|I_a|- \left(2\pi-|I_a|\right)>0. $$
Next, consider the bump function $\phi\in C^\infty(\bR)$ which satisfies the following properties: $\operatorname{supp}\phi\subset[0,1]$, $0\le\phi\le\phi_*$ and $\phi=\phi_*$ on $[\frac{1}{4},\frac{3}{4}]$. The bumps are flat at maxs and mins. We define the values $C_\phi:=\|\phi\|_{C^s}$ and $B_0:=1+\|v_0\|_\infty$.

As before, we construct two functions with the same parameters to get the bound. For $ c \in\bR^{W+1}$ with $|c|_\infty=1$ let
$$H_{c}(y):=\epsilon_W\sum_{j=1}^{W+1}c_j\,\phi\big((W+1)y-j+1\big),\qquad\epsilon_W:=\frac{R_f}{C_\phi(W+1)^s}.$$
The terms have disjoint supports in $[0,1]$, so we have that $H_{c}\in\mathcal H_{R_f}$, $\|H_{c}\|_\infty \le \epsilon_W \phi_*$ and $\|H_{c}'\|_\infty\le R_f \le 1$. The set $\{|c|_\infty=1\}$ is homeomorphic to $S^W$ by an odd map, we can apply Borsuk-Ulam theorem to get a coefficient vector $c$ such that $w(H_{c})=w(H_{-c})$. The two Hamiltonians are approximated by the same parameters and thus output the same values, which must be within $\eta$ of each other. For every $u_0\in\mathcal U$, denote the true solutions by $u_+$ and $u_-$ associated with $H_{c}$ and $H_{-c}$, and so the two true solutions are within $2\eta$ to each other with respect to the $C([0,T];X^\alpha)$. 
Their means $m_\pm(t):=\frac1{2\pi}\int_{\mathbb T} u_\pm(t,x)\,dx$ are close. Recall that constants are eigenfunctions of $L$ with eigenvalue $c_0$ and so $\|v\|_{X^\alpha}\ge c_0^{\alpha/2}\sqrt{2\pi}\,\left|\tfrac1{2\pi}\int_{\mathbb T} v \right|$. Therefore, the difference is given by:
\begin{equation}
|m_+(t)-m_-(t)| \le \frac{2\eta}{c_0^{\alpha/2}\sqrt{2\pi}},
\label{eqn:hj_means_close}
\end{equation}
for $0\le t \le T$.

Pick the index $i$ with $|c_i|=1$ and define
$$S:=\frac{i-1/2}{W+1}$$
for $u_0:=S V_0 = S \int_0^x v_0 \in\mathcal U$. Then $S$ is the center of the $i$th bump, and  spatial derivative is $\partial_x u_0=S v_0=S$ on $I$. The solutions are smooth \cite{evans2022partial}, so we can differentiate the equation in $x$ to get the slopes $z_\pm:=\partial_xu_\pm$ by solving:
$$\partial_t z_\pm + H_{ \pm c}'(z_\pm)  \partial_xz_\pm = \nu\,\partial_{xx} z_\pm, \qquad z_\pm(0)=S v_0.$$
The maximum principle gives $|z_\pm-S|\le B_0$. 

We can find a dynamical system for the mean value. If we integrate the PDE for $u_\pm$, using $\int_{\mathbb T}\partial_{xx}u_\pm\,dx=0$, then
\begin{equation}
m_\pm'(t)=-\frac{1}{2\pi}\int_{\mathbb T} H_{\pm c} \left( z_\pm(t,x) \right) \, dx,
\label{eqn:hj_mean_ode}
\end{equation}
which is controlled by the average slope. Over a short time interval, the slopes are close to $S$. At $t=0$ the slope equals $S$ on all of $I$, but diffusion will eventually change it. We bound this change on the inner interval $I_a$. We can determine a $t_1\in(0,T]$, and set $\omega:=a/(2\nu t_1)$. The perturbation $\tilde z:=z_\pm-S$ solves the linear PDE
$$\tilde z_t+b\tilde z_x-\nu \tilde z_{xx}=0\ \ \ \text{on }[0,t_1]\times I,$$
where $b:=H_{\pm c}'(z_\pm)$ satisfies $|b|\le1$, $\tilde z(0,\cdot)=0$ on $I$, and $|\tilde z|\le B_0$ at endpoints $x=x_L$ and $x=x_R$. We can form a bound using weak maximum principle. Consider a comparison with the function
\[
Z(t,x):=B_0 e ^{(\omega+\nu\omega^2)t} \left( e^{\omega(x-x_R)}+e^{\omega(x_L-x)}\right)
\]
The first term is at least $B_0$ at $x=x_R$ and similarly with the second at left endpoint $x=x_L$, and both decay towards the interior of $I$. Applying the linear PDE to $Z$ to check the residual yields
$$Z_t+bZ_x-\nu Z_{xx}=\omega B_0e^{(\omega+\nu\omega^2)t}\left((1+b)e^{\omega(x-x_R)}+(1-b)e^{\omega(x_L-x)}\right)\ge0,$$
and noting that $Z\ge|\tilde z|$ at $t=0$ and $x=x_L,x_R$. The weak maximum principle applied to $Z-\tilde z$ and $Z+\tilde z$ implies that $|\tilde z|\le Z$ on $[0,t_1]\times I$. For $x\in I_a$, both $x-x_R$ and $x_L-x$ are at most $-a$, so by replacing $\omega t_1=a/(2\nu)$, we have the bound
$$|z_\pm-S| \le 2 B_0 \, e^{a/(2\nu)}e^{-a^2/(4\nu t_1)},$$
on $[0,t_1]\times I_a$. By choosing the constants to be
$$C_2:=8B_0e^{a/(2\nu)},
\qquad 
\ell:=\log\big(C_2(W+1)\big)>1,
\qquad
t_1:=\min\left\{T,\frac{a^2}{4\nu\ell}\right\},$$
we have $e^{-a^2/(4\nu t_1)}\le e^{-\ell}=1/(C_2(W+1))$, and 
\begin{equation}
|z_\pm-S|   \le  \frac{1}{4(W+1)},
\label{eqn:hj_plateau}
\end{equation}
on $[0,t_1]\times I_a$. This shows that the two slopes are close to $S$, and improves as $W$ increases.

We next relate the slopes to the means. By \eqref{eqn:hj_plateau} on $[0,t_1]\times I_a$, the slopes are within a quarter bump, i.e. $(W+1)z_\pm-i+1\in[\frac{1}{4},\frac{3}{4}]$. In this subinterval, the only bump that is nonzero is the $i$th one (by construction) and it equals $\phi_*$ (flat), which yields
$$H_{\pm c}(z_\pm) = \pm \epsilon_W c_i \phi_*,\qquad H_{c}(z_+)-H_{-c}(z_-)= 2 \epsilon_W c_i \phi_*.$$
Taking the differences of \eqref{eqn:hj_mean_ode} and splitting the integral into
$$m_-'(t)-m_+'(t) = \frac {1}{2\pi} \int_{I_a} \left(H_{c}(z_+)-H_{-c}(z_-)\right)\,dx
+\frac{1}{2\pi} \int_{\mathbb T\setminus I_a}\left(H_{c}(z_+)-H_{-c}(z_-) \right)\, dx,$$
we will bound each term. The first integral equals $\frac{1}{\pi}  \epsilon_Wc_i\phi_*|I_a|$, and the second is at most $\frac1\pi\epsilon_W\phi_*(2\pi-|I_a|)$. We have that $|I_a|>2\pi-|I_a|$, so the first integral is larger in absolute value than the second integral. Thus by the reverse triangle inequality $|m_+'(t)-m_-'(t)|\ge\frac{\varsigma\phi_*}{\pi}\epsilon_W$ on $[0,t_1]$. Since $m_+(0)=m_-(0)$, integrating gives the bound
$$|m_+(t_1)-m_-(t_1)|=\int_0^{t_1}|m_+'(t)-m_-'(t)|\,dt \ge \frac{\varsigma\phi_*}{\pi}\,t_1\,\epsilon_W.$$

Lastly, we compare the bounds.  By \eqref{eqn:hj_means_close}, this differences between the means are at most $2\eta/(c_0^{\alpha/2}\sqrt{2\pi})$. Since $\ell>1$, we have $t_1\ge\min\{T,a^2/(4\nu)\}/\ell$, and replacing $\epsilon_W$ yields the inequality:
$$(W+1)^s \log \left(C_2(W+1)\right) \ge \frac{C_1R_f}{\eta}, \qquad C_1:=\frac{c_0^{\alpha/2}\varsigma \phi_* \min\{T,a^2/(4\nu)\}}{\sqrt{2\pi}\,C_\phi}.$$
Choose $c>0$ with $(2c)^s\frac{2}{s}<C_1R_f$. If $W+1<2c\,(\eta\log\eta^{-1})^{-1/s}$, then for sufficiently small $\eta$ we would have  $\log(C_2(W+1))\le \frac{2}{s} \log\eta^{-1}$, so the LHS is smaller than $(2c)^s \frac{2}{s} \,\eta^{-1}$, which is a contradiction. Therefore, the inequality should be $W+1 \ge 2c\,(\eta\log\eta^{-1})^{-1/s}$, and when the RHS is larger than $1$, the result is $W \ge c\,(\eta\log\eta^{-1})^{-1/s}$.
\end{proof}

For $s\ge2$, the exponent $1/s$ from Theorem~\ref{thm:hj_lower_bound} agrees (up to logs) with the upper bound uniform over this class (Section~\ref{sec:examples}). For first-order Hamilton-Jacobi equations, upper bounds that avoid the curse of parametric complexity were derived in \cite{lanthaler2026parametric}. Our lower bound is over a class of Hamiltonians and is complementary. This is a fairly simple PDE and therefore indicates that genuine nonlinearities in $u_x$ will require such a rate. We extend the lower bounds from Theorems~\ref{thm:operator_lower_bound} and \ref{thm:hj_lower_bound} to include the effect of learning the spectrum, which matches the spectral term of Theorem~\ref{thm:total_cost} up to the constant. We can show this by considering just the linear part, setting $f$ to zero.

\begin{proposition}[Spectral Projection is Sharp]
\label{prop:spectral_lower}
Let $\tilde{L}$ satisfy the prior assumptions and have eigenpairs $(\tilde{\lambda}_k,e_k)$. The norms are those induced by $\tilde{L}$. Set $T>0$, $\alpha<\gamma_{\mathrm s}$, $\rho_{\mathrm s}>0$, the governing law $f\equiv0$, and a finite set $\mathcal K$ of indices with
$$\frac{1}{T} \le \tilde{\lambda}_k < \left(\frac{\rho_{\mathrm s}}{9\eta}\right)^{1/\zeta_{\mathrm s}},$$
for $k\in\mathcal K$. For vector $ m\in[1,2]^{\mathcal K}$, suppose $L_{m }$ has eigenvalues $m_k\tilde{\lambda}_k$ for $k\in\mathcal K$ and $\tilde{\lambda}_k$ otherwise, with the same corresponding eigenfunctions. Let $(\mathcal A_w)_{w\in\bR^W}$ be maps from $B:=\{u_0:\|u_0\|_{X^{\gamma_{\mathrm s}}} \le \rho_{\mathrm s}\}$ to $C([0,T];X^\alpha)$, and let $w:[1,2]^{\mathcal K}\to\bR^W$ be continuous weights. If, for every $m \in[1,2]^{\mathcal K}$, the following holds
$$\sup_{u_0\in B} \left\|\mathcal A_{w(m )}(u_0)-e^{-tL_{m }}u_0\right\|_{X^\alpha_T} \le \eta, $$
then $W\ge|\mathcal K|$. Assuming the two-sided Weyl bound $c_Wk^{2/d_x} \le \tilde{\lambda}_k \le C_Wk^{2/d_x}$, we have $W \ge c \eta^{-d_x/(\gamma_{\mathrm s}-\alpha)}$ for sufficiently small $\eta$. More generally, for an operator of order $2\ell$, with $k^{2\ell/d_x}$ in the Weyl bound, the same argument yields $W\ge c\,\eta^{-d_x/(\ell(\gamma_{\mathrm s}-\alpha))}$.
\end{proposition}

\begin{proof}[Proof of Proposition~\ref{prop:spectral_lower}]
We follow the same proof style as before. Let $K:=|\mathcal K|$ and $W\le K-1$. The boundary of $[-1,1]^K$ is homeomorphic to $S^{K-1}$ by a homeomorphism that commutes with $c\mapsto-c$. Applying the Borsuk-Ulam theorem to $c\mapsto w(m(c))$, with $m(c):=\frac{1}{2}(3+c)\in[1,2]^K$, gives $c$ on that boundary with $w(m (c))=w(m (-c))$. On the boundary, some coordinates satisfy $|c_k|=1$, so there is a mode whose multiplier is $1$ for one of the two vectors and $2$ for the other, i.e., $\{m _k(c),m _k(-c)\}=\{1,2\}$.

Next we examine the separation on the mode.  We can choose initial data to have all its mass on that mode, e.g. $u_0:=\rho_{\mathrm s}(\tilde{\lambda}_k)^{-\gamma_{\mathrm s}/2}e_k \in B$. Since $f\equiv0$, we can explicitly solve the PDE and get two solutions: $e^{-\tilde{\lambda}_kt}u_0$ and $e^{-2\tilde{\lambda}_kt}u_0$.  At the time $t=(\tilde{\lambda}_k)^{-1}\le T$, their difference is given by 
$$\left(e^{-1}-e^{-2}\right) u_0= \left(e^{-1}-e^{-2}\right) \rho_{\mathrm s}(\tilde{\lambda}_k)^{-\gamma_{\mathrm s}/2}e_k.$$
By assumption, $e_k$ is an eigenfunction of $\tilde{L}$ with eigenvalue $\tilde{\lambda}_k$, and by construction we have  $\|e_k\|_{X^\alpha} =(\tilde{\lambda}_k)^{\alpha/2}$. Therefore, their difference in the $X^\alpha$ norm induced by $\tilde{L}$ is given by:
$$\left(e^{-1}-e^{-2}\right) \rho_{\mathrm s} (\tilde{\lambda}_k)^{(\alpha-\gamma_{\mathrm s})/2} = \left(e^{-1}-e^{-2}\right) \rho_{\mathrm s}(\tilde{\lambda}_k)^{-\zeta_{\mathrm s}},$$
recalling that $\zeta_{\mathrm s}=\frac{\gamma_{\mathrm s}-\alpha}{2}$. The condition that $\tilde{\lambda}_k<(\rho_{\mathrm s}/(9\eta))^{1/\zeta_{\mathrm s}}$ on $\mathcal K$ gives $\rho_{\mathrm s}(\tilde{\lambda}_k)^{-\zeta_{\mathrm s}}>9\eta$. Using $(e^{-1}-e^{-2}) 9 > 2$, the difference has $X^\alpha$ norm greater than $2\eta$. But both solutions are within $\eta$ of the same trajectory $\mathcal A_{w(m(c))}(u_0)$, which is a contradiction by triangle inequality. Therefore, $W\ge K$. 

The Weyl bounds give $\#\{k:\tilde{\lambda}_k<\Lambda\} \ge (\Lambda/C_W)^{d_x/2}-1$ and $\#\{k:\tilde{\lambda}_k<1/T\} \le (c_WT)^{-d_x/2}$. With $\Lambda:=(\rho_{\mathrm s}/(9\eta))^{1/\zeta_{\mathrm s}}$, the number of modes with $T^{-1}\le\tilde{\lambda}_k<\Lambda$ is at least $c(\rho_{\mathrm s}/\eta)^{d_x/(2\zeta_{\mathrm s})}$ for sufficiently small $\eta$, and $d_x/(2\zeta_{\mathrm s})=d_x/(\gamma_{\mathrm s}-\alpha)$. This completes the proof.
\end{proof}

The bound is sharp, since for $f\equiv0$, exact spectral weights $e^{- m_k\tilde{\lambda}_k\Delta t}$ produce $e^{-tL_{m}}$ on the first $N$ modes, so the only error is from truncating and encoding. Since $e^{-tL_{m}}$ commutes with $P_N$ and is a contraction on $X^\alpha$,
$$\left\|(I-P_N)e^{-tL_{m}}u_0\right\|_{X^\alpha} \le \rho_{\mathrm s} (\tilde{\lambda}_{N+1})^{-\zeta_{\mathrm s}} \le \frac{\eta}{2}$$
when $\tilde{\lambda}_{N+1}\ge(2\rho_{\mathrm s}/\eta)^{1/\zeta_{\mathrm s}}$. By Weyl's law, this holds for some $N\le C(\rho_{\mathrm s}/\eta)^{d_x/(\gamma_{\mathrm s}-\alpha)}$, and choosing $h$ so that the encoding error is at most $\eta/2$ gives accuracy $\eta$. So the spectral term, specifically the exponent in Theorem~\ref{thm:total_cost}, is sharp when the spectrum is unknown.  Learning a self-adjoint linear operator with a known eigenspace was investigated in \cite{de2023convergence}. Proposition~\ref{prop:spectral_lower} is a parametric version, where the number of unknowns is quantified by Weyl's law.

\section{Explicit Examples}
\label{sec:examples}

In this section, we apply the main theorems to several popular equations. The results are summarized in Table~\ref{tab:examples}. The key is that the structural exponent does not depend on the dimension, and even for complicated dynamics considered here, the nonlinear governing equations can be associated with fairly simple DAGs. Upper bounds for the terminal time solutions for Allen-Cahn, Cahn-Hilliard, and the two-dimensional Navier-Stokes were shown through stable spectral schemes emulation in \cite{chandramoorthy2026spectral}. In \cite{kovachki2021universal}, they showed bounds for Navier-Stokes by emulating a pseudo-spectral method. We show that our bounds hold along the whole trajectory with shared weights and that we have lower bounds that match or nearly match.

 We consider the torus $\Omega=\mathbb T^{d_x}$ with $H=L^2(\mathbb T^{d_x};\bR^d)$ and $L=-\nu\Delta+c_0I$, where $\nu,c_0>0$. In this setting,  $L$ is strictly positive, $X^s\simeq H^s(\mathbb T^{d_x};\bR^d)$, the eigenfunctions are trigonometric, and Weyl's gives $\lambda_N\asymp N^{2/d_x}$. Given the equation $u_t-\nu\Delta u=g(u,\nabla u)$, we write it in the form \eqref{eqn:PDE} by shifting $f=g+c_0u$. This adds one affine node to the DAG, but does not change the structural exponent. On the torus the trapezoidal rule satisfies the quadrature assumptions (see Remark~\ref{rem:stable_quadrature_example}), the FFT gives the fast transforms (see Assumption~\ref{assumption:spectral_factorization}), and the Moser estimates give Assumption~\ref{assumption:bootstrap} for smooth governing laws. This reduces the overhead from several assumptions. Unless otherwise stated, we set
$$\beta=\alpha-1,\qquad \alpha<\gamma<\alpha+1,\qquad \theta=\frac{1}{2},\qquad \zeta=\frac{\gamma-\alpha}{2} \in \left(0,\frac{1}{2}\right).$$

\subsection{Allen-Cahn with Explicit Constants}
\label{sec:explicit_constants}

To show that the settings of the main theorems are obtainable, we show one example explicitly. Let $\Omega=\bT^1$, $L = -\partial_{xx}+I$ with $\lambda_k = 1+k^2$, and the Allen-Cahn governing law $f(u)=2u-u^3$. Since it is polynomial, $p_f=0$ is attained.

 The law is gradient-free, so we just need $u$ and thus the embedding simplifies to $X^\alpha\hookrightarrow C^0$, that is, $\alpha>\frac{1}{2}$. We take $\alpha=\frac{7}{4}<2$, so there is no amplification. Since $f$ maps $H^\alpha$ into $H^1$, we may set $\beta=1$, which satisfies $\beta>\frac{1}{2}$ (needed for Assumption~\ref{assumption:quadrature_stability}(ii)) and gives $\theta=\frac{5}{8}$. Setting $\gamma=\frac{11}{4}$, gives us $\zeta=\frac{1}{2}$. The Lipschitz estimate of $f$ into $X^1=H^1$ involves $u'$, so there we use $X^\alpha\hookrightarrow C^1$, which holds since $\alpha>\frac{3}{2}$. With $R_0=1$ and $R=3$, one can show that $\|u\|_{C^1}\le m_K=C_{\mathrm{emb}}R\approx 4$ on the ball, and $f$ has Lipschitz constant $L_R\le 207$ from $X^\alpha$ to $X^1$. The contraction condition of Proposition~\ref{prop:contraction} then yields $T_0\approx 10^{-4}$. The horizon is short because $L_R$ grows like $m_K^2$ and $T_0$ scales like $L_R^{-1/\theta}=L_R^{-8/5}$. Longer times can be reached by the continuation arguments and energy estimate.

For the nonlinear part, since $\alpha<2$, we have $\varepsilon\asymp\eta$. Approximating the cubic term through $u^3=\frac14\big((u+u^2)^2-(u-u^2)^2\big)$ with Yarotsky's construction of the square \cite{yarotsky2017error} gives a block $\widehat f_\varepsilon$ with $206$, $244$,  and $290$ weights at $\varepsilon=10^{-2},10^{-3},10^{-4}$. By Lemma~\ref{lem:sharing}, this one block is shared across every grid point, time level and sweep, and the spectral layers add $3N$ shared weights, with $N\asymp\eta^{-1}$ by Theorem~\ref{thm:main}. The trainable count $W_{\mathrm{train}}=W(\widehat f_\varepsilon)+3N$ is therefore
$$
\begin{array}{c|c|cc|c}
\eta & N & W(\widehat f_\varepsilon) & 3N & W_{\mathrm{train}}\\\hline
10^{-2} & 10^{2} & 206 & 3\times10^{2} & 5\times10^{2}\\
10^{-3} & 10^{3} & 254 & 3\times10^{3} & 3\times10^{3}\\
10^{-4} & 10^{4} & 300 & 3\times10^{4} & 3\times10^{4}
\end{array}
$$
independently of the number of grid points, time levels and sweeps. The block grows by less than $100$ weights over two orders of magnitude in accuracy, so the trainable count is dominated by the spectral weights, which reflect the regularity of the data rather than the law.

The values above use $R_0=1$, which is valid for one window. However, to extend to longer times we need the solution to remain in the ball. The solutions tend to the states $\pm1$ where $\|{\pm1}\|_{X^\alpha}=\sqrt{2\pi}$. Since $\|u\|_{L^\infty}\le0.54\|u\|_{X^{7/4}}$, if we start with $\|u_0\|_{X^\alpha} \le 1$ then by the maximum principle it will satisfy $|u|\le1$ for all time, and using mild form on the unit interval gives the bound $\|u(t)\|_{X^{7/4}}\le9.1$ for all $t$ and a uniform bound in $X^\gamma$ for $t\ge1$. Using this value, we have that  $R_0 \approx 9.1$ implies $m_K\approx 25$, $L_R \approx 8\times10^3$ and $T_0\approx3\times10^{-7}$. Therefore, to move through a unit interval, the method needs about $3\times10^6$ windows by Proposition~\ref{prop:continuation}. If we just used contraction, we would need to rescale per-window error by $2^{-3\times10^6}$. However, using the energy bounds from Lemma~\ref{lem:allen_cahn_window} gives $\Theta_{\mathrm{win}}=e^{3T/2}$ and reduces the accuracy factor to $A_J\approx4\times10^{12}$ for $T_{\mathrm{final}}=1$, provided each window contains at least two time steps so that $\Delta t\,\Gamma_0^2\le1$. This means that we need less than $650$ more parameters for the nonlinear block, due to the log scaling.

\subsection{Reaction-Diffusion Equations and Viscous Conservation Laws}

First we look at reaction-diffusion equations. Consider $u_t-\nu\Delta u=g(u)$ with $g:\bR^d\to\bR^d$ applied component-wise and $g(0)=0$. Since the law is gradient-free, the network needs only $X^\alpha\hookrightarrow C^0$, that is, $\alpha>\frac{d_x}2$. We take $\frac{d_x}{2}<\beta<\alpha$ and $g\in C^s$ with $s\ge\lceil\beta\rceil+1$. We can apply Moser estimates \cite{bahouri2011fourier, runst2011sobolev} to show that $u\mapsto g(u)$ is Lipschitz on bounded sets from $X^\alpha$ to $X^\beta$, and thus the contraction theory Proposition~\ref{prop:contraction} can be applied. The DAG is simply a single component-wise $C^s$ node, so $p_{\max}=1/s$, and Theorem~\ref{thm:main} gives
$$W(\widehat f_\varepsilon)\le C\eta^{-\chi/s}\log\eta^{-1}.$$
For $d_x\le3$ the exponents can be chosen so that $\frac{d_x}2<\beta<\alpha<2$, for example, $\alpha=\frac{7}{4}$ and $\beta=\frac{13}{8}$ in three dimensions, and so the factor is $\chi=1$. The same exponent $s$ yields well-posedness and approximation bounds, so $p_{\max}\le1/(\lceil\beta\rceil+1)$ in this setting.

\begin{example}[Allen-Cahn]
\label{ex:allen_cahn}
For $g(u)=u-u^3$ the graph consists of two multilinear nodes, $u\odot u$ and $(u\odot u)\odot u$, therefore we have $p_f=0$, and using Corollary~\ref{cor:structural}, $W(\widehat f_\varepsilon)\le C \log\eta^{-1}$. The same holds for any polynomial reaction with diffusion in every variable, including the diffusive FitzHugh-Nagumo system.
\end{example}

\begin{example}[Arrhenius Equation]
\label{ex:arrhenius}
For a non-polynomial reaction term, consider $g(u)=A\,e^{-E/(u+a)}$ with $u+a$ bounded away from zero on $K$. The single node is $C^\infty$ near $K$, so $p_f=0$ but it cannot be  attained. Thus we have $W(\widehat f_\varepsilon)\le C_p \eta^{-p}\log\eta^{-1}$ for every $p>0$.
\end{example}

\paragraph{Viscous Conservation Laws.} Consider $u_t+\partial_xF(u)=\nu u_{xx}$ on $\mathbb T$ with $F\in C^s$, $s\ge3$ and $F(0)=F'(0)=0$. The governing law can be written as $f(u,u_x)=-F'(u)\,u_x+c_0u$. In one dimension, we can pick $\alpha\in(\frac{3}{2},2)$, so $X^\alpha\hookrightarrow C^1$ and $\chi=1$. With $\beta=\alpha-1\in(\frac{1}{2},1)$ and $s\ge3$, $f$ is Lipschitz from $X^\alpha$ to $X^\beta$, where $s\ge3$ is needed to control the difference $F'(u)-F'(v)$ in $H^\beta$. The graph (Figure~\ref{fig:dags}(b)) has a component-wise node $F'\in C^{s-1}$ and a bilinear node, so $p_{\max}=\frac1{s-1}$ and
$$W(\widehat f_\varepsilon)\le C\,\eta^{-1/(s-1)}\log\eta^{-1}.$$

In divergence form, we have $f=-\partial_xF(u)$, and the derivative can be moved to a spectral layer (Proposition~\ref{prop:divergence_form}). With $\alpha=1$ the amplification is only logarithmic, and the exponent improves to $\frac{1}{s}$.

\begin{example}[Viscous Burgers]
\label{ex:burgers}
For $F(u)=\frac12u^2$, the node $F'$ is the identity and the graph reduces to one bilinear node (Figure~\ref{fig:dags}(a)). Therefore, $p_f=0$ and $W(\widehat f_\varepsilon)\le C\log\eta^{-1}$. The same holds for $f=-(u\cdot\nabla)u$ in any dimension.
\end{example}

\paragraph{Viscous Hamilton-Jacobi equations.} Consider $u_t+H(u_x)=\nu u_{xx}$ on $\mathbb T$ with $H\in C^s$, $s\ge2$, and $H(0)=0$, so that $f(u,u_x)=c_0u-H(u_x)$. Similar to the conservation laws, we take $\alpha\in(\frac{3}{2},2)$ and $\beta=\alpha-1$, so $X^\alpha\hookrightarrow C^1$ and $\chi=1$. If we write $H(u_x)-H(v_x)=(u_x-v_x)\int_0^1H'(v_x+\tau(u_x-v_x))\,d\tau$ and use the properties of $H^\beta(\mathbb T)$, the condition $s\ge2$ leads to a Lipschitz bound for $f$ from $X^\alpha$ to $X^\beta$. The governing law written as a DAG with a single component-wise $C^s$ node applied to $u_x$ gives $p_{\max}=\frac{1}{s}$ and
$$W(\widehat f_\varepsilon)\le C\,\eta^{-1/s} \log\eta^{-1},$$
where $C$ depends on $H$ only through the norm $\|H\|_{C^s}$. For $s$, the Taylor construction in Remark~\ref{rem:taylor_layer} attains this bound with continuous weights. Thus, the bound is uniform over the class $\mathcal H_{R_f}$ from Theorem~\ref{thm:hj_lower_bound}, where $\|H\|_{C^s} \le R_f$ makes the constants in the discretization estimates uniform. In this setting, we have
$$c\,\left(\eta\log\eta^{-1}\right)^{-1/s}\ \le\ W\ \le\ C\,\eta^{-1/s}\log\eta^{-1}.$$
This result is over the class $H\in\mathcal H_{R_f}$, which is smooth with $p_f=0$ for each element. No uniform bound over this class can improve the $\frac{1}{s}$ exponent.

\subsection{Two-Dimensional Navier-Stokes}

Consider $u_t-\nu\Delta u+\mathbb P[(u\cdot\nabla)u]=0$, $\nabla\cdot u=0$, on $\mathbb T^2$, where $\mathbb P$ is the Leray projection. The space is set to $H=L^2_{0,\sigma}(\mathbb T^2;\bR^2)$, the mean-zero divergence-free fields, and $L=-\nu\Delta$ is the Stokes operator. The linear operator is strictly positive on $H$ and commutes with $\mathbb P$. The governing equation is $f=\mathbb P f_{\mathrm{loc}}$ with $f_{\mathrm{loc}}(u,\nabla u)=-(u\cdot\nabla)u$, so Lemma~\ref{lem:linear_postcomposition} can be applied with $\widetilde H=L^2_0(\mathbb T^2;\bR^2)$. Since $\mathbb P$ is the orthogonal projection onto $H$, $\mu_k\equiv1$ (from Lemma~\ref{lem:linear_postcomposition}) and the Leray projection does not add additional costs. This also shows that the network can approximate the system without seeing the pressure. 

First, we consider the velocity form. The embedding into $C^1$ requires $\alpha>2$, so we take $\alpha=2+\iota$ and $\beta=1+\iota$ for small $\iota>0$. The properties of $H^{1+\iota}(\mathbb T^2)$ allow $f_{\mathrm{loc}}$ to be Lipschitz from $X^\alpha$ to $\widetilde X^\beta$. The DAG is simply one bilinear node (Figure~\ref{fig:dags}(a)), so $p_f=0$. The neural network error is amplified with factor $\xi_\alpha\asymp\lambda_N^{\iota/2}$, so the block needs accuracy $\varepsilon\asymp\eta^\chi$ with $\chi=\frac{\gamma-2}{\gamma-2-\iota}$. However, because $p_f=0$, this just changes the constant, i.e., $$W(\widehat f_\varepsilon)\le C\log\varepsilon^{-1}\le C\chi\log\eta^{-1}+C,$$
so the block has $C\log\eta^{-1}$ parameters.

If we use the divergence form, we can reduce the embedding requirement. The nonlinearity can be written as $f=-\mathbb P\nabla\cdot(u\otimes u)$, where the local law $f_{\mathrm{loc}}(u)=u \otimes u$ only sees $u$, and the divergence part becomes a diagonal spectral layer with $|\mu_k|\asymp|k|$ (see Proposition~\ref{prop:divergence_form}). The embedding only requires $\alpha>1$, so we can set $\alpha=1+\iota$ and $\beta=\iota$. This removes the gradient transforms, but the network error goes through the derivative layer and is amplified by $\xi_{\alpha+1}\asymp\lambda_N^{\iota/2}$, as in velocity form. Since $p_f=0$, the block count again is $C\log\eta^{-1}$.

For the long horizon times, recall that the two-dimensional equations are globally well-posed for smooth data and have $X^\gamma$ bounds on every finite horizon. Thus the continuation theory, Proposition~\ref{prop:continuation}, extends the result to any $T_{\mathrm{final}}$ with the same block. As discussed in Section~\ref{sec:long_time}, the constant depends on $T_{\mathrm{final}}$, and exponentially on the number of windows if only the contraction is used.

\paragraph{Cahn-Hilliard.} We consider a fourth order equation $u_t+\epsilon^2\Delta^2u=\Delta(u^3-u)$ on $\mathbb T^{d_x}$ with $d_x\le3$. Since the mean of $u$ is conserved, we can work with mean-zero fields, i.e., a nonzero mean $m_0$ is removed by writing $u=m_0+w$, and this does not change the law being polynomial. Set $H=L^2_0(\mathbb T^{d_x})$ and $L=\epsilon^2\Delta^2$, which is strictly positive on $H$ with eigenvalues $\lambda_k=\epsilon^2|k|^4$, so that $X^s\simeq H^{2s}_0$. The governing equation is $f=\mathcal M f_{\mathrm{loc}}$ with $\mathcal M=\Delta$ and $f_{\mathrm{loc}}(u)=u^3-u$. Since $\mathcal M^*e_k=-|k|^2e_k$, the spectral multipliers satisfy $|\mu_k|=\epsilon^{-1}\lambda_k^{1/2}$, so we can apply Proposition~\ref{prop:divergence_form} with $\beta=\alpha-1$. Since we use $s=\alpha$ in the quadrature, we need $X^\alpha\hookrightarrow C(\mathbb T^{d_x})$, i.e. $\alpha>\frac{d_x}{4}$, which holds for $\alpha=1$ when $d_x\le 3$. Then the local part $f_{\mathrm{loc}}$ is Lipschitz from $X^1=H^2$ to $\widetilde X^1=H^2$. The DAG for $f_{\mathrm{loc}}$ has two multilinear nodes so $p_f=0$, and the spectral layers need $4N$ weights (for $D_{\mathrm{an}}$, $D_S$, $D_W$ and $D_\mu$). The network error goes through the layer defined by $D_\mu$ so it gets amplified by the factor $\xi_{\alpha+1}=\xi_2$, but it is only logarithmic for $\alpha=1$. Since $p_f=0$, the only change is to constants and we get
$$W(\widehat f_\varepsilon)\le C\log\eta^{-1}.$$
Since $L$ is a fourth order operator, by Remark~\ref{rem:higher_order} with $m=2$ we have $N\asymp M \asymp \eta^{-d_x/(2(\gamma_{\mathrm s}-\alpha))}$. The example shows that the theory applies to higher order $L$.

\subsection{Summary}

Table~\ref{tab:examples} summarizes the two complexity measures used in this work, i.e., the trainable parameters (Theorem~\ref{thm:main}) and the size of the full network (Theorem~\ref{thm:total_cost}). Since we can loop the blocks and apply the network to space-time points using parameter sharing, the complexity is mostly on the block and spectral terms. Let
$$\upsilon_0:=\frac{d_x}{r} + \frac1{\zeta_{\mathrm t}}$$
denote the exponent of the number $MQ$ of space-time points, thus the full network has $W(\mathcal A_n) \le C\eta^{-\upsilon}(\log\eta^{-1})^2$ with $\upsilon=\upsilon_0+p_{\max}\chi_{\mathrm s}$. The best choice for $r$ is $r=\gamma_{\mathrm s}-\alpha$, which leads to $M\asymp N$.

\begin{table}[ht]
\centering
\small
\setlength{\tabcolsep}{4.5pt}
\begin{tabular}{@{}llccclcl@{}}
\toprule
Equation & Local law $f$ & Nodes & $d_x$ & $p_{\max}$ & Block $W(\widehat f_\varepsilon)$ & Spectral & Full network $\upsilon$ \\
\midrule
Allen-Cahn & $u-u^3$ & $2$ mult. & $\ge1$ & $0$ & $C\log\eta^{-1}$ & $3N$ & $\upsilon_0$ \\
Reaction-diffusion, $g\in C^s$ & $g(u)$ & $1$ comp. & $\le3$ & $\frac1s$ & $C\eta^{-1/s}\log\eta^{-1}$ & $3dN$ & $\upsilon_0+\frac1s$ \\
Viscous Burgers & $-u\,\partial_xu$ & $1$ mult. & $1$ & $0$ & $C\log\eta^{-1}$ & $4N$ & $\upsilon_0$ \\
Conservation law, $F\in C^s$ & $-F'(u)\,\partial_xu$ & $1$ c.\ $+\,1$ m. & $1$ & $\frac1{s-1}$ & $C\eta^{-\frac1{s-1}}\log\eta^{-1}$ & $4N$ & $\upsilon_0+\frac1{s-1}$ \\
Viscous Hamilton--Jacobi, $H\in C^s$ & $-H(\partial_xu)$ & $1$ comp. & $1$ & $\frac1s$ & $C\eta^{-1/s}\log\eta^{-1}$ & $4N$ & $\upsilon_0+\frac1s$ \\
Cahn--Hilliard & $\Delta(u^3-u)$ & $2$ mult. & $\le3$ & $0$ & $C\log\eta^{-1}$ & $4N$ & $\upsilon_0$ \\
2D Navier-Stokes & $-\mathbb P[(u\cdot\nabla)u]$ & $1$ mult. & $2$ & $0$ & $C\log\eta^{-1}$ & $10N$ & $\upsilon_0$ \\
\bottomrule
\end{tabular}
\caption{Trainable and Total Complexity for accuracy $\eta$ in $C([0,T];X^\alpha)$, with $\upsilon_0=\frac{d_x}{r}+\frac1{\zeta_{\mathrm t}}$ and $N\asymp\eta^{-d_x/(\gamma_{\mathrm s}-\alpha)}$, except for the (4th order) Cahn-Hilliard where $N\asymp\eta^{-d_x/(2(\gamma_{\mathrm s}-\alpha))}$. }
\label{tab:examples}
\end{table}

\section{Conclusion}
\label{sec:conclusion}

In empirical tests, neural networks of moderate size obtain small errors on the standard nonlinear PDE benchmarks. We showed that for semilinear parabolic equations $u_t+Lu=f(u,\nabla u)$ with $-L$ generating an analytic semigroup, the structure of the equations gives an explanation for this. The trainable cost of approximating the solution operator splits into two terms. The first is the governing law, which appears through a single point-wise block shared across grid points, time points, and Picard sweeps. This block needs $C\eta^{-p\chi_{\mathrm s}}\log\eta^{-1}$ parameters for every $p$ above the structural exponent $p_f$. The exponent is determined by the computation graph of $f$, applies to laws that are nonlinear in $\nabla u$, and does not depend on the spatial dimension or the number of state components. In particular, for the two-dimensional Navier-Stokes equations, the nonlinear block costs only $C\log\eta^{-1}$ parameters, where the Leray projection is absorbed into the spectral layers. The second term is data-based and appears in the learned spectral layers with a bounded number of weights per mode, i.e., $N\asymp\eta^{-d_x/(\gamma_{\mathrm s}-\alpha)}$ total. Since the weights are tied, neither term grows with the grid, the time step or the number of sweeps, which is consistent with the ability of these networks to transfer across resolutions. Since all time levels are updated at once via the Picard iteration, the depth is polylogarithmic in $\eta^{-1}$.

The analysis only uses the max-norm error of the block on the states observed by the iteration, and its Lipschitz constant is not needed. The error is controlled in the stronger norm of $C([0,T];X^\alpha)$ along the whole trajectory, since only the exact discrete map needs to contract and parabolic smoothing improves point-wise errors to $X^\alpha$. All of the other sources of error are derived from the mild formulation and bounded separately. The cost is the factor $\chi_{\mathrm s}\ge1$ in the exponent, which exceeds $1$ only when $\alpha>2$. It is sharp as an estimate, but it reflects the formulation rather than the method, since it disappears for gradient-free laws when $d_x\le3$.

Both terms have lower bounds for architectures whose law-dependent parameters are continuous. Over a class of viscous Hamilton-Jacobi laws, at least $c(\eta\log\eta^{-1})^{-1/s}$ parameters are needed, which for $s\ge2$ matches the upper bound up to logarithmic factors. Over a gradient-free class, at least $c\,\eta^{-1/(s+1)}$ are needed. When the spectrum is not known, one weight per mode is needed, and thus the spectral error term is sharp. For Allen-Cahn the constants are explicitly computed. The block has a few hundred parameters on a single time window, so the trainable count is controlled by the spectral weights. Energy estimates can be used to extend the result to long horizons with growth set by the one-sided rate of the law rather than by the contraction.

The results focus on the approximation rate and not generalization, which is left to future work. In particular, the algebraic decay of the errors in the size of the training set is consistent with the algebraic parameter counts shown here, but proving it requires additional statistical analysis \cite{taniguchi2026generalization,weihs2026generalization,chandramoorthy2026spectral}. The error is measured in $X^\alpha$, which forces $\alpha>2$ for gradient laws when $d_x\ge2$. Bounds in weaker norms follow directly and measuring the error in such norms from start to finish should remove most of the amplification. The proofs use a first-order Lawson-Euler step, which limits the temporal accuracy. The analysis looks at one operator at a time so extending it to families of equations, as in multiple-operator learning \cite{weihs2026deep,weihs2026generalization,weihs2026multiple}, would need additional structural conditions on the family of graphs, as was done for classes of stable spectral schemes in \cite{chandramoorthy2026spectral}.

\section*{Acknowledgments}
This work was supported in part by the National Science Foundation under grant NSF 2606034.

\bibliographystyle{plain}
\bibliography{references}

\appendix

\section{Proofs}
\label{sec:deferred}

The proofs are deferred here for readability. The sections are grouped by similar techniques or results.

\subsection{Functional Analysis and Spaces}
\label{sec:app_spaces}

\begin{proof}[Proof of Proposition~\ref{prop:contraction}]
By the assumptions in the proposition and Assumption~\ref{assumption:semigroup_regular}, we have the semigroup bound $ \|S(\cdot)u_0\|_{X_T^\alpha} \leq \|u_0\|_{X^\alpha},
$
and the Duhamel bound
$$\|\mathcal D[g]\|_{X_T^\alpha}
\leq C T^\theta\|g\|_{X_T^\beta},$$ 
where $\theta:=1-\frac{\alpha-\beta}{2}>0$ and $C$ depends on $\alpha-\beta$ but is independent of $T$.

For a fixed initial function $u_0\in X^\alpha$, we set the radius of the ball to $R:= 1 + 2\|u_0\|_{X^\alpha}$.  Then the semigroup assumption yields
$$\|S(\cdot)u_0\|_{X_T^\alpha} \leq \|u_0\|_{X^\alpha}<\frac{R}{2}.$$
By the local Lipschitz assumption on $f$ with $v=0 \in B_R^\alpha$ and using that $f(0,0)=0$, we can bound the force term by:
$$ \|f(u,\nabla u)\|_{X_T^\beta} = \|f(u,\nabla u)-f(0,0)\|_{X_T^\beta} \leq L_R \|u\|_{X_T^\alpha},$$
for every $u\in B_R^\alpha$. Therefore, $\|f(u,\nabla u)\|_{X_T^\beta} \leq L_R R$. 

To prove that the ball is invariant, i.e., $\Phi(B_R^\alpha)\subseteq B_R^\alpha$, take any $u\in B_R^\alpha$, then
\begin{align*}
\|\Phi[u]\|_{X_T^\alpha} &\leq
\|S(\cdot)u_0\|_{X_T^\alpha}+\|\mathcal D[f(u,\nabla u)]\|_{X_T^\alpha} \\
&\leq \|u_0\|_{X^\alpha}+ C T^\theta \|f(u,\nabla u)\|_{X_T^\beta} \\
&\leq \|u_0\|_{X^\alpha} + C T^\theta L_R R \\
&\leq \frac{R}{2} + \frac{R}{2} =  R,
\end{align*}
holds when $T_0>0$ is chosen so that $CT^\theta L_R\leq\frac{1}{2}$  for every $T\in (0, T_0]$.

Next, to show that the map is contractive, take any $u,v\in B_R^\alpha$. Then by the assumptions on the Duhamel operator and by its linearity
\begin{align*}
\|\Phi[u]-\Phi[v]\|_{X_T^\alpha}
= \|\mathcal D[f(u,\nabla u)-f(v,\nabla v)]\|_{X_T^\alpha}
&\leq C T^\theta \|f(u,\nabla u)-f(v,\nabla v)\|_{X_T^\beta} \\
&\leq C T^\theta L_R \|u-v\|_{X_T^\alpha}.
\end{align*}
Let $\kappa_0=CT^\theta L_R$. Since $T\leq T_0$, by construction we have $\kappa_0 \leq \frac{1}{2}<1$. Thus the map is a (strict) contraction on $B_R^\alpha$,
$$\|\Phi[u]-\Phi[v]\|_{X_T^\alpha} \leq
\kappa_0 \|u-v\|_{X_T^\alpha}.$$
Since $B_R^\alpha$ is a closed subset of the Banach space $X_T^\alpha$, it is complete and thus by the Banach fixed point theorem there exists a unique fixed function $u\in B_R^\alpha$ of $\Phi$ which satisfies $u=\Phi[u]$. By Definition~\ref{def:mild}, $u$ is the unique mild solution on $[0,T]$ in $B_R^\alpha$.
\end{proof}

\begin{proof}[Proof of Proposition~\ref{prop:instantaneous_regularity}]
In the mild formulation, using Assumption~\ref{assumption:semigroup_regular} and the assumptions in the proposition,  we have
\[
\begin{aligned}
\|u(t)\|_{X^\gamma} &\leq \|S(t)u_0\|_{X^\gamma}
+ \left\| \int_0^t S(t-s)f(u(s),\nabla u(s))\,ds
\right\|_{X^\gamma}\\
&\leq C t^{-(\gamma-\alpha)/2}\|u_0\|_{X^\alpha}
+ C t^{1-(\gamma-\beta)/2} \|f(u,\nabla u)\|_{X_T^\beta},
\end{aligned}
\]
for $t>0$.
Since $u\in X_T^\alpha$, by assumption we have
$f(u,\nabla u)\in X_T^\beta$. Since $\gamma-\beta<2$, the right-hand side is finite for every $t>0$ which implies that $u(t)\in X^\gamma$ for every $t\in(0,T]$. To show that $u\in C((0,T];X^\gamma)$,  note that for any interval $[\delta,T]$ the first term of the mild solution will have the term $S(\cdot-\delta)S(\delta)u_0$ with $S(\delta)u_0\in X^\gamma$. This is continuous by using the continuity of the semigroup on $X^\gamma$. The second term is Duhamel component which lies in $C([0,T]; X^\gamma)$ by the continuity argument following Assumption~\ref{assumption:semigroup_regular} with $s=\beta$ and $r=\gamma$.
\end{proof}

\begin{proof}[Proof of Lemma~\ref{lem:derived_regularity}]
The result of (i) is direct, it follows from $u(t)\in B^\alpha_R$ for every $t$ and the definition of $F_R$.

For (ii), if $t>0$ we can apply Proposition~\ref{prop:instantaneous_regularity} to get
$$ \|u(t)\|_{X^\gamma} \le C t^{-\zeta}\|u_0\|_{X^\alpha} + C t^{1-\bar\sigma} \|g\|_{X^\beta_T}
\le C t^{-\zeta} R + C t^{\theta-\zeta} F_R,$$
using $1-\bar\sigma=\theta-\zeta>0$, which is the same as $\gamma-\beta<2$. Then, we multiply by $t^{\zeta}$ and take the supremum which gives (ii) for $t\le T$.

Lastly, for (iii) we can set $\tau=t-s$ and write the mild form as:
$$u(t)-u(s)=\left[S(\tau)-I\right] u(s)+\int_s^tS(t-r)g(r)\,dr .$$
Using the spectral representation and applying the inequality $|e^{-x}-1| \le x^{\zeta}$ for $x\ge0$ and $\zeta\in[0,1]$, the first term is bounded by:
$$\left\|\left[S(\tau)-I\right] u(s) \right\|_{X^\alpha}^2 =\sum_k \lambda_k^\alpha \left| e^{-\lambda_k\tau} - 1 \right|^2 |u_k(s)|^2 
\le\tau^{2\zeta} \sum_k \lambda_k^{\alpha+2\zeta} |u_k(s)|^2 
= \tau^{2\zeta} \|u\|_{X^\gamma}^2,$$
since $\alpha+2\zeta=\gamma$. Using the bound (ii), the first term is controlled by $\tau^{\zeta}R_\gamma s^{-\zeta} = (t-s)^{\zeta} s^{-\zeta} R_\gamma $. For the second term, applying Assumption~\ref{assumption:semigroup_regular} and (i) yields:
$$\left\|\int_s^t S(t-r) g(r) \,dr\right\|_{X^\alpha} \le C F_R \int_s^t(t-r)^{-\sigma} \, dr = \frac{CF_R}{\theta}\tau^{\theta} 
\le C F_R \tau^{\zeta} T^{\theta-\zeta},$$
where $\theta>\zeta$. Since $s\le T$, we have that $T^{\theta-\zeta} \le s^{-\zeta} T^{\theta}$, and adding the two bounds (with $C\ge 1$) yields the results. The bound on $[g]_\zeta$ follows from the (point-wise) Lipschitz inequality \eqref{eqn:point-wise_lipschitz}.
\end{proof}

\begin{proof}[Proof of Lemma~\ref{lem:bootstrap}]
We iterate over $\theta$ sized steps by setting $s_0:=\alpha$, $s_{j+1}:= \min\{s_j+\theta,\gamma'\}$ and $J:= \lceil(\gamma'-\alpha)/\theta \rceil$, so $s_J=\gamma'$. Let $B_j:=\sup_{0\le t\le T} \|u(t)\|_{X^{s_j}}$ with $B_0\le R$ and assume $B_j<\infty$. 

By Assumption~\ref{assumption:bootstrap} with $s=s_j$, the force $g=f(u,\nabla u)$ satisfies the bound $\sup_t\|g(t)\|_{X^{s_j-2\sigma}} \le K_R(1+B_j)$, where $s_j-2\sigma \ge \beta \ge 0$. The mild form gives $u(t)=S(t)u_0+\int_0^tS(t-\tau)g(\tau)\,d\tau$, the integral converges in $X^\beta$ because $g$ is continuous with values in $X^\beta$. The difference between the target exponent and the force's exponent is $s_{j+1}-(s_j-2\sigma)\le\theta+2\sigma=1+\sigma<2$, so applying the smoothing from Assumption~\ref{assumption:semigroup_regular} yields $\|S(t-\tau)g(\tau)\|_{X^{s_{j+1}}}\le C(t-\tau)^{-(s_{j+1}-s_j+2\sigma)/2}\|g(\tau)\|_{X^{s_j-2\sigma}}$. This has an integrable singularity, so after applying it to each truncation $P_K$ of the integrand, integrating over time, and sending $K\to\infty$ (which holds by continuity in $\|\cdot \|_{X^s}$) we get
$$\sup_{0\le t\le T}\left\|\int_0^tS(t-\tau) g(\tau) \,d\tau \right\|_{X^{s_{j+1}}} \le C T^{1-(s_{j+1}-s_j+2\sigma)/2} K_R (1+B_j).$$
The first term in the mild form is controlled by $$\|S(t)u_0\|_{X^{s_{j+1}}} \le \|u_0\|_{X^{s_{j+1}}} \le\lambda_1^{-(\gamma'-s_{j+1})/2}\|u_0\|_{X^{\gamma'}}.$$
Thus $B_{j+1}\le C(1+\|u_0\|_{X^{\gamma'}}+B_j)$, and after $J$ steps $B_J\le C'(1+\|u_0\|_{X^{\gamma'}})$ (where $C'$ depends on $J$).  The bound on the force is from Assumption~\ref{assumption:bootstrap} with $s=\gamma'$.
\end{proof}

\begin{proof}[Proof of Lemma~\ref{lem:bootstrap_time}]
Let $g(\tau):=f(u(\tau),\nabla u(\tau))$, $\kappa=\min\{1,\zeta'\}$ with $\zeta'=(\gamma'-\alpha)/2$, and fix $0\le s<t\le T$ with $\delta:=t-s$. Using the mild form at time $s$, we have the difference
$$u(t)-u(s)=\left(S(\delta)-I\right)u(s)+\int_s^t S(t-\tau) g(\tau)\,d\tau.$$
 For $\lambda>0$ and $\kappa\in[0,1]$ we have the simple inequality $ |e^{-\lambda \delta}-1| \le \min\{1,\lambda\delta\} \le (\lambda\delta)^\kappa $, so using the spectral representation, we bound:
$$\left\|(S(\delta)-I)u(s)\right\|_{X^\alpha}^2=\sum_k\lambda_k^\alpha\left|e^{-\lambda_k\delta}-1\right|^2|\langle u(s),e_k\rangle|^2\le\delta^{2\kappa}\|u(s)\|_{X^{\alpha+2\kappa}}^2.$$
Since $\kappa\le\zeta'$, we have $\alpha+2\kappa\le\gamma'$. The Hilbert space embedding and Lemma~\ref{lem:bootstrap} combine to give
$$\left\|(S(\delta)-I)u(s)\right\|_{X^\alpha}\le\delta^\kappa\lambda_1^{-(\gamma'-\alpha-2\kappa)/2}\|u(s)\|_{X^{\gamma'}}\le\lambda_1^{-(\gamma'-\alpha-2\kappa)/2}R_{\gamma'}\,(t-s)^\kappa .$$

To bound the other term, we apply Lemma~\ref{lem:bootstrap} to get $\sup_\tau\|g(\tau)\|_{X^{\gamma'-2\sigma}}\le K_R(1+R_{\gamma'})$. Since $\alpha-(\gamma'-2\sigma)=2(\sigma-\zeta')$, the smoothing from Assumption~\ref{assumption:semigroup_regular} gives
$$\|S(t-\tau)g(\tau)\|_{X^\alpha}\le C(t-\tau)^{-(\sigma-\zeta')}K_R(1+R_{\gamma'}),$$ for   $\sigma>\zeta'$.

When $\sigma\le\zeta'$, the index $\alpha$ is not larger than the index $\gamma'-2\sigma$ so we can control the term using contractivity of $S$ and the first eigenvalue. Specifically, it has the same form but with power $0$ and constant $\lambda_1^{-(\zeta'-\sigma)}$. In both cases the exponent is $a':=(\sigma-\zeta')_+ \le \sigma < 1$, and thus integration in $\tau\in(s,t)$ yields
\[
\left\|\int_s^tS(t-\tau)g(\tau)\,d\tau\right\|_{X^\alpha}\le\frac{C}{1-a'}\,(t-s)^{1-a'}K_R(1+R_{\gamma'}) .
\]
The exponent is $1-a'=\min\{1,\theta+\zeta'\}\ge\kappa$, since $\theta>0$. Also, the time difference is no more than $T$, i.e. $t-s\le T$, and thus $(t-s)^{1-a'}\le T^{1-a'-\kappa}(t-s)^\kappa$, which allows us to bound the second term by at most $CT^{\min\{1,\theta+\zeta'\}-\kappa} K_R (1+R_{\gamma'}) (t-s)^\kappa$.

Adding the two bounds gives $\|u(t)-u(s)\|_{X^\alpha}\le C_{\mathrm t}(t-s)^\kappa$ with
\[
C_{\mathrm t}:=\lambda_1^{-(\gamma'-\alpha-2\kappa)/2} R_{\gamma'} + C T^{\min\{1,\theta+\zeta'\}-\kappa}K_R(1+R_{\gamma'}).
\]
For the second inequality in the lemma, $u(t)$ and $u(s)$ are both in ball of radius $R$ in $X^\alpha$, so by the Lipschitz bound \eqref{eqn:point-wise_lipschitz} we have

$$ \|g(t)-g(s)\|_{X^\beta} \le L_R \|u(t)-u(s)\|_{X^\alpha} \le L_R C_{\mathrm t} (t-s)^\kappa .$$
\end{proof}

\begin{proof}[Proof of Proposition~\ref{prop:galerkin_approximation}]
The proof is direct. Let $u_k:=\langle u,e_k\rangle_H,$
then
\begin{align*}
\|u-P_Nu\|_{X^\beta}^2
=
\sum_{k=N+1}^\infty
\lambda_k^{-(\alpha-\beta)}
\lambda_k^\alpha |u_k|^2 \le
\lambda_{N+1}^{-(\alpha-\beta)}
\sum_{k=N+1}^\infty
\lambda_k^\alpha |u_k|^2 \le
\lambda_{N+1}^{-(\alpha-\beta)}
\|u\|_{X^\alpha}^2,
\end{align*}
which proves the bound.
\end{proof}

\subsection{Neural Networks and DAG approximation}
\label{sec:app_networks}

\begin{proof}[Proof of Lemma~\ref{lem:nn_calculus}]
Both part (1) and the equal depth case in part (3) are implied directly from equation \eqref{eqn:relu_network}, since the first layers can be stacked on the common input and all other layers can be written block-wise. If the depths differ, the affine output layer $y=Az+b$ for the shallower network is replaced by the ReLU layer $(\sigma(Az+b),\sigma(-Az-b))$, which at most doubles the weights for that layer, and the $2 m_i$ nonnegative terms are moved through $\Delta-1$ more layers by diagonal identity layer ($2 m_i$ per layer). At the output layer, the terms are mapped by using $y=\sigma(y)-\sigma(-y)$ with $ 2 m_i$ weights. This adds up to no more than $W(\mathcal{N}_i)+2m_i\Delta$ additional weights. 

For part (2), since the product of sparse matrices is not necessarily sparse, we combine the output layer of $\mathcal{N}_1$ with the input layer of $\mathcal{N}_2$ with the identity layer used in part (4). That is, the output $y=Az+b$ of $\mathcal{N}_1$ is replaced by a split ReLU layer $(\sigma(Az+b),\sigma(-Az-b))$, and the first layer $A'y+b'$ of $\mathcal{N}_2$ uses $$A'\left(\sigma(Az+b)-\sigma(-Az-b)\right)+b'.$$ This at most doubles each of the modified steps' number of nonzero entries, where the depth is at most $\dep(\mathcal{N}_1)+\dep(\mathcal{N}_2)$, and $W \le 2 W(\mathcal{N}_1)+2 W(\mathcal{N}_2)$. 

For part (4), we have the identity $y=\sigma(y)-\sigma(-y)$ for all real $y$, and together with the recombination has four nonzero weights per coordinate. For part (5), the hidden layer has weights $(1,1)$ and biases $(\mathrm{a},-\mathrm{a})$. The output layer has weights $(1,-1)$ and bias $-\mathrm{a}$, which totals seven nonzero terms. Note that if $|y|\le\mathrm{a}$ then $\sigma(y+\mathrm{a})-\sigma(y-\mathrm{a})-\mathrm{a}=(y+\mathrm{a})-0-\mathrm{a}=y$, and if $y>\mathrm{a}$ we have $(y+\mathrm{a})-(y-\mathrm{a})-\mathrm{a}=\mathrm{a}$. By  symmetry it holds for $y<-\mathrm{a}$.
\end{proof}

\begin{proof}[Proof of Lemma~\ref{lem:sharing}]
Both versions map $X$ to the array whose $i$-th entry is obtained through the same composition of affine maps and activations to $X_i$. Since $\sigma$ acts element-wise and \eqref{eqn:conv1x1} applies $(A_{l},b_{l})$ at each point, the two formulation agrees at each point. The depth does not change since the layers are a one-to-one correspondence and the number of distinct parameters are the same since it is the same action. When written as dense layers, the convolution stores one copy of each layer per grid point.
\end{proof}

\begin{proof}[Proof of Theorem~\ref{thm:relu_f}]
We construct the approximating network by going through each node and creating a local approximation. Then we control the total error by propagating approximation error inductively through the topological ordering of the graph.

First, we derive the relevant constants. Since the graph $G=(\mathcal V,\mathcal E)$ is finite and acyclic, there exists a topological ordering of its vertices. We relabel the vertices so that
$$(i,j)\in\mathcal E \quad\Longrightarrow\quad
i<j.$$
Let $\mathcal V_{\mathrm{in}}$ denote the input nodes, and let $\Vnl$ denote the set of non-affine nodes. For every node $j\in\mathcal V$, let $\Psi_j$ be the map computed at node $j$, and let $z_j(x)$ for $x\in K$ be the exact intermediate variable produced during the evaluation of the DAG.

To bound the error propagation through the DAG, we need to define the Lipschitz constant and the corresponding input/output spaces for each node.  Since the intermediate variables are continuous, the input space $K$ is compact, and the graph contains finitely many nodes, every intermediate variable's range 
$$K_j := \{z_j(x): x\in K\}$$
is compact. For every node $j$ which is not an input node, let $\mathrm{pa}(j)$ denote the set of parent nodes of $j$, and define the set of admissible inputs as
$$K_j^{\mathrm{in}} := \prod_{i\in\mathrm{pa}(j)}K_i,$$
which is compact by construction.  The product is equipped with the maximum norm. By Assumption~\ref{assumption:structured_f}, every nodal map is either affine, multilinear, or component-wise $C^{s_j}$ on an open neighborhood of $K_j^{\mathrm{in}}$. Thus, each map $\Psi_j$ is continuously differentiable on some open neighborhood $U_j\supset K_j^{\mathrm{in}}$. Since $K_j^{\mathrm{in}}$ is compact, we can choose a compact set $\overline U_j$ so that
$$K_j^{\mathrm{in}} \subset
\operatorname{int}(\overline U_j) \subset
\overline U_j \subset
U_j \subset
\operatorname{dom}(\Psi_j).$$
The set $\overline U_j$ is needed as a buffer to ensure that perturbation can extend outside of the exact range while remaining valid. The map $\Psi_j$ is $C^1$, so it is locally Lipschitz on the open set $U_j$ and a locally Lipschitz map is Lipschitz on every compact subset. Therefore, $\Psi_j$ is Lipschitz on $\overline U_j$ with Lipschitz constant $L_j:=\Lip(\Psi_j, \overline U_j) <\infty$ with respect to the maximum norm.
For the affine nodes, $L_j$ can be set to operator norm.

Since the graph is finite, the maximum Lipschitz constant of a node map
$$L_{\max}:=
\max_{j\in\mathcal V \setminus\mathcal V_{\mathrm{in}}} L_j$$
is finite. Once every buffer set is fixed, let $\Lambda_1:=\max\{1,L_{\max}\}$ and let $\rho$ be the minimal distance between the admissible inputs and the boundary of the buffer sets, i.e.,
$$\rho :=
\min_{j\in\mathcal V\setminus\mathcal V_{\mathrm{in}}}
\operatorname{dist}\left( K_j^{\mathrm{in}}, \partial\overline U_j \right).$$
Since $K_j^{\mathrm{in}} \subset
\operatorname{int}(\overline U_j)$ and both sets are compact, the distance for each $j$ is positive, and since there are only finitely many nodes $\rho>0$.

This construction shows that if every parent variable at a node is perturbed by at most $\rho$ in the maximum norm, then the perturbation does not leave $\overline U_j$, where the Lipschitz estimate for $\Psi_j$ remains valid. This is needed for propagation and evaluation.

Next we construct a deep ReLU approximation of each node. Let $\delta\in(0,1)$ be the error parameter for the nodes. For each $j\in\mathcal V$, the ReLU realization
$\widetilde\Psi_{j,\delta}$ is determined by the type of node as follows. If $j$ is affine, then $\Psi_j$ is exactly realizable by linear layers, and thus we set $\widetilde\Psi_{j,\delta}:=\Psi_j$, and incur no error
$$\sup_{y\in \overline U_j}
\|\Psi_j(y)-\widetilde\Psi_{j,\delta}(y)\|=0.$$
The depth and number of weights are independent of $\delta$ and so we set $p_j=q_j=0$. If there is at least one nonlinear node, then we can neglect any fixed parameter costs associated with the finitely many linear nodes.

If node $j$ is multilinear, then since $K_j^{\mathrm{in}}$ is compact, there exists $b_j>0$ with
$K_j^{\mathrm{in}}\subset\left([-b_j,b_j]^n\right)^{r_j}$, where $r_j$ is the number of input variables associated with
node $j$. The buffer set is the cube,
$$\overline U_j := \left([-2b_j,2b_j]^n\right)^{r_j},
\qquad
U_j := \left((-3b_j,3b_j)^n\right)^{r_j},$$
and thus $\operatorname{dist}\left(K^{\mathrm{in}}_j,\partial\overline U_j\right)\ge b_j>0$. Using Theorem~\ref{thrm:quantitative_multilinear},
$\Psi_j$ can be approximated on this cube by a ReLU network
$\widetilde\Psi_{j,\delta}$ with
$$\sup_{y\in \overline U_j} \|\Psi_j(y)-\widetilde\Psi_{j,\delta}(y)\| \le \delta,$$
and with parameter counts
\[
\dep(\widetilde\Psi_{j,\delta})
\le c_j\log(\delta^{-1}),
\qquad
W(\widetilde\Psi_{j,\delta})
\le c_j\log(\delta^{-1}),
\]
where \(c_j>0\) depends only on \(\Psi_j\), \(\overline U_j\), and the node
dimensions. Thus for each multilinear node we set $p_j=0$ and $q_j=1$.

Suppose that node $j$ is a component-wise nonlinear node induced by a scalar function $g_j\in C^{s_j}$. Let $d_j$ denote the dimension of the input space of node $j$, so that
$\overline U_j\subset\bR^{d_j}$
and $\Psi_j(y_1,\ldots,y_{d_j})= \left(g_j(y_1),\ldots,g_j(y_{d_j}) \right)$.
Since $\overline U_j\subset\bR^{d_j}$ is compact, its coordinate-wise projections are compact subsets of $\bR$ and so there exists a compact interval
$[a_j,b_j]\subset\bR$
containing every coordinate of every point in $\overline U_j$.  We denote $[a_j,b_j]=:[b^-_j,b^+_j]$ to avoid confusion.
We can apply Theorem~\ref{thrm:quantitative_smooth_scalar} to $g_j$ on $[b^-_j,b^+_j]$, to get that for every $\delta\in(0,1)$, there is a ReLU network $N_{j,\delta}$ such that
$$\|g_j-N_{j,\delta}\|_{L^\infty([b^-_j,b^+_j])} \le
\frac{\delta}{\sqrt{d_j}},$$
with parameter counts
$$\dep(N_{j,\delta}) \le C_j\log(\delta^{-1}), \qquad
W(N_{j,\delta}) \le C_j\,\delta^{-1/s_j}\log(\delta^{-1}),$$
where $C_j>0$ depends only on $g_j$, $s_j$, and $[b^-_j,b^+_j]$. The full node approximation is given by
$$\widetilde\Psi_{j,\delta}(y_1,\ldots,y_{d_j})
:=\big( N_{j,\delta}(y_1), \ldots, N_{j,\delta}(y_{d_j})
\big).$$
For every $y\in\overline U_j$, the difference is 
\[
\begin{aligned}
\|\Psi_j(y)-\widetilde\Psi_{j,\delta}(y)\|^2
=\sum_{i=1}^{d_j} |g_j(y_i)-N_{j,\delta}(y_i)|^2
 &\leq d_j\left(\frac{\delta}{\sqrt{d_j}}\right)^2 =\delta^2,
\end{aligned}
\]
yielding
$$\sup_{y\in \overline U_j}
\|\Psi_j(y)-\widetilde\Psi_{j,\delta}(y)\| \le\delta.$$
Since $\widetilde\Psi_{j,\delta}$ contains  $d_j$ parallel copies of $N_{j,\delta}$, the depth is the same, while the number of weights increases by at most a factor of $d_j$. Therefore, there exists a constant $c_j>0$, depending only on
$g_j$, $\overline U_j$, and the node dimension $d_j$, such that
$$\dep(\widetilde\Psi_{j,\delta}) \le c_j\log(\delta^{-1}),
\qquad
W(\widetilde\Psi_{j,\delta}) \le
c_j\,\delta^{-1/s_j}\log(\delta^{-1}),$$
where we can set $p_j={s_j}^{-1}$ and $q_j=1$. 

Next we estimate the propagation of error through the DAG. Let $\mathcal V=\{j_1,\ldots,j_M\}$ be the the topological ordering of the graph so that every parent of $j_m$ is in
$\{j_1,\ldots,j_{m-1}\}$. For each $m=1,\ldots,M$, we define the node-wise error by
$$e_m:= \sup_{x\in K} \left\| z_{j_m}(x)-\widetilde z_{j_m}(x) \right\|,$$
and the cumulative error (with $E_0=0$) by
$$E_m:=\max_{1\le {l}\le m} e_{l}.$$
If $j_m$ is an input node, then $e_m=0$. If $j_m$ is not an input node, we can write
$$\mathbf z_{\mathrm{pa}(j_m)}(x)
:= \left(z_i(x)\right)_{i\in\mathrm{pa}(j_m)}, \qquad
\widetilde{\mathbf z}_{\mathrm{pa}(j_m)}(x)
:= \left(\widetilde z_i(x)\right)_{i\in\mathrm{pa}(j_m)},$$
and so
$$z_{j_m}(x) = \Psi_{j_m} \left( \mathbf z_{\mathrm{pa}(j_m)}(x) \right), \qquad
\widetilde z_{j_m}(x) = \widetilde\Psi_{j_m,\delta} \left( \widetilde{\mathbf z}_{\mathrm{pa}(j_m)}(x) \right).$$
We use induction to derive the error bound. Assuming that
$E_{m-1}\leq \rho$, we have
$$\left\|
\mathbf z_{\mathrm{pa}(j_m)}(x)-\widetilde{\mathbf z}_{\mathrm{pa}(j_m)}(x)\right\|_{\infty}
\leq E_{m-1} \leq \rho $$
for $x\in K$. Since
$$\mathbf z_{\mathrm{pa}(j_m)}(x)\in K_{j_m}^{\mathrm{in}}
\subset \operatorname{int}(\overline U_{j_m})$$
and
$$\operatorname{dist}\left(K_{j_m}^{\mathrm{in}},\partial\overline U_{j_m}\right)\ge \rho,$$
it follows that $\widetilde{\mathbf z}_{\mathrm{pa}(j_m)}(x)\in \overline U_{j_m}$ for all $x\in K$. In addition, the Lipschitz bound for $\Psi_{j_m}$ and the (local) approximation error for $\widetilde\Psi_{j_m,\delta}$ hold. The node-wise error is bounded by
$$\begin{aligned}
e_m &= \sup_{x\in K} \left\| \Psi_{j_m} \left(\mathbf z_{\mathrm{pa}(j_m)}(x)\right)-\widetilde\Psi_{j_m,\delta} \left(\widetilde{\mathbf z}_{\mathrm{pa}(j_m)}(x)\right) \right\| \\
&\leq \sup_{x\in K} \left\| \Psi_{j_m} \left(\mathbf z_{\mathrm{pa}(j_m)}(x)\right)-\Psi_{j_m}\left( \widetilde{\mathbf z}_{\mathrm{pa}(j_m)}(x)\right) \right\|+
\sup_{x\in K} \left\| \Psi_{j_m} \left( \widetilde{\mathbf z}_{\mathrm{pa}(j_m)}(x) \right)
-\widetilde\Psi_{j_m,\delta} \left(\widetilde{\mathbf z}_{\mathrm{pa}(j_m)}(x) \right)\right\|\\
&\leq L_{j_m}E_{m-1} + \sup_{x\in K}
\left\| \Psi_{j_m} \left( \widetilde{\mathbf z}_{\mathrm{pa}(j_m)}(x) \right) -\widetilde\Psi_{j_m,\delta} \left(\widetilde{\mathbf z}_{\mathrm{pa}(j_m)}(x) \right)\right\|,
\end{aligned} $$
using the Lipschitz constant for the exact node function $\Psi_{j_m}$. If $j_m$ is an affine node, then the second term is zero. Otherwise, by the approximation bounds we have
$$\sup_{x\in K} \left\| \Psi_{j_m} \left(
\widetilde{\mathbf z}_{\mathrm{pa}(j_m)}(x) \right)
-\widetilde\Psi_{j_m,\delta} \left(\widetilde{\mathbf z}_{\mathrm{pa}(j_m)}(x) \right)\right\| 
\leq \sup_{y\in \overline U_{j_m}} \left\| \Psi_{j_m}(y)-\widetilde\Psi_{j_m,\delta}(y) \right\|
\leq \delta,$$
and thus $e_m \leq L_{j_m}E_{m-1}+\delta$. Since we set
$\Lambda_1=\max\{ 1, L_{\max}\} \ge 1$, either case gives the bound $e_m \leq \Lambda_1 E_{m-1} + \delta$.
The cumulative error is bounded by
$$ E_m = \max\{E_{m-1},e_m\}
\le \max\{E_{m-1},\Lambda_1 E_{m-1}+\delta\}
= \Lambda_1 E_{m-1}+\delta.$$
Iterating the bound yields
$$E_m \leq \delta \sum_{{l}=0}^{m-1} \Lambda_1^{l}, $$
and the largest value of the sum is defined as $C_1$, given by
$$C_{1} := \sum_{{l}=0}^{|\mathcal V|-1} \Lambda_1^{l}.$$
Since $ m \le |\mathcal V|$ the error is bounded by $E_m
\le C_1\delta$, uniformly in $m$. In particular, if $j_{\mathrm{out}}$ is the output node, then the error is
$$\sup_{x\in K} \|f(x)-\mathcal{N}_\delta(x)\|
= e_{j_{\mathrm{out}}} \le E_{|\mathcal V|}
\le C_{1}\delta.$$
We need to choose the parameters so that the induction and bounds hold. We can set
$$\varepsilon_0 := \min\left\{ 1, \frac{\rho}{C_1} \right\},
\qquad
\delta := \frac{\varepsilon}{C_1}, \qquad \varepsilon\in(0,\varepsilon_0],$$
then $C_1\delta = \varepsilon \le \varepsilon_0\le\rho$ and the bound $E_m\le C_1\delta\le\rho$ thus closes the induction step and error step, and we have
$$\sup_{x\in K} \|f(x)-\mathcal{N}_\delta(x)\| \le \varepsilon$$
where we set $\mathcal{N}_\varepsilon :=\mathcal{N}_\delta$ to finish the construction.

Lastly, we need to show that the network $\mathcal{N}_\varepsilon$ is obtained through realization of the graph and through $\widetilde\Psi_{j,\delta}$.  Since the DAG is finite and acyclic, the nodes can be evaluated in the topological order. Along every directed path, the network depth is bounded by the sum of the depths of the nodewise networks, up to a fixed constant taking into account any affine (routing) layers. The total number of nonzero weights is at most twice the sum of the node-wise weight counts (Lemma~\ref{lem:nn_calculus}, Part 2) plus the weights for any of the identity layers (Lemma~\ref{lem:nn_calculus}, Part 4). Each identity step costs at most four weights per coordinate per layer, and every node-wise network has depth at most $c_j\log(\delta^{-1})$, so the total cost is at most $C\log(\delta^{-1})$ weights, with $C$ depending only on the DAG and the node dimensions. Since every nonlinear node adds at least $c_j\log(\delta^{-1})$ weights, we can include both together by redefining $c_j$. Therefore, there exist constants $C_D,C_W>0$, depending only on the DAG and the node dimensions, such that
$$\dep(\mathcal{N}_\varepsilon) \le C_D+ \sum_{j\in\Vnl} c_j\log(\delta^{-1}),
\qquad
W(\mathcal{N}_\varepsilon) \le C_W+ \sum_{j\in\Vnl} c_j\,\delta^{-p_j} \left(\log(\delta^{-1})\right)^{q_j}.$$
Since $\delta=\varepsilon/C_1$, we have that $\log(\delta^{-1})
= \log(\varepsilon^{-1}) + \log(C_1)$ and so $\delta^{-p_j} =
C_1^{p_j}\varepsilon^{-p_j}$.
Again, by redefining of the constants $c_j$, we obtain
$$\dep(\mathcal{N}_\varepsilon) \le C_D+ \sum_{j\in\Vnl} c_j\log(\varepsilon^{-1}), \qquad
W(\mathcal{N}_\varepsilon)\le C_W+ \sum_{j\in\Vnl} c_j\,\varepsilon^{-p_j} \left(\log(\varepsilon^{-1})\right)^{q_j}.$$
We can remove the constant terms by noting that if $\varepsilon\in(0,1)$, then $\log(\varepsilon^{-1})\ge0$ and by replacing $\varepsilon_0$ with $\min\{\varepsilon_0,e^{-1}\}$ we can assume $\log(\varepsilon^{-1}) \ge 1$. Therefore, 
$C_D \le C_D\log(\varepsilon^{-1})$ and $\varepsilon^{-p_j}
\left(\log(\varepsilon^{-1})\right)^{q_j}\ge1$, so we can combine the terms into
$$\dep(\mathcal{N}_\varepsilon) \le \sum_{j\in\Vnl} c_j\log(\varepsilon^{-1}),
\qquad
W(\mathcal{N}_\varepsilon) \le \sum_{j\in\Vnl} c_j\,\varepsilon^{-p_j} \left(\log(\varepsilon^{-1})\right)^{q_j}.$$
We note that this bound assumes that there is at least one non-affine node. This completes the proof.
\end{proof}

\subsection{Discretization and Complexity}
\label{sec:app_discrete}

\begin{proof}[Proof of Proposition~\ref{prop:temporal_quadrature}]
Fix the time index $q\in\{1,\ldots,Q\}$ and decompose the difference over the subintervals:
\begin{align}
\mathcal D[g](t_q)-(\mathcal D_Q[g])^q
&=\sum_{r=0}^{q-1}\int_{t_r}^{t_{r+1}}\left[S(t_q-s)g(s)-S(t_q-t_r)g(t_r)\right]ds
\notag\\
&=\underbrace{\sum_{r=0}^{q-1}\int_{t_r}^{t_{r+1}}S(t_q-s)\left(g(s)-g(t_r)\right)ds}_{=:E_1}
+\underbrace{\sum_{r=0}^{q-1}\left(\int_{t_r}^{t_{r+1}}\left[S(t_q-s)-S(t_q-t_r)\right]ds\right)g(t_r)}_{=:E_2} .
\label{eqn:quad_split}
\end{align}
We will repeatedly use the smoothing estimates from Assumption~\ref{assumption:semigroup_regular} in the two forms
\begin{equation}
\left\|S(v)w\right\|_{X^\alpha}\le Cv^{-\sigma}\|w\|_{X^\beta},
\qquad
\left\|LS(v)w\right\|_{X^\alpha}=\left\|S(v)w\right\|_{X^{\alpha+2}}\le Cv^{-1-\sigma}\|w\|_{X^\beta},
\label{eqn:smoothing_two}
\end{equation}
noting that $\frac{(\alpha+2)-\beta}{2}=1+\sigma$.

To bound $E_1$, we separate the first subinterval from the rest since the weighted modulus degenerates there. For $r=0$, we use only boundedness so if $q=1$
$$\left\|\int_0^{\Delta t}S(t_1-s)\left(g(s)-g(0)\right)ds\right\|_{X^\alpha}
\le2\|g\|_{\infty,\beta}\int_0^{\Delta t}C(\Delta t-s)^{-\sigma}ds
=\frac{2C}{\theta}\|g\|_{\infty,\beta}(\Delta t)^{\theta}, $$
and if $q\ge2$ then $t_q-s\ge t_{q-1}\ge\Delta t$ for $s\le\Delta t$, so the same term is at most $2C\|g\|_{\infty,\beta}\Delta t\,(\Delta t)^{-\sigma}=2C\|g\|_{\infty,\beta}(\Delta t)^{\theta}$. For $r\ge1$ and $s\in[t_r,t_{r+1}]$ the assumptions give $\|g(s)-g(t_r)\|_{X^\beta}\le[g]_\zeta(\Delta t)^{\zeta}t_r^{-\zeta}$, thus
\begin{align}
\left\|\sum_{r=1}^{q-2}\int_{t_r}^{t_{r+1}}S(t_q-s)\left(g(s)-g(t_r)\right)ds\right\|_{X^\alpha}
&\le C[g]_\zeta(\Delta t)^{\zeta}\sum_{r=1}^{q-2}t_r^{-\zeta}\int_{t_r}^{t_{r+1}}(t_q-s)^{-\sigma}ds
\notag\\
&\le C[g]_\zeta(\Delta t)^{1+\zeta}\sum_{r=1}^{q-2}t_r^{-\zeta}\,(t_q-t_{r+1})^{-\sigma},
\label{eqn:quad_beta_sum}
\end{align}
where in the last step we bounded the integrand by its supremum on the subinterval, which is finite because $t_{r+1}<t_q$ for $r\le q-2$. The last subinterval, $r=q-1$, is estimated directly by $\int_{t_{q-1}}^{t_q}(t_q-s)^{-\sigma}ds=\theta^{-1}(\Delta t)^{\theta}$, so it adds at most $C[g]_\zeta(\Delta t)^{\zeta+\theta}t_{q-1}^{-\zeta}\le C[g]_\zeta(\Delta t)^{\theta}$, since $t_{q-1}\ge\Delta t$ for $q\ge2$. Substituting $t_r=r\Delta t$ in \eqref{eqn:quad_beta_sum} yields 
$$(\Delta t)^{1+\zeta}\sum_{r=1}^{q-2}(r\Delta t)^{-\zeta}\left((q-r-1)\Delta t\right)^{-\sigma}
=(\Delta t)^{1-\sigma}\sum_{r=1}^{q-2}r^{-\zeta}(q-r-1)^{-\sigma}
\le C(\Delta t)^{1-\sigma}q^{1-\zeta-\sigma},$$
the last inequality uses the discrete Beta function bound which is valid for $\zeta<1$, $\sigma<1$ and $\zeta+\sigma=\bar\sigma<1$. Since $q\Delta t=t_q\le T$,
$$(\Delta t)^{1-\sigma}q^{1-\zeta-\sigma}=t_q^{1-\zeta-\sigma}(\Delta t)^{1-\sigma-(1-\zeta-\sigma)}=t_q^{\theta-\zeta}(\Delta t)^{\zeta}\le T^{\theta-\zeta}(\Delta t)^{\zeta}.$$
Adding up all the terms and using $(\Delta t)^{\theta}=(\Delta t)^{\theta-\zeta}(\Delta t)^{\zeta}\le T^{\theta-\zeta}(\Delta t)^{\zeta}$ for the contribution $C[g]_\zeta(\Delta t)^{\theta}$ of the term $r=q-1$, we obtain $\|E_1\|_{X^\alpha}\le C\|g\|_{\infty,\beta}(\Delta t)^{\theta}+C[g]_\zeta T^{\theta-\zeta}(\Delta t)^{\zeta}$.

Next, for $E_2$, boundedness is enough. For indices $r\le q-2$, the interval $[t_r,t_{r+1}]$ is at least $\Delta t$ away from the singularity. We can use that $S(t_q-s)-S(t_q-t_r)=\int_{t_q-s}^{t_q-t_r}LS(v)\,dv$ and the second estimate in \eqref{eqn:smoothing_two} to get
$$\left\|\left(\int_{t_r}^{t_{r+1}}\left[S(t_q-s)-S(t_q-t_r)\right]ds\right)g(t_r)\right\|_{X^\alpha}
\le C\|g\|_{\infty,\beta}\int_{t_r}^{t_{r+1}}\int_{t_q-t_{r+1}}^{t_q-t_r}v^{-1-\sigma}\,dv\,ds.$$
The inner integral does not depend on $s$, so the outer integral becomes $\Delta t$. As $r$ goes from $0$ to $q-2$ the intervals $[t_q-t_{r+1},t_q-t_r]$ partition $[\Delta t,t_q]$ and so
$$\sum_{r=0}^{q-2}\left\|\cdots\right\|_{X^\alpha}
\le C\|g\|_{\infty,\beta}\,\Delta t\int_{\Delta t}^{t_q}v^{-1-\sigma}\,dv
\le \frac{C}{\sigma}\|g\|_{\infty,\beta}\,\Delta t\,(\Delta t)^{-\sigma}
=\frac{C}{\sigma}\|g\|_{\infty,\beta}(\Delta t)^{\theta}.$$
For the last subinterval where $r=q-1$, we estimate the two terms in $E_2$ separately, using the first bound in \eqref{eqn:smoothing_two}:
$$\left\|\int_{t_{q-1}}^{t_q}S(t_q-s)ds\,g(t_{q-1})\right\|_{X^\alpha}\le C\|g\|_{\infty,\beta}\int_0^{\Delta t}v^{-\sigma}dv=\frac{C}{\theta}\|g\|_{\infty,\beta}(\Delta t)^{\theta},$$
and
$$ \left\|\Delta t\,S(\Delta t)g(t_{q-1})\right\|_{X^\alpha}\le C\|g\|_{\infty,\beta}\Delta t\,(\Delta t)^{-\sigma}=C\|g\|_{\infty,\beta}(\Delta t)^{\theta}.$$
Thus $\|E_2\|_{X^\alpha}\le C\|g\|_{\infty,\beta}(\Delta t)^{\theta}$, with $C$ depending only on $\sigma$.

Adding the two bounds together and taking the maximum completes the first part, where both sides are zero when $q=0$.

For the second part of the proposition, write the error as $E_1+E_2$ as in \eqref{eqn:quad_split}, with integration variable $\tau$. Since the modulus of $g$ does not use the weight, the subintervals do not need to be separated. For $E_1$, the assumed bound $\|g(\tau)-g(t_r)\|_{X^\beta}\le H(\Delta t)^\kappa$ and the smoothing gives
$$\|E_1\|_{X^\alpha}\le CH(\Delta t)^\kappa\int_0^{t_q}(t_q-\tau)^{-\sigma}d\tau\le CT^\theta H(\Delta t)^\kappa .$$
For $E_2$, write $S(t_q-\tau)-S(t_q-t_r)=S(t_q-\tau)(I-S(\delta))$ with $\delta=\tau-t_r\le\Delta t$. Let $s:=\beta'-2\kappa$. A spectral argument yields
$$\|(I-S(\delta))w\|_{X^{s}} \le \delta^\kappa \|w\|_{X^{s+2\kappa}} =\delta^\kappa \|w\|_{X^{\beta'}},$$
and the smoothing bound from $X^{s}$ to $X^\alpha$ includes a $(t_q-\tau)^{-a}$ term,  where $a=(\alpha-s)_+/2=\left(\sigma-\frac{\beta'-\beta}{2}+\kappa\right)_+$. Therefore,
$$\|E_2\|_{X^\alpha}\le CG(\Delta t)^\kappa\int_0^{t_q}(t_q-\tau)^{-a}d\tau\le\frac{C}{1-a}T^{1-a}G(\Delta t)^\kappa,$$
and the integral converging because $a<1$, which is the condition $\kappa<\theta+\frac{\beta'-\beta}2$.

If $s<0$, $X^s$ is the space from Definition~\ref{def:negative_spaces}, and both spectral bounds still hold. If $s>\alpha$, then $a=0$ and the bound uses $\|w\|_{X^\alpha}\le\lambda_1^{-(s-\alpha)/2}\|w\|_{X^s}$, so $C$ also depends on $\lambda_1$.
\end{proof}

\begin{proof}[Proof of Proposition~\ref{prop:spectral_exact}]
With the exact values, the encoder gives $D^{\mathrm{ex}}_{\mathrm{an}}\mathcal T_{\mathrm{an}}w=\mathcal C^h_Nw$, the synthesis step gives $\mathcal S^h_Na^r$ and $\nabla_h\mathcal S^h_Na^r$, and the analysis step gives $\mathcal C^h_N G^r$, using Assumption~\ref{assumption:spectral_factorization}. Writing out the recurrence steps yields $v^q=S_N(\Delta t)^q a^h_0+\Delta t\sum_{r<q}S_N(\Delta t)^{q-r}\widehat g^{r}$, which is step (2d) since $S_N(t_q-t_r)=S_N(\Delta t)^{q-r}$. Lastly, $S_N(\Delta t)^{\tau/\Delta t}=S_N(\tau)$, since $S_N(\Delta t)$ is diagonal with positive entries, so the output agrees with Definition~\ref{def:semigroup_lift}.
\end{proof}

\begin{proof}[Proof of Proposition~\ref{prop:discrete_ball_containment}]
Take $t \in [t_q,t_{q+1}]$, and set the interpolation weight to
$s_0:=\frac{t-t_q}{t_{q+1}-t_q}\in[0,1]$
which yields the piecewise linear interpolant in time:
$$\mathcal L_{N,Q}a(t)=
(1-s_0)\mathcal{S}_N a^q + s_0\mathcal{S}_N a^{q+1}.
$$
By the triangle inequality and homogeneity of the
\(X^\alpha\)-norm,
$$\|\mathcal L_{N,Q}a(t)\|_{X^\alpha}\leq
\max_{0\leq j\leq Q}
\|\mathcal S_Na^j\|_{X^\alpha}.$$
The supremum over \(t\in[0,T]\) yields the bound
$$\|\mathcal L_{N,Q}a\|_{X_T^\alpha}
\leq\max_{0\leq j\leq Q}
\|\mathcal S_Na^j\|_{X^\alpha} = \max_{0\leq j\leq Q}
\|a^j\|_{\ell_\alpha^2},$$
where we used that the discrete norm and function norms are equal.

Next, since $\mathcal L_{N,Q}a(t_j)=\mathcal{S}_N a^j$,
we have the lower bound
$$
\|\mathcal L_{N,Q}a\|_{X_T^\alpha}
= \sup_{t\in[0,T]}
\|\mathcal L_{N,Q}a(t)\|_{X^\alpha}
\geq \max_{0\leq j\leq Q} \|\mathcal S_Na^j\|_{X^\alpha}.$$
Using both inequalities shows that the norms are equal, and the rest of the claim follows.
\end{proof}

\begin{proof}[Proof of Corollary~\ref{cor:inputs_in_K}]
By Proposition~\ref{prop:discrete_ball_containment}, the lifting $u_a:=\mathcal L_{N,Q}a$ is in $B_R^\alpha$ and $\mathcal S_Na^q=u_a(t_q)$. Thus
$((\mathcal S^h_Na^q)_i,(\nabla_h\mathcal S^h_Na^q)_i)=(u_a(t_q,x_i),\nabla u_a(t_q,x_i))$, which lies in $K$ by definition.
\end{proof}
\begin{proof}[Proof of Lemma~\ref{lem:encoder_bounded}]
The point-wise values $v(x_i)$ are well defined by Assumption~\ref{assumption:regularity_embedding}. Since $\mathcal S_N$ is an isometry from $(\bR^N,\|\cdot\|_{\ell^2_\alpha})$ onto $(E_N,\|\cdot\|_{X^\alpha})$, duality on the finite-dimensional space $E_N$ gives
$$\left\|\Pi_{N,h}v\right\|_{X^\alpha}
=\sup_{\substack{\phi\in E_N\\ \|\phi\|_{X^{-\alpha}}\le1}}
\left|\left\langle \mathcal S_N\mathcal C^h_N\mathcal I_hv,\phi\right\rangle\right|
=\sup_{\substack{\phi\in E_N\\ \|\phi\|_{X^{-\alpha}}\le1}}
\left|\sum_{i=1}^{M(h)}\omega_i^{N,h}v(x_i)\phi(x_i)\right|
\le C^\alpha_{\mathrm{stab}}\|v\|_{X^\alpha},
$$
using the $s=\alpha$ case from Assumption~\ref{assumption:quadrature_stability}.
\end{proof}

\begin{proof}[Proof of Theorem~\ref{thm:fully_discrete_fixed_point}]
We first derive a discrete semigroup estimate. Let
$v\in X^\beta\cap C(\overline{\Omega})$
and $\tau>0$. Since $\mathcal S_N$ is an isometry from
$(\bR^N,\|\cdot\|_{\ell_\alpha^2})$ onto
$(E_N,\|\cdot\|_{X^\alpha})$, the following norms are equal:
$$\|S_N(\tau)\mathcal C_N^h v\|_{\ell_\alpha^2}
= \|\mathcal S_NS_N(\tau)\mathcal C_N^h v\|_{X^\alpha}.$$
Using duality of the $X^\alpha$ norm on $E_N$, we can characterize the norm by:
$$\|\mathcal S_NS_N(\tau)\mathcal C_N^h v\|_{X^\alpha}
= \sup_{\substack{\phi\in E_N\\
\|\phi\|_{X^{-\alpha}}\leq 1}} \left|
\left\langle \mathcal S_NS_N(\tau)\mathcal C_N^h v,\phi \right\rangle \right|.$$
The semigroup satisfies $S(\tau)\mathcal S_N=\mathcal S_NS_N(\tau)$ and is self-adjoint, so
\begin{align*}
\left\langle
\mathcal S_NS_N(\tau)\mathcal C_N^h v,\phi
\right\rangle
=\left\langle S(\tau)\mathcal S_N\mathcal C_N^h v,\phi \right\rangle
= \left\langle \mathcal S_N\mathcal C_N^h v,S(\tau)\phi \right\rangle
=\sum_{i=1}^{M(h)} \omega_i^{N,h} v(x_i)(S(\tau)\phi)(x_i),
\end{align*}
using the quadrature rule and the definitions of $\mathcal C_N^h$ and $\mathcal S_N$. Since $\phi\in E_N$ and $S(\tau)E_N\subseteq E_N$, using Assumption~\ref{assumption:quadrature_stability} leads to
$$\|S_N(\tau)\mathcal C_N^h v\|_{\ell_\alpha^2}
\leq
C_{\mathrm{stab}} \, \|v\|_{X^\beta}
\sup_{\substack{\phi\in E_N\\
\|\phi\|_{X^{-\alpha}}\leq 1}}
\|S(\tau)\phi\|_{X^{-\beta}}\leq
C C_{\mathrm{stab}}
\tau^{-(\alpha-\beta)/2}
\|v\|_{X^\beta},$$
where we used the semigroup property for negative spaces and the supremum.

Next, we prove a bound on the fully discrete map. Recall that $a=(a^0,\ldots,a^Q)\in B_R^{N,Q}$, and the continuous variable is defined by the lifting $u_a:=\mathcal L_{N,Q}a$. Proposition~\ref{prop:discrete_ball_containment} implies that
$\|u_a\|_{X_T^\alpha}=\|a\|_{X_{N,Q}^\alpha}\leq R$, so
$u_a\in B_R^\alpha$. For each $q=0,\ldots,Q$, define the continuous reconstruction
$\widetilde u_a^q:=\mathcal S_Na^q\in E_N$,
whose values and derivatives at the sample points are $\mathcal S_N^h a^q =\mathcal I_h\widetilde u_a^q$ and
$\nabla_h\mathcal S_N^h a^q=\mathcal I_h^{\text{grad}}\widetilde u_a^q$. We define the continuous forcing function by
$$F_a^q(x):=f\left( \widetilde u_a^q(x),
\nabla\widetilde u_a^q(x)\right),$$ whose value at the sample point $x_i$ is exactly
$$F_a^q(x_i) = f\left( (\mathcal S_N^h a^q)_i, (\nabla_h\mathcal S_N^h a^q)_i \right).$$
Therefore, applying the discrete coefficient map gives
$$\mathcal C_N^h\mathcal I_hF_a^q
= \mathcal C_N^h\left[ \mathcal F^h\left( \mathcal S_N^h a^q,
\nabla_h\mathcal S_N^h a^q \right)\right],$$
which are the values inside the discrete fixed-point map. Since $\widetilde u_a^q=u_a(t_q)$, the force is also equal to
$F_a^q =f\left(u_a(t_q), \nabla u_a(t_q)\right)$
and so by assumption $\|F_a^q\|_{X^\beta} \leq F_R$.

Taking norms in the definition of $\Psi^h_{N,Q}$ leads to
\begin{align*}
\|(\Psi_{N,Q}^h(a))^q\|_{\ell_\alpha^2}
&\leq
\|a_0^h\|_{\ell_\alpha^2} +\Delta t
\sum_{r=0}^{q-1} \|S_N(t_q-t_r)\mathcal C_N^h\mathcal I_hF_a^r\|_{\ell_\alpha^2}\\
&\leq \|a_0^h\|_{\ell_\alpha^2} + C C_{\mathrm{stab}}F_R \Delta t \sum_{r=0}^{q-1} (t_q-t_r)^{-(\alpha-\beta)/2},
\end{align*}
where we used that $S_N(t)$ is contractive on $\ell_\alpha^2$ and the discrete semigroup estimate above.
For exponents $\sigma=\frac{\alpha-\beta}{2}<1$ and $\theta=1-\sigma$, and recalling that $t_q-t_r=(q-r)\Delta t$, we have
\begin{align*}
\Delta t \sum_{r=0}^{q-1}(t_q-t_r)^{-\sigma}
=\Delta t^{1-\sigma} \sum_{r=0}^{q-1}(q-r)^{-\sigma}
=\Delta t^{1-\sigma}
\sum_{j=1}^{q}j^{-\sigma}
\leq (1-\sigma)^{-1}(q\Delta t)^{1-\sigma}
\leq (1-\sigma)^{-1} T^\theta.
\end{align*}
We redefine the constant $C$ to absorb the $\sigma$ terms, thus yielding
$$\|(\Psi_{N,Q}^h(a))^q\|_{\ell_\alpha^2}
\leq
\|a_0^h\|_{\ell_\alpha^2}
+
C C_{\mathrm{stab}}F_RT^\theta
\quad \Longrightarrow \quad
\|\Psi_{N,Q}^h(a)\|_{X_{N,Q}^\alpha}
\leq
\|a_0^h\|_{\ell_\alpha^2}
+ C C_{\mathrm{stab}}F_RT^\theta.$$

To show that the discrete map is contractive, take  $a,b\in B_R^{N,Q}$, and let $u_a:=\mathcal L_{N,Q}a$, $u_b:=\mathcal L_{N,Q}b$, $\widetilde u_a^q:=\mathcal S_Na^q$, $\widetilde u_b^q:=\mathcal S_Nb^q$. In the same way as before, $\mathcal S_N^h a^q =\mathcal I_h\widetilde u_a^q$ and $\mathcal S_N^h b^q=\mathcal I_h\widetilde u_b^q$, and we define $F_a^q$ and $F_b^q$ accordingly. Since $\widetilde u_a^q=u_a(t_q)$ and
$\widetilde u_b^q=u_b(t_q)$, we can apply the local Lipschitz assumption and the definition of the discrete norm to bound
$$\|F_a^q-F_b^q\|_{X^\beta}\leq L_R\|u_a(t_q)-u_b(t_q)\|_{X^\alpha}=L_R\|a^q-b^q\|_{\ell_\alpha^2}
\leq L_R \|a-b\|_{X_{N,Q}^\alpha}.$$
The difference of the full maps is bounded by
\begin{align*}
\|(\Psi_{N,Q}^h(a))^q - (\Psi_{N,Q}^h(b))^q\|_{\ell_\alpha^2} 
&= \Delta t \left\|\sum_{r=0}^{q-1} S_N(t_q-t_r) \mathcal C_N^h\mathcal I_h\left( F_a^r-F_b^r\right) \right\|_{\ell_\alpha^2}\\
&\leq C C_{\mathrm{stab}}L_R \Delta t \sum_{r=0}^{q-1}
(t_q-t_r)^{-(\alpha-\beta)/2} \|a-b\|_{X_{N,Q}^\alpha}
\\
&\leq C C_{\mathrm{stab}}L_RT^\theta \|a-b\|_{X_{N,Q}^\alpha}.
\end{align*}
Taking the maximum over $q$ gives $$\|\Psi_{N,Q}^h(a)-\Psi_{N,Q}^h(b)\|_{X_{N,Q}^\alpha}
\leq \kappa \|a-b\|_{X_{N,Q}^\alpha},$$ where $\kappa:=C C_{\mathrm{stab}}L_RT^\theta$ is independent of $N$, $h$, and $Q$.

It follows directly that if $\|a_0^h\|_{\ell_\alpha^2}+C C_{\mathrm{stab}}F_RT^\theta \leq R$, then  $\Psi_{N,Q}^h(B_R^{N,Q}) \subseteq B_R^{N,Q}$.
In addition, if $\kappa<1$, then $\Psi_{N,Q}^h$ is a contraction on the complete set $B_R^{N,Q}$, and thus by the Banach fixed-point theorem there exists a unique fixed point $a^h_{N, Q} \in B_R^{N,Q}$.
\end{proof}

\begin{proof}[Proof of Theorem~\ref{thm:fully_discrete_consistency}]
Recall that $\sigma=\frac{\alpha-\beta}{2}$ and $\theta:=1-\sigma$, so that the conditions are $0<\sigma<1$ and $\theta>0$. Define the following functions:
$$F(t):=f(u(t),\nabla u(t)),
\qquad F_N(t):=f(P_N u(t),\nabla P_N u(t)),
\qquad \widetilde u_N^h(t) :=\mathcal S_N\mathcal C_N^h\mathcal I_h u(t).$$

The proof uses two (discrete) semigroup estimates. The first estimate was derived at the beginning of the proof of Theorem~\ref{thm:fully_discrete_fixed_point}, i.e., for any $v\in X^\beta\cap C(\overline\Omega)$ and $\tau>0$, we have
$$\left\|\mathcal S_NS_N(\tau)\mathcal C_N^hv\right\|_{X^\alpha}\le CC^\beta_{\mathrm{stab}}\tau^{-\sigma}\|v\|_{X^\beta}. $$
The second estimate uses Lemma~\ref{lem:encoder_bounded} and the contractivity of $S_N(\tau)$ to get:
$$\|\mathcal S_N S_N(\tau) \mathcal C_N^hv\|_{X^\alpha} \le C^\alpha_{\mathrm{stab}}\|v\|_{X^\alpha}$$
for any $\tau\ge0$. Specifically, this shows that $\|\widetilde u_N^h(t)\|_{X^\alpha}\le C^\alpha_{\mathrm{stab}}R$. Therefore $\widetilde u_N^h(t)$ and $P_Nu(t)$ lie in a fixed bounded subset of $X^\alpha$, uniformly in $N$, $h$ and $t$, and the Lipschitz constant of $f$ on this set can be absorbed by $C_R$.

We estimate directly the difference between one step of the fully discrete map and the solution $u$. Pick an index $q\in\{0,\ldots,Q\}$, then
$$\begin{aligned}
\left(\Phi_{N,Q}^h[u]\right)(t_q)-u(t_q)
=& S(t_q)\mathcal S_N a_0^h-S(t_q) u_0 +
\Delta t\sum_{{l}=0}^{q-1} S(t_q-t_{l})\mathcal S_N
\mathcal C_N^h\mathcal I_h \left[ f(\widetilde u_N^h(t_{l}), \nabla\widetilde u_N^h(t_{l})) \right]
\\
&\ \ -\int_0^{t_q}S(t_q-s)F(s)\,ds.
\end{aligned}
$$
Since $\mathcal S_N\mathcal C_N=P_N$, we have
$$S(t_q)\mathcal S_N a_0^h-S(t_q) u_0 = S(t_q)\mathcal S_N
\left(a_0^h-\mathcal C_Nu_0\right) + S(t_q)(P_N-I)u_0.$$
We decompose the difference into several parts:
$$\left(\Phi_{N,Q}^h[u]\right)(t_q)-u(t_q)
=
E_{\mathrm{init},q} +E_{\mathrm{Gal},q} +E_{\mathrm{nl},q}
+E_{\mathrm{quad},q} +E_{\mathrm{proj},q}  +E_{\mathrm{time},q},
$$
where each term represents a specific approximation step:
$$E_{\mathrm{init},q} := S(t_q)\mathcal S_N \left(a_0^h-\mathcal C_N u_0\right),$$
$$E_{\mathrm{Gal},q} := S(t_q)(P_N-I)u_0,$$
$$ E_{\mathrm{nl},q} :=\Delta t\sum_{{l}=0}^{q-1} S(t_q-t_{l}) \mathcal S_N\mathcal C_N^h \mathcal I_h \left[ f(\widetilde u_N^h(t_{l}), \nabla\widetilde u_N^h(t_{l})) -F_N(t_{l}) \right],$$

$$E_{\mathrm{quad},q} := \Delta t\sum_{{l}=0}^{q-1} S(t_q-t_{l})\mathcal S_N \left( \mathcal C_N^h\mathcal I_hF_N(t_{l}) -\mathcal C_NF_N(t_{l}) \right),$$
$$E_{\mathrm{proj},q} := \Delta t\sum_{{l}=0}^{q-1} S(t_q-t_{l})
\left( P_NF_N(t_{l})-F(t_{l}) \right),$$
and
$$E_{\mathrm{time},q} := \Delta t\sum_{{l}=0}^{q-1}
S(t_q-t_{l})F(t_{l}) - \int_0^{t_q}S(t_q-s)F(s)\,ds.$$
Each term can be bounded separately. When $q=0$, all sums are set to zero, so in the bounds that deal with singular kernel, we will assume $q\geq1$.

First, by the contraction property of $S(t)$ on $X^\alpha$ and since $\mathcal S_N$ is an isometry, we have:
$$ \|E_{\mathrm{init},q}\|_{X^\alpha} \leq
\|a_0^h-\mathcal C_Nu_0\|_{\ell_\alpha^2} 
= E_{\mathrm{init}}^h,
\qquad \|E_{\mathrm{Gal},q}\|_{X^\alpha} \leq \|(I-P_N)u_0\|_{X^\alpha}.$$

Next, consider $E_{\mathrm{nl},q}$, which measures the effect of evaluating the law at encoded samples $\widetilde u^h_N$ instead of $P_Nu$. Since $P_N=\mathcal S_N\mathcal C_N$, we have
$$\widetilde u_N^h(t)-P_Nu(t) =
\mathcal S_N \left( \mathcal C_N^h\mathcal I_hu(t) -\mathcal C_Nu(t) \right),$$
so by the isometry of $\mathcal S_N$ and by
Assumption~\ref{assumption:quadrature_consistency} with $s=\alpha$, we get
$$\|\widetilde u_N^h(t)-P_Nu(t)\|_{X^\alpha}
= \left\| \mathcal C_N^h\mathcal I_hu(t) -\mathcal C_Nu(t) \right\|_{\ell_\alpha^2}
\leq C_{\mathrm{quad}}h^r \|u(t)\|_{X^{\alpha+r}}.$$
For $t>0$, since $r\leq\gamma-\alpha$ when $\gamma_{\mathrm s}=\gamma$, the continuous embedding gives
$\|u(t)\|_{X^{\alpha+r}} \leq C\|u(t)\|_{X^\gamma}$. The bootstrapped case where $r$ can be larger than $\gamma-\alpha$ is left for the end of the proof. By the local Lipschitz continuity of $f$, 
$$\left\| f(\widetilde u_N^h(t),\nabla\widetilde u_N^h(t)) -
F_N(t) \right\|_{X^\beta}
\leq C_Rh^r\|u(t)\|_{X^\gamma},$$
for all $t>0$, where the Lipschitz constant is absorbed in the constant above. At $t=0$, we use boundedness
$$ \| f(\widetilde u_N^h(0),\nabla\widetilde u_N^h(0)) - F_N(0) \|_{X^\beta} \leq C_R.$$
The $l=0$ contribution thus gives
$$\|E_{\mathrm{nl},q}^{(0)}\|_{X^\alpha}
\leq C_R\Delta t\,t_q^{-\sigma} 
\leq C_R(\Delta t)^{1-\sigma}
= C_R(\Delta t)^\theta,$$
since $t_q\geq\Delta t$ and the others terms give
$$\|E_{\mathrm{nl},q}^{(>0)}\|_{X^\alpha}
\leq C_R h^r \Delta t \sum_{{l}=1}^{q-1} (t_q-t_{l})^{-\sigma} \|u(t_{l})\|_{X^\gamma}.$$
We set the exponents to $\zeta:=\frac{\gamma-\alpha}{2}$ and $\bar\sigma:=\frac{\gamma-\beta}{2}$, so
$0<\zeta<1$, $\bar\sigma<1$, and $\sigma+\zeta=\bar\sigma$.
By Proposition~\ref{prop:instantaneous_regularity},
$\|u(t)\|_{X^\gamma} \leq C_R t^{-\zeta} + C_R t^{1-\bar\sigma},$ and thus we will have contributions from both terms within the sum.
For the first contribution, when $q\geq2$,
$$\Delta t \sum_{{l}=1}^{q-1} (t_q-t_{l})^{-\sigma}t_{l}^{-\zeta}
= (\Delta t)^{1-\sigma-\zeta} \sum_{{l}=1}^{q-1} (q-{l})^{-\sigma}{l}^{-\zeta}.$$

We split the sum at $l=\lfloor q/2\rfloor$ to get
$$\sum_{{l}=1}^{q-1} (q-{l})^{-\sigma}{l}^{-\zeta}
\leq C q^{-\sigma} \sum_{{l}=1}^{\lfloor q/2\rfloor}{l}^{-\zeta} +
Cq^{-\zeta} \sum_{j=1}^{\lceil q/2\rceil}j^{-\sigma}
\leq Cq^{1-\sigma-\zeta},$$
since $0<\zeta<1$ and $0<\sigma<1$. So the first sum is bounded by 
$$\Delta t \sum_{{l}=1}^{q-1} (t_q-t_{l})^{-\sigma}t_{l}^{-\zeta}
\leq C (\Delta t\,q)^{1-\sigma-\zeta} = Ct_q^{1-\bar\sigma}
\leq CT^{1-\bar\sigma}. $$
The second sum is bounded by
$$\Delta t \sum_{{l}=1}^{q-1} (t_q-t_{l})^{-\sigma}t_{l}^{1-\bar\sigma} \leq
T^{1-\bar\sigma} \Delta t \sum_{{l}=1}^{q-1}(t_q-t_{l})^{-\sigma}
= T^{1-\bar\sigma}(\Delta t)^\theta \sum_{j=1}^{q-1}j^{-\sigma}
\leq CT^{1-\bar\sigma}t_q^\theta
\leq CT^{1-\bar\sigma+\theta}.$$
And together the max error for the nonlinear reconstruction is
$\max_{0\leq q\leq Q} \|E_{\mathrm{nl},q}\|_{X^\alpha}
\leq C_R(h^r+(\Delta t)^\theta)$.

For the spatial quadrature error, set $d_l:=\mathcal C_N^h\mathcal I_h F_N(t_l)-\mathcal C_NF_N(t_l),$
and thus the error is 
$$E_{\mathrm{quad},q}=\Delta t\sum_{l=0}^{q-1}S(t_q-t_l)\,\mathcal S_Nd_l .$$
Since $\mathcal S_N$ is an isometry from $(\bR^N,\|\cdot\|_{\ell^2_\beta})$ to $(E_N,\|\cdot\|_{X^\beta})$, the smoothing from Assumption~\ref{assumption:semigroup_regular} gives
\begin{equation}
\left\|S(\tau)\mathcal S_Nd_l\right\|_{X^\alpha}\le C\tau^{-\sigma}\|d_l\|_{\ell^2_\beta},
\label{eqn:quad_smooth}
\end{equation}
for all $\tau>0$. For every node with $t_l>0$, Assumption~\ref{assumption:quadrature_consistency} with $s=\beta$ and the weighted hypothesis on $F_N$ give
\begin{equation}
\|d_l\|_{\ell^2_\beta}\le C_{\mathrm{quad}}h^r\|F_N(t_l)\|_{X^{\beta+r}}\le C_Rh^r\,t_l^{-\varrho}.
\label{eqn:quad_node}
\end{equation}
The rest needs to deal with $l=0$ and the sum, which takes several cases.

First, if $\varrho=0$ we use the bootstrap. Note that \eqref{eqn:quad_node} also holds at $l=0$. This is due to lower semicontinuity, in particular, $F_N(t)\to F_N(0)$ in $X^\beta$ as $t\rightarrow 0^+$, by \eqref{eqn:point-wise_lipschitz} and the continuity of $u$ in $X^\alpha$, so the coefficients of $F_N(t)$ converge to those of $F_N(0)$. Lower semicontinuity of the $X^{\beta+r}$ norm for the coefficient convergence follows from Fatou's lemma applied to the series $\|v\|_{X^{\beta+r}}^2=\sum_k\lambda_k^{\beta+r}|\langle v,e_k\rangle|^2$ (see also Lemma~\ref{lem:compact_embedding}) and yields $\|F_N(0)\|_{X^{\beta+r}}\le\liminf_{t\rightarrow 0^+}\|F_N(t)\|_{X^{\beta+r}}\le C_R$. Putting \eqref{eqn:quad_node} into \eqref{eqn:quad_smooth} for every $l$ gives
$$\|E_{\mathrm{quad},q}\|_{X^\alpha}
\le C_Rh^r\,\Delta t\sum_{l=0}^{q-1}(t_q-t_l)^{-\sigma}
=C_Rh^r(\Delta t)^\theta\sum_{j=1}^{q}j^{-\sigma}
\le \frac{C_R}{\theta}h^r\,t_q^\theta
\le C_Rh^r .$$

Next, for the case $\varrho=r/2$, which is the case without the bootstrap, $\gamma_{\mathrm s}=\gamma$. Since $r\le\gamma_{\mathrm s}-\alpha=\gamma-\alpha=2\zeta$ by \eqref{eqn:quad_order_cap}, we have $\varrho\le\zeta$ and $\sigma+\varrho\le\bar\sigma<1$. For the nodes $l\ge1$, \eqref{eqn:quad_smooth} and \eqref{eqn:quad_node} give
$$\Delta t\sum_{l=1}^{q-1}\left\|S(t_q-t_l)\mathcal S_N d_l \right\|_{X^\alpha}
\le C_R h^r\,\Delta t\sum_{l=1}^{q-1}(t_q-t_l)^{-\sigma}t_l^{-\varrho}
\le C_R h^r\,T^{1-\sigma-\varrho}, $$
using the weighted convolution bound for $E_{\mathrm{nl},q}$, applied with $\zeta$ replaced by $\varrho$. The node $l=0$ is estimated without using the quadrature consistency assumption, since the weight $t_l^{-\varrho}$ in \eqref{eqn:quad_node} is not finite. Using $S(t_q)\mathcal S_N=\mathcal S_NS_N(t_q)$ and $\mathcal S_N\mathcal C_N=P_N$, we split $d_0$ into two parts:
$$
\Delta t\left\|S(t_q)\mathcal S_Nd_0\right\|_{X^\alpha}
\le\Delta t\left\|\mathcal S_NS_N(t_q)\mathcal C^h_N\mathcal I_hF_N(0)\right\|_{X^\alpha}
+\Delta t\left\|S(t_q)P_NF_N(0)\right\|_{X^\alpha}
\le C\left(C^\beta_{\mathrm{stab}}+1\right)\Delta t\,t_q^{-\sigma}\,\|F_N(0)\|_{X^\beta}. $$
The first term is bounded by the discrete smoothing estimate stated at the start of this proof. The second is bounded by Assumption~\ref{assumption:semigroup_regular} and since $P_N$ is a contraction on $X^\beta$. Since $\|P_Nu_0\|_{X^\alpha} \le \|u_0\|_{X^\alpha} \le R$, we have $\|F_N(0)\|_{X^\beta}\le F_R$, and $t_q\ge\Delta t$ gives $\Delta t\,t_q^{-\sigma}\le(\Delta t)^{\theta}$. Thus it adds at most $C_R(\Delta t)^\theta$.

In both cases,
$$\max_{0\le q\le Q}\|E_{\mathrm{quad},q}\|_{X^\alpha}\le C_R\left(h^r+(\Delta t)^\theta\right), $$
where the term $(\Delta t)^\theta$ appears only when $\varrho=r/2$. In that case there is no bootstrap, so $\zeta_{\mathrm t}=\zeta$, and $(\Delta t)^\theta \le T^{\theta-\zeta}(\Delta t)^{\zeta}$ is absorbed into the temporal error term of the theorem.

Next for the force projection error, we decompose it using
$P_NF_N-F=P_N(F_N-F)+(P_N-I)F$, and consider the first part. This leads to three bounds that are needed.

First, using the Lipschitz condition from \eqref{eqn:point-wise_lipschitz} and the fact that $P_N$ is an orthogonal projection and contraction on $X^\alpha$ (so both inputs are $B^\alpha_R$), we have
\begin{equation}
\|F_N(t_{l})-F(t_{l})\|_{X^\beta}
\leq L_R\left\|P_Nu(t_{l})-u(t_{l})\right\|_{X^\alpha}.
\label{eqn:proj_lip}
\end{equation}
Second, that projection measured in $X^\alpha$ is the tail bound from Proposition~\ref{prop:galerkin_approximation} with regularity exponent $\gamma$ and thus we have
\begin{equation}
\left\|P_Nu(t_{l})-u(t_{l})\right\|_{X^\alpha}
\leq \lambda_{N+1}^{-\zeta}\left\|u(t_{l})\right\|_{X^\gamma}
\leq C_R\,\lambda_{N+1}^{-\zeta}\,t_{l}^{-\zeta},
\label{eqn:proj_gal}
\end{equation}
where in the last inequality we apply Lemma~\ref{lem:derived_regularity} Part 2 for $l>0$. At ${l}=0$, we use instead the boundedness $\|F_N(0)-F(0)\|_{X^\beta}\leq C_R$. For the third bound, the resulting $X^\beta$ bounds from the previous steps are in $X^\beta$ but we did not include the semigroup part of the projection error, which lifts back to $X^\alpha$. By Assumption~\ref{assumption:semigroup_regular} and since $P_N$ is a contraction on $X^\beta$, the semigroup smooths by:
\begin{equation}
\left\|S(\tau)P_Nv\right\|_{X^\alpha}\leq C\tau^{-\sigma}\|v\|_{X^\beta}.
\label{eqn:proj_smooth}
\end{equation}
Putting all three steps together, the $l>0$ terms gives us
\[
\Delta t\sum_{{l}=1}^{q-1}
\left\|S(t_q-t_{l})P_N\left(F_N(t_{l})-F(t_{l})\right)\right\|_{X^\alpha}
\leq
C_R\lambda_{N+1}^{-\zeta}\,
\Delta t\sum_{{l}=1}^{q-1}(t_q-t_{l})^{-\sigma}t_{l}^{-\zeta}
\leq
C_RT^{1-\bar\sigma}\lambda_{N+1}^{-\zeta},
\]
the last step is the weighted convolution estimate shown in $E_{\mathrm{nl},q}$, which converges since $\sigma+\zeta=\bar\sigma<1$. The term $l=0$ cannot use \eqref{eqn:proj_gal}, and is estimated through boundedness:
$$\Delta t\left\|S(t_q)P_N\left(F_N(0)-F(0)\right)\right\|_{X^\alpha}
\leq C_R\,\Delta t\,t_q^{-\sigma}
\leq C_R(\Delta t)^{1-\sigma}=C_R(\Delta t)^{\theta},$$
where $t_q\geq\Delta t$, with the worst case when $q=1$. Thus,
$$\max_{0\leq q\leq Q} \left\| \Delta t \sum_{{l}=0}^{q-1}
S(t_q-t_{l}) P_N\left(F_N(t_{l})-F(t_{l})\right) \right\|_{X^\alpha}
\leq C_R \left( \lambda_{N+1}^{-\zeta} + (\Delta t)^{\theta} \right)
\leq C_R \left( \lambda_{N+1}^{-\zeta} + (\Delta t)^{\zeta} \right),$$
where $T$ is absorbed in the constants and the last inequality used $(\Delta t)^{\theta}\le T^{\theta-\zeta}(\Delta t)^{\zeta}$.

The second part of the projection error measures the tail $(I-P_N)F$.
For $\tau>0$ and $v\in X^\beta$,
$$\|S(\tau)(I-P_N)v\|_{X^\alpha}^2
= \sum_{k>N} \lambda_k^{\alpha-\beta} e^{-2\lambda_k\tau} \lambda_k^\beta|v_k|^2.$$
For any fixed $c\in(0,1)$, $\lambda_k^\sigma e^{-\lambda_k\tau}
= \tau^{-\sigma} (\lambda_k\tau)^\sigma e^{-(1-c)\lambda_k\tau} e^{-c\lambda_k\tau}.$ Since the sup is bounded, $\sup_{x\geq0}x^\sigma e^{-(1-c)x}<\infty$, and $\lambda_k\geq\lambda_{N+1}$, we have 
$\lambda_k^\sigma e^{-\lambda_k\tau}
\leq C \tau^{-\sigma} e^{-c\lambda_{N+1}\tau}$, and thus
$$\|S(\tau)(I-P_N)v\|_{X^\alpha}
\leq C \tau^{-\sigma} e^{-c\lambda_{N+1}\tau} \|v\|_{X^\beta}.$$
Using the bound on $F$ in $X^\beta$ (uniform) we have
$$\left\| \Delta t \sum_{{l}=0}^{q-1} S(t_q-t_{l})(I-P_N)F(t_{l}) \right\|_{X^\alpha}
\leq C_R\Delta t \sum_{j=1}^{q} (j\Delta t)^{-\sigma} e^{-c\lambda_{N+1}j\Delta t}.$$
The function $g(s)=s^{-\sigma}e^{-c\lambda s}$ is positive and decreasing on $\bR_+$ so we use the integral bound
$$\Delta t \sum_{j=1}^{\infty} (j\Delta t)^{-\sigma} e^{-c\lambda j\Delta t}
\leq \int_0^\infty s^{-\sigma} e^{-c\lambda s}\,ds
= C\lambda^{\sigma-1},$$
which shows
$$\max_{0\leq q\leq Q} \left\| \Delta t \sum_{{l}=0}^{q-1}
S(t_q-t_{l})(I-P_N)F(t_{l}) \right\|_{X^\alpha}
\leq C_R\lambda_{N+1}^{-\theta}.$$
Since $\gamma-\beta<2$, we have
$\zeta=\frac{\gamma-\alpha}{2} < 1-\frac{\alpha-\beta}{2} = \theta$.
Since $\zeta<\theta$ and $\lambda_{N+1}\ge\lambda_1$, we have $\lambda_{N+1}^{-\theta}=\lambda_{N+1}^{-\zeta}\lambda_{N+1}^{-(\theta-\zeta)}\le\lambda_1^{-(\theta-\zeta)}\lambda_{N+1}^{-\zeta}$ and therefore,
$$ \max_{0\leq q\leq Q} \|E_{\mathrm{proj},q}\|_{X^\alpha}
\leq C_R \left[ \lambda_{N+1}^{-(\gamma-\alpha)/2}
+ (\Delta t)^{\zeta} \right],$$
where the constants absorb $\lambda_1$ factors only.

The last error term is the temporal quadrature error. Using Proposition~\ref{prop:temporal_quadrature} with $g=F$ gives
$$\max_{0\leq q\leq Q}
\|E_{\mathrm{time},q}\|_{X^\alpha}
\leq C\left(\|F\|_{\infty,\beta}(\Delta t)^{\theta}+[F]_\zeta T^{\theta-\zeta}(\Delta t)^{\zeta}\right)
\leq C_R(\Delta t)^{\zeta}, $$
with $\zeta<\theta$. Combining all the previous bounds proves the first total bound.

For the second result in the theorem, denote the piece-wise linear interpolant by $\mathcal J_Qu$. On each subinterval $[t_q,t_{q+1}]$, both $\Phi_{N,Q}^h[u]$ and $\mathcal J_Qu$ are affine functions in time. So taking the line segment 
$t=(1-s_0)t_q+s_0 t_{q+1}$ with $s_0\in[0,1]$, the norm is the max over endpoints:
$$\left\| \Phi_{N,Q}^h[u](t)-\mathcal J_Qu(t) \right\|_{X^\alpha}
\leq \max_{j=q,q+1} \left\| \Phi_{N,Q}^h[u](t_j)-u(t_j) \right\|_{X^\alpha}.$$
Taking the sup over each interval and then the maximum over
$q$ gives the second results and completes the proof.

If $\gamma_{\mathrm s}=\gamma'>\gamma$, Lemma~\ref{lem:bootstrap} gives, for every $t\in[0,T]$
$$\|u(t)\|_{X^{\gamma_{\mathrm s}}}\le R_{\gamma'},
\qquad \|F(t)\|_{X^{\gamma_{\mathrm s}-(\alpha-\beta)}}\le C_R.$$
Since these bounds hold at $t=0$, we avoid the previous issues at $l=0$. For brevity, the following is the list of estimates that change:
\begin{enumerate}
\item The Galerkin term bound is now $\|(I-P_N)u_0\|_{X^\alpha} \le \lambda_{N+1}^{-\zeta_{\mathrm s}} \|u_0\|_{X^{\gamma_{\mathrm s}}}$.
\item The nonlinear reconstruction bound changes as follows. Since $\alpha+r\le\gamma_{\mathrm s}$, the factor $\|u(t_l)\|_{X^{\alpha+r}}$ is at most $CR_{\gamma'}$ at every $l$, so 
$$\|E_{\mathrm{nl},q}\|_{X^\alpha}\le C_R h^r\,\Delta t \sum_{l=0}^{q-1}(t_q-t_l)^{-\sigma}\le C_RT^\theta h^r.$$
\item For the quadrature error, the weighted condition of the theorem holds with $\varrho=0$. Since $\alpha+r\le\gamma_{\mathrm s}$ and $P_N$ is a contraction on $X^{\alpha+r}$, so $\|P_Nu(t)\|_{X^{\alpha+r}}\le\|u(t)\|_{X^{\alpha+r}} \le CR_{\gamma'}$ for all $t\in[0,T]$. Assumption~\ref{assumption:bootstrap} with $s=\alpha+r$ then gives us
$$\sup_{0\le t\le T}\|f(P_Nu(t),\nabla P_Nu(t))\|_{X^{\beta+r}}\le K_R\left(1+CR_{\gamma'}\right).$$ The estimate for $\varrho=0$ above applies, and $\max_q\|E_{\mathrm{quad},q}\|_{X^\alpha} \le C_Rh^r$.
\item For the projection term: In \eqref{eqn:proj_gal} the weighted bound is replaced by $\|P_Nu(t_l)-u(t_l)\|_{X^\alpha}\le\lambda_{N+1}^{-\zeta_{\mathrm s}}R_{\gamma'}$ for every $l$, so the first sum is at most $C_R \lambda_{N+1}^{-\zeta_{\mathrm s}}\,\Delta t\sum_{l=0}^{q-1}(t_q-t_l)^{-\sigma }\le C_RT^\theta\lambda_{N+1}^{-\zeta_{\mathrm s}}$. For the tail bound, we apply Proposition~\ref{prop:galerkin_approximation} with $\beta + 2\zeta_{\mathrm s}=\gamma_{\mathrm s}-(\alpha-\beta)$ which gives us
$$\|S(\tau)(I-P_N)F\|_{X^\alpha}\le C\tau^{-\sigma}\|(I-P_N)F\|_{X^\beta}\le C\tau^{-\sigma}\lambda_{N+1}^{-\zeta_{\mathrm s}}\|F\|_{X^{\gamma_{\mathrm s}-(\alpha-\beta)}}, $$
and summing this part over $l$ gives $C_RT^\theta\lambda_{N+1}^{-\zeta_{\mathrm s}}$.
\item Lastly, we update the temporal error term. We apply the second part of Proposition~\ref{prop:temporal_quadrature} with $g=F$, $\beta':=\gamma_{\mathrm s}-(\alpha-\beta)$, and $\kappa:=\zeta_{\mathrm t}$, so that $\frac{\beta'-\beta}{2}=\zeta_{\mathrm s}$. Its properties hold:
\begin{itemize}
\item Lemma~\ref{lem:bootstrap} gives $\sup_t\|F(t)\|_{X^{\beta'}}\le G:=K_R(1+R_{\gamma'})$.
\item Lemma~\ref{lem:bootstrap_time}, whose exponent $\min\{1,\zeta_{\mathrm s}\}$ is $\zeta_{\mathrm t}$, gives $\|F(t)-F(s)\|_{X^\beta}\le H\,(t-s)^{\zeta_{\mathrm t}}$ with $H:=L_RC_{\mathrm t}$.
\item The condition $\kappa<\theta+\frac{\beta'-\beta}{2}$ is $\zeta_{\mathrm t}<\theta+\zeta_{\mathrm s}$, which holds because $\zeta_{\mathrm t}\le\zeta_{\mathrm s}$ and $\theta>0$.
\end{itemize}
Therefore, the error is $\max_q\|E_{\mathrm{time},q}\|_{X^\alpha}\le C\left(T^\theta H+T^{1-a}G\right)(\Delta t)^{\zeta_{\mathrm t}}\le C_R(\Delta t)^{\zeta_{\mathrm t}}$.
\end{enumerate}
In this case, $(\Delta t)^\theta$ does not appear, which gives the bound of the theorem with $\lambda_{N+1}^{-\zeta_{\mathrm s}}$ and $(\Delta t)^{\zeta_{\mathrm t}}$, depending on the extra regularity induced via bootstrapping. 
\end{proof}

\begin{proof}[Proof of Lemma~\ref{lem:interpolation}]
For $t\in[t_q,t_{q+1})$, set the difference to be $\tau:=t-t_q\in[0,\Delta t)$. Using the property that $S(\tau)\mathcal S_N=\mathcal S_NS_N(\tau)$, we split the error into
$$\left(\mathcal L^S_{N,Q}a\right)(t)-u(t)
=S(\tau)\left(\mathcal S_Na^q-u(t_q)\right) +\left(S(\tau)u(t_q)-u(t_q+\tau)\right). $$
The first difference is mapped by $S(\tau)$ (contractive in $X^\alpha$), and thus is a bound on the node-wise difference directly. For the second term, we can start the mild form at $t_q$ to get
$$u(t_q+\tau)=S(\tau)u(t_q)+\int_0^\tau S(\tau-s) f\left(u(t_q+s),\nabla u(t_q+s)\right)\,ds,$$
and thus applying Assumption~\ref{assumption:semigroup_regular} yields
$$\left\|S(\tau)u(t_q)-u(t_q+\tau)\right\|_{X^\alpha}
\le C\tau^{1-(\alpha-\beta)/2}\left\|f(u,\nabla u)\right\|_{X^\beta_T} 
\le CF_R(\Delta t)^{\theta}.$$
Taking the supremum over $t$ and $q$ completes the first part. 

 Using the bootstrap, the force is in $X^{\gamma_{\mathrm s}-(\alpha-\beta)}$ and has norm at most $K_R(1+R_{\gamma'})$. The smoothing from that space into $X^\alpha$ includes an additional $(t-\tau)^{-(\sigma-\zeta_{\mathrm s})_+}$ factor so the Duhamel term is at most $CK_R(1+R_{\gamma'})(\Delta t)^{1-(\sigma-\zeta_{\mathrm s})_+}$, where $1-(\sigma-\zeta_{\mathrm s})_+=\min\{1,\theta+\zeta_{\mathrm s}\}\ge\zeta_{\mathrm t}$.
 \end{proof}

\begin{proof}[Proof of Lemma~\ref{lem:sampling}]
We look at the error induced by sampling the function and its coefficients, splitting it into two pieces: $\mathcal S_Nb^q-u(t_q)=\mathcal S_N(\mathcal C^h_N\mathcal I_h-\mathcal C_N)u(t_q)+(P_N-I)u(t_q).$ Using the mild form, we have
\begin{equation}
u(t)=S(t)u_0+\mathcal D[g](t),
\qquad
g(t)=f\left(u(t),\nabla u(t)\right).
\label{eqn:samp_split}
\end{equation}
For the spectral error (first term), $(I-P_N)$ commutes with $S(t)$, so we have
$$\left\|(I-P_N)S(t)u_0\right\|_{X^\alpha}=\left\|S(t)(I-P_N)u_0\right\|_{X^\alpha}\le\left\|(I-P_N)u_0\right\|_{X^\alpha}
= \varepsilon^{\mathrm{tail}}_N(u_0),$$
where for the inequality we used the contraction $S(t)$ on $X^\alpha$. For the Duhamel term, by Assumption~\ref{assumption:semigroup_regular} and Lemma~\ref{lem:derived_regularity}(i), we get
$$\left\|\mathcal D[g](t)\right\|_{X^\gamma}\le C\int_0^t(t-s)^{-\bar\sigma}\,ds\;F_R=\frac{C}{1-\bar\sigma}\,t^{1-\bar\sigma}F_R\le C\,T^{\theta-\zeta}F_R,$$
which is finite since $\bar\sigma<1$, so Proposition~\ref{prop:galerkin_approximation} gives $\|(I-P_N)\mathcal D[g](t)\|_{X^\alpha}\le\lambda_{N+1}^{-\zeta}\|\mathcal D[g](t)\|_{X^\gamma}\le C_R\lambda_{N+1}^{-\zeta}$.

The next part is splitting the quadrature error (coefficient error) in the same way. One part gives $\|(\mathcal C^h_N\mathcal I_h-\mathcal C_N)S(t)u_0\|_{\ell^2_\alpha}\le\varepsilon^{\mathrm{enc}}_{N,h}(u_0)$ directly by Definition~\ref{def:data_moduli}. On the Duhamel term we use Assumption~\ref{assumption:quadrature_consistency} to get the bound $C_{\mathrm{quad}}h^r\|\mathcal D[g](t)\|_{X^{\alpha+r}}$, and 
$$\|\mathcal D[g](t)\|_{X^{\alpha+r}}\le C\|\mathcal D[g](t)\|_{X^\gamma}\le C_R$$
as before, where $r\le\gamma-\alpha$. If $\gamma_{\mathrm s}=\gamma'>\gamma$, Lemma~\ref{lem:bootstrap} improves the estimate to   $$\|\mathcal D[g](t)\|_{X^{\gamma_{\mathrm s}}} \le \|u(t)\|_{X^{\gamma_{\mathrm s}}} + \|S(t)u_0\|_{X^{\gamma_{\mathrm s}}}
\le R_{\gamma'}+\|u_0\|_{X^{\gamma'}},$$
instead. The same estimates yield $\lambda_{N+1}^{-\zeta_{\mathrm s}}$ and allows for $r\le\gamma_{\mathrm s}-\alpha$. Adding together the four parts and taking the maximum over $q$ gives the result. The final part is directly from the triangle inequality and inclusion of the discrete ball Proposition~\ref{prop:discrete_ball_containment}.
\end{proof}

\begin{proof}[Proof of Corollary~\ref{cor:structural}]
Given $p>p_f$, pick a valid representative DAG $G$ with $p_{\max}(G)<p$. Such graphs exist by the definition of the infimum. If we apply Theorem~\ref{thm:main} to $G$, we get $W\le C_{G}\eta^{-p_{\max}(G)\chi_{\mathrm s}}\log\eta^{-1}\le C_{G}\eta^{-p\chi_{\mathrm s}}\log \eta^{-1}$ for $\eta\le 1$ and $\alpha\ne2$. For $\alpha=2$, where $\chi=1$, the additional factor $(\log\eta^{-1})^{p_{\max}(G)}$ is at most $C\eta^{-(p-p_{\max}(G))}$ since $p>p_{\max}(G)$. In both cases $C_p:=C_{G}$'s dependence on $p$ is through the choice of $G$. So when the infimum is attained, we can pick the $G$ that attains it. The last part is the case $p_{\max}=0$ (Theorem~\ref{thm:main}), since then $\varepsilon^{-p_{\max}}=1$.
\end{proof}

\begin{proof}[Proof of Proposition~\ref{prop:continuation}]
First we look at one window. Let $a^n_\varepsilon(c)$ be the $n$-th neural iterate starting from $c$, so $\tilde u_{j+1}=\mathcal S_N[a^n_\varepsilon(b_j)]^Q$. The difference can be split into
$$\tilde u_{j+1}-u_{j+1}
=\mathcal S_N \left(a^n_\varepsilon(b_j)-x(b_j)\right)^Q
+\mathcal S_N \left(\Phi_{\mathrm{win}}(b_j)-\Phi_{\mathrm{win}}(\beta_j)\right)
+\left(\mathcal S_N\Phi_{\mathrm{win}}(\beta_j)-u_{j+1}\right).$$
For the first term, recall that $x(b_j)$ is the fixed point of $\Psi^h_{N,Q}$ started from $b_j$ (admissible by (ii)), so \eqref{eqn:neural_drift} and \eqref{eqn:total_error_eq4} give
$$\left\|\mathcal S_N\left[a^n_\varepsilon(b_j)-x(b_j)\right]^Q\right\|_{X^\alpha}
\le \left\|a^n_\varepsilon(b_j)-x(b_j)\right\|_{X^\alpha_{N,Q}}
\le \frac{E_{\mathrm{net}}}{1-\kappa}+2R\,\kappa^n.$$
For the third term, $\mathcal S_N\Phi_{\mathrm{win}}(\beta_j)-u_{j+1}$ is the nodal error at $q=Q$ of the fixed point $x(\beta_j)$ using the data $u_j$. Thus using \eqref{eqn:total_error_eq5}, \eqref{eqn:total_error_eq7}, and Lemma~\ref{lem:sampling} yields
$$\left\|\mathcal S_N\Phi_{\mathrm{win}}(\beta_j)-u_{j+1}\right\|_{X^\alpha}
\le\frac{E_{\mathrm{cons}}+E_{\mathrm{samp}}}{1-\kappa}+E_{\mathrm{samp}},$$
where $E_{\mathrm{cons}}$ and $E_{\mathrm{samp}}$ use the data $u_j$. Adding them together, they can be bounded by Theorem~\ref{thm:total_error} (without $E_{\mathrm{rec}}$), which is at most $\eta_{\mathrm{win}}$ by (i). Since $\|v\|_{X^s} \le \lambda_1^{-(\alpha-s)/2}\|v\|_{X^\alpha}$ for $s \le \alpha$, the terms are at most $\lambda_1^{-(\alpha-s)/2}\eta_{\mathrm{win}}$ in $X^s$. We use (iii) to bound the middle term (in both norms). Writing the difference as $e_j:=\|b_j-\beta_j\|_{\ell^2_s}$, we have
\begin{equation}
\|\tilde u_{j+1}-u_{j+1}\|_{X^\alpha}\le\eta_{\mathrm{win}}+C_{\mathrm{lift}}e_j,
\qquad
\|\tilde u_{j+1}-u_{j+1}\|_{X^s}\le\lambda_1^{-(\alpha-s)/2}\eta_{\mathrm{win}}+\Theta_{\mathrm{win}}e_j .
\label{eqn:continuation_step}
\end{equation}

Next we look over the windows. Since $\tilde u_{j+1}\in E_N$,
$$e_{j+1}\le\|P_N(\tilde u_{j+1}-u_{j+1})\|_{X^s}+\|(\mathcal C_N-\mathcal C^h_N\mathcal I_h)u_{j+1}\|_{\ell^2_s}
\le\|\tilde u_{j+1}-u_{j+1}\|_{X^s}+\lambda_1^{-(\alpha-s)/2}\varepsilon^{\mathrm{enc}}_{N,h}(u_{j+1}).$$
 Then using \eqref{eqn:continuation_step}, we get $e_{j+1}\le\Theta_{\mathrm{win}}e_j+\lambda_1^{-(\alpha-s)/2}\bar\eta$. Since $e_0=0$, we get $e_j \le \lambda_1^{-(\alpha-s)/2}\bar \eta \sum_{i<j}\Theta_{\mathrm{win}}^i$, and the first inequality of \eqref{eqn:continuation_step} produces the result.

Note that the steps above used (ii) and (iii) with $b_j$, which requires $\|b_j-\beta_j\|_{\ell^2_\alpha}\le\delta$. By the same splitting, we have  
$$\|b_j-\beta_j\|_{\ell^2_\alpha}\le\|\tilde u_j-u_j\|_{X^\alpha}+\varepsilon^{\mathrm{enc}}_{N,h}(u_j)\le\bar\eta+C_{\mathrm{lift}}e_{j-1}.$$
The condition that this is less than $\delta$ in the theorem statement allows for the induction to be valid.
\end{proof}

\noindent The condition (iii) holds when we set $s=\alpha$ and $\Theta_{\mathrm{win}}=C_{\mathrm{lift}}=1/(1-\kappa)$. Under an energy estimate, it holds with $s=0$, see Lemmas~\ref{lem:energy_stability} and \ref{lem:lifting}.

\subsection{Realization as a ReLU Network}
\label{sec:app_realization}

\begin{proof}[Proof of Theorem~\ref{thm:realization}]
 The proof combines the prior results and is a bit of bookkeeping. By \eqref{eqn:admissible_data} and Proposition~\ref{prop:discrete_ball_containment}, $A^0(w)\in B^{N,Q}_{R_0}$ is an admissible starting point for Theorem~\ref{thm:total_error}. Let $A^k(w)=a^k_\varepsilon$ denote the $k$-th neural Picard sweep. The accuracy bound (i) from this theorem is the result of Theorem~\ref{thm:main}.

We next count the weights of a single sweep of $\Psi^h_{N,Q,\varepsilon}$. This has three parts. For the synthesis, at each of the $Q$ time points the pair $\left(\mathcal S^h_Na^r,\nabla_h\mathcal S^h_Na^r\right)$ is computed by applying $d_x+1$ transformation to a coefficient vector $a^r$ and are from Assumption~\ref{assumption:fast_transform}. In the Fourier case they are repeated transforms, adjusted by multipliers $\mathrm i k_j$, and thus a single factorization is enough. For a general eigenbasis $\partial_je_k$ may not be directly proportional to $e_k$ and Assumption~\ref{assumption:fast_transform} applied to each of the $d_x+1$ operations. By Assumption~\ref{assumption:fast_transform}, this process costs at most $(1+d_x)d\left(\mathcal T(M)+N\right)$ weights and $C\log M$ layers per time stamp, and so it is at most $(1+d_x)dQ\left(\mathcal T(M)+N\right)$ weights for the synthesis.

For the nonlinear part, the neural network $\widehat f_\varepsilon$ is applied in parallel at the $MQ$ space-time points, which contributes $MQ\,W(\widehat f_\varepsilon)$ weights and depth $\dep(\widehat f_\varepsilon)$. Since the copies are identical, the number of distinct weights is just $W(\widehat f_\varepsilon)$.

For the analysis and the temporal parts, there are $Q$ applications of $\mathcal C^h_N$ which costs at most $dQ\mathcal T(M)$ weights and $C\log M$ layers. We need to also compute
$$v^q:=S_N(t_q)\,a_0^h+\Delta t\sum_{r=0}^{q-1}S_N(t_q-t_r)\,\mathcal C^h_Ng^r,$$
for $q=0,\ldots,Q$. These are exactly the iterates from \eqref{eqn:duhamel_recurrence}, with diagonal $S_N(\Delta t)$ and so it decouples into $dN$ independent scalar iterations $v^{q+1}_k=\rho_kv^q_k+c^q_k$ with $\rho_k:=e^{-\lambda_k\Delta t}$. The process is associative, i.e., the maps $x\mapsto\rho x+c$ can be composed by $(\rho_1,c_1)\circ(\rho_2,c_2)=(\rho_1\rho_2,\rho_1c_2+c_1)$, and since the multipliers $\rho$ are constants, each composition is an affine map of the data components $(c_1,c_2)$ with two nonzero weights. A parallel scan can produce $v^0,\ldots,v^Q$ from $c^0,\ldots,c^{Q-1}$ and $a^h_0$ using $CQ$ nodes in $C\log(2Q)$ layers. In a deep network, the quantities that are not updated at a given layer are carried forward using the identity process from Lemma~\ref{lem:nn_calculus}(iv), which is at most $Q+1$ values per layer with four weights each. Therefore, one scalar equation costs at most $CQ\log(2Q)$ weights. Summing over all the $dN$ recurrence equations, the total cost is (at most) $CdNQ\log(2Q)$ weights and $C\log(2Q)$ layers. 

Putting this together, one neural Picard sweep has depth at most $\dep(\widehat f_\varepsilon)+C\log(MQ)$ and size at most
$$C\left[(2+d_x)dQ\mathcal T(M)+dNQ\log(2Q)+MQ\,W(\widehat f_\varepsilon)\right],$$
where $(1+d_x)dQN$ from the derivative synthesis have been absorbed into $dNQ\log(2Q)$.  Each sweep also maps $a^h_0$ (dimension $dN$) through the layers using the identity (Lemma~\ref{lem:nn_calculus}(iv)), since it starts from $v^0=a^h_0$. This costs $CdN\left(\dep(\widehat f_\varepsilon)+\log(MQ)\right)\le CdN\left(W(\widehat f_\varepsilon)+\log(MQ)\right)$ weights, using $\dep(\widehat f_\varepsilon)\le W(\widehat f_\varepsilon)$, and the first part is at most $MQ\,W(\widehat f_\varepsilon)$ if $N\le MQ$. Composing $n$ sweeps and including the encoder step $w\mapsto\mathcal C^h_Nw$, gives (ii) and the first bound in (iii). If $N\le CM$, then $dN\left(W(\widehat f_\varepsilon)+\log(MQ)\right)\le C\left(MQ\,W(\widehat f_\varepsilon)+dMQ\log(MQ)\right)$, so the identity cost can be absorbed into \eqref{eqn:realization_size}.

For the second bound in (iii), we assumed that $N\le CM$ and $\mathcal T(M)\le CM\log M$. Thus (up to constants) $dNQ\log(2Q)\le CdMQ\log(MQ)$, $dN\log(MQ)\le dMQ\log(MQ)$, $(2+d_x)dQ\,\mathcal T(M)\le C(2+d_x)dMQ\log M$, and the encoder costs $\mathcal T(M)\le CM\log M$. Combining that together gives \eqref{eqn:realization_size}. Part (iv) follows from Lemma~\ref{lem:sharing} and Definition~\ref{def:spectral_layers}, where the block is one neural network applied at $nMQ$ points, and the spectral layers are the same across time-stamps and sweeps.
\end{proof}

\begin{proof}[Proof of Lemma~\ref{lem:linear_postcomposition}]
For any $g\in\widetilde X^\beta$ and $k\le N$, using \eqref{eqn:multiplier} gives
$$(\mathcal C_N\mathcal Mg)_k=\langle\mathcal Mg,e_k\rangle_H=\langle g,\mathcal M^{*}e_k\rangle=\mu_k\langle g,e_k\rangle=(D_\mu\mathcal C_Ng)_k,$$
which is result (i). The matrix is diagonal, and diagonal matrices commute with $S_N(t)=e^{-t\Lambda_N}$. Also, the norm is
$$\|D_\mu c\|_{\ell^2_s}^2=\textstyle\sum_k\lambda_k^s|\mu_k|^2|c_k|^2\le C_{\mathcal M}^2\|c\|^2_{\ell^2_s}.$$

For (ii), we look at where $\mathcal C^h_N$ is used. In Theorem~\ref{thm:fully_discrete_fixed_point}, it appears only through the discrete smoothing
$\|S_N(\tau)\mathcal C_N^hv\|_{\ell^2_\alpha}\le CC_{\mathrm{stab}}\tau^{-(\alpha-\beta)/2}\|v\|_{\widetilde X^\beta}$, and so putting $D_\mu$ into it and using (i) would just multiply the RHS by at most $C_{\mathcal M}$. The term $v=f_{\mathrm{loc}}(u,\nabla u)$ would now be estimated in $\widetilde X^\beta$, giving a $L^{\mathrm{loc}}_R$ and $F^{\mathrm{loc}}_R$ factor instead of $L_R$ and $F_R$. The same applies in Theorems~\ref{thm:neural_consistency} and \ref{thm:fully_discrete_consistency}, where $D_\mu$ commutes with $S_N(t_q-t_r)$ and can be absorbed into the constants. Theorems~\ref{thm:total_error} and \ref{thm:main} use only the conclusions of these results. 

Part (iii) is direct: $D_\mu=\operatorname{diag}(\mu_1,\ldots,\mu_N)$ is determined only by $\mathcal M$ and the eigenfunctions $e_1,\ldots,e_N$ of $L$ through \eqref{eqn:multiplier}. If $\mathcal M$ is the orthogonal projection of $\widetilde H$ onto $H$, then $\mathcal M^*e_k=e_k$ since $e_k\in H$, so we have $\mu_k=1$ and $D_\mu=I$.
\end{proof}

\subsection{Solution Operator Proofs}
\label{sec:app_operator}

\begin{proof}[Proof of Lemma~\ref{lem:uniform_wp}]
In the proof of Proposition~\ref{prop:contraction}, we can replace the radius by the value  $R=1+2R_0$. The invariance part holds since $\|S(\cdot)u_0\|_{X^\alpha_T}\le R_0<R/2$ for every $u_0$ with $\|u_0\|_{X^\alpha}\le R_0$. The value of $T_0$ is set by $CT_0^\theta L_R\le\frac{1}{2}$, which only sees $u_0$ through $R$, which is uniform.
\end{proof}

\begin{proof}[Proof of Lemma~\ref{lem:compact_embedding}]
By Proposition~\ref{prop:galerkin_approximation}, $\|I-P_N\|_{\mathcal L(X^\gamma,X^\alpha)}\le\lambda_{N+1}^{-(\gamma-\alpha)/2}\to0$ as $N\to\infty$. This map is the limit of finite-rank operators $P_N$ (measured in the operator norm) and thus is compact, so the set $\mathcal U_{R_0}$ is totally bounded in $X^\alpha$. It is also closed in $X^\alpha$, since if we take $u_m\to u$ in $X^\alpha$ with $\|u_m\|_{X^\gamma}\le R_0$, then $\langle u_m,e_k\rangle\to\langle u,e_k\rangle$ for every $k$. By Fatou's lemma, the series is bounded by $$\|u\|_{X^\gamma}^2=\sup_K\sum_{k\le K}\lambda_k^\gamma|\langle u,e_k\rangle|^2=\sup_K\lim_{m\to\infty}\sum_{k\le K}\lambda_k^\gamma|\langle u_m,e_k\rangle|^2\le\liminf_{m\to\infty}\|u_m\|_{X^\gamma}^2\le R_0^2.$$
Thus by completeness, we have shown compactness.
\end{proof}

\begin{proof}[Proof of Lemma~\ref{lem:uniform_encoding}]
Set $E_h:=\mathcal C^h_N\mathcal I_h-\mathcal C_N$, so by Lemma~\ref{lem:encoder_bounded} and the bound $\|\mathcal C_Nv\|_{\ell^2_\alpha} \le \|v\|_{X^\alpha}$, we have the uniform in $h$ estimate $\|E_hv\|_{\ell^2_\alpha}\le C_E\|v\|_{X^\alpha}$ (with $C_E:=C^\alpha_{\mathrm{stab}}+1$). The set $\bigcup_{K\ge1}E_K$ (finite eigen-expansion) is dense in $X^\alpha$. For $w\in\bigcup_{K\ge1}E_K$, Assumption~\ref{assumption:quadrature_consistency} gives
$$\|E_hw\|_{\ell^2_\alpha}\le C_{\mathrm{quad}}h^r\|w\|_{X^{\alpha+r}} \to 0,$$
as $h \to 0$. Given $\epsilon>0$, choose a finite $\frac{\epsilon}{3C_E}$-net $\{v_1,\ldots,v_J\}$ of $\mathcal V$ (exists by compactness) and a finite set of elements $w_j\in\bigcup_{K\ge1}E_K$ where $\|v_j-w_j\|_{X^\alpha}\le\frac{\epsilon}{3C_E}$. For $v\in\mathcal V$ and the nearest $v_j$, we have
$$\|E_hv\|_{\ell^2_\alpha}\le C_E\|v-w_j\|_{X^\alpha}+\|E_hw_j\|_{\ell^2_\alpha}\le\frac{2\epsilon}{3}+\max_{1\le i\le J}\|E_hw_i\|_{\ell^2_\alpha}.$$
The max is over finitely many terms and is at most $\frac{\epsilon}{3}$ for $h$ small enough. Therefore, we have $\sup_{v\in\mathcal V}\|E_hv\|_{\ell^2_\alpha}\le\epsilon$ for $h$ sufficiently small.
\end{proof}

\begin{proof}[Proof of Theorem~\ref{thm:operator_rates}]
We just need to show that every constant used in Theorems~\ref{thm:total_error}, \ref{thm:main}, and \ref{thm:realization} depends only on the data via $R_0$ and $\rho_\gamma$. Thus, the parameter can be controlled uniformly and the neural network can be applied to every $u_0\in\mathcal U$.

Since $\|u_0\|_{X^\alpha}\le R_0$, Lemma~\ref{lem:uniform_wp} gives  $R=1+2R_0$, $T_0$, and a contraction constant $\kappa_0$ which is the same to all $u_0\in\mathcal U$, and that $G_T[u_0]\in B^\alpha_R$. Because of this, we have several implications. The compact set $K$ from Assumption~\ref{assumption:structured_f} depends only on $R$, so one neural network $\widehat f_\varepsilon$ is valid for every input data, and using Corollary~\ref{cor:inputs_in_K} it is evaluated only on set $K$. The constants $L_R$ and $F_R$, and thus $\kappa=CC_{\mathrm{stab}}L_RT^\theta$ in Theorem~\ref{thm:fully_discrete_fixed_point}, are the same for the data, and by Assumption~\ref{assumption:uniform_data} with $\rho_\gamma$, are the values $R_\gamma$ and $C_R$ from Lemmas~\ref{lem:interpolation} and \ref{lem:sampling} and from Theorem~\ref{thm:fully_discrete_consistency}. The initialization terms are uniformly controlled. First, Assumption~\ref{assumption:quadrature_consistency} gives $$E^h_{\mathrm{init}} \le C_{\mathrm{quad}}h^r\|u_0\|_{X^{\alpha+r}}\le C\rho_\gamma\,h^r,$$ since $\alpha+r\le\gamma$ by \eqref{eqn:quad_order_cap}. Second by Proposition~\ref{prop:galerkin_approximation}, we have $\|(I-P_N)u_0\|_{X^\alpha}\le\lambda_{N+1}^{-(\gamma-\alpha)/2}\rho_\gamma$. 

Lastly, the admissibility condition \eqref{eqn:admissible_data} holds uniformly over the input set by Lemma~\ref{lem:encoder_bounded}, namely, $\|\mathcal C^h_N\mathcal I_hu_0\|_{\ell^2_\alpha}\le C^\alpha_{\mathrm{stab}}R_0$ for every $u_0\in\mathcal U$ and admissible $(N,h)$. We need to increase the radius by replacing $R_0$ by $C^\alpha_{\mathrm{stab}}R_0$ and $R$ by $1+2C^\alpha_{\mathrm{stab}}R_0$ in Lemma~\ref{lem:uniform_wp}. But fortunately, this can be done once and can be fixed before any parameter is set. This will change constants but will not affect the exponent. Therefore the bound of Theorem~\ref{thm:total_error} holds with one constant $C_R$ for all $u_0\in\mathcal U$, and the balancing of terms in Theorem~\ref{thm:main} picks $(N,h,Q,n,\varepsilon)$ once and uniformly. Theorem~\ref{thm:realization} then applies directly and completes the proof.
\end{proof}

\begin{proof}[Proof of Proposition~\ref{prop:divergence_form}]
Result (i) is directly implied by $\|\mathcal Mg\|_{X^\beta}^2=\sum_k\lambda_k^{\alpha-1} |\mu_k|^2 |\langle g,\tilde e_k\rangle|^2\le C^2_{\mathcal M}\|g\|_{\widetilde X^{\alpha}}^2$. For part (ii), note that $D_\mu$ and $S_N(\tau)$ are both diagonal and so they commute, which yields
$$\left\|S_N(\tau)D_\mu \widetilde{\mathcal C}^h_Nv\right\|_{\ell^2_\alpha}
= \left\|D_\mu S_N(\tau) \widetilde{\mathcal C}^h_N v \right\|_{\ell^2_\alpha}
\le C_{\mathcal M}\left\| S_N(\tau)\widetilde{\mathcal C}^h_N v\right\|_{\ell^2_{\alpha+1}},$$
where we used that $\|D_\mu c\|^2_{\ell^2_\alpha}=\sum_k\lambda_k^\alpha|\mu_k|^2|c_k|^2\le C^2_{\mathcal M}\sum_k\lambda_k^{\alpha+1}|c_k|^2$ and the condition that $|\mu_k| \le C_{\mathcal M}\lambda_k^{1/2}$.  Next, by duality in the proof of Theorem~\ref{thm:fully_discrete_fixed_point} applied to the ambient space and over $\widetilde E_N:=\operatorname{span}\{\tilde e_1,\ldots,\tilde e_N\}$, we have
$$ \left\|S_N(\tau)\widetilde{\mathcal C}^h_Nv\right\|_{\ell^2_{\alpha+1}}
=\sup_{\substack{\phi\in\widetilde E_N\\ \|\phi\|_{\widetilde X^{-\alpha-1}}\le1}}\left|\sum_i\omega^{N,h}_iv(x_i)(S(\tau)\phi)(x_i)\right|
\le C^\alpha_{\mathrm{stab}}\|v\|_{\widetilde X^{\alpha}}\sup_\phi\|S(\tau)\phi\|_{\widetilde X^{-\alpha}},$$
where $S(\tau)$ maps into itself since $\tilde e_k$ are eigenfunctions and $\|S(\tau)\phi\|_{X^{-\alpha}}\le C\tau^{-1/2}\|\phi\|_{X^{-\alpha-1}}$ by the negative-order spaces from Definition~\ref{def:negative_spaces}. Since the exponent $\tfrac{1}{2}$ agrees with $\sigma=\frac{\alpha-\beta}{2}$ for $\beta=\alpha-1$, the other estimate hold.

For Part (iv), the network error $g_\varepsilon$ is only controlled in $\ell^\infty$, so $\|\widetilde{\mathcal C}^h_Ng_\varepsilon\|_{\ell^2_0}\le C\varepsilon$ from the proof of Theorem~\ref{thm:neural_consistency}, and
$$\left\|S_N(\tau)D_\mu\widetilde{\mathcal C}^h_N g_\varepsilon\right\|_{\ell^2_\alpha}
\le C_{\mathcal M}\left\|S_N(\tau) \widetilde{\mathcal C}^h_N g_\varepsilon \right\|_{\ell^2_{\alpha+1}}
\le C \varepsilon \min \left\{\lambda_N^{(\alpha+1)/2},\tau^{-(\alpha+1)/2}\right\}.$$
The rest of that proof applies, where we put $\alpha+1$ in place of $\alpha$, and thus have the factor $\xi_{\alpha+1}$.

Lastly, part (iii) holds since the network now only sees the values $\mathcal S^h_N a^r$, which are well-defined point-wise using continuity alone.
\end{proof}

\subsection{Window Stability for General Governing Laws}
\label{sec:window_stability_general}

This subsection gives the condition (iii) of Proposition~\ref{prop:continuation} for general laws, where Lemma~\ref{lem:allen_cahn_window} is the Allen-Cahn case.

\begin{lemma}[Discrete Energy Stability]
\label{lem:energy_stability}
Let $B\subset\bR^N$, and suppose there are $\xi\in\bR$ and $\Gamma_0\ge0$ such that, for $a',b'\in B$,
$$\langle G_N(a')-G_N(b'),a'-b'\rangle_0 \le \xi\,\|a'-b'\|^2_{\ell^2_0},
\qquad
\|G_N(a')-G_N(b')\|_{\ell^2_0}\le\Gamma_0\|a'-b'\|_{\ell^2_0}.$$
If $\Delta t\,\Gamma_0^2\le1$ and $x^q(a),x^q(b)\in B$ for $0\le q<Q$, then
$$\|z^q\|_{\ell^2_0}\le e^{\psi t_q}\|z^0\|_{\ell^2_0}$$
for $0 \le q\le Q$ where $\psi:=\xi+\tfrac{1}{2}-\lambda_1$.
\end{lemma}

\begin{proof}[Proof of Lemma~\ref{lem:energy_stability}]
By \eqref{eqn:window_stepper} and \eqref{eqn:smoothing}, $\|z^{q+1}\|_{\ell^2_0}\le e^{-\lambda_1\Delta t}\|z^q+\Delta t\,d^q\|_{\ell^2_0}$. Using the assumption and expanding the expression gives
$$\|z^q+\Delta t\,d^q\|^2_{\ell^2_0}=\|z^q\|^2_{\ell^2_0}+2\Delta t\langle d^q,z^q\rangle_0+\Delta t^2\|d^q\|^2_{\ell^2_0}\le\left(1+(2\xi+1)\Delta t\right)\|z^q\|^2_{\ell^2_0}\le e^{(2\xi+1)\Delta t}\|z^q\|^2_{\ell^2_0},$$
where we used $\Delta t^2\Gamma_0^2\le\Delta t$ and $1+x\le e^x$. Thus $\|z^{q+1}\|_{\ell^2_0}\le e^{\psi\Delta t}\|z^q\|_{\ell^2_0}$, and the claim follows using induction.
\end{proof}

The one-sided constant $\xi$ can be much smaller than a Lipschitz constant. If $f=f(u)$ satisfies $\langle f(y)-f(y'),y-y'\rangle\le\xi_f|y-y'|^2$ and has Lipschitz constant $\operatorname{Lip}(f)$ on a convex neighborhood of $K$, then \eqref{eqn:analysis_pairing}, \eqref{eqn:sampling_pairing}, and Assumption~\ref{assumption:quadrature_stability}(i) gives the conditions $\xi\le C^0_{\mathrm{stab}}\max\{\xi_f,0\}$ and $\Gamma_0\le C^0_{\mathrm{stab}}\operatorname{Lip}(f)$, independently of $N$ and $h$. For laws involving $\nabla u$, $\Gamma_0$ grows like $\lambda_N^{1/2}$, and an $N$-independent $\xi$ would require some control of the quadrature aliasing. While this is possible, we do not pursue it here.

\begin{lemma}[$X^\alpha$ Over Windows]
\label{lem:lifting}
Let $T\le1$ and $B$ be from Lemma~\ref{lem:energy_stability}. If $x^q(a),x^q(b) \in B$ for $0\le q<Q$ and the following hold
\begin{enumerate}
\item[(i)] $\|d^0\|_{\ell^2_0} \le \Gamma_0\|z^0\|_{\ell^2_0}$ with $\Delta t\,\Gamma_0\le1$;
\item[(ii)] $\|G_N(a')-G_N(b')\|_{\ell^2_\beta} \le \Gamma_\alpha\|a'-b'\|_{\ell^2_\alpha}$ for $a',b'\in B$, and $2^{\alpha/2}c_\alpha\Gamma_\alpha T^\theta/\theta\le\frac{1}{2}$;
\item[(iii)] $\|d^p\|_{\ell^2_0} \le \Gamma_*\, t_p^{-1/2}\|z^0\|_{\ell^2_0}$ for $1\le p<Q$,
\end{enumerate}
then
$$\|\Phi_{\mathrm{win}}(a)-\Phi_{\mathrm{win}}(b)\|_{\ell^2_\alpha}\le C_{\mathrm{lift}}\,\|a-b\|_{\ell^2_0},$$
where
$C_{\mathrm{lift}} := 4c_\alpha T^{-\alpha/2} \left( 1 + 2^{(\alpha-1)/2} \Gamma_*T^{1/2}\right)$.
\end{lemma}

Condition (i) is weaker than the condition of Lemma~\ref{lem:energy_stability} and it holds under the choices of Theorem~\ref{thm:main} when $u_0\in X^{\gamma_{\mathrm s}}$. This holds even for laws involving the gradient, there $\Gamma_0\le C(1+\lambda_N^{1/2})$, while $\Delta t\,\lambda_N\le C(c_{\mathrm{err}}\eta)^{1/\zeta_{\mathrm t}-1/\zeta_{\mathrm s}}\le C$ since $\zeta_{\mathrm t}\le\zeta_{\mathrm s}$, thus $\Delta t\,\Gamma_0\le C\Delta t^{1/2}\le1$ for sufficiently small $\eta$. Condition (iii) of Lemma~\ref{lem:lifting} follows from either of two bounds on $G_N$. Using the energy estimate (Lemma~\ref{lem:energy_stability}) it holds with $\Gamma_*=\Gamma_0e^{\psi_+T}T^{1/2}$ and $\psi_+:=\max\{\psi,0\}$, since $\|d^p\|_{\ell^2_0}\le\Gamma_0\|z^p\|_{\ell^2_0}\le\Gamma_0e^{\psi_+T}\|z^0\|_{\ell^2_0}$ and $t_p\le T$. Using a Lipschitz bound $\|G_N(a')-G_N(b')\|_{\ell^2_0} \le \Gamma_1\|a'-b'\|_{\ell^2_1}$ on $B$, it holds with $\Gamma_*=4c_\alpha \Gamma_1e^{64c_\alpha^2\Gamma_1^2T}$, as shown after the proof.

\begin{proof}[Proof of Lemma~\ref{lem:lifting}]
We can write out the time-stepper \eqref{eqn:window_stepper} which gives $$z^q=S_N(t_q)z^0+ \Delta t\sum_{p<q}S_N(t_q-t_p)d^p.$$ 
Let $W:=\max_{1\le q\le Q } \, t_q^{\alpha/2} \|z^q\|_{\ell^2_\alpha}$ and fix $1\le q\le Q$. We split the sum into the term $p=0$, the early terms with $1\le p$ and $t_p\le t_q/2$, and the remaining terms $t_q/2<t_p<t_q$.

First $p=0$ and the $z^0$ part. By \eqref{eqn:smoothing} and (i),  they give
$$c_\alpha t_q^{-\alpha/2}(1+\Delta t\Gamma_0) \|z^0\|_{\ell^2_0} \le 2 c_\alpha t_q^{-\alpha/2} \|z^0\|_{\ell^2_0},$$
since $\Delta t\,\Gamma_0\le 1$ by (i).
For the early terms, we have $t_q-t_p\ge t_q/2$, so using \eqref{eqn:smoothing}, (iii), and $\Delta t\sum_{t_p\le t_q/2}t_p^{-1/2}\le2(t_q/2)^{1/2}$, they are bounded by $2c_\alpha\Gamma_*(t_q/2)^{(1-\alpha)/2}\|z^0\|_{\ell^2_0}$. The remaining terms satisfy $\|z^p\|_{\ell^2_\alpha} \le t_p^{-\alpha/2}W\le(t_q/2)^{-\alpha/2}W$. Thus using (ii) and \eqref{eqn:smoothing} with $r=\alpha-\beta$, they give at most
$$c_\alpha\Gamma_\alpha\left(\frac{t_q}2\right)^{-\alpha/2}W\,\Delta t \sum_{t_q/2<t_p<t_q} (t_q-t_p)^{-\sigma}     \le c_\alpha \Gamma_\alpha \left(\frac{t_q}2\right)^{-\alpha/2} \frac{T^\theta} \theta W.$$

By definition of $W$, we take the maximum after multiplying by $t_q^{\alpha/2}$, using $t_q^{1/2}\le T^{1/2}$ and condition (ii) to get
$$ W \le 2 c_\alpha \|z^0\|_{\ell^2_0} + 2^{(\alpha+1)/2} c_\alpha \Gamma_*T^{1/2} \|z^0\|_{\ell^2_0}+\tfrac{1}{2} W,$$
so $W\le C_{\mathrm{lift}}T^{\alpha/2}\|z^0\|_{\ell^2_0}$, and $\|\Phi_{\mathrm{win}}(a)-\Phi_{\mathrm{win}}(b)\|_{\ell^2_\alpha}=\|z^Q\|_{\ell^2_\alpha}\le T^{-\alpha/2}W$.
\end{proof}

\noindent\emph{Proof of the second case of (iii).} We can also show this just with Lipschitz. We need $\|z^p\|_{\ell^2_1}\lesssim t_p^{-1/2}\|z^0\|_{\ell^2_0}$, since then $\|d^p\|_{\ell^2_0}\le\Gamma_1\|z^p\|_{\ell^2_1}$ would yield (iii). 

By \eqref{eqn:smoothing} with $r=1$, valid since $\alpha\ge1$, the $z^0$ and the $p=0$ term are at most $2c_\alpha t_q^{-1/2}\|z^0\|_{\ell^2_0}$, exactly as above. The remaining terms give
\begin{equation}
\|z^q\|_{\ell^2_1}\le2c_\alpha t_q^{-1/2}\|z^0\|_{\ell^2_0}+c_\alpha\Gamma_1\,\Delta t\sum_{p=1}^{q-1}(t_q-t_p)^{-1/2}\|z^p\|_{\ell^2_1},
\label{eqn:volterra}
\end{equation}
for $1\le q\le Q$.
We bound $V:=\max_{1\le q\le Q}t_q^{1/2}e^{-c_\Gamma^2t_q}\|z^q\|_{\ell^2_1}$ with $c_\Gamma:=8c_\alpha\Gamma_1$ to close the system. The factor $t_q^{1/2}$ removes the singularity at $t=0$, and $e^{-c_\Gamma^2t_q}$ makes the sum in \eqref{eqn:volterra} small without shrinking the window. Multiplying \eqref{eqn:volterra} by $t_q^{1/2} e^{-c_\Gamma^2t_q}$ and using $\|z^p\|_{\ell^2_1}\le t_p^{-1/2} e^{c_\Gamma^2t_p}V$, we have the inequality
$$t_q^{1/2}e^{-c_\Gamma^2t_q} \|z^q\|_{\ell^2_1} \le 2 c_\alpha\|z^0\|_{\ell^2_0} + c_\alpha \Gamma_1 V\,\Delta t \sum_{p=1}^{q-1} \frac{t_q^{1/2}}{t_p^{1/2}(t_q-t_p)^{1/2}}\,e^{-c_\Gamma^2(t_q-t_p)}.$$

We consider two cases. For $t_p \le t_q/2$, we have $t_q-t_p \ge t_q/2$ and $\Delta t \sum_{t_p\le t_q/2} t_p^{-1/2} \le 2(t_q/2)^{1/2}$, so in this case we have  $2t_q^{1/2}e^{-c_\Gamma^2t_q/2} \le 2e^{-1/2} c_\Gamma^{-1}$. For the other case where $t_p>t_q/2$, we have $t_q^{1/2} t_p^{-1/2} \le \sqrt2$, so this part gives $\sqrt2\int_0^\infty s^{-1/2}e^{-c_\Gamma^2s}\,ds = \sqrt{2\pi}\,c_\Gamma^{-1}$. Since $2e^{-1/2}+\sqrt{2\pi} < 4$, the total sum is at most $4 c_\Gamma^{-1}$, and taking the max over $q$ yields
$$V \le 2 c_\alpha \|z^0\|_{\ell^2_0} + \frac{4c_\alpha\Gamma_1}{c_\Gamma}V = 2 c_\alpha \|z^0\|_{\ell^2_0}+\frac{1}{2} V.$$
Therefore $V\le4c_\alpha\|z^0\|_{\ell^2_0}$, i.e., $\|z^p\|_{\ell^2_1}\le 4 c_\alpha e^{c_\Gamma^2T} t_p^{-1/2}\|z^0\|_{\ell^2_0}$, which shows (iii) with $\Gamma_*=4c_\alpha\Gamma_1 e^{64c_\alpha^2\Gamma_1^2T}$.

The second case applies to every $f(u,\nabla u)$ considered here except those in divergence form, since it only needs the Lipschitz bound. In divergence form, $D_\mu$ is unbounded and so $\Gamma_1$ scales like $\lambda_N^{1/2}$. Thus condition (iii) of Proposition~\ref{prop:continuation} is taken from the contraction with $s=\alpha$. With $v=\mathcal S_Na'$ and $w=\mathcal S_Nb'$, \eqref{eqn:analysis_pairing}, \eqref{eqn:sampling_pairing} and Assumption~\ref{assumption:quadrature_stability}(i), applied to $v-w$ and $\nabla(v-w)$, give $\Gamma_1 \le C^0_{\mathrm{stab}}\operatorname{Lip}(f)C_1$, where $\|y\|_H+\|\nabla y\|_H \le C_1\|y\|_{X^1}$. Condition (ii) also always holds, with $\Gamma_\alpha = C^\beta_{\mathrm{stab}}L_R$, by the quadrature stability Assumption~\ref{assumption:quadrature_stability}(ii) with $s=\beta$, but possibly with a shorter window.

\section{Quantitative ReLU Approximation Results}
\label{sec:appendix}

We recall two approximation results used for the DAG approximation in Section~\ref{sec:dag}, that is, the classical rate for smooth univariate functions, and the approximation of multilinear maps. 

\begin{theorem}[Quantitative Approximation of Smooth Scalar Functions]
\label{thrm:quantitative_smooth_scalar}
If $s \ge 1$ is an integer, $[a,b]\subset\bR$ is a compact interval, and $g\in C^{s}([a,b])$, then there exist constants $C>0$ and $\delta_0\in(0,1)$, depending only on $g$, $s$, and $[a,b]$, such that for every $\delta\in(0,\delta_0]$ there is a ReLU network $N_\delta:\bR\to\bR$ with
$$\|g-N_\delta\|_{L^\infty([a,b])}\le\delta,
\qquad
\dep(N_\delta)\le C\log(\delta^{-1}),
\qquad
W(N_\delta)\le C\,\delta^{-1/s}\log(\delta^{-1}).$$
\end{theorem}

This is a one-dimensional version of \cite[Theorem~1]{yarotsky2017error}, where one uses $g/\max\{1,\|g\|_{W^{s,\infty}}\}$ and notes that $C^s([a,b])\subset W^{s,\infty}((a,b))$. The restriction to integer $s$ is from the referenced theorem and is used in Assumption~\ref{assumption:structured_f}.

\begin{remark}[Continuous Weights via Local Approximation]
\label{rem:taylor_layer}
Let $M_\delta:=\lceil\delta^{-1/s}\rceil$, and let $y_i$ be nodes with spacing $h_\delta:=(b-a)/M_\delta$ covering $[a,b]$. The hat functions $\phi_i(y):=h_\delta^{-1}\left(\sigma(y-y_{i-1})-2\sigma(y-y_i)+\sigma(y-y_{i+1})\right)$ are realized exactly by ReLU networks and form a partition of unity on $[a,b]$. Since $\phi_i$ is supported where $|y-y_i|\le h_\delta$, the combination of local Taylor polynomials
$$\sum_i\phi_i(y)\sum_{l<s}\frac{g^{(l)}(y_i)}{l!}(y-y_i)^l$$
is within $\frac{\|g^{(s)}\|_\infty}{s!}h_\delta^{s}\le C\delta$ of $g$.

Each of the $O(\delta^{-1/s})$ local terms $\phi_i(y)(y-y_i)^l$ is a product of at most $s$ factors bounded on $[a,b]$, so the multiplication networks of \cite{yarotsky2017error}, as in Theorem~\ref{thrm:quantitative_multilinear}, approximate it with size and depth $O(\log\delta^{-1})$. At each $y$ only two hat functions are active, so accuracy $\delta/(2s\max\{1,\max_{l<s}\|g^{(l)}\|_\infty\})$ per term is sufficient. The resulting network has the size bound of Theorem~\ref{thrm:quantitative_smooth_scalar}. The hidden layers do not depend on $g$, and $g$ appears only through the output weights $g^{(l)}(y_i)/l!$, which depend continuously on $g$ in $C^s$. The continuity is needed for comparisons in the lower bound results of Remark~\ref{rem:block_lower_bound}, Remark~\ref{rem:operator_lower_bound}, and Theorem~\ref{thm:hj_lower_bound}.  The weight continuity of this construction was noted in \cite{yarotsky2017error}.
\end{remark}

\begin{theorem}[Quantitative Approximation of Multilinear Maps]
\label{thrm:quantitative_multilinear}
Let $r\ge2$, $n,p\ge1$, and let $T:(\bR^n)^{r}\to\bR^{p}$
be multilinear. Let $b>0$ and set $\mathcal Q:=([-b,b]^n)^{r}$. Then there exist constants $C>0$ and $\delta_0\in(0,1)$, depending only on $T$, $r$, $n$, $p$, and $b$, such that for every $\delta\in(0,\delta_0]$ there is a ReLU network $\widetilde T_\delta:(\bR^n)^r\to\bR^p$ with
$$\sup_{y\in\mathcal Q} \left \| T(y) - \widetilde T_\delta(y )\right\|_{\bR^p} \le \delta,
\qquad
\dep(\widetilde T_\delta) \le  C \log(\delta^{-1}),
\qquad
W(\widetilde T_\delta) \le C \log(\delta^{-1}).$$
\end{theorem}

\begin{proof}[Proof of Theorem~\ref{thrm:quantitative_multilinear}]
Denote $y=(y^{(1)},\ldots,y^{(r)})$ with $y^{(l)}\in\bR^n$ and thus the multilinear function can be written as
$$T(y) = \sum_{k_1,\ldots,k_r} c_{k_1\ldots k_r}\,y^{(1)}_{k_1}\cdots y^{(r)}_{k_r}, 
\qquad c_{k_1\ldots k_r} := T(e_{k_1},\ldots,e_{k_r}) \in \bR^p,$$
i.e., a fixed linear combination of $n^r$ products of $r$ coordinates, each bounded by $b$. The coefficients depend only on $T$, $r$ and $n$. It is therefore enough to approximate $P(x_1,\ldots,x_r):=x_1\cdots x_r$ on $[-B,B]^r$, $B:=\max\{b,1\}$, to collect all $n^r$ copies in parallel and apply an exact affine output layer to multiply the terms by the proper coefficients.

From \cite{yarotsky2017error}, for every $\mu\in(0,1)$ and $B'\ge1$ there is a ReLU network $\widetilde\times_\mu$ with $$\sup_{|x|,|y|\le B'}|xy-\widetilde\times_\mu(x,y)| \le \mu$$ and $\dep(\widetilde\times_\mu),W(\widetilde\times_\mu)\le C(B')\log(\mu^{-1})$. We next construct each of the $r$-fold products using a standard binary tree with $\lceil\log_2r\rceil$ levels and by replacing each multiplication with the approximation $\widetilde\times_\mu$. Note that if $r$ is not a power of two, then we will have one extra node that can be transferred to the next level in the tree using the identity $x=\sigma(x)-\sigma(-x)$ (which adds no error). Let $P_l$ and $\widetilde P_l$ be the exact and computed values of a node at level $l$ (which are products of at most $2^l$ of the inputs). Set the error at that level to be $e_l:=\max|P_l-\widetilde P_l|$ and let $B_l:=B^{2^l}+1$. 

 The polynomial is bounded by $|P_l|\le B^{2^l}$, since every term is at most $B\ge1$ in magnitude. If two nodes at level $l$ have errors at most $e_l$ and satisfy $|\widetilde P_l|,|\widetilde P_l'|\le B_l$, then

$$|P_l P_l'-\widetilde P_l \widetilde P_l'| \le |P_l|\,|P_l'-\widetilde P_l'| + |\widetilde P_l'| \,|P_l - \widetilde P_l| \le 2 B_l\,e_l.$$
The multiplication network is accurate up to error $\mu$ on the box of radius $B_{\max}:=B_{\lceil\log_2r\rceil}$, so it adds at most $\mu$ to the error, thus $e_{l+1} \le 2 B_l e_l+\mu$. Using induction starting with $e_0=0$ gives,
$$ e_l\le C_1(B,r)\mu,   \qquad
C_1(B,r):=\sum_{l=0}^{\lceil \log_2r \rceil-1}\ \prod_{m=l+1}^{\lceil\log_2r\rceil-1} 2 B_m,$$
for every level $l\le\lceil\log_2r\rceil$ and since every term in $C_1$ is at least $1$. The bound $|\widetilde P_l|\le B_l$ is valid by induction on $l$ since it holds for $l=0$ and if $e_l \le 1$ then $|\widetilde P_l| \le |P_l| + e_l \le B^{2^l}+1=B_l$. Setting $C_1\mu\le1$, implies that every $e_l\le1$ and every step is evaluated inside each of the valid boxes. The network $\widetilde P_\mu$ has depth at most $\lceil \log_2r \rceil\,C(B_{\max})\log(\mu^{-1})$. Its size is at most $r-1$ multiplication networks with identity steps of four weights per layer, which are at most $C(B,r)\log(\mu^{-1})$.

Lastly, we denote the coefficient magnitude by $$c_{\text{sum}}:=\sum_{k_1,\ldots,k_r}\|c_{k_1\ldots k_r}\|_{\bR^p},$$ and set $\mu:= \delta/(C_1(B,r) \max\{c_{\text{sum}},1\})$, which is valid for all $\delta\le\delta_0$ since then $C_1\mu\le1$. By putting the layers together in parallel (one tree network per monomial) and including the exact output layer, we get $\widetilde T_\delta$ with $\sup_{y\in\mathcal Q}\|T(y)-\widetilde T_\delta(y)\|_{\bR^p}\le c_{\text{sum}} C_1\mu\le\delta$. Up to constants, the depth is that of a single tree, and the parallel combination multiplies the size by $n^r$. Since $\log(\mu^{-1})\le C\log(\delta^{-1})$ for $\delta\le\delta_0$, it follows that $\dep(\widetilde T_\delta) \le C\log(\delta^{-1})$ and $W(\widetilde T_\delta) \le C\log(\delta^{-1})$, where $C$ depends only on $T$, $r$, $n$, $p$ and $b$.
\end{proof}

\end{document}